\documentclass[12pt]{article}
\usepackage{graphicx}
\usepackage{amsmath,amsthm,amssymb,enumerate}%, esint}
\usepackage{euscript,mathrsfs}
\usepackage{color}
\usepackage{dsfont}
\usepackage[citecolor=blue,colorlinks=true]{hyperref}
\usepackage[left=2cm,right=2cm,top=3.5cm,bottom=3.5cm]{geometry}
\usepackage{color}
\usepackage[framemethod=tikz]{mdframed}
\allowdisplaybreaks

\usepackage{multirow}
\usepackage{subcaption}

\usepackage{bm}
\usepackage{bbm}
\usepackage{soul}

\catcode`\@=11 \@addtoreset{equation}{section}

\catcode`\@=12

\allowdisplaybreaks

\newtheorem{Theorem}{Theorem}[section]
\newtheorem{Proposition}[Theorem]{Proposition}
\newtheorem{Lemma}[Theorem]{Lemma}
\newtheorem{Corollary}[Theorem]{Corollary}

\theoremstyle{definition}
\newtheorem{Definition}[Theorem]{Definition}

\newtheorem{Remark}[Theorem]{Remark}

\newcommand{\bTheorem}[1]{
\begin{Theorem} \label{T#1} }
\newcommand{\eT}{\end{Theorem}}

\newcommand{\bProposition}[1]{
\begin{Proposition} \label{P#1}}
\newcommand{\eP}{\end{Proposition}}

\newcommand{\bLemma}[1]{
\begin{Lemma} \label{L#1} }
\newcommand{\eL}{\end{Lemma}}

\newcommand{\bCorollary}[1]{
\begin{Corollary} \label{C#1} }
\newcommand{\eC}{\end{Corollary}}

\newcommand{\bRemark}[1]{
\begin{Remark} \label{R#1} }
\newcommand{\eR}{\end{Remark}}

\newcommand{\bDefinition}[1]{
\begin{Definition} \label{D#1} }
\newcommand{\eD}{\end{Definition}}

\newcommand{\Del}{\Delta_x}

\newcommand{\Ds}{\mathbb{D}_x}

\newcommand{\bfphi}{\boldsymbol{\varphi}}

\newcommand{\bFormula}[1]{
\begin{equation} \label{#1}}
\newcommand{\eF}{\end{equation}}

\newcommand{\Ov}[1]{\overline{#1}}

\newcommand{\vc}[1]{{\bm #1}}

\newcommand{\Divh}{{\rm div}_h}
\newcommand{\Gradh}{\nabla_h}

\newcommand{\aleq}{\stackrel{<}{\sim}}

\DeclareMathOperator{\divv}{div}

\newcommand{\vr}{\varrho}
\newcommand{\vre}{\vr_\ep}

\newcommand{\vte}{\vt_\ep}
\newcommand{\vue}{\vu_\ep}
\newcommand{\tvr}{\tilde \vr}

\newcommand{\tvt}{\tilde \vt}

\newcommand{\vt}{\vartheta}
\newcommand{\vu}{\vc{u}}
\newcommand{\vm}{\vc{m}}

\newcommand{\vn}{\vc{n}}

\newcommand{\Div}{{\rm div}_x}
\newcommand{\Grad}{\nabla_x}

\newcommand{\dx}{\,{\rm d} {x}}

\newcommand{\dt}{\,{\rm d} t }

\newcommand{\intO}[1]{\int_{\Omega} #1 \ \dx}

\newcommand{\intTd}[1]{\int_{\mathbb{T}^d} #1 \, \dx}

\newcommand{\D}{{\rm d}}

\newcommand{\ep}{\varepsilon}

\newcommand{\vtB}{\vt_B}

\newcommand{\br}{ \nonumber \\ }

\def\softd{{\leavevmode\setbox1=\hbox{d}%
          \hbox to 1.05\wd1{d\kern-0.4ex{\char039}\hss}}}
\definecolor{Cgrey}{rgb}{0.85,0.85,0.85}
\definecolor{Cblue}{rgb}{0.50,0.85,0.85}
\definecolor{Cred}{rgb}{1,0,0}
\definecolor{fancy}{rgb}{0.10,0.85,0.10}

\newcommand\Cbox[2]{%
    \newbox\contentbox%
    \newbox\bkgdbox%
    \setbox\contentbox\hbox to \hsize{%
        \vtop{
            \kern\columnsep
            \hbox to \hsize{%
                \kern\columnsep%
                \advance\hsize by -2\columnsep%
                \setlength{\textwidth}{\hsize}%
                \vbox{
                    \parskip=\baselineskip
                    \parindent=0bp
                    #2
                }%
                \kern\columnsep%
            }%
            \kern\columnsep%
        }%
    }%
    \setbox\bkgdbox\vbox{
        \color{#1}
        \hrule width  \wd\contentbox %
               height \ht\contentbox %
               depth  \dp\contentbox
        \color{black}
    }%
    \wd\bkgdbox=0bp%
    \vbox{\hbox to \hsize{\box\bkgdbox\box\contentbox}}%
    \vskip\baselineskip%
}

\mdfdefinestyle{MyFrame}{%
	linecolor=black,
	outerlinewidth=1pt,
	roundcorner=5pt,
	innertopmargin=\baselineskip,
	innerbottommargin=\baselineskip,
	innerrightmargin=10pt,
	innerleftmargin=10pt,
	backgroundcolor=white!20!white}

\date{}

\begin{document}

%%%%%%%%%%%%%%%%%%%%%%%%%%%%%%%%

\title{Penalization method for the Navier--Stokes--Fourier system}

\author{Danica Basari\' c\,\thanks{The work of D.B., E.F. and H.M.~was partially supported by the
		Czech Sciences Foundation (GA\v CR), Grant Agreement
		21--02411S. The Institute of Mathematics of the Academy of Sciences of
		the Czech Republic is supported by RVO:67985840. \hspace*{1em} $^\spadesuit$M.L. has been funded by the Deutsche Forschungsgemeinschaft (DFG, German Research Foundation) - Project number 233630050 - TRR 146 as well as by  TRR 165 Waves to Weather. She is grateful to the Gutenberg Research College and  Mainz Institute of Multiscale Modelling for supporting her research.
The research of Y.Y. was funded by Sino-German (CSC-DAAD) Postdoc Scholarship Program in 2020 - Project
number 57531629.} \and Eduard Feireisl\,$^{*}$
\and M\' aria Luk\' a\v cov\' a -- Medvi\softd ov\' a$^{\spadesuit}$\and Hana Mizerov\' a\,$^{*,\clubsuit}$ \and Yuhuan Yuan$^{\spadesuit}$
}

\date{}

\maketitle

\medskip

\centerline{$^{*}$ Institute of Mathematics of the Academy of Sciences of the Czech Republic}
\centerline{\v Zitn\' a 25, CZ-115 67 Praha 1, Czech Republic}
\centerline{basaric@math.cas.cz, feireisl@math.cas.cz, hana.mizerova@fmph.uniba.sk}

\bigskip

\centerline{$^\spadesuit$ Institute of Mathematics, Johannes Gutenberg-University Mainz}
\centerline{Staudingerweg 9, 55 128 Mainz, Germany}
\centerline{lukacova@uni-mainz.de, yuhuyuan@uni-mainz.de}

\bigskip
\centerline{$^\clubsuit$ Department of Mathematical Analysis and Numerical Mathematics}
\centerline{Faculty of Mathematics, Physics and Informatics of the Comenius University}
\centerline{Mlynsk\' a dolina, 842 48 Bratislava, Slovakia}

\begin{abstract}
	
We apply the method of penalization to the Dirichlet problem for the Navier--Stokes--Fourier system governing the motion of a general viscous compressible fluid
confined to a bounded Lipschitz domain. The physical domain is embedded into a large cube on which the periodic boundary  conditions are imposed. The original boundary conditions are enforced through
a singular friction term in the momentum equation and a heat source/sink term in the internal energy balance. The solutions of the penalized problem are shown to converge to the solution of the limit problem.
Numerical experiments are performed to illustrate the efficiency of the method.

\end{abstract}

%\bigskip


%{\bf MSC:}
%\bigskip

\tableofcontents

\section{Introduction}
\label{i}

	Let us consider the \emph{Navier--Stokes--Fourier system} in the entropy formulation,
\begin{align}
	\partial_t \varrho +\Div (\varrho \vu)&=0, \label{i1}\\
	\partial_t (\varrho \vu) + \Div (\varrho \vu \otimes \vu)+ \nabla_x p&= \Div \mathbb{S}, \label{i2}\\
	\partial_t (\varrho s ) + \Div (\varrho s \vu) + \Div \left( \frac{\bm{q}}{\vartheta} \right) &= \frac{1}{\vartheta} \left( \mathbb{S}: \mathbb{D}_x \vu - \frac{\bm{q} \cdot \Grad \vartheta}{\vartheta} \right) \label{i3}.
\end{align}
The unknowns are the standard variables: the density $\vr = \vr(t,x)$, the temperature $\vt = \vt(t,x)$, and the velocity $\vu = \vu(t,x)$, whereas the thermodynamic functions: the pressure
$p = p(\vr, \vt)$, the entropy $s = s(\vr, \vt)$ as well as the viscous stress tensor $\mathbb{S} = \mathbb{S}(\vt, \mathbb{D}_x \vu)$, and the heat flux
$\vc{q} = \vc{q}(\vt, \Grad \vt)$ are determined through suitable constitutive relations.

The fluid is confined to a bounded domain $\Omega \subset R^d$, $d=2,3$, on the boundary of which the Dirichlet boundary conditions
\begin{equation} \label{i4}
	\vu|_{\partial \Omega} = 0,\ \vt|_{\partial \Omega} = \vtB,\ \vtB = \vtB(t,x)
\end{equation}
are imposed. Our goal is to approximate solutions of problem \eqref{i1}--\eqref{i4} via penalization of the spatial domain. Specifically, we suppose
\[
\Omega \subset  \mathbb{T}^d,
\]
where $\mathbb{T}^d$ is sufficiently large ``flat'' torus, and replace the field equations \eqref{i1}--\eqref{i3} by the penalized system
\begin{align}
	\partial_t \varrho +\Div (\varrho \vu)&=0, \label{i5}\\
	\partial_t (\varrho \vu) + \Div (\varrho \vu \otimes \vu)+ \nabla_x p&= \Div \mathbb{S} - \frac{1}{\ep} \mathds{1}_{\mathbb{T}^d \setminus \Omega} \vu, \label{i6}\\
	\partial_t (\varrho s ) + \Div (\varrho s \vu) + \Div \left( \frac{\bm{q}}{\vartheta} \right) &= \frac{1}{\vartheta} \left( \mathbb{S}: \mathbb{D}_x \vu - \frac{\bm{q} \cdot \Grad \vartheta}{\vartheta} - \frac{1}{\ep} \mathds{1}_{\mathbb{T}^d \setminus \Omega} |\vt - \vtB|^k (\vt - \vtB) \right) \label{i7}
\end{align}
on the set $(0,T) \times \mathbb{T}^d$, where $\vtB$ is a smooth extension of the boundary temperature on $\mathbb{T}^d \setminus \Ov{\Omega}$. Obviously, solving the problem on the flat torus
$\mathbb{T}^d$ is equivalent to imposing the space periodic boundary conditions. The solution of the original problem is then recovered by letting $\ep \to 0$ in \eqref{i6}, \eqref{i7}.

The penalization method is a popular simulation tool when the boundary of $\Omega$ has a complicated structure and its approximation by polygons may be problematic. The difficulty with domain approximation and construction of a suitable mesh is transformed
to the forcing terms that are much easier to handle.
This idea has been used quite often in the literature. Domain penalization is realized in the immersed boundary method \cite{Peskin:1972,Peskin:2002} and the (Lagrange-multiplier based) fictitious domain method \cite{Glowinski-Pan-Periaux:1994,Glowinski-Pan-Periaux:1994a,Hyman:1952}. Both approaches have been originally developed in the context of
 incompressible Navier--Stokes equations.
In the context of  fluid-structure interaction problem a  penalization method is applied on a moving domain
in  \cite{Bruneau:2018}. Penalization of boundary conditions in a spectral method approximating one- and multidimensional compressible
Navier--Stokes--Fourier equations was discussed in \cite{Hesthaven1, Hesthaven2}.
 Related numerical analysis
 for one-dimensional heat equation with a singular forcing term was presented in \cite{Beyer-LeVeque:1992}.
 For elliptic boundary problems the error estimates between the exact solution and (numerical) solutions of $L^2$- or $H^1$ penalization problems were presented in \cite{Maury:2009, Saito-Zhou:2015,Zhang:2006,Zhou-Saito:2014}.
 %For elliptic boundary problems the error estimates between the exact solution and solutions of $L^2$- or $H^1$-penalty problem, also corresponding numerical solution were presented in \cite{Maury:2009, Saito-Zhou:2015,Zhang:2006,Zhou-Saito:2014}.

Even at the purely theoretical level, penalization can be useful for problems with low regularity (Lipschitz) of the boundary, where suitable
approximation by regularization is hampered by the absence of smooth approximate solutions. Our goal in this paper is twofold:

\begin{itemize}
	\item Using the framework of \emph{weak solutions}, developed in \cite{FeNov6A} for the penalized problem and, more recently, in \cite{ChauFei} for the original Dirichlet problem, we show that
	weak solutions of \eqref{i5}--\eqref{i7} converge to a weak solution of \eqref{i1}--\eqref{i4} as $\ep \to 0$.
	
	\item We perform numerical experiments illustrating the abstract results.
	
	\end{itemize}

As far as we are aware, our paper is the first one to provide analytical results as well as finite volume simulations on complex domains for the penalization method applied to the multidimensional Navier--Stokes--Fourier system.
%As far as we are aware our paper presents the first result on applications of the penalization method for multidimensional Navier--Stokes--Fourier system, both from analytical point of view as well as for finite volume simulations on complex domains.

The rest of the paper is organized in the following way: In Section~\ref{M} we define the concept of weak solution to the Dirichlet problem for the Navier--Stokes--Fourier system and formulate the main theoretical result on the strong convergence of the penalized solutions. Sections~\ref{U} and \ref{C} are devoted to the derivation of the uniform bounds using the ballistic energy inequality and to the convergence analysis of the penalized solutions, respectively. Section~\ref{numerics} presents a series of numerical simulations illustrating robustness and efficiency of the proposed penalization strategy when solving the Dirichlet problem for the Navier--Stokes--Fourier system in complex domains.

\section{Constitutive equations, weak solutions, main theoretical result}
\label{M}

Before stating our main analytical result, let us introduce the basic hypotheses imposed on the physical domain and constitutive equations.

\subsection{Physical domain}

We suppose that $\Omega \subset R^d$, $d=2,3$ is a bounded domain with Lipschitz boundary. In addition, we suppose that the boundary datum $\vtB$ can be extended on $[0,T] \times \mathbb{T}^d$ in the following way:
{%\cred
\begin{align}
&	\vtB \in W^{1, \infty}([0,T] \times \mathbb{T}^d), \ \vtB \in  W^{2, \infty}([0,T] \times \Omega) \cap W^{2, \infty}([0,T] \times (\mathbb{T}^d \setminus \Omega)), \br
&	\inf \vtB > 0,\ \Del \vtB(t, \cdot) = 0 \ \ \mbox{ a.a. in} \ [0,T] \times \Omega.
	\label{M1}
	\end{align}}
%\begin{equation} \label{M1}
%	\vtB \in W^{2, \infty}([0,T] \times \Ov{\Omega}),\ \inf \vtB > 0,\ \Del \vtB(t, \cdot) = 0 \ \mbox{a.a. in} \ [0,T] \times \Omega.
%	\end{equation}
Note that such an extension \emph{always} exists as long as $\Omega$ as well as the boundary datum are smooth of class at least $C^{2 + \nu}$. For less regular domains, condition \eqref{M1} must be imposed as
a hypothesis.

\subsection{Constitutive equations}\label{ss:ce}

The equations of state interrelating the thermodynamic functions $p$, $s$, to the internal energy $e$ are motivated by the existence theory developed in \cite{ChauFei}. Specifically, we suppose:
\begin{align}
	p(\varrho, \vartheta) =p_{\textup{m}}(\varrho,\vartheta)+ p_{\textup{rad}}(\vartheta), \quad &\mbox{with} \quad p_{\textup{m}}(\varrho, \vartheta)= \vartheta^{\frac{5}{2}} P\left( \frac{\varrho}{\vartheta^{\frac{3}{2}}} \right), \ p_{\textup{rad}}(\vartheta)=  \frac{a}{3} \vartheta^4, \label{M2}\\
	e(\varrho, \vartheta) =e_{\textup{m}}(\varrho,\vartheta)+ e_{\textup{rad}}(\varrho, \vartheta), \quad &\mbox{with} \quad e_{\textup{m}}(\varrho, \vartheta)= \frac{3}{2}\frac{\vartheta^{\frac{5}{2}}}{\varrho} P\left( \frac{\varrho}{\vartheta^{\frac{3}{2}}} \right), \ e_{\textup{rad}}(\varrho, \vartheta)=   \frac{a}{\varrho} \vartheta^4, \label{M3}\\
	s(\varrho, \vartheta) =s_{\textup{m}}(\varrho,\vartheta)+ s_{\textup{rad}}(\varrho, \vartheta), \quad &\mbox{with} \quad s_{\textup{m}}(\varrho, \vartheta)= \mathcal{S} \left( \frac{\varrho}{\vartheta^{\frac{3}{2}}} \right), \ s_{\textup{rad}}(\varrho, \vartheta)=   \frac{4a}{3} \frac{\vartheta^3}{\varrho}, \label{M4}
\end{align}
where $a>0$, $P \in C^1[0,\infty)$ satisfies
\begin{equation} \label{M5}
	P(0)=0, \quad P'(Z)>0 \mbox{ for } Z \geq 0, \quad 0< \frac{\frac{5}{3} P(Z)-P'(Z)Z}{Z} \leq c \mbox{ for } Z \geq 0
\end{equation}
and
\begin{equation} \label{M6}
	\mathcal{S}'(Z) = - \frac{3}{2} \frac{\frac{5}{3} P(Z)-P'(Z)Z}{Z^2}.
\end{equation}
It follows from \eqref{M5} and \eqref{M6} that the functions
\[
Z \mapsto \frac{P(Z)}{Z^{\frac{5}{3}}} \ \mbox{and}\ Z \mapsto \mathcal{S}(Z)
\]
are decreasing, and we assume
\begin{align}
	\lim_{Z \rightarrow \infty} \frac{P(Z)}{Z^{\frac{5}{3}}} &= p_{\infty}>0, \label{M7} \\
		\lim_{Z \rightarrow \infty} \mathcal{S}(Z) &= 0.
	\label{M8}
\end{align}

We refer to \cite[Chapter 2]{FeNov6A} for the physical background of the hypotheses \eqref{M2}--\eqref{M8}. In particular, \eqref{M7}, \eqref{M8}
describe the behaviour of the fluid in the degenerate area, where \eqref{M8} is in agreement with the Third law of thermodynamics.

\subsection{Transport terms}

We suppose the fluid is Newtonian (linearly viscous), with the viscous stress
\[
	\mathbb{S}(\vt, \mathbb{D}_x \vu) = \mu(\vt) \left( \Grad \vu + \Grad^t \vu - \frac{2}{d} \Div \vu \,\mathbb{I} \right) + \eta(\vt) \, \Div \vu \,\mathbb{I}.
\]
Here $\mathbb{D}_x \vu=(\Grad\vu+\Grad^t\vu)/2$ stands for the symmetric velocity gradient.

Similarly, the heat flux is given by Fourier's law
\[
\vc{q}(\vt, \Grad \vt) = - \kappa(\vt) \Grad \vt.
\]
As for the transport coefficients $\mu$, $\eta$ and $\kappa$, we suppose they are continuously differentiable functions of temperature $\vartheta$ satisfying
\begin{align}
	0< \underline{\mu} (1+\vartheta) & \leq \mu(\vartheta) \leq \overline{\mu} (1+\vartheta), \quad |\mu'(\vartheta) | \leq c \mbox{ for all } \vartheta \geq 0, \label{M9}\\
	0 & \leq \eta(\vartheta) \leq \overline{\eta} (1+\vartheta), \label{M10} \\
	0< \underline{\kappa} (1+\vartheta^{\beta}) & \leq \kappa(\vartheta) \leq \overline{\kappa} (1+\vartheta^{\beta}), \quad \beta>6. \label{M11}
\end{align}

\subsection{Weak solutions}

The concept of \emph{weak solution} of the penalized system \eqref{i5}--\eqref{i7} was introduced in \cite{FeNov6A}.

	\begin{Definition} [{\bf Weak solution of the penalized problem}] \label{MD1}
	We say that the trio of functions $[\vr, \vt, \vu ]$ is a \emph{weak solution} of the penalized Navier--Stokes--Fourier system \eqref{i5}--\eqref{i7} with the initial data
	\[
	\vr(0, \cdot) = \vr_0,\ (\vr \vu) (0, \cdot) = \vm_0,\ \vr s(\vr, \vt) (0, \cdot) = S_0
	\]	
if the following holds.
	\begin{itemize}
		\item[(i)] \textit{\textit{Weak formulation of the continuity equation}}: The integral identity
		\begin{equation} \label{M12}
			- \int_{\mathbb{T}^d} \varrho_0 \varphi(0,\cdot) \ \textup{d}x  = \int_{0}^{T} \int_{\mathbb{T}^d} [\varrho \partial_t \varphi + \varrho \vu \cdot \nabla_x \varphi] \ \textup{d}x \textup{d}t,
		\end{equation}
		holds for any $\varphi \in C^1_c([0,T) \times \mathbb{T}^d)$.
		\item[(ii)] \textit{Weak formulation of the renormalized continuity equation}: For any function
		\[
			b \in C^1[0, \infty),\ b' \in C_c[0, \infty)
		\]
				the integral identity
		\begin{equation} \label{M13}
			- \int_{\mathbb{T}^d} b(\vr_0) \varphi(0,\cdot) \dx = \int_{0}^{T} \int_{\mathbb{T}^d} [b(\varrho) \partial_t\varphi+ b(\varrho)\vu \cdot \nabla_x\varphi + \varphi \ \Big( b(\varrho) -
			b'(\vr) \vr \Big)\Div \vu ] \ \textup{d}x \textup{d}t
		\end{equation}
		holds for any $\varphi \in C_c^1([0,T) \times \mathbb{T}^d)$.
		\item[(iii)] \textit{\textit{Weak formulation of the momentum equation}}: The integral identity
		\begin{equation} \label{M14}
			\begin{aligned}
				- \int_{\mathbb{T}^d} \vm_0 \cdot \bfphi (0,\cdot) \dx &= \int_{0}^{T} \int_{\mathbb{T}^d} [\varrho \vu \cdot \partial_t \bfphi + (\varrho \vu \otimes \vu) : \nabla_x \bfphi+ p(\varrho, \vartheta)\Div \bfphi] \ \textup{d}x \textup{d}t \\
				&- \int_{0}^{T} \int_{\mathbb{T}^d} \mathbb{S}(\vt, \mathbb{D}_x \vu): \mathbb{D}_x \bfphi \ \textup{d}x \textup{d}t
				-\frac{1}{\varepsilon} \int_{0}^{T} \int_{\mathbb{T}^d \setminus \Omega} \vu \cdot \bfphi \ \textup{d}x \textup{d}t
			\end{aligned}
		\end{equation}
		holds for any $\bfphi \in C^1_c([0,T) \times \mathbb{T}^d; R^d)$.
		
		\item[(iv)] \textit{\textit{Weak formulation of the entropy inequality}}: The integral inequality
		\begin{equation} \label{M15}
			\begin{aligned}
				-\int_{\mathbb{T}^d} S_0 \ \varphi(0,\cdot) \ \textup{d}x &\geq \int_{0}^{T} \int_{\mathbb{T}^d} \left[\varrho s(\varrho, \vartheta) \big( \partial_t\varphi +  \vu \cdot \nabla_x \varphi \big)+ \frac{\bm{q}(\vartheta, \nabla_x \vartheta)}{\vartheta} \cdot \nabla_x \varphi \right] \textup{d}x \textup{d}t \\
				&+ \int_{0}^{T} \int_{\mathbb{T}^d} \frac{\varphi}{\vartheta} \left( \mathbb{S}(\vt, \mathbb{D}_x \vu): \mathbb{D}_x \vu -\frac{\bm{q}(\vartheta, \nabla_x \vartheta) \cdot \nabla_x \vartheta}{\vartheta}\right)  \ \textup{d}x \textup{d}t \\
				&-\frac{1}{\varepsilon} \int_{0}^{T} \int_{\mathbb{T}^d \setminus \Omega} \frac{\varphi}{\vartheta} \ |\vartheta-\vartheta_B|^k (\vartheta-\vartheta_B)  \ \textup{d}x \textup{d}t
			\end{aligned}
		\end{equation}
		holds for any $\varphi \in C^1_c([0,T) \times \mathbb{T}^d)$, $\varphi \geq 0$.
		\item[(v)] \textit{\textit{Total energy balance}}:
		The integral inequality
		\begin{equation} \label{M16}
			\begin{aligned}
				&\psi (\tau)\int_{\mathbb{T}^d}\left( \frac{1}{2} \varrho |\vu|^2 +\varrho 	e(\varrho,\vartheta)  \right)(\tau,\cdot) \  \textup{d}x  \\
				&\quad- \int_{0}^{\tau} \partial_t \psi \int_{\mathbb{T}^d}\left[ \frac{1}{2} \varrho |\vu|^2 +\varrho e(\varrho,\vartheta) \right] \textup{d}x \textup{d}t \\
				&\quad + \frac{1}{\varepsilon} \int_{0}^{\tau} \psi  \int_{\mathbb{T}^d \setminus \Omega} |\vu|^2 \ \textup{d}x \textup{d}t + \frac{1}{\varepsilon} \int_{0}^{\tau} \psi \int_{\mathbb{T}^d \setminus \Omega} |\vartheta -\vartheta_B|^k (\vartheta -\vartheta_B) \ \textup{d}x \textup{d}t\\
				& \quad   \leq \psi (0)\int_{\mathbb{T}^d}\left( \frac{1}{2} \frac{|\vm_0|^2}{\vr_0} + \vr_0 	e(\vr_0,S_0)  \right) \  \textup{d}x
			\end{aligned}
		\end{equation}
		holds for a.e. $\tau \in (0,T)$ and any $\psi \in C^1[0,T]$, $\psi \geq 0$.
	\end{itemize}
\end{Definition}

A suitable concept of a weak solution for the system \eqref{i1}--\eqref{i3} endowed with the Dirichlet boundary conditions \eqref{i4} has been developed only recently in \cite{ChauFei}.

	\begin{Definition} [{\bf Weak solution of the Dirichlet problem}] \label{MD2}
	We say that the trio of functions $[\vr, \vt, \vu ]$ is a \emph{weak solution} of the Navier--Stokes--Fourier system \eqref{i1}--\eqref{i3}, with the Dirichlet boundary
	conditions \eqref{i4} and the initial data
	\[
	\vr(0, \cdot) = \vr_0,\ (\vr \vu) (0, \cdot) = \vm_0,\ \vr s(\vr, \vt) (0, \cdot) = S_0
	\]	
	if the following holds.
	\begin{itemize}
		\item[(i)] \textit{\textit{Weak formulation of the continuity equation}}: The integral identity
		\begin{equation} \label{M17}
			- \intO{ \varrho_0 \varphi(0,\cdot) }  = \int_{0}^{T} \intO{ [\varrho \partial_t \varphi + \varrho \vu \cdot \nabla_x \varphi] } \textup{d}t,
		\end{equation}
		holds for any $\varphi \in C^1_c([0,T) \times \Ov{\Omega})$.
		\item[(ii)] \textit{Weak formulation of the renormalized continuity equation}: For any function
		\[
		b \in C^1[0, \infty),\ b' \in C_c[0, \infty),
		\]
		the integral identity
		\begin{equation} \label{M18}
			- \intO{ b(\vr_0) \varphi(0,\cdot) } = \int_{0}^{T} \intO{ [b(\varrho) \partial_t\varphi+ b(\varrho)\vu \cdot \nabla_x\varphi + \varphi \ \Big( b(\varrho) -
			b'(\vr) \vr \Big)\Div \vu ] } \textup{d}t
		\end{equation}
		holds for any $\varphi \in C_c^1([0,T) \times \Ov{\Omega})$.
		\item[(iii)] \textit{\textit{Weak formulation of the momentum equation}}: The integral identity
		\begin{equation} \label{M19}
			\begin{aligned}
				- \intO{ \vm_0 \cdot \bfphi (0,\cdot) } &= \int_{0}^{T} \intO{ [\varrho \vu \cdot \partial_t \bfphi + (\varrho \vu \otimes \vu) : \nabla_x \bfphi+ p(\varrho, \vartheta)\Div \bfphi] } \textup{d}t \\
				&- \int_{0}^{T} \intO{ \mathbb{S}(\vt, \mathbb{D}_x \vu): \mathbb{D}_x \bfphi } \textup{d}t
				\end{aligned}
		\end{equation}
		holds for any $\bfphi \in C^1_c([0,T) \times \Omega; R^d)$.
		\item[(iv)] \textit{\textit{Weak formulation of the entropy inequality}}: The integral inequality
		\begin{equation} \label{M20}
			\begin{aligned}
				-\intO{ S_0 \ \varphi(0,\cdot) } &\geq \int_{0}^{T} \intO{ \left[\varrho s(\varrho, \vartheta) \big( \partial_t\varphi +  \vu \cdot \nabla_x \varphi \big)+ \frac{\bm{q}(\vartheta, \nabla_x \vartheta)}{\vartheta} \cdot \nabla_x \varphi \right] } \textup{d}t \\
				&+ \int_{0}^{T} \intO{ \frac{\varphi}{\vartheta} \left( \mathbb{S}(\vt, \mathbb{D}_x \vu): \mathbb{D}_x \vu -\frac{\bm{q}(\vartheta, \nabla_x \vartheta) \cdot \nabla_x \vartheta}{\vartheta}\right)  } \textup{d}t
			\end{aligned}
		\end{equation}
		holds for any $\varphi \in C^1_c([0,T) \times \Omega)$, $\varphi \geq 0$.
		\item[(v)] \textit{\textit{Ballistic energy balance}}:
		For any
		\begin{equation} \label{M21}
		\tvt \in C^1([0,T] \times \Ov{\Omega}),\ \inf \tvt > 0,\ \tvt|_{\partial \Omega} = \vtB
		\end{equation}
	the integral inequality
		\begin{equation} \label{M22}
			\begin{aligned}
				&\intO{ \left( \frac{1}{2} \varrho |\vu|^2 +\varrho 	e(\varrho,\vartheta) - \tvt \vr s(\vr, \vt) \right)(\tau,\cdot) }  \\
				&\quad + \int_{0}^{\tau} \intO{ \frac{\tvt}{\vartheta} \left( \mathbb{S}(\vt, \mathbb{D}_x \vu): \mathbb{D}_x \vu -\frac{\bm{q}(\vartheta, \nabla_x \vartheta) \cdot \nabla_x \vartheta}{\vartheta}\right)  } \textup{d}t\\
				& \quad   \leq \intO{ \left( \frac{1}{2} \frac{|\vm_0|^2}{\vr_0} + \vr_0 	e(\vr_0,S_0) - \tvt(0, \cdot) S_0  \right) } \\
				&\quad- \int_0^\tau \intO{ \left[ \vr s(\vr, \vt) \Big( \partial_t \tvt + \vu \cdot
					\Grad \tvt \Big) + \frac{\vc{q} (\vt, \Grad \vt) \cdot \Grad \tvt }{\vt} \right] } \dt
			\end{aligned}
		\end{equation}
		holds for a.e. $\tau \in (0,T)$.
	\end{itemize}
\end{Definition}

Apparently, the main difference between the two concepts of weak solutions is the total, ballistic energy balance \eqref{M16}, \eqref{M22}, respectively. In addition, the pointwise inequality
\eqref{M22} is weaker than its counterpart \eqref{M16} stated in the differential form. Still \eqref{M22} is sufficient for showing the weak--strong uniqueness property, see \cite{ChauFei}.

\subsection{Main result}

Having collected the necessary material, we are ready to state the main theoretical result of the paper.

\begin{Theorem}[\bf Convergence of the penalization method] \label{MT1}
	
	Let $\Omega \subset R^d$, $d=2,3$, be a bounded Lipschitz domain. Suppose that the boundary function $\vtB$ admits the extension \eqref{M1}. Let the thermodynamic functions $p$, $e$, $s$ as well as
	the transport coefficients $\mu$, $\eta$, $\kappa$ satisfy the hypotheses \eqref{M2}--\eqref{M11}. Consider a family of measurable initial data
	\[
	\vr_{0,\ep} > 0,\ \vm_{0, \ep},\ S_{0,\ep}
	\]
	defined on $\mathbb{T}^d$
	and satisfying
	\begin{align}
		\vr_{0,\ep} &\to \vr_0 \ \mbox{in}\ L^1(\mathbb{T}^d), \br
		\vm_{0, \ep} &\to \vm_0 \ \mbox{weakly in} \ L^q(\mathbb{T}^d; R^d) \ \mbox{for some}\ q > 1, \br
		S_{0, \ep} &\to S_0 \ \mbox{weakly in} \
		 L^q(\mathbb{T}^d) \ \mbox{for some}\ q > 1,\ \br
		\intTd{ \left( \frac{1}{2} \frac{|\vm_{0,\ep}|^2}{\vr_{0,\ep}} + \vr_{0, \ep} e(\vr_{0,\ep}, S_{0, \ep}) \right) } &\to
		\intTd{ \left( \frac{1}{2} \frac{|\vm_{0}|^2}{\vr_{0}} + \vr_{0}e(\vr_0, S_{0}) \right) }
		\label{M23}
		\end{align}
	as $\ep \to 0$, where
	\begin{itemize}
	
	\item
		\begin{equation} \label{M24bis}
			\vr_0 \geq 0,\ \vm_0(x) = 0,\ S_0 (x) = \vr_0(x) s (\vr_0(x), \vtB(x))  \ \mbox{for any}\ x \in \mathbb{T}^d \setminus \Ov{\Omega}
		\end{equation}
	if $\vtB = \vtB(x)$ is independent of $t$;
	
	\item
	\begin{equation} \label{M24}
			\vr_0 (x) = 0,\ \vm_0(x) = 0,\ S_0 (x) = \vr_0(x) s (\vr_0(x), \vtB(0, x)) = \frac{4a}{3} \vtB^3(0, x) \ \mbox{for any}\ x \in \mathbb{T}^d \setminus \Ov{\Omega}
		\end{equation}
	if $\vtB = \vtB(t,x)$.

\end{itemize}	
%	\begin{equation} \label{M24}
%		\vr_0 (x) = 0,\ \vm_0(x) = 0,\ S_0 (x) = \vr_0 (x) s (\vr_0(x), \vtB(0,x)) = \frac{4a}{3} \vtB^3(0, x) \ \mbox{for any}\ x \in \mathbb{T}^d \setminus \Ov{\Omega}.
%	\end{equation}
	Let $(\vre, \vte, \vue)_{\ep > 0}$ be the corresponding family of weak solutions to the penalized problem specified in Definition \ref{MD1}, with the parameter $k >  \beta - 1$.
	
	Then, up to a suitable subsequence,
	\begin{align}
		\vre &\to \vr \ \mbox{in}\ C_{\rm weak}([0,T]; L^{\frac{5}{3}}(\Omega)) \ \mbox{and (strongly) in}\ L^1((0,T) \times \Omega), \br
		\vue &\to \vu \ \mbox{weakly in}\ L^2(0,T; W^{1,2}(\Omega; R^d)),\ \br
		\vte &\to \vt \ \mbox{weakly in}\  L^2(0,T; W^{1,2}(\Omega)) \ \mbox{and strongly in}\ L^2((0,T) \times \Omega),
		\nonumber
		\end{align}
where $(\vr, \vt, \vu)$ is a weak solution of the Navier--Stokes--Fourier system \eqref{i1}--\eqref{i3}, with the boundary conditions \eqref{i4} and the initial conditions
$(\vr_0, \vm_0, S_0)$ in the sense of Definition \ref{MD2}.

	\end{Theorem}

\begin{Remark} \label{MR1}
	
	The existence theory \cite[Chapter 3]{FeNov6A} can be adapted in a straightforward manner to provide a family of weak solutions $(\vre, \vte, \vue)_{\ep > 0}$ to the penalized problem
	\emph{assumed} in the hypotheses of Theorem \ref{MT1}. Accordingly, Theorem \ref{MT1} represents an alternative proof of existence of a weak solution for the limit system, cf. \cite{ChauFei}.

\end{Remark}

\begin{Remark} \label{MR2}
	
	As shown in \cite[Section 3.2, formula (3.39)]{FeNov6A},
	\begin{equation} \label{M25}
	0 \leq s_m(\vr, \vt) \aleq 1 + |\log(\vr)| + [\log(\vt)]^+.
	\end{equation}
	Consequently, we may set
	\[
	\vr_0 s(\vr_0,  \vtB(0,\cdot)) = \vr_0 s_m(\vr_0,  \vtB(0,\cdot)) + \frac{4a}{3}  \vtB^3(0,\cdot) = \frac{4a}{3} \vtB^3(0,\cdot) \ \mbox{whenever}\ \vr_0 = 0
	\]
	in agreement with \eqref{M24}
	
	\end{Remark}

\begin{Remark} \label{MR3}

The hypothesis requiring the limit initial density to be zero outside $\Omega$ can be dropped provided $\vtB$ is independent of time, see Section~\ref{nonzero}.

\end{Remark}

The forthcoming two sections are devoted to the proof of Theorem \ref{MT1}.

\section{Uniform bounds}
\label{U}

Our goal is to establish uniform bounds for the sequence of solutions of the penalized problem independent of $\ep \to 0$.

\subsection{Ballistic energy inequality}

The crucial point is rewriting the ballistic energy inequality in terms of the solutions of the penalized problem. To this end, consider
\[
\varphi (t,x) = \psi(t) \tvt(t,x), \ \psi \in C^1_c[0,T),\ \psi \geq 0,
\]
where
\begin{equation} \label{U1}
\tvt \in C^1([0,T] \times \mathbb{T}^d), \ \inf \tvt > 0,\ \tvt|_{\mathbb{T}^d \setminus \Omega} = \vtB
\end{equation}
as a test function in the penalized entropy inequality \eqref{M15}. We get
\[
	\begin{aligned}
		- \psi(0) \int_{\mathbb{T}^d} S_{0,\ep} \ \tvt(0,\cdot) \ \textup{d}x &\geq \int_{0}^{T} \psi \int_{\mathbb{T}^d} \left[\vre s(\vre, \vte) \big( \partial_t\tvt +  \vue \cdot  \Grad \tvt \big)+ \frac{\bm{q}(\vte, \Grad \vte)}{\vte} \cdot \nabla_x \tvt \right] \textup{d}x \textup{d}t \\
		&+ \int_{0}^{T} \psi \int_{\mathbb{T}^d} \frac{\tvt}{\vte} \left( \mathbb{S}(\vte, \mathbb{D}_x \vue): \mathbb{D}_x \vue -\frac{\bm{q}(\vte, \nabla_x \vte) \cdot \Grad \vte}{\vte}\right)  \ \textup{d}x \textup{d}t \\
		&-\frac{1}{\varepsilon} \int_{0}^{T} \psi \int_{\mathbb{T}^d \setminus \Omega} \frac{\vtB}{\vte} \ |\vte -\vartheta_B|^k (\vte -\vartheta_B)  \ \textup{d}x \textup{d}t \\
		&+ \int_0^T \partial_t \psi \int_{\mathbb{T}^d} \vre s(\vre, \vte) \tvt \dx \dt;
	\end{aligned}
\]
whence, after a straightforward manipulation,
\begin{equation} \label{U2}
\begin{aligned}
	- \int_{\mathbb{T}^d} S_{0,\ep} \ \tvt(0,\cdot) \ \textup{d}x &\geq \int_{0}^{\tau} \int_{\mathbb{T}^d} \left[\vre s(\vre, \vte) \big( \partial_t\tvt +  \vue \cdot  \Grad \tvt \big)+ \frac{\bm{q}(\vte, \Grad \vte)}{\vte} \cdot \nabla_x \tvt \right] \textup{d}x \textup{d}t \\
	&+ \int_{0}^{\tau} \int_{\mathbb{T}^d} \frac{\tvt}{\vte} \left( \mathbb{S}(\vte, \mathbb{D}_x \vue): \mathbb{D}_x \vue -\frac{\bm{q}(\vte, \nabla_x \vte) \cdot \Grad \vte}{\vte}\right)  \ \textup{d}x \textup{d}t \\
	&-\frac{1}{\varepsilon} \int_{0}^{\tau}  \int_{\mathbb{T}^d \setminus \Omega} \frac{\vtB}{\vte} \ |\vte -\vartheta_B|^k (\vte -\vartheta_B)  \ \textup{d}x \textup{d}t \\
	&- \int_{\mathbb{T}^d} \vre s(\vre, \vte) \tvt (\tau, \cdot) \dx \  \  \ \mbox{for a.a.}\ \tau \in (0,T).
\end{aligned}
\end{equation}

Finally, we subtract \eqref{U2} from the penalized energy balance \eqref{M16} obtaining
	\begin{equation} \label{U3}
	\begin{aligned}
		&\intTd{ \left( \frac{1}{2} \vre |\vue|^2 +\vre 	e(\vre,\vte) - \tvt \vre s(\vre, \vte) \right)(\tau,\cdot) }  \\
		&\quad + \int_{0}^{\tau} \intTd{ \frac{\tvt}{\vte} \left( \mathbb{S}(\vte, \mathbb{D}_x \vue): \mathbb{D}_x \vue -\frac{\bm{q}(\vte, \nabla_x \vte) \cdot \nabla_x \vte}{\vte}\right)  } \textup{d}t\\
		&\quad + \frac{1}{\varepsilon} \int_{0}^{\tau}  \int_{\mathbb{T}^d \setminus \Omega} \frac{1}{\vte} \ |\vte -\vartheta_B|^{k + 2}   \ \textup{d}x \textup{d}t + \frac{1}{\varepsilon} \int_{0}^{\tau}  \int_{\mathbb{T}^d \setminus \Omega} |\vue |^2   \ \textup{d}x \textup{d}t
		\\
		& \quad   \leq \intTd{ \left( \frac{1}{2} \frac{|\vm_{0,\ep}|^2}{\vr_{0,\ep}} + \vr_{0,\ep} 	e(\vr_{0,\ep} ,S_{0,\ep}) - \tvt(0, \cdot) S_{0,\ep}  \right) } \\
		&\quad- \int_0^\tau \intTd{ \left[ \vre s(\vre, \vte) \Big( \partial_t \tvt + \vue \cdot
			\Grad \tvt \Big) + \frac{\vc{q} (\vte, \Grad \vte) \cdot \Grad \tvt }{\vte} \right] } \dt
	\end{aligned}
\end{equation}
for a.e. $\tau \in (0,T)$.

\subsection{Mass conservation}

It follows from the equation of continuity \eqref{M12} that the total mass of the fluid is a constant of motion. Specifically, in accordance with hypothesis \eqref{M24},
\begin{equation} \label{U4}
	M_{0,\ep} = \intTd{ \vre (\tau, \cdot) } = \intTd{ \vr_{0,\ep} } \to \intO{ \vr_0 } = M_0 \ \mbox{as}\ \ep \to 0.
	\end{equation}

\subsection{Energy estimates}

In accordance with hypothesis \eqref{M1}, we may consider $\tvt = \vtB$, where $\vtB$ is the extension of the boundary temperature specified in \eqref{M1}, as a ``test'' function in the ballistic energy
inequality \eqref{U3}.
{%\cred
Strictly speaking, \eqref{U3} was originally derived for $C^1$ functions, however, extension of its validity to Lipschitz test functions is straightforward.}
We obtain
	\begin{equation} \label{U5}
	\begin{aligned}
		&\intTd{ \left( \frac{1}{2} \vre |\vue|^2 +\vre 	e(\vre,\vte) - \vtB \vre s(\vre, \vte) \right)(\tau,\cdot) }  \\
		&\quad + \int_{0}^{\tau} \intTd{ \frac{\vtB}{\vte} \left( \mathbb{S}(\vte, \mathbb{D}_x \vue): \mathbb{D}_x \vue -\frac{\bm{q}(\vte, \nabla_x \vte) \cdot \nabla_x \vte}{\vte}\right)  } \textup{d}t\\
		&\quad + \frac{1}{\varepsilon} \int_{0}^{\tau}  \int_{\mathbb{T}^d \setminus \Omega} \frac{1}{\vte} \ |\vte -\vartheta_B|^{k + 2}   \ \textup{d}x \textup{d}t + \frac{1}{\varepsilon} \int_{0}^{\tau}  \int_{\mathbb{T}^d \setminus \Omega} |\vue |^2   \ \textup{d}x \textup{d}t
		\\
		& \quad   \leq \intTd{ \left( \frac{1}{2} \frac{|\vm_{0,\ep}|^2}{\vr_{0,\ep}} + \vr_{0,\ep} 	e(\vr_{0,\ep} ,S_{0,\ep}) - \vtB(0, \cdot) S_{0,\ep}  \right) } \\
		&\quad- \int_0^\tau \intTd{ \left[ \vre s(\vre, \vte) \Big( \partial_t \vtB + \vue \cdot
			\Grad \vtB \Big) + \frac{\vc{q} (\vte, \Grad \vte) \cdot \Grad \vtB }{\vte} \right] } \dt
	\end{aligned}
\end{equation}
for a.e. $\tau \in (0,T)$.

It follows from the hypotheses \eqref{M9}, \eqref{M10} and the Korn--Poincar\' e inequality (see e.g. \cite[Proposition 2.1]{FeNov6A}) that
\begin{equation} \label{U6}
	\| \vue \|^2_{W^{1,2}(\mathbb{T}^d; R^d)} \aleq
\intTd{ \frac{\vtB}{\vte} \mathbb{S}(\vte, \mathbb{D}_x \vue): \mathbb{D}_x \vue } + \frac{1}{\ep} \int_{\mathbb{T}^d \setminus \Omega} |\vue |^2   \ \textup{d}x.
	\end{equation}
Similarly, in accordance with hypothesis \eqref{M11},
\begin{align}
\| \vte^{\frac{\beta}{2}} \|_{W^{1,2}(\mathbb{T}^d)}^2 &+
\| \log( \vte) \|_{W^{1,2}(\mathbb{T}^d)}^2 \br &\leq c(\vtB) \left[ 1  -  \intTd{\vtB \frac{\bm{q}(\vte, \nabla_x \vte) \cdot \nabla_x \vte}{\vte^2} } +
		\frac{1}{\ep} \int_{\mathbb{T}^d \setminus \Omega} \frac{1}{\vte} \ |\vte -\vartheta_B|^{k + 2} \right].
		\label{U7}
	\end{align}

In view of hypothesis \eqref{M23}, the energy of the initial data is uniformly bounded, and we may use \eqref{U6}, \eqref{U7} to reduce the inequality \eqref{U5} to
	\begin{equation} \label{U8}
	\begin{aligned}
		&\intTd{ \left( \frac{1}{2} \vre |\vue|^2 +\vre 	e(\vre,\vte) - \vtB \vre s(\vre, \vte) \right)(\tau,\cdot) }  \\
		&\quad + \int_0^\tau \left( 	\| \vue \|^2_{W^{1,2}(\mathbb{T}^d; R^d)} + \| \vte^{\frac{\beta}{2}} \|_{W^{1,2}(\mathbb{T}^d)}^2 +
		\| \log( \vte) \|_{W^{1,2}(\mathbb{T}^d)}^2  \right) \dt \\
		&\quad + \frac{1}{\varepsilon} \int_{0}^{\tau}  \int_{\mathbb{T}^d \setminus \Omega} \frac{1}{\vte} \ |\vte -\vartheta_B|^{k + 2}   \ \textup{d}x \textup{d}t + \frac{1}{\varepsilon} \int_{0}^{\tau}  \int_{\mathbb{T}^d \setminus \Omega} |\vue |^2   \ \textup{d}x \textup{d}t
		\\
		& \quad   \aleq \left[ 1 +
		\left|  \int_0^\tau \intTd{ \left[ \vre s(\vre, \vte) \Big( \partial_t \vtB + \vue \cdot
			\Grad \vtB \Big) + \frac{\vc{q} (\vte, \Grad \vte) \cdot \Grad \vtB }{\vte} \right] } \dt \right| \right]
	\end{aligned}
\end{equation}
for a.e. $\tau \in (0,T)$.

Next, recalling $\Del \vtB = 0$ in $\Omega$ we may integrate
{%\cred
\begin{align}
- \intTd{ \frac{ \vc{q}(\vte, \Grad \vte) \cdot \Grad \vtB }{\vte} } &= \intO{ \frac{ \kappa (\vte) \Grad \vte \cdot \Grad \vtB }{\vte} } +
\int_{\mathbb{T}^d \setminus \Omega} \frac{ \kappa (\vte) \Grad \vte \cdot \Grad \vtB }{\vte} \dx \br &=
\int_{\partial \Omega} K(\vte) [\Grad \vtB \cdot \vc{n} \pm ] \ \D \sigma_x -
\int_{\mathbb{T}^d \setminus \Omega} { K (\vte) \Del \vtB } \dx,
\nonumber
\end{align}
where we introduced a function $K$, $K'(\vt) = \frac{\kappa (\vt)}{\vt}$, and where $[\Grad \vtB \cdot \vc{n} \pm ]$ denotes a possible jump
of the normal derivative of $\vtB$ across $\partial \Omega$.
As
\[
|K(\vte)| \aleq \left(1 + |\log(\vte)| + \vte^\beta \right)
\]
the volume integral
\[
\int_{\mathbb{T}^d \setminus \Omega} { K (\vte) \Del \vtB } \dx
\]
is controlled by the left--hand side of \eqref{U8} as soon as $k > \beta - 1$. To control the surface integral, we need the following standard result.

\begin{Lemma} \label{LU1}
Let $Q \subset R^d$ be a bounded Lipschitz domain and $\delta > 0$ arbitrary.
Then there exists $c(\delta)$ such that
\[
\| v \|^2_{L^2(\partial Q)} \leq \delta \| \Grad v \|^2_{L^2(Q; R^d)} + c(\delta) \| v \|^2_{L^2(Q)}
\]
for any $v \in W^{1,2}(Q)$.
\end{Lemma}
	
\begin{proof}	
Assuming the contrary, we get $\delta_0 > 0$ and sequences $(v_n )_{n=1}^\infty \subset W^{1,2}(Q)$, $C_n \to \infty$ such that
	\[
\| v_n \|^2_{L^2(\partial Q)} \geq \delta_0 \| \Grad v_n \|^2_{L^2(Q; R^d)} + C_n \| v_n \|^2_{L^2(Q)}.
\]	
Introducing $w_n = v_n / \| v_n \|_{L^2(\partial Q)}$, we get
\[
\delta_0 \| \Grad w_n \|^2_{L^2(Q; R^d)} + C_n \| w_n \|^2_{L^2(Q)} \leq 1,\ \| w_n \|^2_{L^2(\partial Q)} = 1.
\]
Consequently, by compactness of the trace operator,
\[
w_n \to 0 \ \mbox{weakly in}\ W^{1,2}(Q) \ \mbox{and} \ w_n|_{\partial Q} \to 0 \ \mbox{(strongly) in}\ L^2(\partial Q),
\]
which is a contradiction.
\end{proof}

Thus applying Lemma \ref{LU1} to
\[
v = \vte^{\frac{\beta}{2}},\ Q = \mathbb{T}^d \setminus \Omega
\]
we may infer that the surface integral
\[
\int_{\partial \Omega} K(\vte) [\Grad \vtB \cdot \vc{n} \pm ] \ \D \sigma_x
\]
is controlled by the left--hand side of \eqref{U8}.
}

%\[
%- \intTd{ \frac{ \vc{q}(\vte, \Grad \vte) \cdot \Grad \vtB }{\vte} } = \intTd{ \frac{ \kappa (\vte) \Grad \vte \cdot \Grad \vtB }{\vte} } = -
%\int_{\mathbb{T}^d \setminus \Omega} { K (\vte) \Del \vtB } \dx,
%\]
%where we have set $K'(\vt) = \frac{\kappa (\vt)}{\vt}$. The integral on the right--hand side is controlled by the left--hand side of \eqref{U8} as soon as $k > \beta - 1$.

We conclude by rewriting
\eqref{U8} in the form
	\begin{equation} \label{U9}
	\begin{aligned}
		&\intTd{ \left( \frac{1}{2} \vre |\vue|^2 +\vre 	e(\vre,\vte) - \vtB \vre s(\vre, \vte) \right)(\tau,\cdot) }  \\
		&\quad + \int_0^\tau \left( 	\| \vue \|^2_{W^{1,2}(\mathbb{T}^d; R^d)} + \| \vte^{\frac{\beta}{2}} \|_{W^{1,2}(\mathbb{T}^d)}^2 +
		\| \log( \vte) \|_{W^{1,2}(\mathbb{T}^d)}^2  \right) \dt \\
		&\quad + \frac{1}{\varepsilon} \int_{0}^{\tau}  \int_{\mathbb{T}^d \setminus \Omega} \frac{1}{\vte} \ |\vte -\vartheta_B|^{k + 2}   \ \textup{d}x \textup{d}t + \frac{1}{\varepsilon} \int_{0}^{\tau}  \int_{\mathbb{T}^d \setminus \Omega} |\vue |^2   \ \textup{d}x \textup{d}t
		\\
		& \quad   \aleq \left[ 1 +
		\int_0^\tau \intTd{ \vre s(\vre, \vte) \Big| \partial_t \vtB + \vue \cdot
			\Grad \vtB \Big|  } \dt  \right]
	\end{aligned}
\end{equation}
for a.e. $\tau \in (0,T)$.

Finally, in accordance with hypothesis \eqref{M4} and by virtue of \eqref{M25},
\begin{equation} \label{U10}
\vre s(\vre, \vte) |\vue|  \ = \  \vre \mathcal{S} \left( \frac{\vre}{\vte^{\frac{3}{2}}} \right) |\vue| + \frac{4a}{3} \vte^3 |\vue| \leq
\vre \mathcal{S} \left( \frac{\vre}{\vte^{\frac{3}{2}}} \right) |\vue| + \delta |\vue|^2 + c(\delta) \vte^6
	\end{equation}
for any $\delta > 0$. Moreover, in accordance with the hypotheses \eqref{M6}, \eqref{M8},
\begin{align}
0 \leq \vre \mathcal{S} \left( \frac{\vre}{\vte^{\frac{3}{2}}} \right) &\aleq \left( 1 + \vte^{\frac{3}{2}} \left( 1 + [\log(\vte)]^+ \right) \right) \ \ \ \mbox{if}\ \frac{\vre}{\vte^{\frac{3}{2}}} \leq 1, \br
0 \leq \vre \mathcal{S} \left( \frac{\vre}{\vte^{\frac{3}{2}}} \right) &\aleq \vre \ \ \ \mbox{if}\ \frac{\vre}{\vte^{\frac{3}{2}}} > 1 .
 \nonumber
\end{align}
Thus, going back to \eqref{U9}, we conclude
\begin{equation} \label{U11}
\intTd{ \vre s(\vre, \vte) |\vue| } \aleq \delta \| \vue \|^2_{L^2(\mathbb{T}^d; R^d)} + c(\delta) \| \vte \|_{L^6(\mathbb{T}^d)}^6 +  \intTd{ \vre |\vue|^2 } + M_{0,\ep}
	\end{equation}
for any $\delta > 0$.

The bound \eqref{U11} together with \eqref{U9} yield the desired conclusion
	\begin{equation} \label{U12}
	\begin{aligned}
		&\intTd{ \left( \frac{1}{2} \vre |\vue|^2 +\vre 	e(\vre,\vte) - \vtB \vre s(\vre, \vte) \right)(\tau,\cdot) }  \\
		&\quad + \int_0^\tau \left( 	\| \vue \|^2_{W^{1,2}(\mathbb{T}^d; R^d)} + \| \vte^{\frac{\beta}{2}} \|_{W^{1,2}(\mathbb{T}^d)}^2 +
		\| \log( \vte) \|_{W^{1,2}(\mathbb{T}^d)}^2  \right) \dt \\
		&\quad + \frac{1}{\varepsilon} \int_{0}^{\tau}  \int_{\mathbb{T}^d \setminus \Omega} \frac{1}{\vte} \ |\vte -\vartheta_B|^{k + 2}   \ \textup{d}x \textup{d}t + \frac{1}{\varepsilon} \int_{0}^{\tau}  \int_{\mathbb{T}^d \setminus \Omega} |\vue |^2   \ \textup{d}x \textup{d}t
		\\
		& \quad   \aleq \left[ 1 +
		\int_0^\tau \intTd{\left( \frac{1}{2} \vre |\vue|^2 +\vre 	e(\vre,\vte) \right) } \dt  \right]
	\end{aligned}
\end{equation}
for a.e. $\tau \in (0,T)$. As the entropy is dominated by the energy, we may use Gronwall's argument to conclude
\begin{align}
	{\rm ess} \sup_{\tau \in (0,T)}  \intTd{\left( \frac{1}{2} \vre |\vue|^2 +\vre 	e(\vre,\vte) \right) }  &\aleq 1, \br
\int_0^T \left( 	\| \vue \|^2_{W^{1,2}(\mathbb{T}^d; R^d)} + \| \vte^{\frac{\beta}{2}} \|_{W^{1,2}(\mathbb{T}^d)}^2 +
\| \log( \vte) \|_{W^{1,2}(\mathbb{T}^d)}^2  \right) \dt &\aleq 1,\ \br	
\int_{0}^{T}  \int_{\mathbb{T}^d \setminus \Omega} \frac{1}{\vte} \ |\vte -\vartheta_B|^{k + 2}   \ \textup{d}x \textup{d}t +  \int_{0}^{T}  \int_{\mathbb{T}^d \setminus \Omega} |\vue |^2   \ \textup{d}x \textup{d}t &\aleq \ep
	\label{U13}
	\end{align}
uniformly for $\ep > 0$.

\section{Convergence}
\label{C}

The ultimate step in the proof of Theorem \ref{MT1} is letting $\ep \to 0$ in the sequence of approximate solutions $(\vre, \vte, \vue)_{\ep > 0}$. We claim that this reduces essentially to performing the limit
in the ballistic energy inequality \eqref{U3}. Indeed the weak formulation of the equation of continuity is the same for the penalized and the limit system, while the momentum and the entropy
balance \eqref{M19}, \eqref{M20} do not see the penalization terms in \eqref{M14}, \eqref{M15}, respectively, as the relevant test functions are compactly supported. Consequently, the limit
in the equation of continuity \eqref{M12}, the momentum equation \eqref{M14} as well as the entropy inequality \eqref{M15} can be performed using the known compactness arguments
as in \cite[Chapter 3]{FeNov6A}. Thus our task reduces to:
\begin{itemize}
	\item verifying the renormalized equation of continuity for the limit $(\vr, \vu)$;
	\item performing the limit in the ballistic energy inequality \eqref{U3}.
\end{itemize}

In the following text, we focus on the general case of a time dependent $\vtB = \vtB(t,x)$ as specified in hypothesis \eqref{M24bis}. The necessary modifications to handle hypothesis \eqref{M24} are discussed in Section \ref{nonzero}.

\subsection{Renormalized equation of continuity}

Basically each step in the following arguments requires passing to a suitable subsequence in the family of approximate solutions we will not relabel for simplicity. First, as a consequence of the hypotheses
\eqref{M3}, \eqref{M7},  we have
\[
\vr^{\frac{5}{3}} + \vt^4 \aleq \vr e(\vr,\vt).
\]
Consequently, in view of the uniform bounds \eqref{U13}
\begin{align}
\vre &\to \vr \ \mbox{weakly-(*) in}\ L^\infty(0,T; L^{\frac{5}{3}}(\mathbb{T}^d)),\br
\vte &\to \vt \ \mbox{weakly-(*) in}\ L^\infty(0,T; L^{4}(\mathbb{T}^d)), \br
\vue &\to \vu \ \mbox{weakly in} \ L^2(0,T; W^{1,2}(\mathbb{T}^d; R^d)), \br
\vre \vue &\to \Ov{\vr \vu} \ \mbox{weakly-(*) in}\ L^\infty(0,T; L^{\frac{5}{4}}(\mathbb{T}^d; R^d)).
\label{C1}	
	\end{align}
In addition, as $(\vre, \vue)$ satisfy the equation of continuity \eqref{M12}, we get
\begin{align}
	\vre &\to \vr \ \mbox{in}\ C_{\rm weak}([0,T]; L^{\frac{5}{3}}(\mathbb{T}^d)), \br
	\Ov{\vr \vu} &= \vr \vu.
\label{C2}
\end{align}
In particular, the equation of continuity \eqref{M12} is satisfied in $(0,T) \times \mathbb{T}^d$ by the limit $(\vr, \vu)$.

Next, using \eqref{U13} again
we get
\begin{equation} \label{C3}
	\vue \to 0 \ \mbox{in}\ L^2((0,T) \times (\mathbb{T}^d \setminus \Omega); R^d)
	\end{equation}
yielding
\[
\vu \in L^2(0,T; W^{1,2}_0 (\Omega; R^d)).
\]
Moreover, as $\vr$ satisfies the equation of continuity, we get
\[
\partial_t \vr = 0 \ \mbox{in}\ \mathcal{D}'((0,T) \times (\mathbb{T}^d \setminus \Omega) ).
\]
By virtue of hypothesis \eqref{M24},
\[
\vr(0, \cdot) = \vr_0 = 0 \ \mbox{in}\ \mathbb{T}^d \setminus \Omega,
\]
and we may infer that
\begin{align}
\vr &= 0 \ \mbox{in}\ (0,T) \times (\mathbb{T}^d \setminus \Omega), \br
\vre &\to 0 \ \mbox{in}\ L^q((0,T) \times (\mathbb{T}^d \setminus \Omega)) \ \mbox{for any}\ 1 \leq q < \frac{5}{3}.	
	\label{C4}
	\end{align}

We conclude that $(\vr, \vu) = (0,0)$ satisfy the renormalized equation of continuity in $(0,T) \times (\mathbb{T}^d \setminus \Omega)$, while the known arguments yield the same conclusion in
$(0,T) \times \Omega$. By the same token, we may suppose
\begin{equation} \label{C5}
	\vre \to \vr \ \mbox{(strongly) in}\ L^q((0,T) \times \Omega)) \ \mbox{for any}\ 1 \leq q < \frac{5}{3}.
	\end{equation}

\subsection{Ballistic energy}
\label{ballistic}

Our ultimate goal in the proof of Theorem \ref{MT1} is to perform the limit $\ep \to 0$ in the ballistic energy inequality \eqref{U3}. The first issue to discuss is the strong convergence of the temperature.
In view of the uniform bounds \eqref{U13}, we may suppose
\begin{align}
	\vte &\to \vt \ \mbox{weakly in}\ L^2(0,T; W^{1,2}(\mathbb{T}^d)), \br
 \log(\vte) &\to \Ov{\log(\vt)} \ \mbox{weakly in}\ L^2(0,T; W^{1,2}(\mathbb{T}^d)).
 \label{C6}
 \end{align}
In addition, we have strong (a.a. pointwise) convergence in $\Omega$ by the compactness arguments of \cite[Chapter 3]{FeNov6A}, say,
\[
\vte \to \vt \ \mbox{(strongly) in}\ L^2((0,T) \times \Omega),
\]
and, again by \eqref{U13},
\begin{equation} \label{C7}
	\vte \to \vtB \ \mbox{in}\ L^{k+1}((0,T) \times (\mathbb{T}^d \setminus \Omega)).
	\end{equation}
Thus we may infer
\begin{equation} \label{C8}
	\vte \to \vt \ \mbox{in}\ L^2((0,T) \times \mathbb{T}^d),\ (\vt - \vtB) \ \mbox{in}\
	L^2(0,T; W^{1,2}_0(\Omega)),\ \Ov{\log(\vt)} = \log(\vt),
\end{equation}
in particular, $\vt > 0$ a.a. in $(0,T) \times \mathbb{T}^d$.

We are ready to perform the limit in the ballistic energy inequality \eqref{U3}. To begin, in view of the hypotheses \eqref{M23}, \eqref{M24},
\begin{align}
&\intTd{ \left( \frac{1}{2} \frac{|\vm_{0,\ep}|^2}{\vr_{0,\ep}} + \vr_{0, \ep} e(\vr_{0,\ep}, S_{0, \ep}) - \tvt(0, \cdot) S_{0,\ep} \right) } \br
&\quad \to \intO{ \left( \frac{1}{2} \frac{|\vm_{0}|^2}{\vr_{0}} + \vr_{0} e(\vr_{0}, S_{0}) - \tvt(0, \cdot) S_{0} \right) }
- \int_{\mathbb{T}^d \setminus \Omega} \frac{a}{3} \vtB^4(0, \cdot) \dx.
\label{C9}
\end{align}

Similarly, using the strong convergence of $(\vre, \vte)$ established in \eqref{C5}, \eqref{C7} we get
\begin{align}
\liminf_{\ep \to 0}
\int_{\tau - \delta}^{\tau + \delta} &\intTd{ \left( \frac{1}{2} \vre |\vue|^2 + \vr_{\ep} e(\vr_{\ep}, \vte) - \tvt \vre s(\vre, \vte) \right) } \dt \br
&\geq \int_{\tau - \delta}^{\tau + \delta} \intO{ \left( \frac{1}{2} \vr |\vu|^2 + \vr e(\vr, \vt) - \tvt \vr s(\vr, \vt) \right) } \dt
- \int_{\tau - \delta}^{\tau + \delta} \int_{\mathbb{T}^d \setminus \Omega} \frac{a}{3} \vtB^4 \ \dx \dt.
\label{C10}
\end{align}
for any $\delta > 0$.

Next, by the same argument,
\begin{align}
-	\int_0^\tau &\intTd{\vre s(\vre, \vte) \partial_t \tvt } \dt \to - \int_0^\tau \intO{ \vr s (\vr, \vt) \partial_t \tvt } \dt -
\int_0^\tau \int_{\mathbb{T}^d \setminus \Omega} \frac{4a}{3} \vtB^3 \partial_t \vtB \ \dx \dt  \br
&= - \int_0^\tau \intO{ \vr s (\vr, \vt) \partial_t \tvt } \dt -
\int_0^\tau \frac{\D }{\dt} \int_{\mathbb{T}^d \setminus \Omega} \frac{a}{3} \vtB^4  \ \dx \dt \br
&= - \int_0^\tau \intO{ \vr s (\vr, \vt) \partial_t \tvt } \dt - \int_{\mathbb{T}^d \setminus \Omega} \frac{a}{3} \vtB^4 (\tau, \cdot) \dx
+ \int_{\mathbb{T}^d \setminus \Omega} \frac{a}{3} \vtB^4 (0, \cdot) \dx.
 \label{C11}
	\end{align}
In addition, by virtue of \eqref{C3},
\begin{equation} \label{C12}
\int_0^\tau \intTd{ \vre s(\vre, \vte) \vue \cdot \Grad \tvt } \dt \to \int_0^\tau \intO{ \vr s(\vr, \vt) \vu \cdot \Grad \tvt } \dt.	
	\end{equation}

Summing up \eqref{C9}--\eqref{C12} and plugging the result in the ballistic energy inequality \eqref{U3} we obtain
	\begin{align}
		&\int_{\tau - \delta}^{\tau + \delta} \intO{ \left( \frac{1}{2} \vr |\vu|^2 +\vr 	e(\vr,\vt) - \tvt \vr s(\vr, \vt) \right)(s,\cdot) } \ \D s  \br
		&\quad + \liminf_{\ep \to 0} \int_{\tau - \delta}^{\tau + \delta} \int_{0}^{s} \intTd{ \frac{\tvt}{\vte} \left( \mathbb{S}(\vte, \mathbb{D}_x \vue): \mathbb{D}_x \vue -\frac{\bm{q}(\vte, \nabla_x \vte) \cdot \nabla_x \vte}{\vte}\right)  }\ \textup{d}t\ \D s\br
				& \quad   \leq 2 \delta \intO{ \left( \frac{1}{2} \frac{|\vm_{0}|^2}{\vr_{0}} + \vr_{0} 	e(\vr_{0} ,S_{0}) - \tvt(0, \cdot) S_{0}  \right) } \br
		&\quad - \int_{\tau - \delta}^{\tau + \delta} \int_0^s \intO{ \left( \vr s (\vr, \vt) \Big( \partial_t \tvt + \vu \cdot \Grad \tvt \Big) \right) }\ \dt \ \D s \br		
		&\quad- \lim_{\ep \to 0} \int_{\tau - \delta}^{\tau + \delta} \int_0^s \intTd{
		\frac{\vc{q} (\vte, \Grad \vte) \cdot \Grad \tvt }{\vte}  } \ \dt \ \D s
	\label{C13}
	\end{align}
for any $\delta > 0$.

Next,
\begin{align}
\intTd{ \frac{\tvt}{\vte} \mathbb{S} (\vte,   \mathbb{D}_x \vue) : \Ds \vue } &=
\intTd{\tvt \mu(\vte) \frac{1}{2 \vte} \left| \Grad \vue + \Grad^t \vue - \frac{2}{d} \Div \vue \mathbb{I} \right|^2 } \br
& +\intTd{\tvt \eta (\vte) \frac{1}{\vte} \left| \Div \vue \right|^2 }.
\nonumber
\end{align}
Consequently, combining convexity of the function
\[
(\vt, Z) \mapsto \frac{|Z|^2}{\vt}
\]
with the strong convergence of the temperature established in \eqref{C8}, we may infer that
\begin{align}
\liminf_{\ep \to 0} \int_{\tau - \delta}^{\tau + \delta} & \int_{0}^{s} \intTd{ \frac{\tvt}{\vte} \mathbb{S}(\vte, \mathbb{D}_x \vue): \mathbb{D}_x \vue } \ \dt \ \D s
\br &\geq
\int_{\tau - \delta}^{\tau + \delta} \int_{0}^{s} \intO{ \frac{\tvt}{\vt} \mathbb{S}(\vt, \mathbb{D}_x \vu): \mathbb{D}_x \vu } \ \dt \ \D s.
\label{C14}
\end{align}

Similarly, we rewrite
\[
- \intTd{ \frac{\tvt}{\vte^2} \vc{q}(\vte, \Grad \vte ) \cdot \Grad \vte } = \intTd{ \tvt \kappa (\vte) |\Grad \log (\vte)|^2 }.
\]
Consequently, using arguments similar to \eqref{C14}, we get
\begin{align}
- \liminf_{\ep \to 0} \int_{\tau - \delta}^{\tau + \delta} \int_{0}^{s} &\intTd{ \frac{\tvt}{\vte}	\frac{\bm{q}(\vte, \nabla_x \vte) \cdot \nabla_x \vte}{\vte} } \ \dt \ \D s \br
&\geq - \int_{\tau - \delta}^{\tau + \delta} \int_{0}^{s} \intO{ \frac{\tvt}{\vt}	\frac{\bm{q}(\vt, \nabla_x \vt) \cdot \nabla_x \vt}{\vt} } \ \dt \ \D s \br
&- \int_{\tau - \delta}^{\tau + \delta} \int_{0}^{s} \int_{\mathbb{T}^d \setminus \Omega} \frac{\bm{q}(\vtB, \nabla_x \vtB) \cdot \nabla_x \vtB}{\vtB} \ \dx \ \dt \ \D s.
	\label{C15}
	\end{align}

Our ultimate goal is to perform the limit in the last integral in \eqref{C13}, namely
\begin{align} \
& \lim_{\ep \to 0} \int_{\tau - \delta}^{\tau + \delta} \int_0^s \intTd{
	\frac{\vc{q} (\vte, \Grad \vte) \cdot \Grad \tvt }{\vte}  } \, \dt \, \D s \br
=& \int_{\tau - \delta}^{\tau + \delta} \int_0^s \intO{
	\frac{\vc{q} (\vt, \Grad \vt) \cdot \Grad \tvt }{\vt}  } \, \dt \, \D s + \int_{\tau - \delta}^{\tau + \delta} \int_0^s \int_{\mathbb{T}^d \setminus \Omega} {
	\frac{\vc{q} (\vtB, \Grad \vtB) \cdot \Grad \vtB }{\vtB}  } \, \dx\, \dt \, \D s.
\label{C16}	
	\end{align}
In view of the strong convergence of $(\vte)_{\ep > 0}$ and weak convergence of $(\Grad \vte )_{\ep > 0}$, it is enough to observe that the terms
\[
\frac{\kappa (\vte) }{\vte} \Grad \vte,\ \ep > 0
\]
are equi--integrable. Seeing that for small values of $\vte$ we have $\frac{\kappa (\vte) }{\vte} \Grad \vte \approx \Grad \log(\vte)$ controlled by \eqref{U13}, we focus on
\[
\vte^{\beta - 1} \Grad \vte ,\ \vte \geq 1,\,\ \ep > 0.
\]
Writing
\[
\vte^{\beta - 1} \Grad \vte = \vte^{\frac{\beta}{2}} \Grad \vte^{\frac{\beta}{2}}
\]
we observe, by virtue of the estimates \eqref{U13}, that it is enough to control the norm
\[
\left\| \vte^{\frac{\beta}{2} } \right\|_{L^q((0,T) \times \mathbb{T}^d )} \aleq 1 \ \mbox{for some}\ q > 2.
\]
This is definitely possible on the set $(0,T) \times (\mathbb{T}^d \setminus \Omega)$ since $k + 1 > \beta$. As for $\Omega$, we have, by virtue of \eqref{U13}
and the standard Sobolev emebedding,
\[
(\vte^{\frac{\beta}{2}} )_{\ep > 0} \ \mbox{bounded in}\ L^2(0,T; L^6(\Omega)),\ (\vte )_{\ep > 0}
\ \mbox{bounded in}\ L^\infty(0,T; L^4(\Omega));
\]
whence the desired conclusion follows by interpolation.

Plugging \eqref{C14}--\eqref{C16} in \eqref{C13} we get
	\begin{align}
	&\int_{\tau - \delta}^{\tau + \delta} \intO{ \left( \frac{1}{2} \vr |\vu|^2 +\vr 	e(\vr,\vt) - \tvt \vr s(\vr, \vt) \right)(s,\cdot) } \ \D s  \br
	&\quad + \int_{\tau - \delta}^{\tau + \delta} \int_{0}^{s} \intO{ \frac{\tvt}{\vt} \left( \mathbb{S}(\vt, \mathbb{D}_x \vu): \mathbb{D}_x \vu -\frac{\bm{q}(\vt, \nabla_x \vt) \cdot \nabla_x \vte}{\vt}\right)  }\ \textup{d}t\ \D s\br
	& \quad   \leq 2 \delta \intO{ \left( \frac{1}{2} \frac{|\vm_{0}|^2}{\vr_{0}} + \vr_{0} 	e(\vr_{0} ,S_{0}) - \tvt(0, \cdot) S_{0}  \right) } \br
	&\quad - \int_{\tau - \delta}^{\tau + \delta} \int_0^s \intO{ \left( \vr s (\vr, \vt) \Big( \partial_t \tvt + \vu \cdot \Grad \tvt \Big) + \frac{\vc{q} (\vt, \Grad \vt) \cdot \Grad \tvt }{\vt} \right) }\ \dt \ \D s
	\nonumber
\end{align}
for any $\delta > 0$. Finally, we divide the above inequality by $2 \delta$ and let $\delta \to 0$ obtaining
	\begin{align}
	&\intO{ \left( \frac{1}{2} \vr |\vu|^2 +\vr 	e(\vr,\vt) - \tvt \vr s(\vr, \vt) \right)(\tau,\cdot) }  \br
	&\quad + \int_{0}^{\tau} \intO{ \frac{\tvt}{\vt} \left( \mathbb{S}(\vt, \mathbb{D}_x \vu): \mathbb{D}_x \vu -\frac{\bm{q}(\vt, \nabla_x \vt) \cdot \nabla_x \vte}{\vt}\right)  }\ \textup{d}t \br
	& \quad   \leq  \intO{ \left( \frac{1}{2} \frac{|\vm_{0}|^2}{\vr_{0}} + \vr_{0} 	e(\vr_{0} ,S_{0}) - \tvt(0, \cdot) S_{0}  \right) } \br
	&\quad - \int_0^\tau \intO{ \left( \vr s (\vr, \vt) \Big( \partial_t \tvt + \vu \cdot \Grad \tvt \Big) + \frac{\vc{q} (\vt, \Grad \vt) \cdot \Grad \tvt }{\vt} \right) }\ \dt
	\nonumber
\end{align}
for a.a. $\tau \in (0,T)$,
which is nothing other than the ballistic energy inequality \eqref{M22}.

We have proved Theorem \ref{MT1}.

%\subsection{Concluding remark}
\subsection{Time independent $\vt_B$}
\label{nonzero}

Finally, we consider the time independent boundary temperature as specified in hypothesis \eqref{M24bis}.
Under these circumstances, the main stumbling block is the fact that the strong convergence of the density outside
$\Omega$ stated in \eqref{C4} is no longer valid.

Fortunately, as $\vtB$ is independent of time, the integrals over $\mathbb{T}^d \setminus \Ov{\Omega}$ vanish in \eqref{C11}. Moreover, the convergence in \eqref{C9} reads now
\begin{align}
	&\intTd{ \left( \frac{1}{2} \frac{|\vm_{0,\ep}|^2}{\vr_{0,\ep}} + \vr_{0, \ep} e(\vr_{0,\ep}, S_{0, \ep}) - \tvt(0, \cdot) S_{0,\ep} \right) } \br
	&\quad \to \intO{ \left( \frac{1}{2} \frac{|\vm_{0}|^2}{\vr_{0}} + \vr_{0} e(\vr_{0}, S_{0}) - \tvt(0, \cdot) S_{0} \right) }
	+ \int_{\mathbb{T}^d \setminus \Omega} \Big( \vr_0 e(\vr_0, S_0) - \vtB S_0 \Big) \dx.
	\label{C18}
\end{align}
Next, in view of the strong convergence of the approximate solutions in $\Omega$, we get
 \begin{align}
	\liminf_{\ep \to 0}
	\int_{\tau - \delta}^{\tau + \delta} &\intTd{ \left( \frac{1}{2} \vre |\vue|^2 + \vr_{\ep} e(\vr_{\ep}, \vte) - \tvt \vre s(\vre, \vte) \right) } \dt \br
	&\geq \int_{\tau - \delta}^{\tau + \delta} \intO{ \left( \frac{1}{2} \vr |\vu|^2 + \vr e(\vr, \vt) - \tvt \vr s(\vr, \vt) \right) } \dt \br
	& + \liminf_{ \ep \to 0 } \int_{\tau - \delta}^{\tau + \delta} \int_{\mathbb{T}^d \setminus \Omega} \Big( \vre e(\vre, \vte) - \vtB \vre s(\vre, \vte) \Big) \dx \dt.
	\label{C19}
\end{align}
Finally, we first use Fatou's lemma to get
\[
\liminf_{ \ep \to 0 } \int_{\tau - \delta}^{\tau + \delta}\int_{\mathbb{T}^d \setminus \Omega} \Big( \vre e(\vre, \vte) - \vtB \vre s(\vre, \vte) \Big) \dx
\geq \int_{\tau - \delta}^{\tau + \delta}\int_{\mathbb{T}^d \setminus \Omega} \left< \nu_{t,x};  \tvr e(\tvr, \tvt) - \vtB \tvr s(\tvr, \tvt)  \right> \dx \dt,
\]
where $(\nu_{t,x} )_{t \in (0,T), x \in \mathbb{T}^d \setminus \Omega}$ is a Young measure generated by the sequence $(\vre, \vte)_{\ep > 0}$. In addition, as $\vte \to \vtB$ strongly,
the Young measure can be written as
\[
\nu_{t,x} = \delta_{\vtB(x)} \otimes \omega_{t,x},
\]
where $(\omega_{t,x} )_{t \in (0,T), x \in \mathbb{T}^d \setminus \Omega}$ is the Young measure associated to $(\vre )_{\ep > 0}$. Accordingly,
\[
\int_{\tau - \delta}^{\tau + \delta}\int_{\mathbb{T}^d \setminus \Omega} \left< \nu_{t,x};  \tvr e(\tvr, \tvt) - \vtB \tvr s(\tvr, \tvt)  \right> \dx \dt =
\int_{\tau - \delta}^{\tau + \delta}\int_{\mathbb{T}^d \setminus \Omega} \left< \omega_{t,x};  \tvr e(\tvr, \vtB ) - \vtB \tvr s(\tvr, \vtB)  \right> \dx \dt.
\]
However, the function
\[
\tvr \mapsto  \tvr e(\tvr, \vtB ) - \vtB \tvr s(\tvr, \vtB)
\]
is convex, see \cite[Chapter 2, Section 2.2.3]{FeNov6A}; whence, by Jensen's inequality,
\[
\int_{\tau - \delta}^{\tau + \delta}\int_{\mathbb{T}^d \setminus \Omega} \left< \omega_{t,x};  \tvr e(\tvr, \vtB ) - \vtB \tvr s(\tvr, \vtB)  \right> \dx \dt
\geq 2 \delta \int_{\mathbb{T}^d \setminus \Omega} \Big( \vr_0 e(\vr_0, \vtB ) - \vtB \vr_0 s(\vr_0, \vtB) \Big) \dx.
\]
Consequently, the last integrals in \eqref{C18}, \eqref{C19} cancel out and the desired conclusion follows.
%Finally, in view of convexity of the internal energy in $(\vr, S)$, we get
% \begin{align}
%	\liminf_{\ep \to 0}
%	\int_{\tau - \delta}^{\tau + \delta} &\intTd{ \left( \frac{1}{2} \vre |\vue|^2 + \vr_{\ep} e(\vr_{\ep}, \vte) - \tvt \vre s(\vre, \vte) \right) } \dt \br
%	&\geq \int_{\tau - \delta}^{\tau + \delta} \intO{ \left( \frac{1}{2} \vr |\vu|^2 + \vr e(\vr, \vt) - \tvt \vr s(\vr, \vt) \right) } \dt \br
%	&- 2 \delta \int_{\mathbb{T}^d \setminus \Omega} \Big( \vr_0 e(\vr_0, S_0) - \vtB S_0 \Big) \dx.
%	\label{C19}
%\end{align}
%Consequently, the last integrals in \eqref{C18}, \eqref{C19} cancel out and the desired conclusion follows.

\section{Finite volume simulations}
\label{numerics}

The aim of this section is to illustrate convergence of the penalization method when simulating viscous compressible flows in complex geometries.
To this end we apply the \emph{finite volume method} (FV) that has been recently developed in~\cite{Lukacova-Mizerova-She:2021} and analysed in~\cite{Feireisl-Lukacova-Mizerova-She:2021}. In what follows we describe the FV method and present four numerical experiments for the Navier--Stokes--Fourier system. We will also study the experimental rate of convergence of  the FV scheme with respect to both the discretization and the  penalty parameters.

We consider for simplicity the perfect gas pressure law
%\begin{equation} \label{EOS}
$	p  =\varrho \vartheta, $
%	\end{equation}
internal energy and entropy are given by  $e = c_v \vartheta,~ s = \log\left( \dfrac{\vartheta^{c_v}}{\varrho}\right)
= \dfrac{1}{\gamma-1} \log \left( \dfrac{p}{\varrho^\gamma}\right)$, respectively, where $c_v = \dfrac1{\gamma-1}$. In the numerical experiments presented below
$\varrho, \vt$ are bounded from below and above. In this case theoretical results on the convergence of the penalization method can be obtained for the perfect gas pressure law as well.

\subsection{Finite volume method}
Our computational domain $\mathbb{T}^d \subset R^d$ is approximated by regular cuboids: finite volumes K of size $h^d,$  where $h  \in (0,1)$ is a mesh step. %Numerical solutions {\cblue $U_h^\epsilon\equiv(\vr_h, \vu_h, \vt_h),$} $\bm{m}_h=\vr_h \vu_h,\ e_h = c_v \vt_h$ obtained by {\cblue the VFV method} are {\cblue in time continuous} piecewise constant functions on $\mathbb{T}^d$ satisfying the following equations:
%Numerical solutions { $U_h^\varepsilon\equiv(\vr_h^\varepsilon, \vu_h^\varepsilon, \vt_h^\varepsilon),$} $\bm{m}_h^\varepsilon=\vr_h^\varepsilon \vu_h^\varepsilon,\ e_h^\varepsilon = c_v \vt_h^\varepsilon$ %obtained by {\cblue the VFV method} are {\cblue in time continuous} piecewise constant functions on $\mathbb{T}^d$ satisfying the following equations:
The piecewise constant numerical solutions $U_h^\varepsilon\equiv(\vr_h^\varepsilon, \vu_h^\varepsilon, \vt_h^\varepsilon)$ satisfy the following semi-discrete first order FV method
\begin{subequations}\label{eq:FVmethod}
	\begin{align}
	D_t \varrho_h^\varepsilon +  \divv_h^{\rm{up}} F(\varrho_h^\varepsilon, \vu_h^\varepsilon) & = 0, \\
	D_t \bm{m}_h^\varepsilon +  \divv_h^{\rm{up}} \bm{F}(\bm{m}_h^\varepsilon, \vu_h^\varepsilon) + \nabla_h p_h^\varepsilon & = 2\divv_h \left( \mu_h^\varepsilon \mathbb{D}_h\vu_h^\varepsilon \right) +
\nabla_h \left( \lambda_h^\varepsilon \divv_h  \vu_h^\varepsilon \right)  \nonumber \\
&\ +\varrho_h^\varepsilon \bm{g}_h^\varepsilon - \frac{1}{\varepsilon}  \mathds{1}_{\mathbb{T}^d \setminus \Omega}  (K)\vu_h^\varepsilon,\\
D_t (\varrho_h^\varepsilon e_h^\varepsilon) +  \divv_h^{\rm{up}} F(\varrho_h^\varepsilon e_h^\varepsilon, \vu_h^\varepsilon) +p_h^\varepsilon \divv_h \vu_h^\varepsilon   & = \divv_h \left(\kappa_h^\varepsilon \nabla_h \vartheta_h^\varepsilon \right) + 2\mu_h^\varepsilon | \mathbb{D}_h\vu_h^\varepsilon |^2 + \lambda_h^\varepsilon | \divv_h  \vu_h^\varepsilon|^2  \nonumber\\
	&\ - \frac{1}{\varepsilon}   \mathds{1}_{\mathbb{T}^d \setminus \Omega} (K) |\vartheta_h^\varepsilon-\vartheta_B|^{k}(\vartheta_h^\varepsilon-\vartheta_B)
	\end{align}
\end{subequations}
with $\bm{m}_h^\varepsilon=\vr_h^\varepsilon \vu_h^\varepsilon,\ e_h^\varepsilon = c_v \vt_h^\varepsilon$,  coefficients $\mu_h^\varepsilon = \mu(\vt_h^\varepsilon),\ \lambda_h^\varepsilon = \lambda(\vt_h^\varepsilon),$ $\lambda(\vt) = \eta(\vt) - 2\mu(\vt)/d,$ $\kappa_h^\varepsilon = \kappa(\vartheta_h^\varepsilon),$ and the characteristic function
\begin{equation*}
\mathds{1}_{\mathbb{T}^d \setminus \Omega}(K) =
\begin{cases}
1, & \mbox{if}~ K_c \in \mathbb{T}^d \setminus \Omega, \\
0, & \mbox{otherwise},
\end{cases}	
\end{equation*}
where $K_c$ is the gravity center of $K$.
Function $\bm{g}_h^\varepsilon$ stands for a piecewise constant  approximation of an external (gravitational) force.

%\begin{subequations}\label{eq:FVmethod}
%	\begin{align}
%	D_t \varrho_h + {\cblue \divv_h^{\rm{up}} (\varrho_h, \vu_h)} & = 0, \\
%	D_t \bm{m}_h + {\cblue \divv_h^{\rm{up}} (\bm{m}_h, \vu_h)} + \nabla_h p_h & = 2\divv_h \left( \mu_h \mathbb{D}_h\vu_h \right) +
%\nabla_h \left( \lambda_h \divv_h  \vu_h \right)  \nonumber \\
%&\ +\varrho_h \bm{g}_h - \frac{1}{\varepsilon} {\cblue \mathds{1}}_{\mathbb{T}^d \setminus \Omega} \vu_h,\\
%D_t (\varrho_h e_h) + {\cblue \divv_h^{\rm{up}} (\varrho_h e_h, \vu_h)} +p_h \divv_h \vu_h   & = \divv_h \left(\kappa_h \nabla_h \vartheta_h \right) + 2\mu_h | \mathbb{D}_h\vu_h |^2 + \lambda_h | \divv_h  \vu_h|^2  \nonumber\\
%	&\ - \frac{1}{\varepsilon}  {\cblue \mathds{1}}_{\mathbb{T}^d \setminus \Omega} (\vartheta_h-\vartheta_B)^{k+1}
%	\end{align}
%\end{subequations}
%with coefficients $\mu_h = \mu(\vt_h),\ \lambda_h = \lambda(\vt_h),$ ${\cblue \lambda(\vt) = \eta(\vt) - 2\mu(\vt)/d},$ and $\kappa_h = \kappa(\vartheta_h).$
%{\cblue Function $\bm{g}_h$ stands for a piecewise constant  approximation of an external (gravitational) force.}

Let us define the discrete operators and the numerical flux function used in the scheme \eqref{eq:FVmethod}. By $\overline{r_h}$ and  $[[ r_h]]$ we denote the average and jump along a cell interface for any $r_h$ piecewise constant function, respectively.
The discrete differential operators
 $\divv_h$ and $\nabla_h $ are adjoint operators defined on each finite volume $K$ in the following way
\begin{align*}
&(\Divh \bm{v}_h)_K = \sum_{\sigma\in \partial K} \frac{|\sigma|}{|K|} \overline{\bm{v}_h} \cdot \vn,
&
&(\Gradh r_h)_K = \sum_{\sigma\in \partial K} \frac{|\sigma|}{|K|} \overline{r_h} \vn,
&
&\mathbb{D}_h r_h= \frac 1 2 \left(\Gradh r_h + (\Gradh r_h)^t\right),
%\divv_h^{\rm{up}}(r_h \bm{v}_h)_K =
%\sum_{\sigma\in \partial K} \frac{|\sigma|}{|K|} F_h(r_h,\bm{v}_h),
\end{align*}
where  $\bm{v}_h$ is a vector-valued piecewise constant function on $\mathbb{T}^d,$ and  $\vn$  denotes an outer normal vector to  $\partial K.$
For flux approximation we apply the viscosity upwind numerical flux $F_h$ defined by
\begin{align*}
&F_h(r_h,\bm{v}_h) = Up[r_h, \bm{v}_h]  - h^{\alpha} ~ [[ r_h]], \quad 0< \alpha <1, \\
&Up[r_h, \bm{v}_h] = \overline{r_h}\, \overline{\bm{v}_h} \cdot \textbf{n} - \frac12 |\overline{\bm{v}_h} \cdot \textbf{n} | ~ [[ r_h]].
\end{align*}
Moreover, we set
\begin{align*}
\big( \divv_h^{\rm{up}}F(r_h, \bm{v}_h) \big)_K = \sum_{\sigma\in \partial K} \frac{|\sigma|}{|K|} F_h(r_h,\bm{v}_h).
\end{align*}
For a vector-valued piecewise constant function $\bm{w}_h$  on $\mathbb{T}^d$ the discrete divergence $\divv_h^{\rm{up}}\bm{F}(\bm{w}_h, \bm{v}_h)_K$ is  defined componentwisely.
%\begin{eqnarray*}
%&&(\Gradh r_h)_K = \sum_{\sigma\in \partial K} \frac{|\sigma|}{|K|} \overline{r_h} \vn, \hspace{2cm}
%(\Divh \bm{v}_h)_K = \sum_{\sigma\in \partial K} \frac{|\sigma|}{|K|} \overline{\bm{v}_h} \cdot \vn, \\
%&&\mathbb{D}_h r_h= \frac 1 2 \left(\Gradh r_h + (\Gradh r_h)^t\right),\hspace{1.2cm}
%\divv_h^{\rm{up}}(r_h \bm{v}_h)_K = \sum_{\sigma\in \partial K} \frac{|\sigma|}{|K|} F_h(r_h,\bm{v}_h),
%\end{eqnarray*}
See~\cite{Lukacova-Mizerova-She:2021} for further details.

%{\cblue Symbol $D_t$ in \eqref{eq:FVmethod} stands for the continuous time derivative which is in our simulations approximated by the forward Euler method with the time step} $\Delta t = 10^{-6}$ for all experiments.
In our simulations we take the symbol $D_t$ in \eqref{eq:FVmethod} as the forward Euler discretization with the time step $\Delta t = 10^{-6}$.
Initial data $(\vr_{h,0}, \vu_{h,0}, \vt_{h,0})$  are taken as piecewise constant projections of the exact initial data $(\vr_0, \vu_0, \vt_0)$ computed directly from $ (\vr_0, \bm{m}_0, S_0)$, cf.~Theorem~\ref{MT1}. Thus,
$\vu_0 = \frac{\bm{m}_0}{\vr_0},$ $\vr_0>0$ and $\vt_0 = \exp\left[(\gamma - 1) (S_0/\vr_0 + \log \vr_0 )\right].$
We set the transport coefficients and other parameters in \eqref{eq:FVmethod},
\begin{equation*}
\mu = \lambda = \kappa = 0.001,~ \gamma = 1.4, ~  k = 6,~ \alpha = 0.6.
\end{equation*}

The computational domain $\mathbb{T}^2$ is taken as $[-1,1]^2$ applying the periodic boundary condition.
Following the theoretical part, cf.~Theorem~\ref{MT1}  and Section~\ref{nonzero}, we will test two cases for the extension of the density outside the fluid domain $\Omega$:
\begin{itemize}
\item constant density, i.e.
\begin{equation*}
\mathds{1}_{\mathbb{T}^d \setminus \Omega} \varrho_{h,0} = 1,
\end{equation*}
\item ``zero" density, i.e.
\begin{equation*}
\mathds{1}_{\mathbb{T}^d \setminus \Omega} \varrho_{h,0} = 0.01.
\end{equation*}
\end{itemize}
The latter is chosen as an approximation, since $\mathds{1}_{\mathbb{T}^d \setminus \Omega} \varrho_{h,0}  = 0$ is numerically unstable.

In the following, we present numerical solutions obtained by the FV method with the penalty parameter $\varepsilon = 10^{-m}, m = 1,\dots,4$ on different meshes with $10^2,\dots,160^2$ cells.
Correspondingly, we compute two $L^1$-errors with respect to the mesh step $h$ and the penalization parameter $\varepsilon$,
\begin{equation}
E(U_h^{\varepsilon}) = \| U_h^{\varepsilon} - U_{h_{ref}}^{\varepsilon} \|_{{L^1(\mathbb{T}^2)}}, \quad
P(U_{h}^\varepsilon) = \| U_{h}^{\varepsilon} - U_{h}^{\varepsilon_{ref}} \|_{{ L^1(\mathbb{T}^2)}},
%E(U_h^{\varepsilon}) = \| U_h^{\varepsilon} - U_{h/2}^{\varepsilon} \|_{{L^1(\mathbb{T}^2)}}, \quad
%P(U_{h}^\varepsilon) = \| U_{h}^{\varepsilon} - U_{h}^{\varepsilon/10} \|_{{ L^1(\mathbb{T}^2)}},
%E(U_h^{\varepsilon}) = \| U_h^{\varepsilon} - U_{h/2}^{\varepsilon} \|_{{\cblue L^1(\mathbb{T}^2;R^{4})}}, \quad
%P(U_{h}^\varepsilon) = \| U_{h}^{\varepsilon} - U_{h}^{\varepsilon/10} \|_{{\cblue L^1(\mathbb{T}^2;R^{4})}},
\end{equation}
respectively, with $h_{ref} = 2/160$, $\varepsilon_{ref} = 10^{-4}$.
Thus, $E(U_h^{\varepsilon})$  verifies the convergence rate of the FV method for a fixed penalization parameter $\varepsilon$ while $P(U_{h}^\varepsilon)$  is used to study  the convergence rate of the penalization technique for a fixed $h$.

\subsection{Experiment~1: Ring domain - constant density outside $\Omega$}

In this experiment we consider the physical fluid domain to be a ring, i.e. $\Omega\equiv B_{0.7}\setminus B_{0.2},$ where $ B_{r} = \{x ~\big|~ |x| < r\}$.
The initial data are given by
\begin{equation*}
	(\varrho, \vu, \vartheta)(0,x)
	\; = \; \begin{cases}
	(1,0, 0 , 1 ) , & x \in B_{0.2}, \\
	\left(1, \frac{ \sin(4\pi (|x|-0.2)) x_2}{|x|} ,-\frac{ \sin(4\pi (|x|-0.2)) x_1}{|x|}, 0.2 + 4|x| \right) , & x \in  \Omega \equiv {B}_{0.7}\setminus B_{0.2}, \\
	(1,0 , 0,  3) , & x \in \mathbb{T}^2\setminus B_{0.7}.
	\end{cases}
\end{equation*}
%The direction of velocity $\vu$ is tangent to the circles but differs between $|x| \in [0.2, 0.45]$ and $|x| \in [ 0.45, 0.7]$.
The finial time is taken as $T=0.1$.
Figure~\ref{fig:ex1-1} shows the errors with respect to the mesh step $h$, i.e. $E(U_h^{\varepsilon}).$
Figure~\ref{fig:ex1-2} presents the errors with respect to  the penalization parameter $\varepsilon$, i.e. $P(U_h^{\varepsilon})$. This experimental
convergence study indicates that both convergence rates are 1.
Effects of different penalization parameters $\varepsilon = 10^{-1},\dots,10^{-4}$ are illustrated in Figure \ref{fig:ex1} which depicts the  numerical solutions  on the  mesh with $80^2$ cells.

%we can observe small oscillations near the inner boundary in the density plots in Figure \ref{fig:ex1}, whereas the temperature plots are quite %clean. This may be because of the penalization heat term acted in the internal energy balance.

\begin{figure}[htbp]
	\setlength{\abovecaptionskip}{0.cm}
	\setlength{\belowcaptionskip}{-0.cm}
	\centering
	\begin{subfigure}{0.32\textwidth}
		\includegraphics[width=\textwidth]{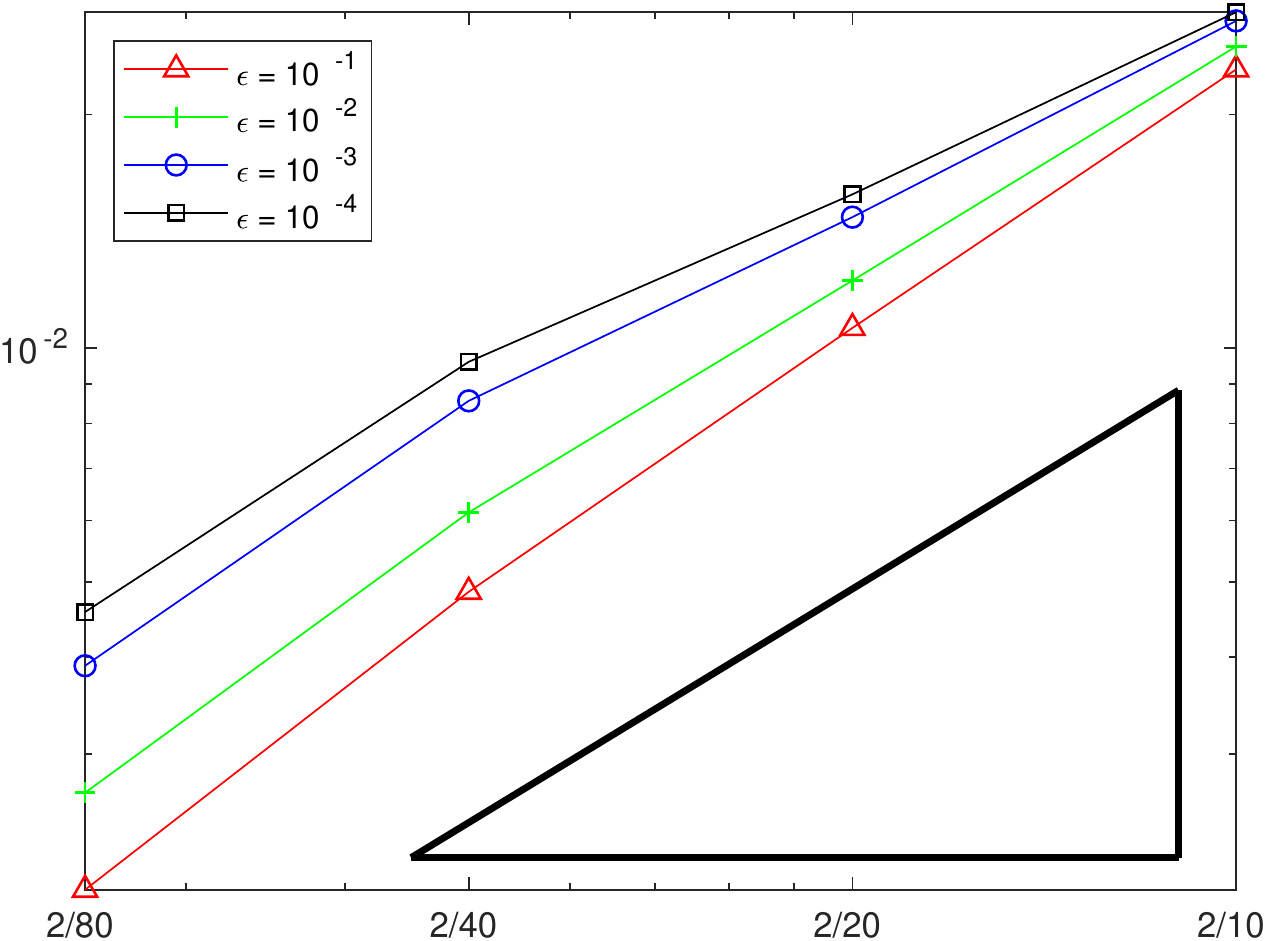}
		\caption{ \bf $\varrho$}
	\end{subfigure}
	\begin{subfigure}{0.32\textwidth}
		\includegraphics[width=\textwidth]{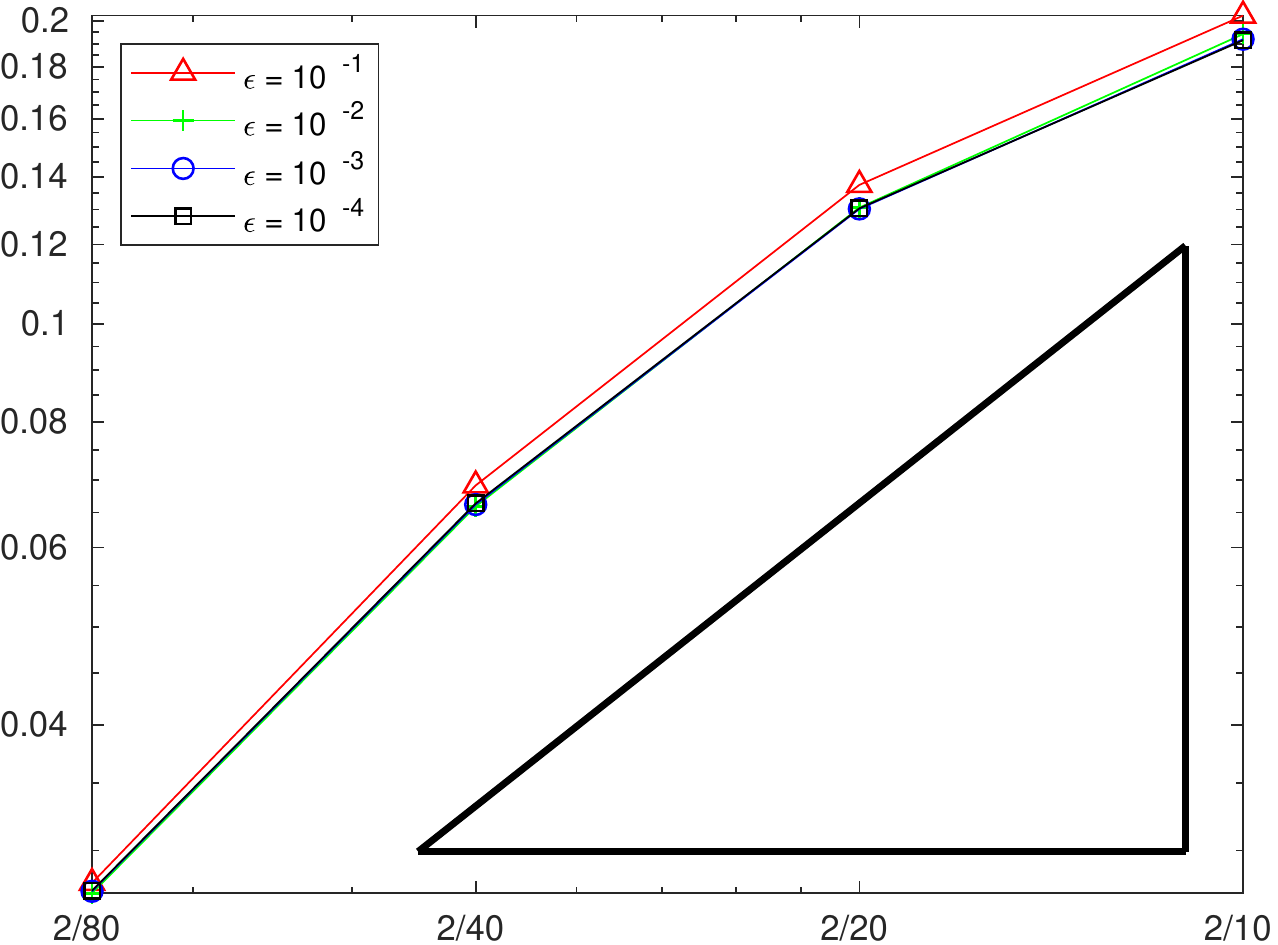}
		\caption{ \bf $\vu$}
	\end{subfigure}
	\begin{subfigure}{0.32\textwidth}
		\includegraphics[width=\textwidth]{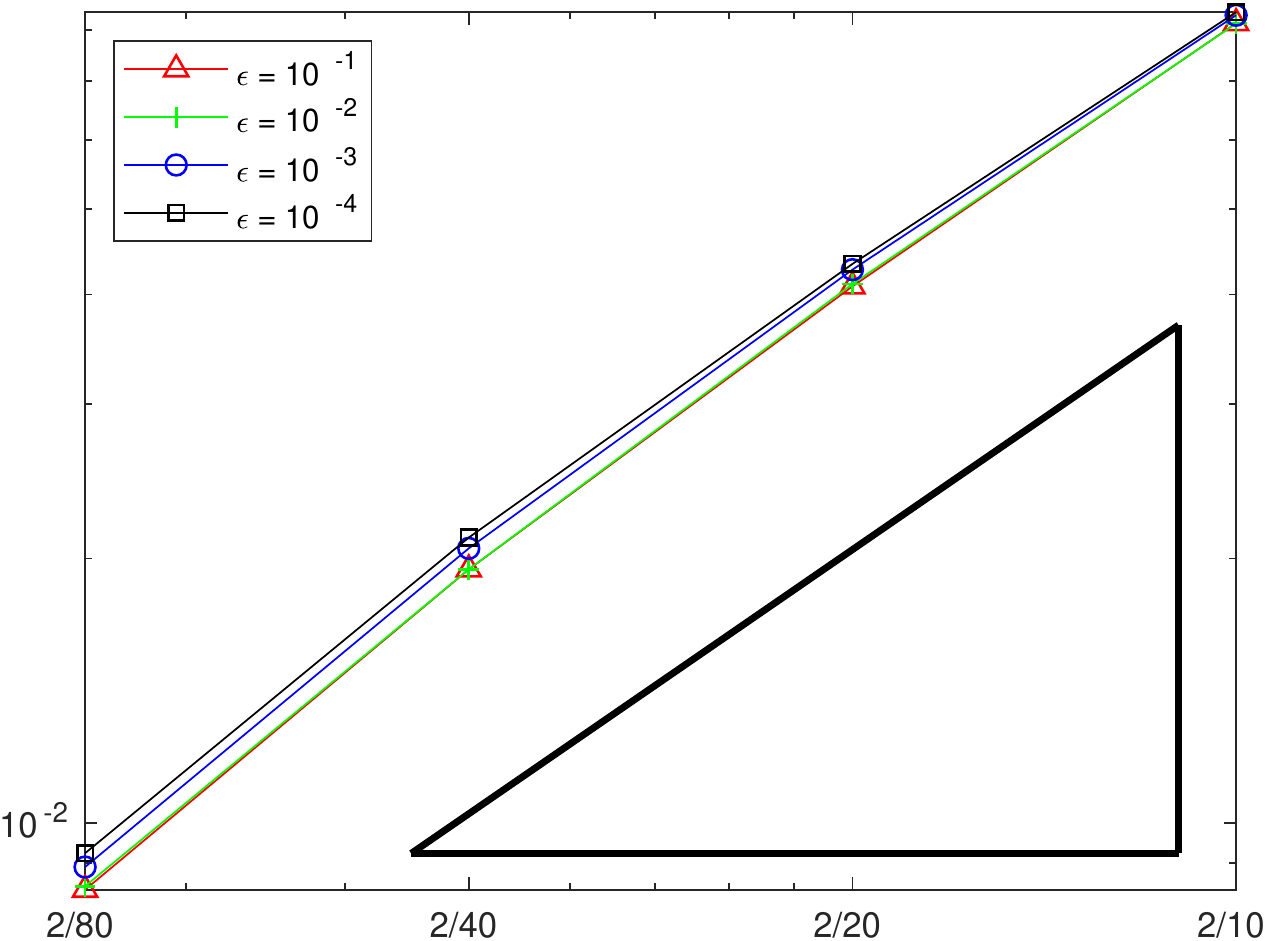}
		\caption{ \bf $\vartheta$}
	\end{subfigure}
	\caption{\small{Experiment~1:  $E(U_h^{\varepsilon})$ errors  with respect to $h.$}}\label{fig:ex1-1}
\end{figure}

\begin{figure}[htbp]
	\setlength{\abovecaptionskip}{0.cm}
	\setlength{\belowcaptionskip}{-0.cm}
	\centering
	\begin{subfigure}{0.32\textwidth}
		\includegraphics[width=\textwidth]{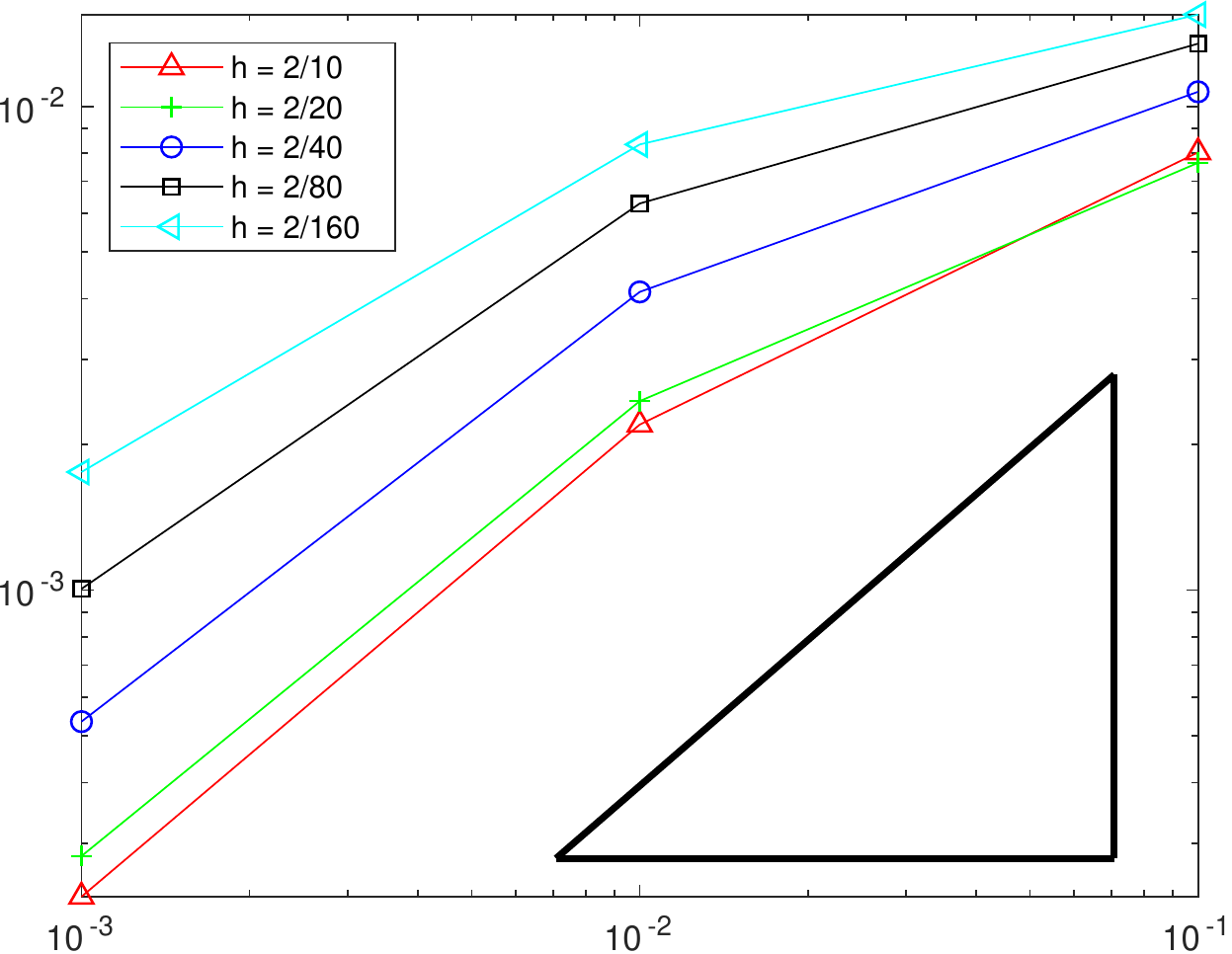}
		\caption{ \bf $\varrho$}
	\end{subfigure}
	\begin{subfigure}{0.32\textwidth}
		\includegraphics[width=\textwidth]{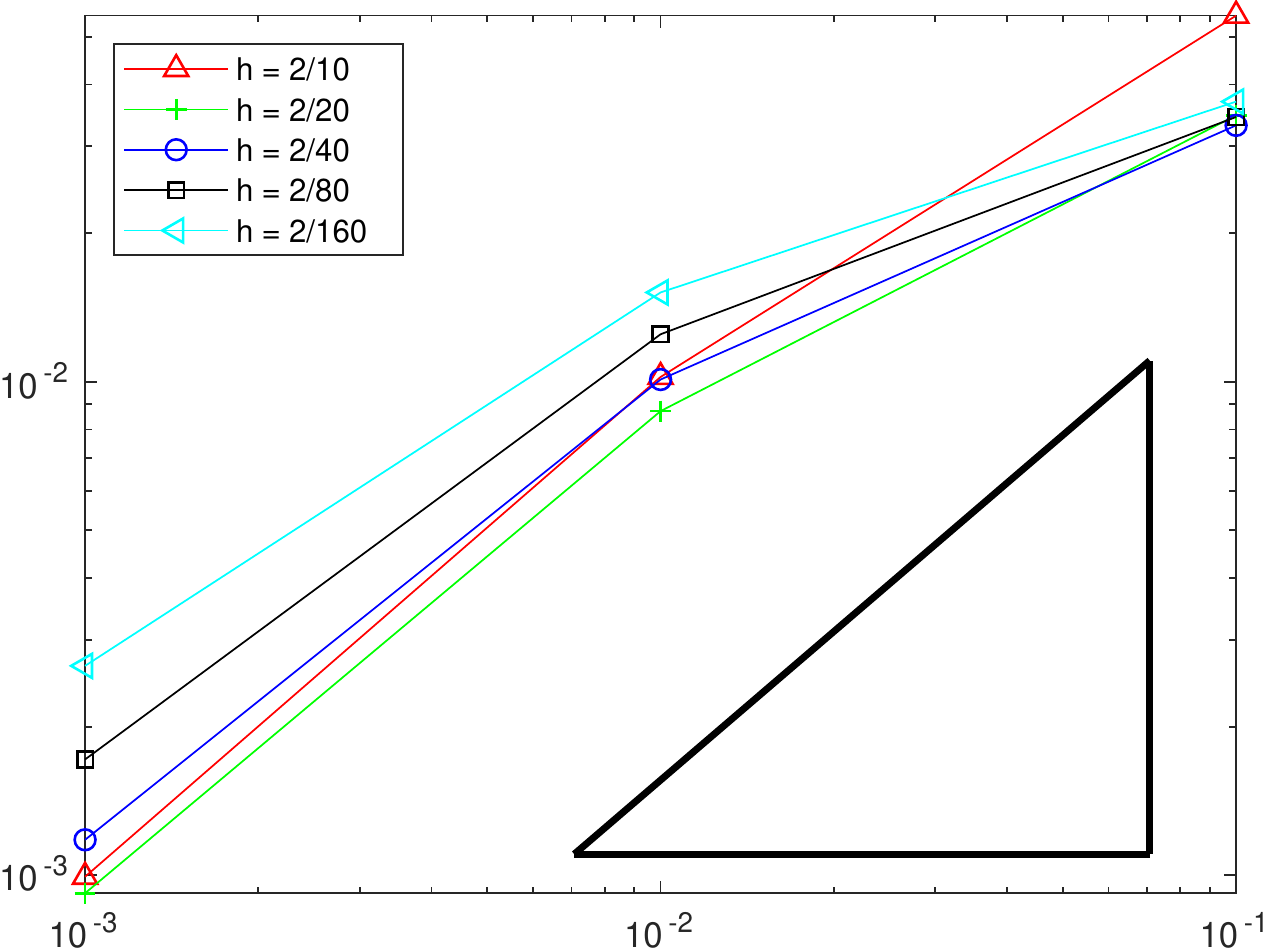}
		\caption{ \bf $\vu$}
	\end{subfigure}
	\begin{subfigure}{0.32\textwidth}
		\includegraphics[width=\textwidth]{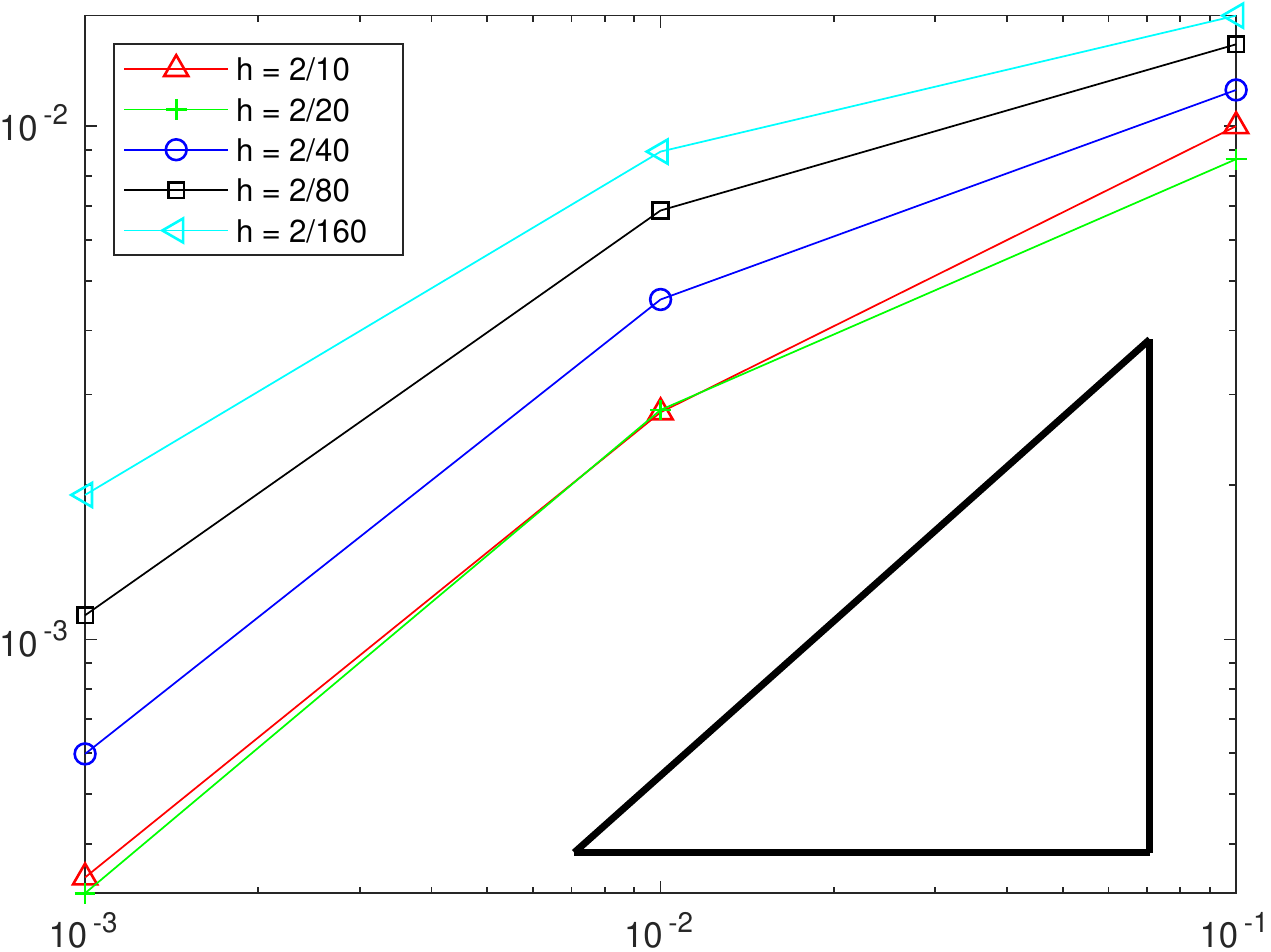}
		\caption{ \bf $\vartheta$}
	\end{subfigure}
	\caption{\small{Experiment~1:  $P(U_h^{\varepsilon})$ errors with respect to $\varepsilon.$}}\label{fig:ex1-2}
\end{figure}

\begin{figure}[htbp]
	\setlength{\abovecaptionskip}{0.cm}
	\setlength{\belowcaptionskip}{-0.cm}
	\centering
	\begin{subfigure}{0.2\textwidth}
		\includegraphics[width=\textwidth]{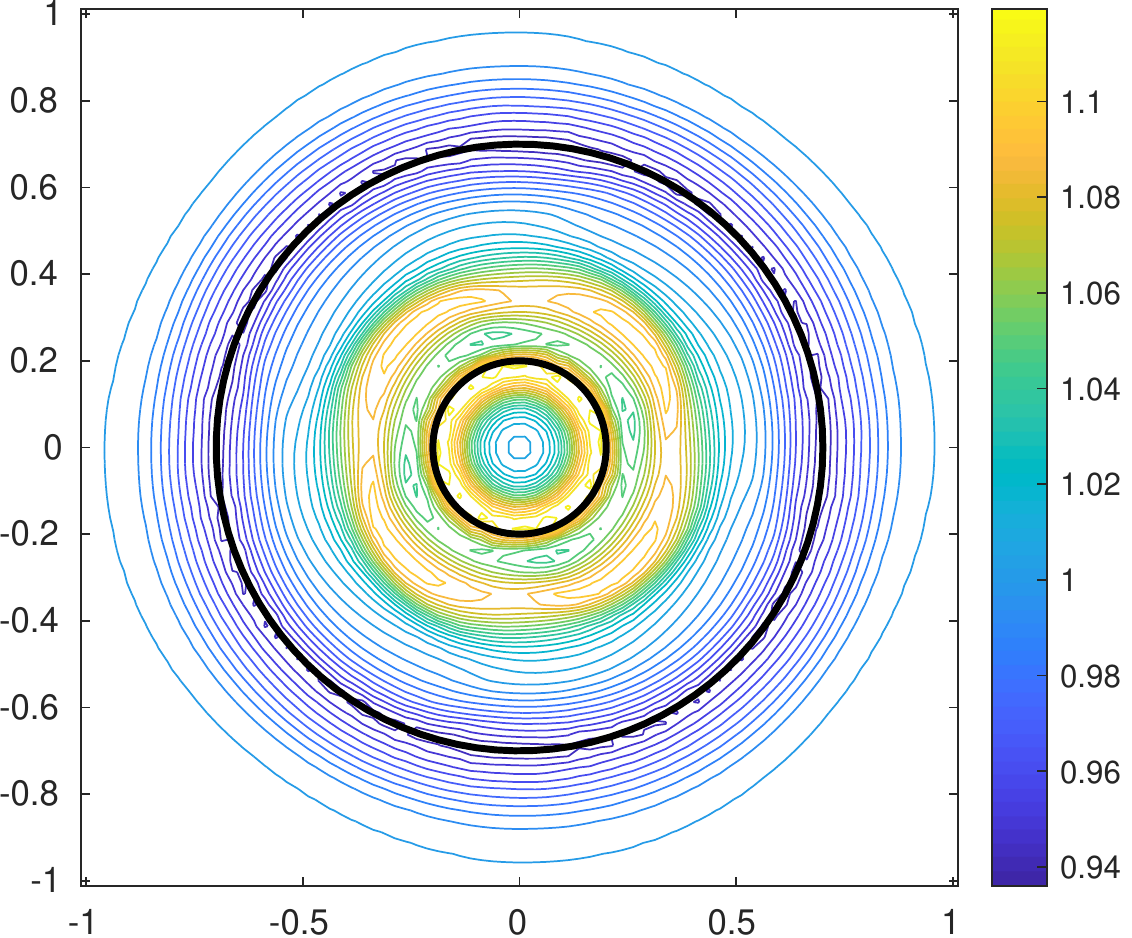}
		\caption{ \bf $\varrho, \varepsilon = 10^{-1}$}
	\end{subfigure}
	\begin{subfigure}{0.2\textwidth}
		\includegraphics[width=\textwidth]{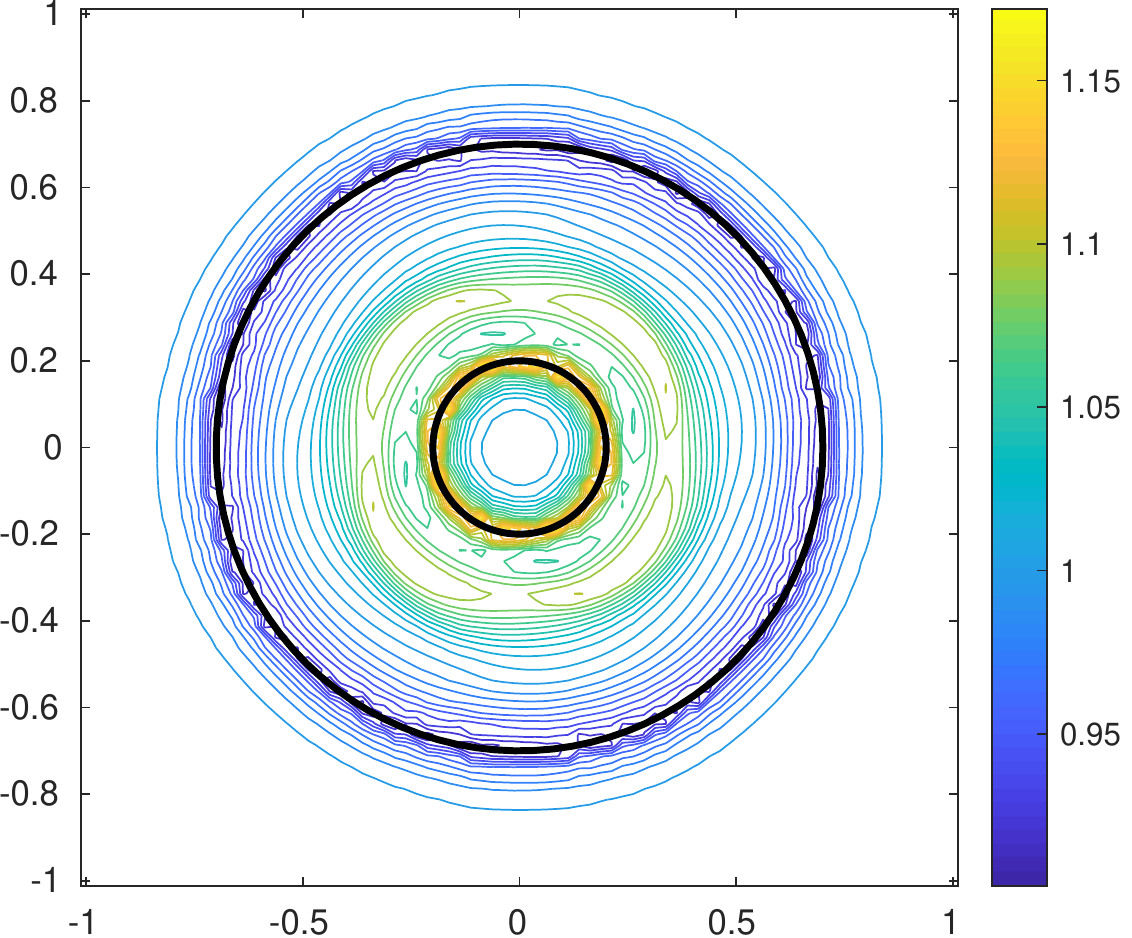}
		\caption{ \bf $\varrho, \varepsilon = 10^{-2}$}
	\end{subfigure}
	\begin{subfigure}{0.2\textwidth}
		\includegraphics[width=\textwidth]{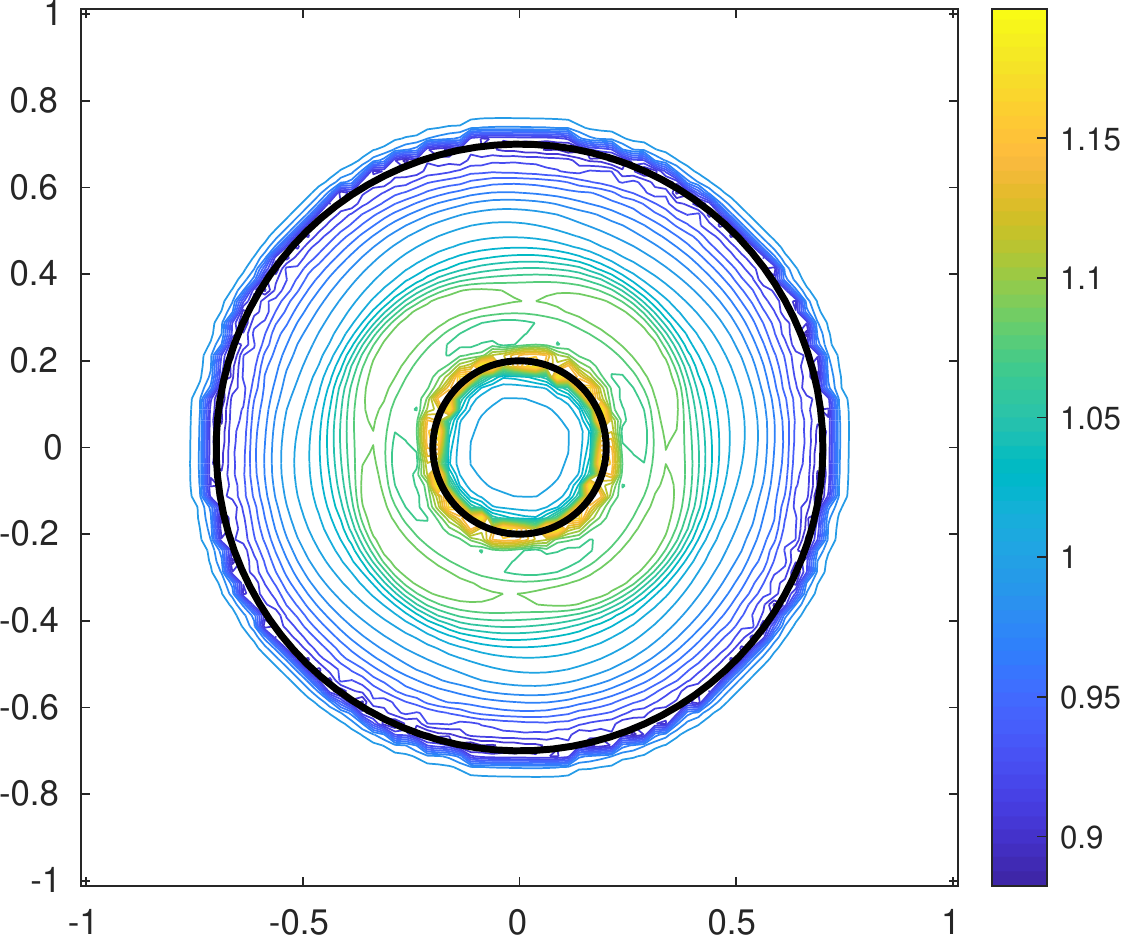}
		\caption{ \bf $\varrho, \varepsilon = 10^{-3}$}
	\end{subfigure}
	\begin{subfigure}{0.2\textwidth}
		\includegraphics[width=\textwidth]{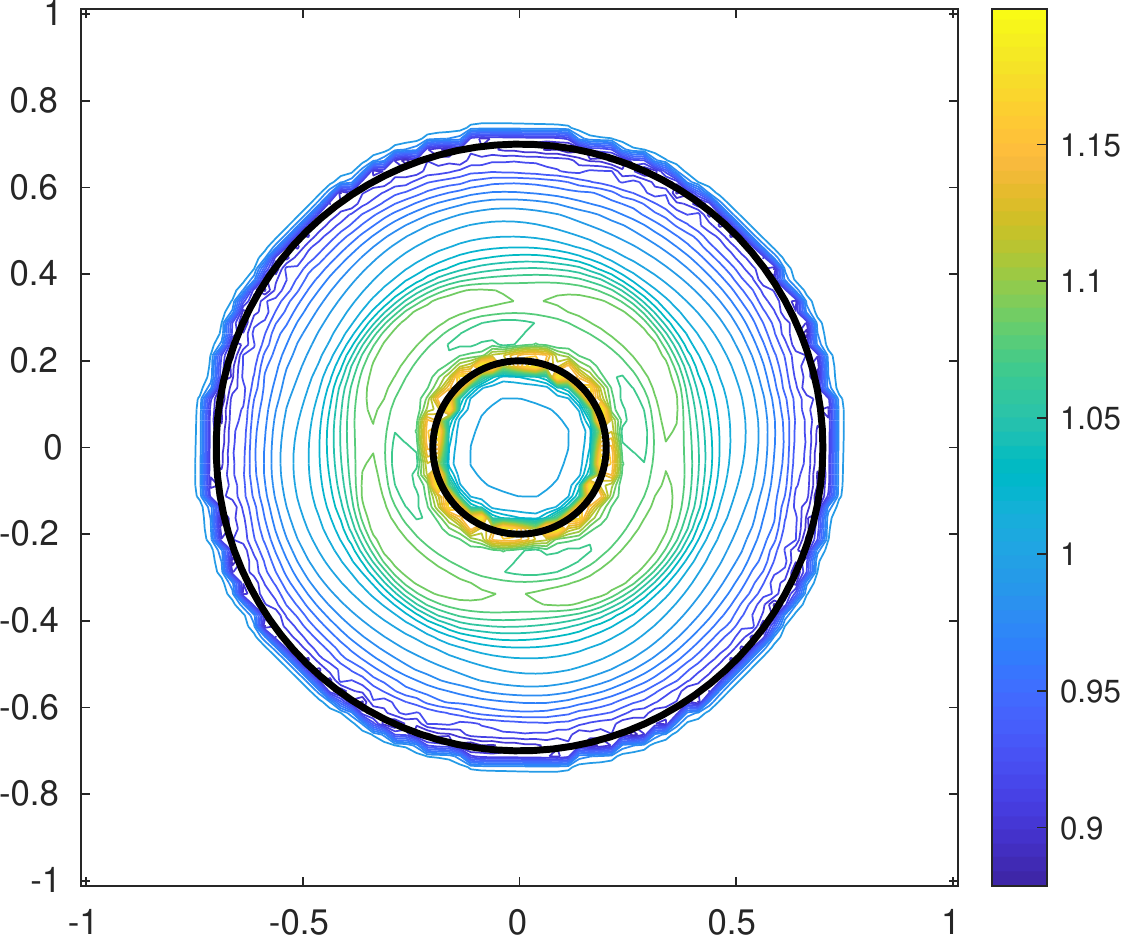}
		\caption{ \bf $\varrho, \varepsilon = 10^{-4}$}
	\end{subfigure}\\
	\begin{subfigure}{0.2\textwidth}
		\includegraphics[width=\textwidth]{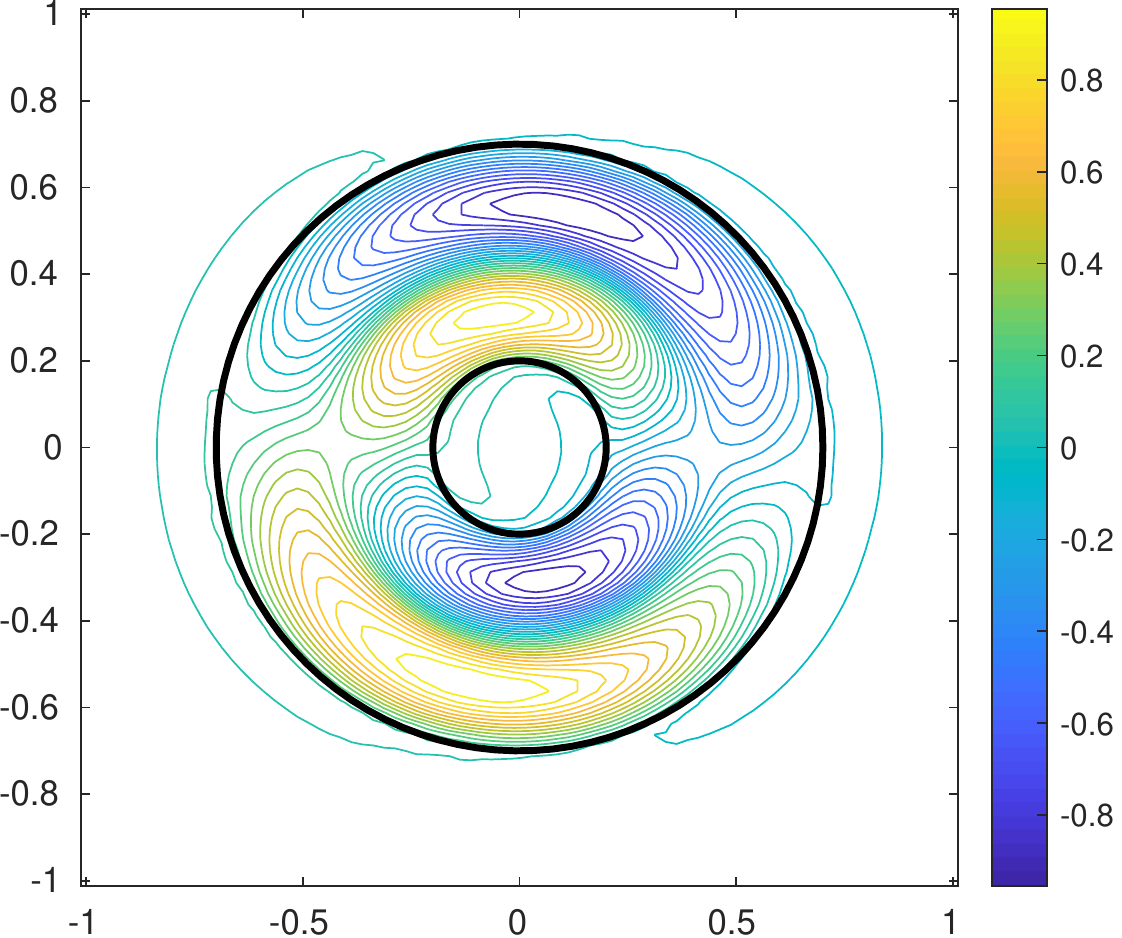}
		\caption{\bf $u_1, \varepsilon = 10^{-1}$}
	\end{subfigure}	
	\begin{subfigure}{0.2\textwidth}
		\includegraphics[width=\textwidth]{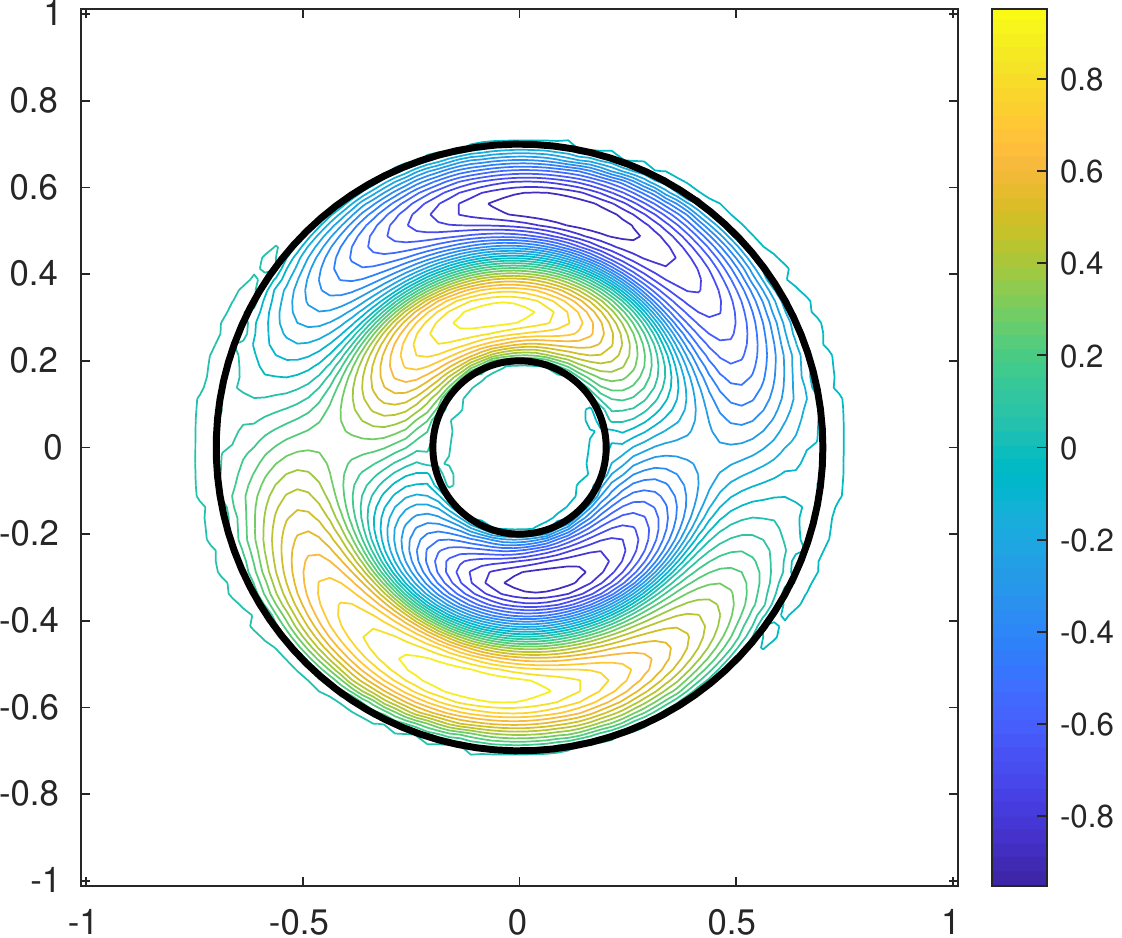}
		\caption{\bf $u_1, \varepsilon = 10^{-2}$}
	\end{subfigure}
	\begin{subfigure}{0.2\textwidth}
		\includegraphics[width=\textwidth]{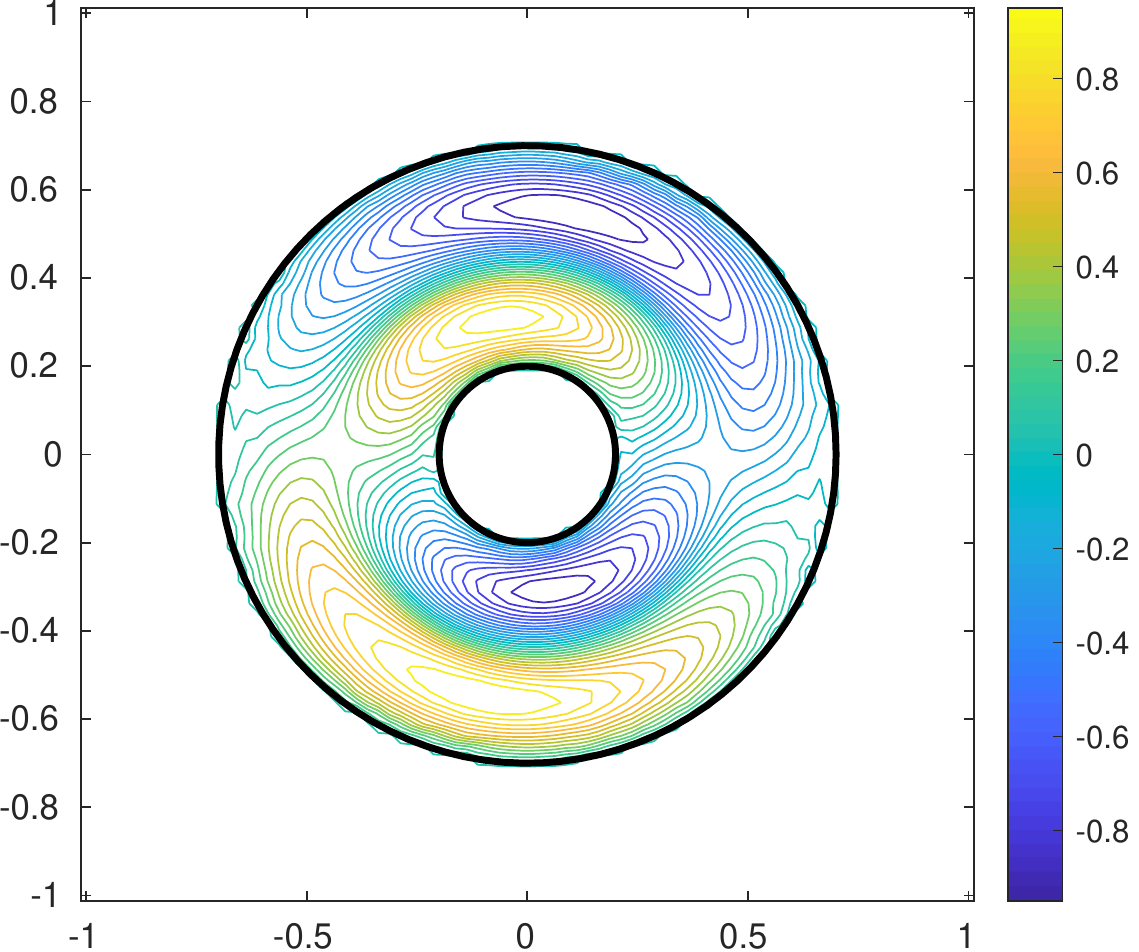}
		\caption{\bf $u_1, \varepsilon = 10^{-3}$}
	\end{subfigure}
	\begin{subfigure}{0.2\textwidth}
		\includegraphics[width=\textwidth]{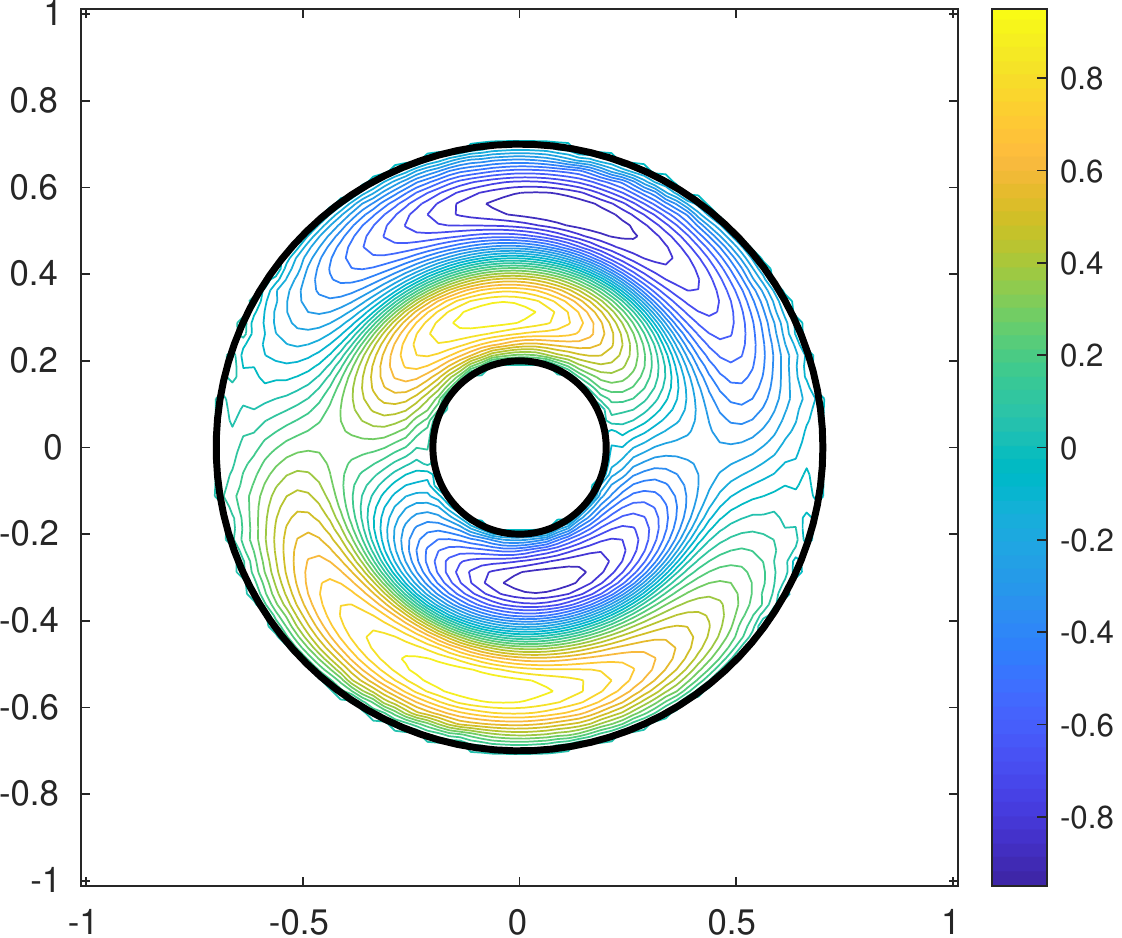}
		\caption{\bf $u_1, \varepsilon = 10^{-4}$}
	\end{subfigure}\\
	\begin{subfigure}{0.2\textwidth}
		\includegraphics[width=\textwidth]{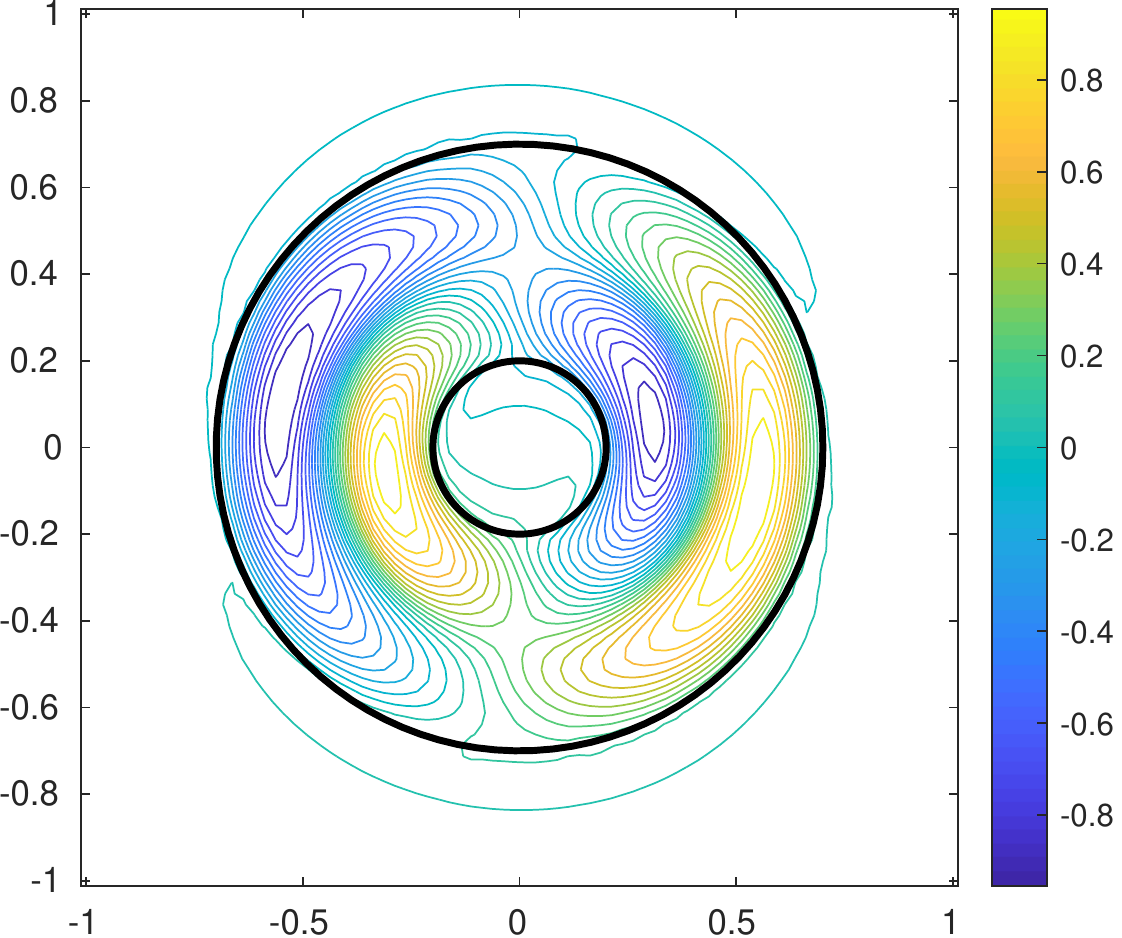}
		\caption{ \bf $u_2, \varepsilon = 10^{-1}$}
	\end{subfigure}	
	\begin{subfigure}{0.2\textwidth}
		\includegraphics[width=\textwidth]{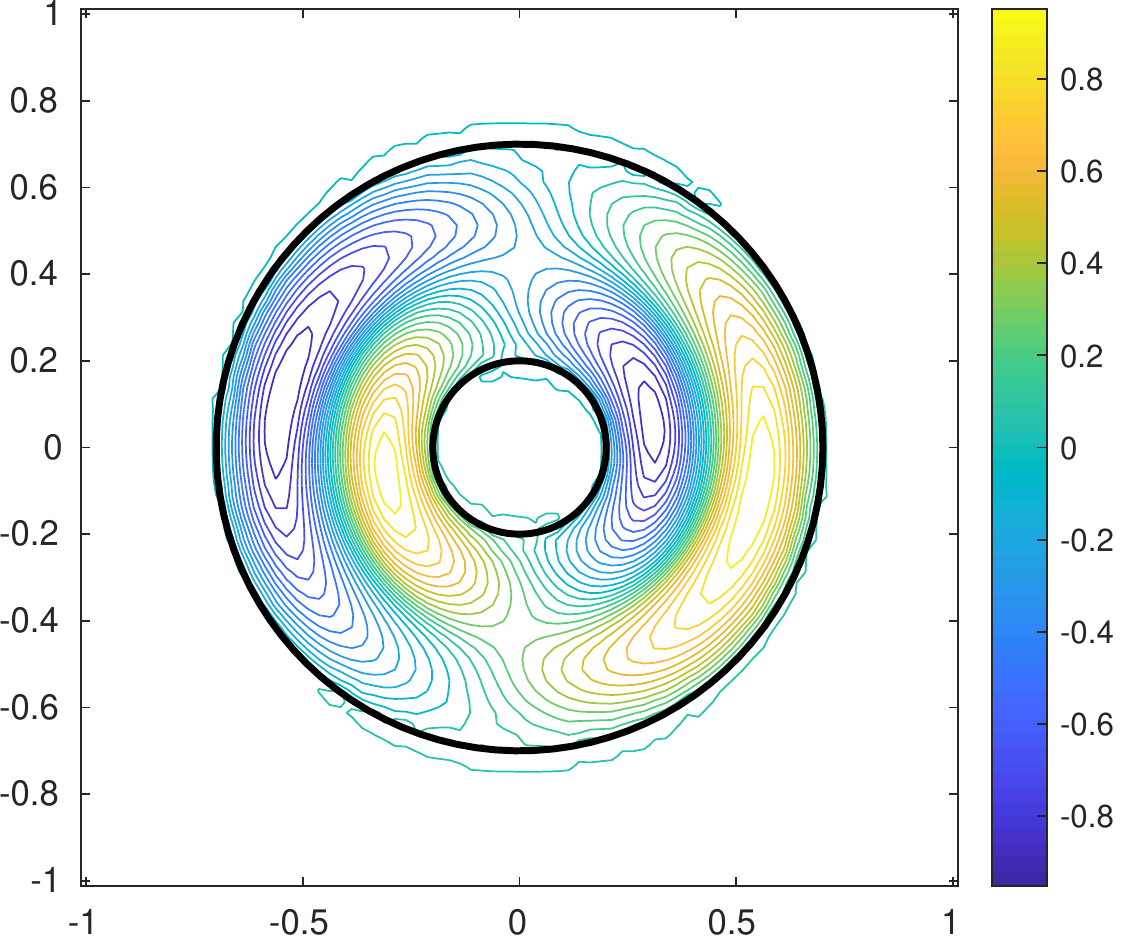}
		\caption{ \bf $u_2, \varepsilon = 10^{-2}$}
	\end{subfigure}
	\begin{subfigure}{0.2\textwidth}
		\includegraphics[width=\textwidth]{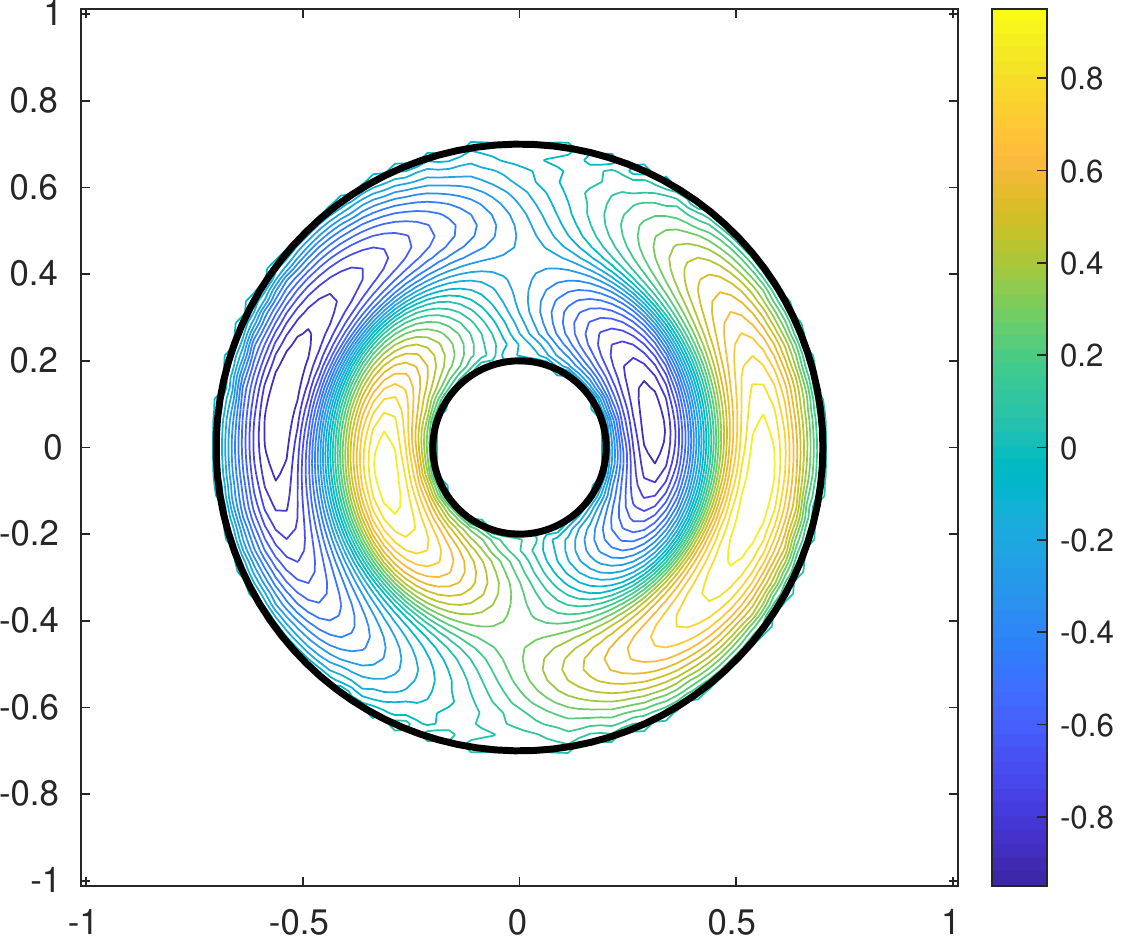}
		\caption{ \bf $u_2, \varepsilon = 10^{-3}$}
	\end{subfigure}
	\begin{subfigure}{0.2\textwidth}
		\includegraphics[width=\textwidth]{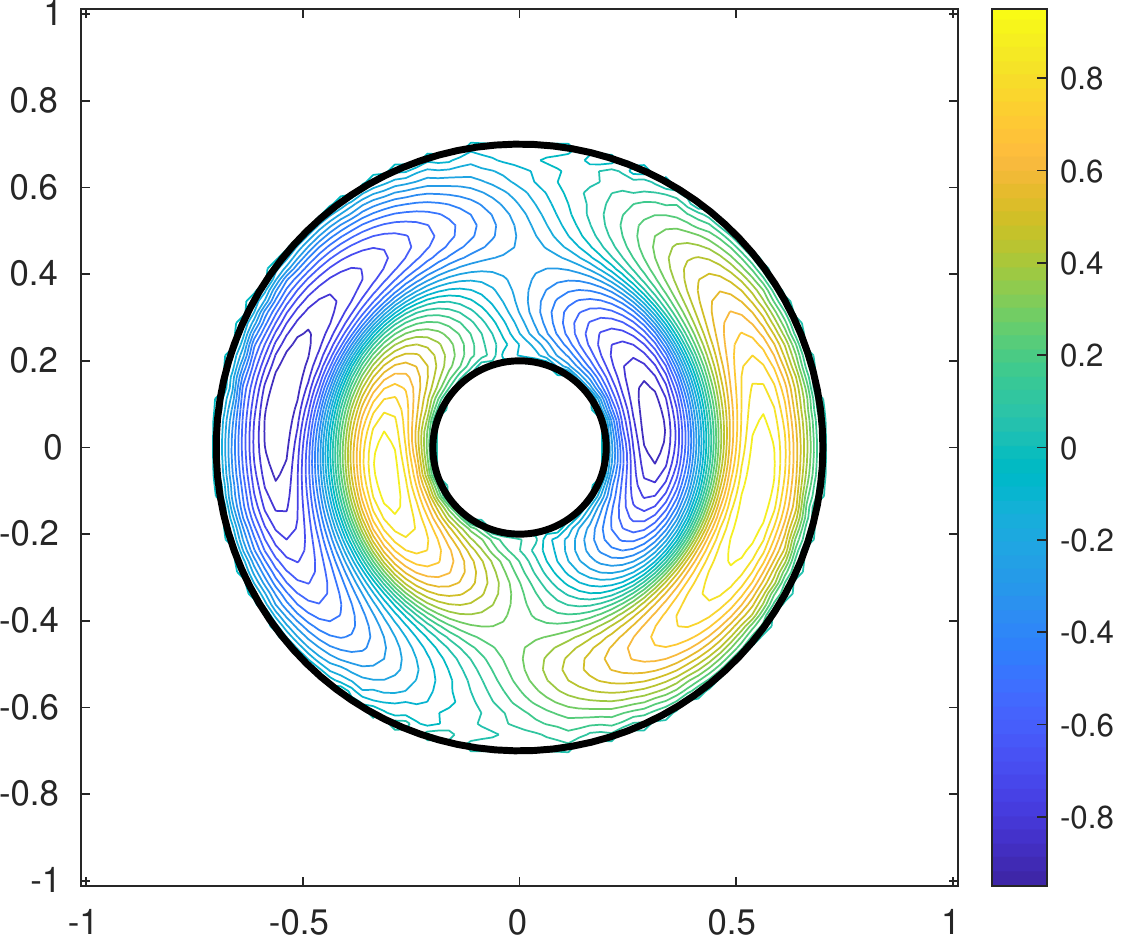}
		\caption{ \bf $u_2, \varepsilon = 10^{-4}$}
	\end{subfigure}\\
	\begin{subfigure}{0.2\textwidth}
		\includegraphics[width=\textwidth]{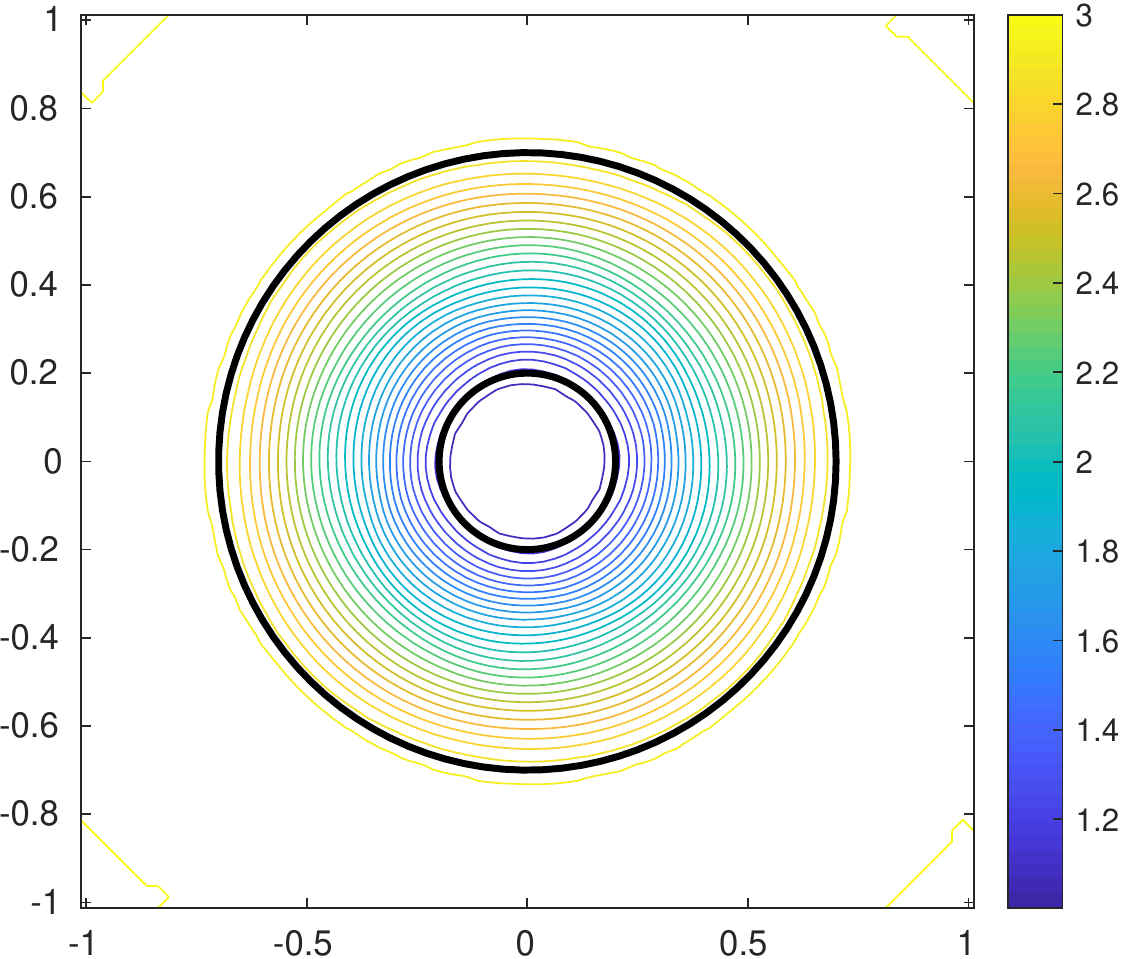}
		\caption{\bf $\vartheta, \varepsilon = 10^{-1}$}
	\end{subfigure}
	\begin{subfigure}{0.2\textwidth}
		\includegraphics[width=\textwidth]{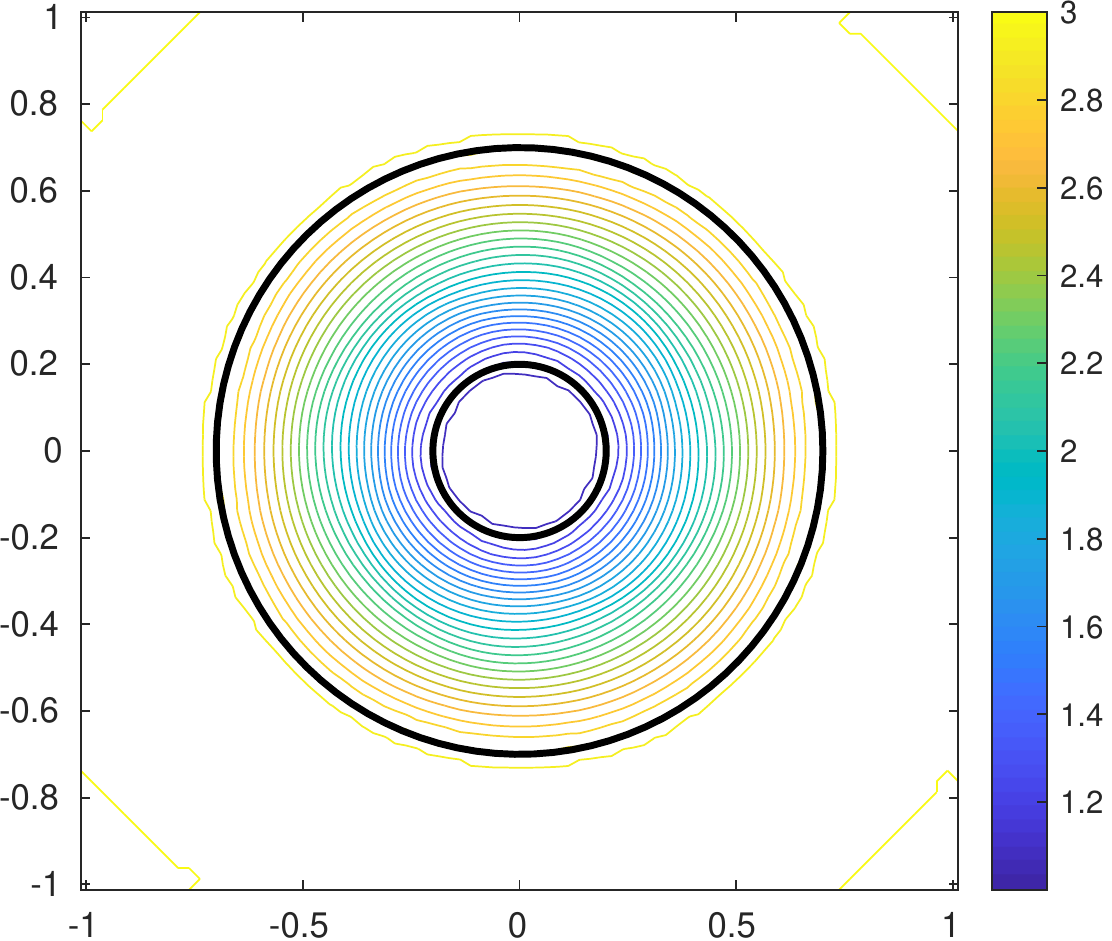}
		\caption{\bf $\vartheta, \varepsilon = 10^{-2}$}
	\end{subfigure}
	\begin{subfigure}{0.2\textwidth}
		\includegraphics[width=\textwidth]{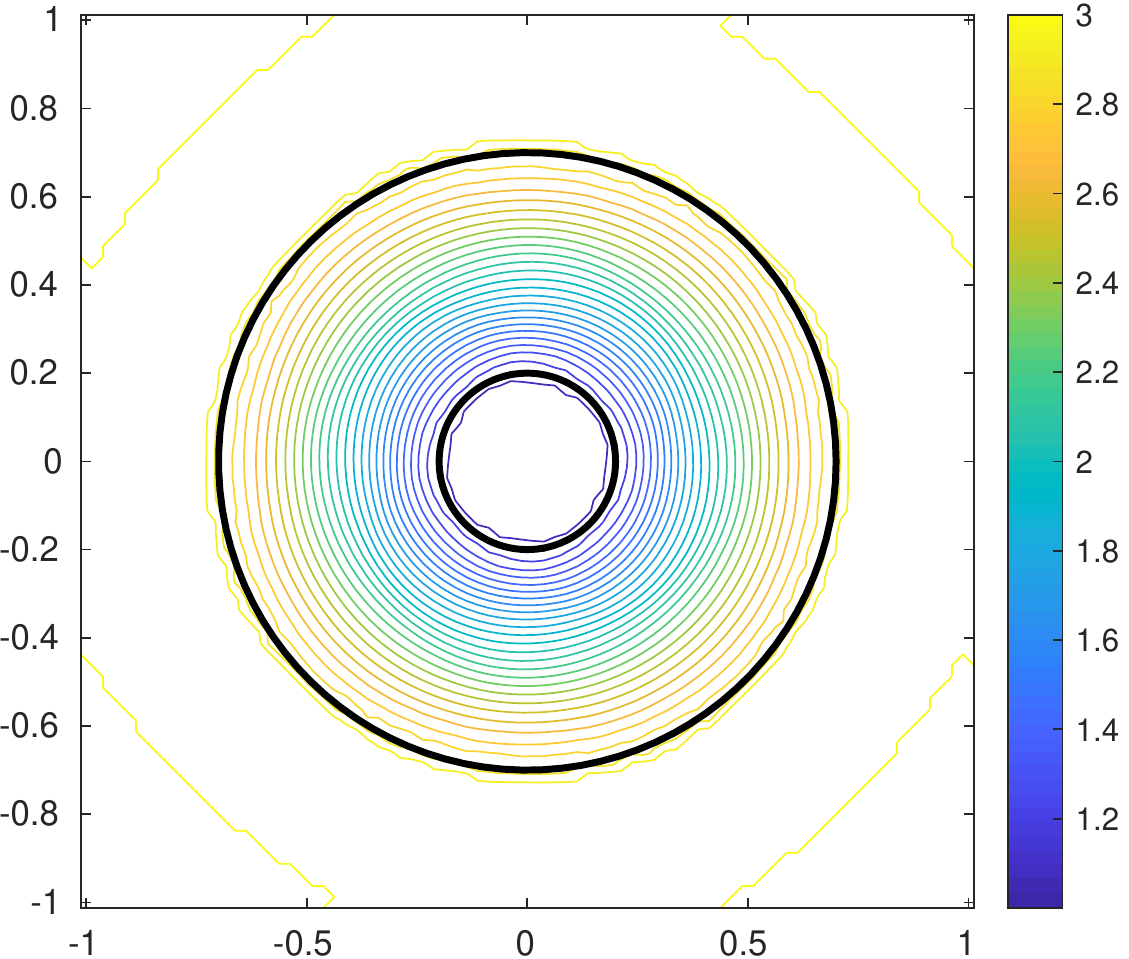}
		\caption{\bf $\vartheta, \varepsilon = 10^{-3}$}
	\end{subfigure}
	\begin{subfigure}{0.2\textwidth}
		\includegraphics[width=\textwidth]{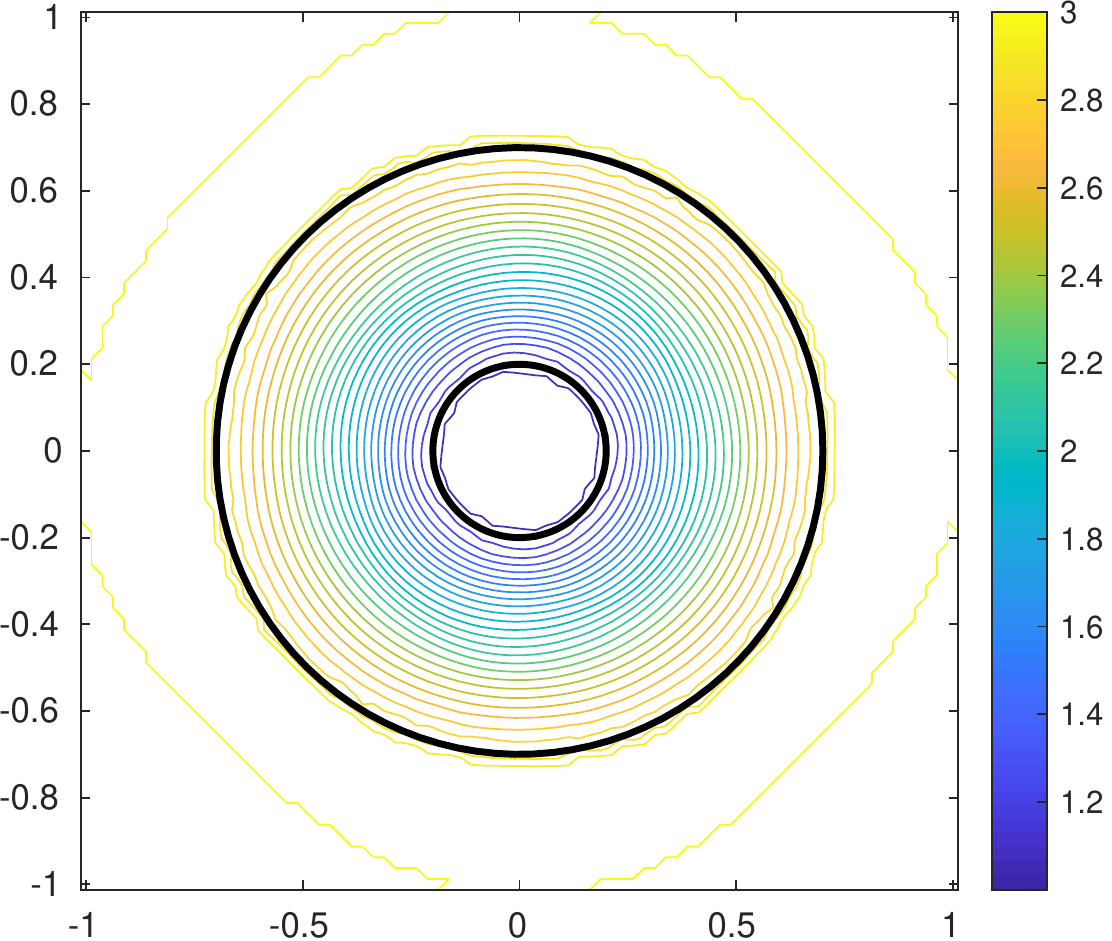}
		\caption{\bf $\vartheta, \varepsilon = 10^{-4}$}
	\end{subfigure}\\
	\begin{subfigure}{0.2\textwidth}
		\includegraphics[width=\textwidth]{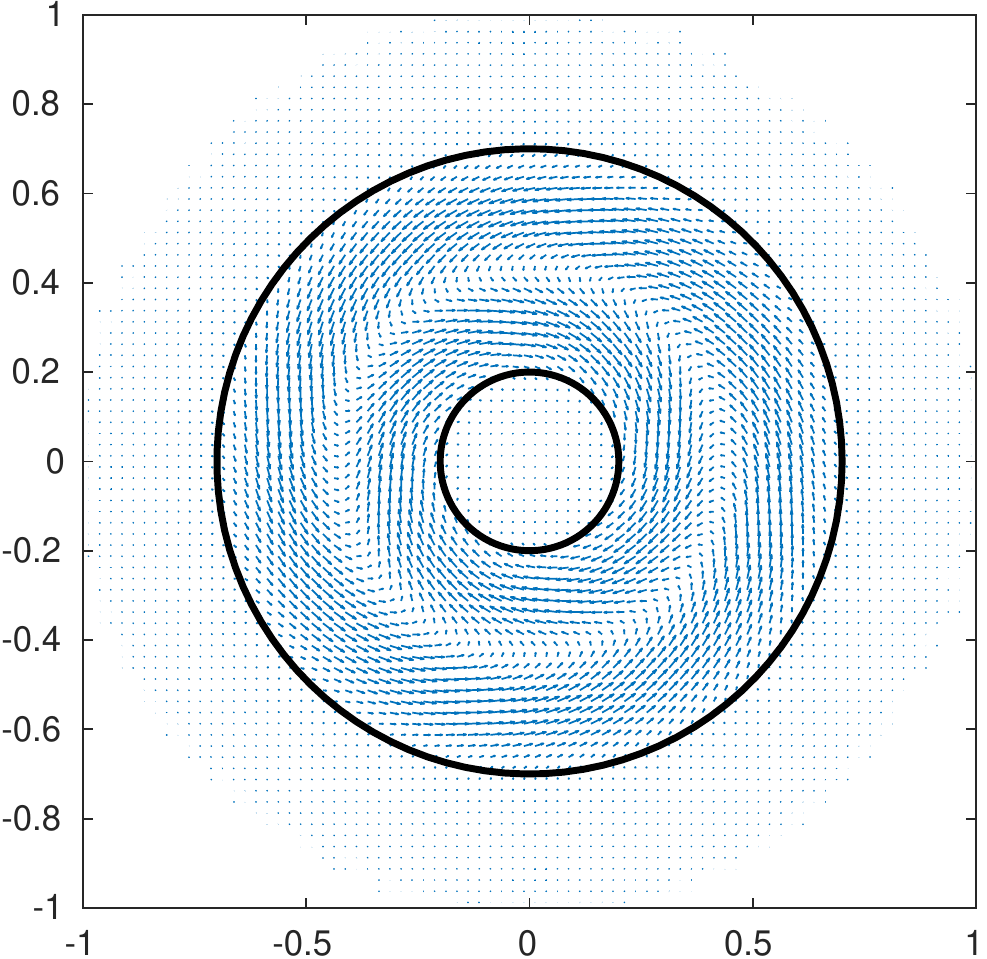}
		\caption{\bf $\vu, \varepsilon = 10^{-1}$}
	\end{subfigure}
	\begin{subfigure}{0.2\textwidth}
		\includegraphics[width=\textwidth]{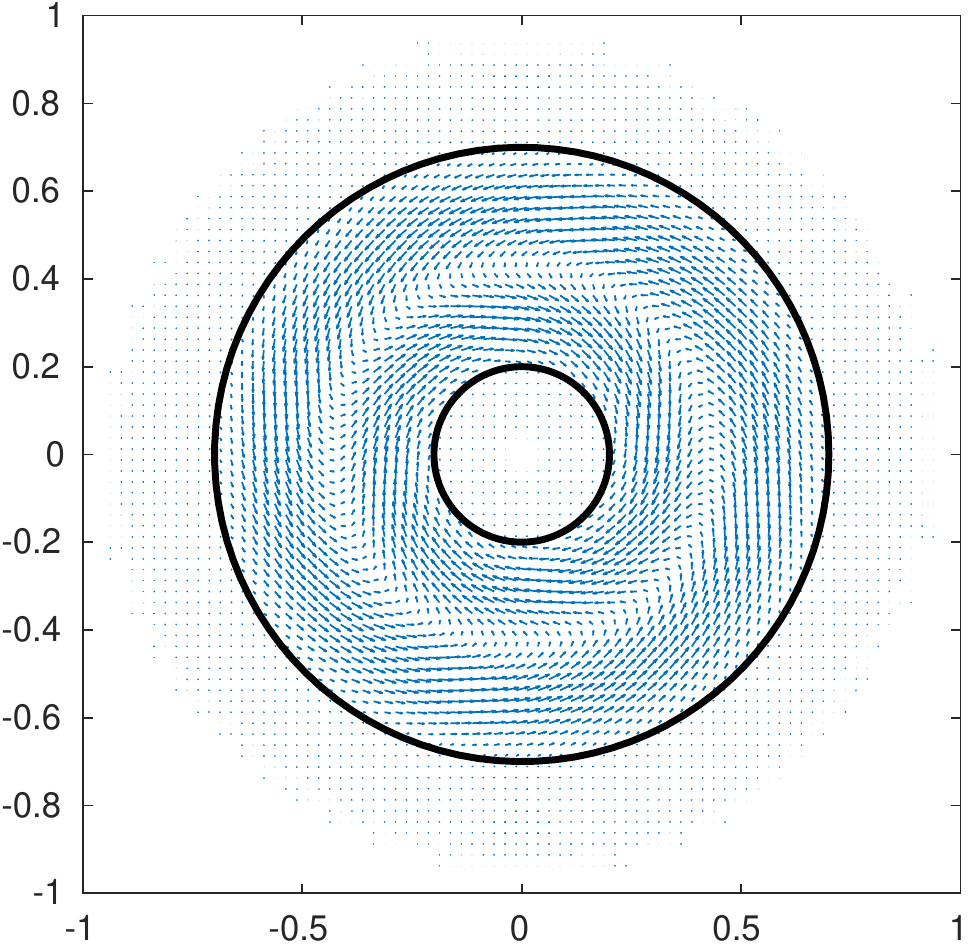}
		\caption{\bf $\vu, \varepsilon = 10^{-2}$}
	\end{subfigure}	
	\begin{subfigure}{0.2\textwidth}
		\includegraphics[width=\textwidth]{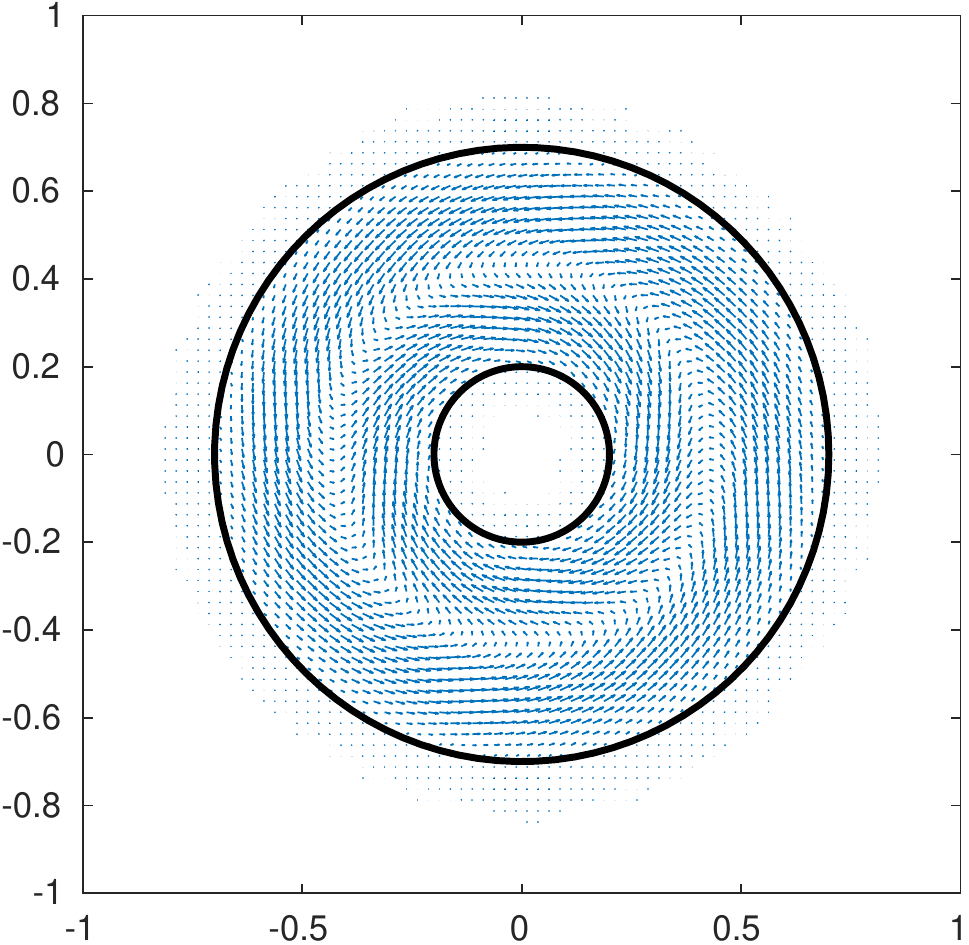}
		\caption{\bf $\vu, \varepsilon = 10^{-3}$}
	\end{subfigure}		
	\begin{subfigure}{0.2\textwidth}
		\includegraphics[width=\textwidth]{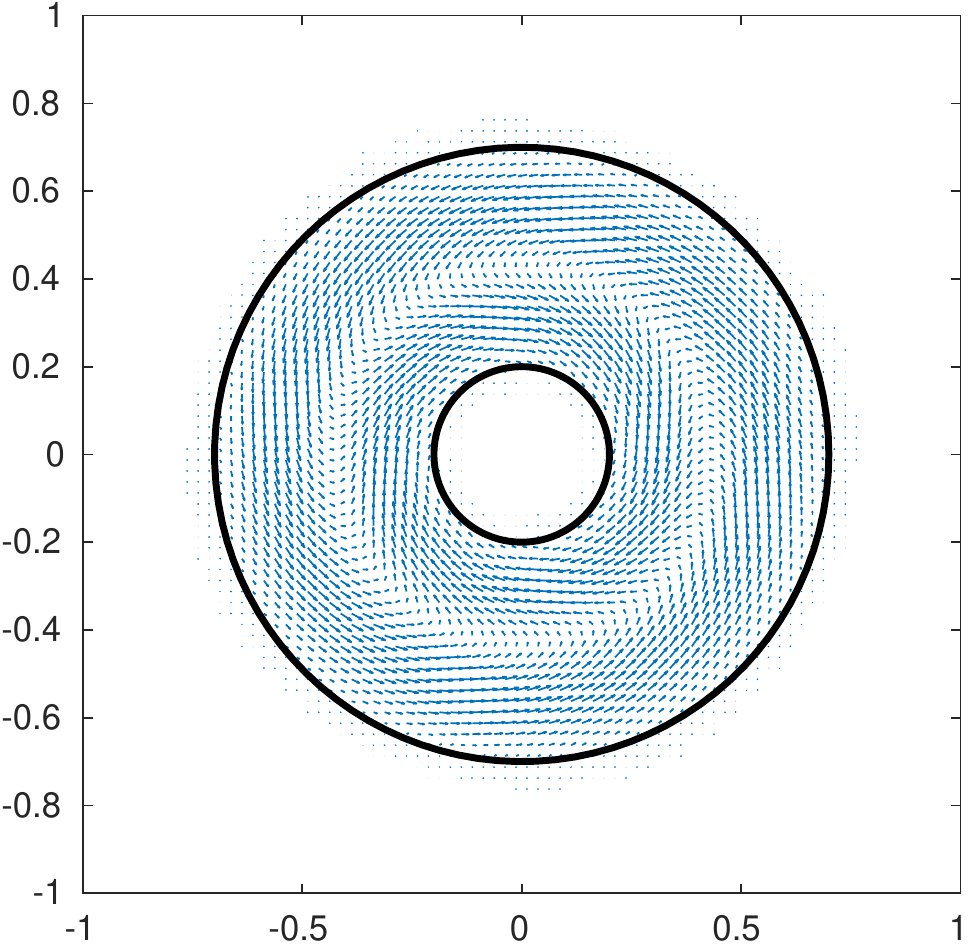}
		\caption{\bf $\vu, \varepsilon = 10^{-4}$}
	\end{subfigure}
	\caption{\small{Experiment~1:  ${U}_h^{\varepsilon}$ with $h = 2/80$ and $\varepsilon = 10^{-1},\dots, 10^{-4}$.}}\label{fig:ex1}
\end{figure}

\subsection{Experiment~2: Ring domain - zero density outside $\Omega$}
In this experiment we use a small initial density outside the fluid domain which approximates the case of zero density outside $\Omega$ studied in the theoretical part.
The rest of the set up  is the same as in Experiment~1, the initial data read
\begin{equation*}
	(\varrho,\vu, \vartheta)(0,x)
	\; = \; \begin{cases}
	(10^{-2},0,0, 1 ) , & x \in B_{0.2}, \\
	\left(1, \frac{ \sin(4\pi (|x|-0.2)) x_2}{|x|} ,-\frac{ \sin(4\pi (|x|-0.2)) x_1}{|x|}, 0.2 + 4|x| \right) , & x \in  \Omega \equiv {B}_{0.7}\setminus B_{0.2}, \\
	(10^{-2},0 ,0,  3) , & x \in \mathbb{T}^2\setminus B_{0.7}.
	\end{cases}
\end{equation*}
Analogously as above,  Figure~\ref{fig:ex2-1} and Figure~\ref{fig:ex2-2} present the errors with respect to  $h$ and $\varepsilon$, respectively.
The convergence rate 1  with respect to both $h$ and $\varepsilon$ is numerically confirmed.
Figure~\ref{fig:ex2} demonstrates the influence of the penalization parameter $\varepsilon = 10^{-1},\dots,10^{-4}$ on the numerical solution computed on the mesh with $80^2$ cells.
Due to small value of the outside density, the fluid tends to flow out of the fluid region $\Omega$ which acts against the penalization and consequently leads to small oscillation near the boundary.

\begin{figure}[htbp]
	\setlength{\abovecaptionskip}{0.cm}
	\setlength{\belowcaptionskip}{-0.cm}
	\centering
	\begin{subfigure}{0.32\textwidth}
		\includegraphics[width=\textwidth]{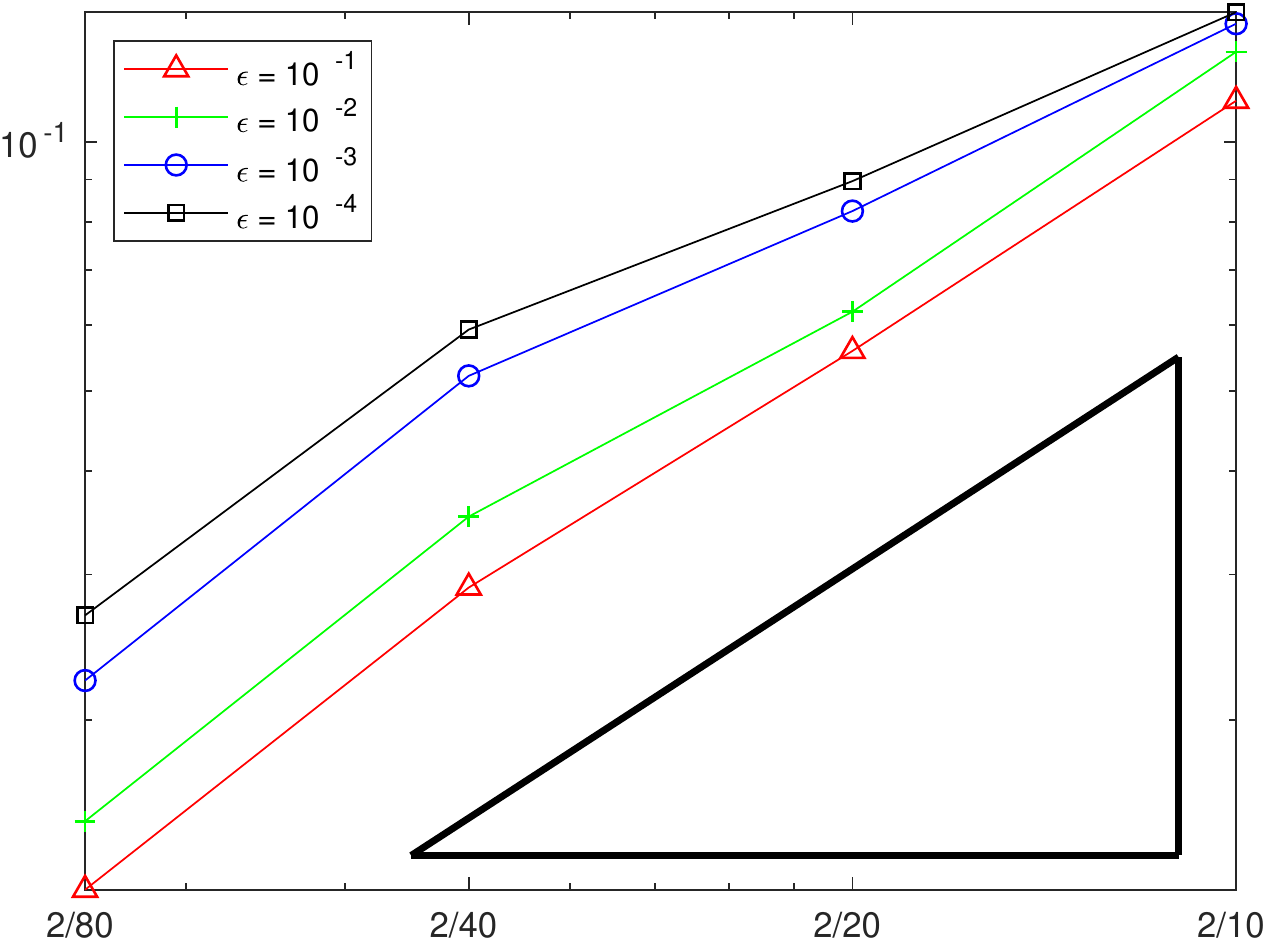}
		\caption{ \bf $\varrho$}
	\end{subfigure}
	\begin{subfigure}{0.32\textwidth}
		\includegraphics[width=\textwidth]{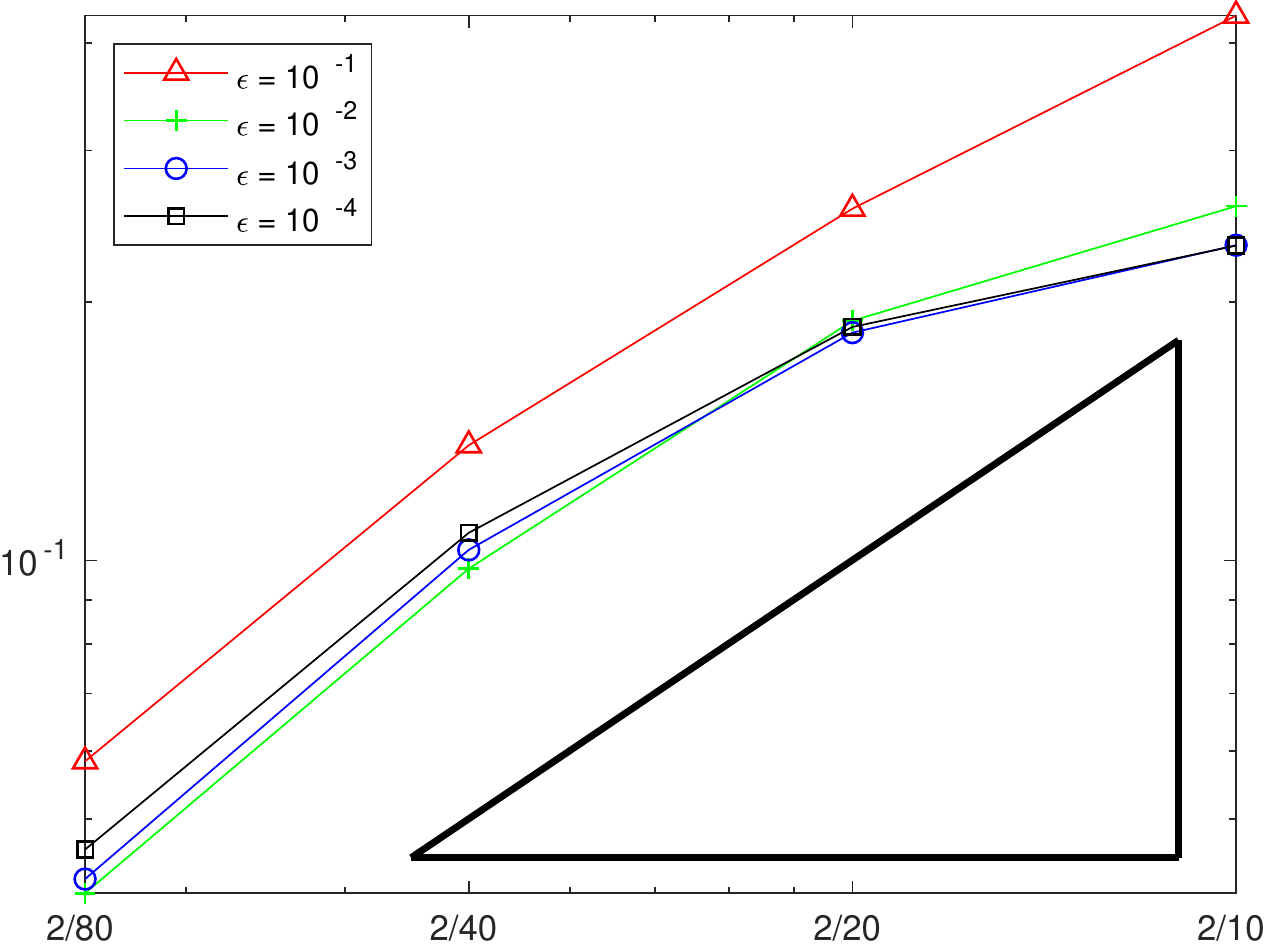}
		\caption{ \bf $\vu$}
	\end{subfigure}
	\begin{subfigure}{0.32\textwidth}
		\includegraphics[width=\textwidth]{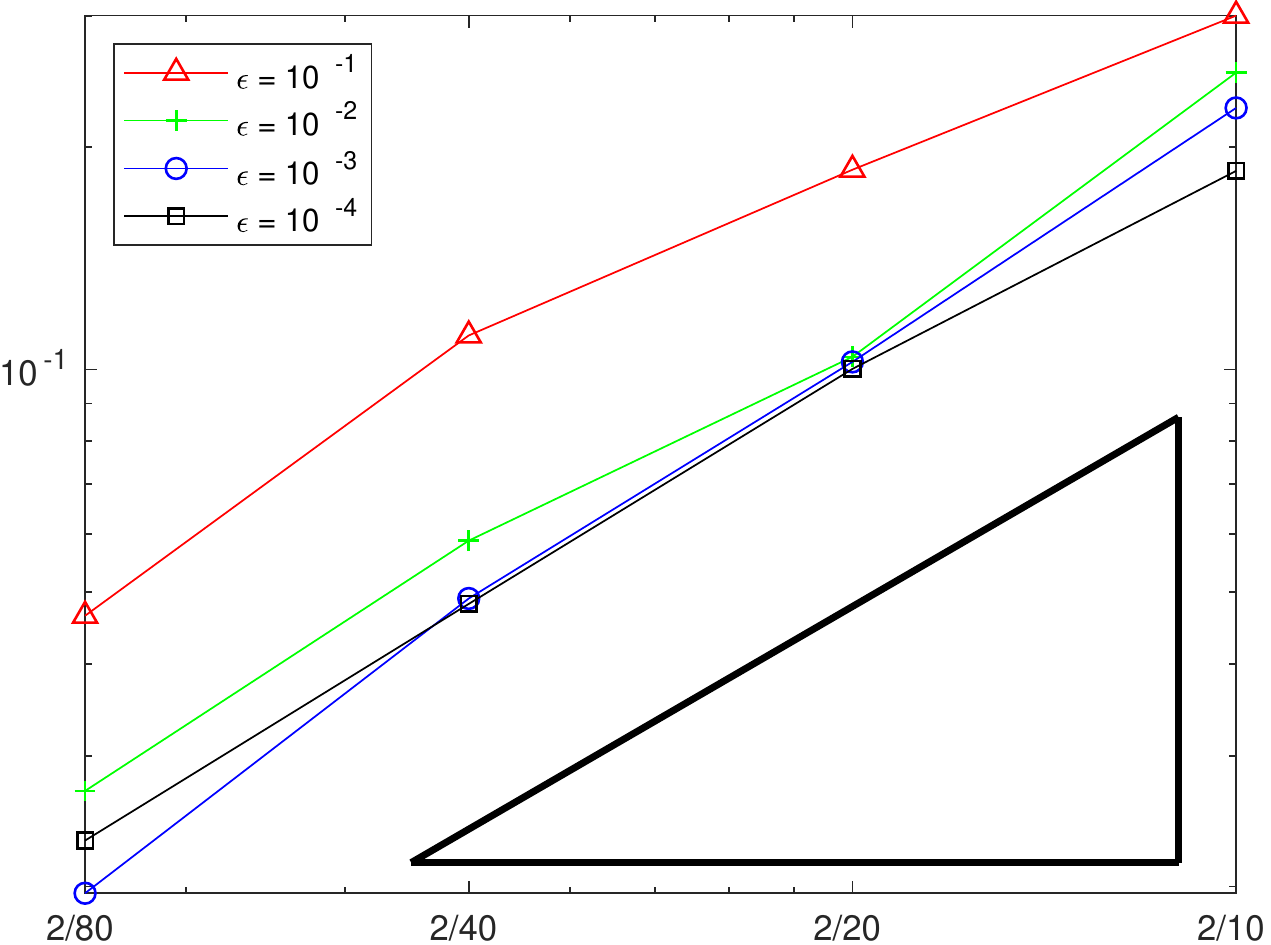}
		\caption{ \bf $\vartheta$}
	\end{subfigure}
	\caption{\small{Experiment~2:  $E(U_h^{\varepsilon})$ errors  with respect to $h.$}}\label{fig:ex2-1}
\end{figure}

\begin{figure}[htbp]
	\setlength{\abovecaptionskip}{0.cm}
	\setlength{\belowcaptionskip}{-0.cm}
	\centering
	\begin{subfigure}{0.32\textwidth}
		\includegraphics[width=\textwidth]{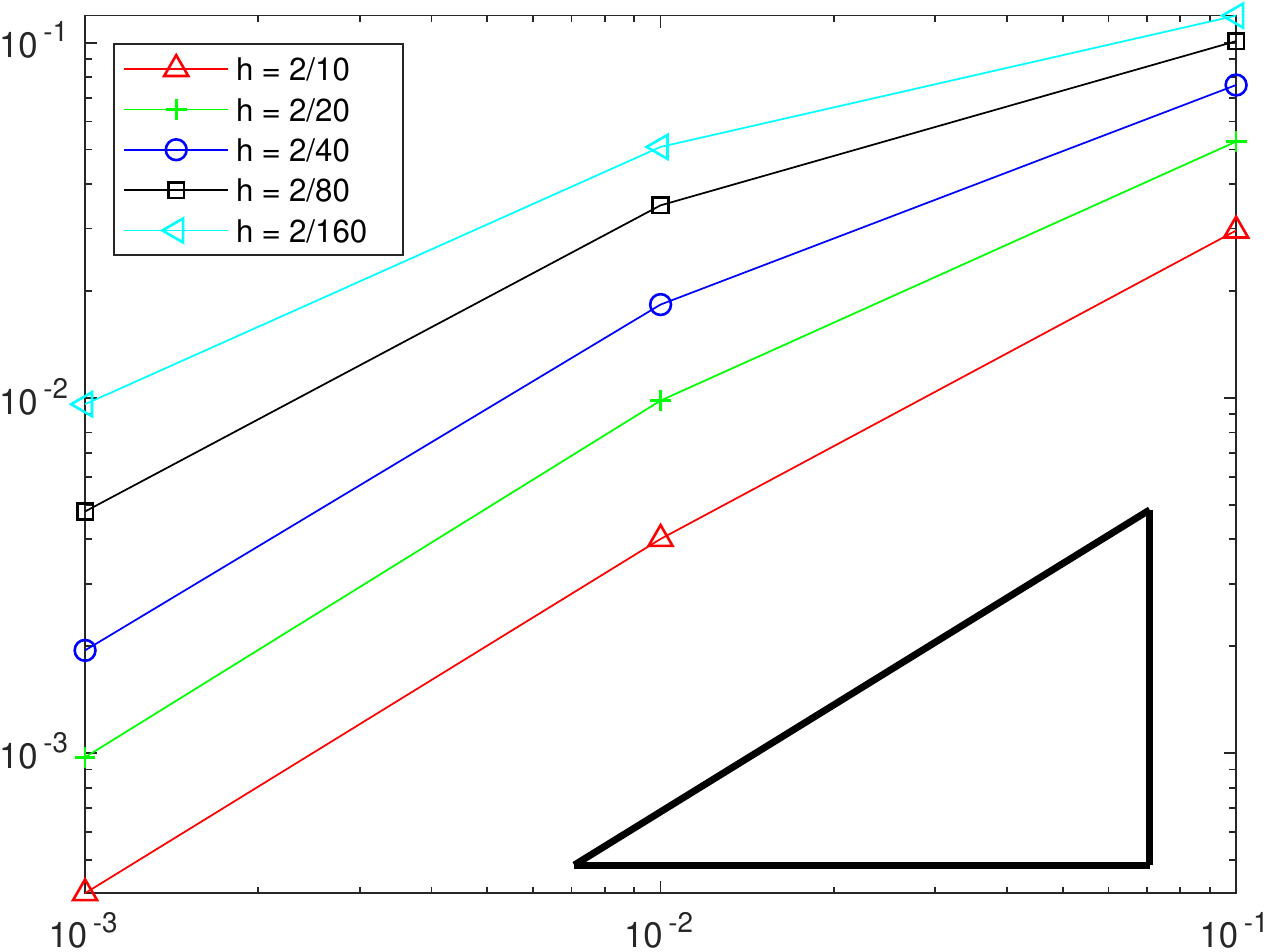}
		\caption{ \bf $\varrho$}
	\end{subfigure}
	\begin{subfigure}{0.32\textwidth}
		\includegraphics[width=\textwidth]{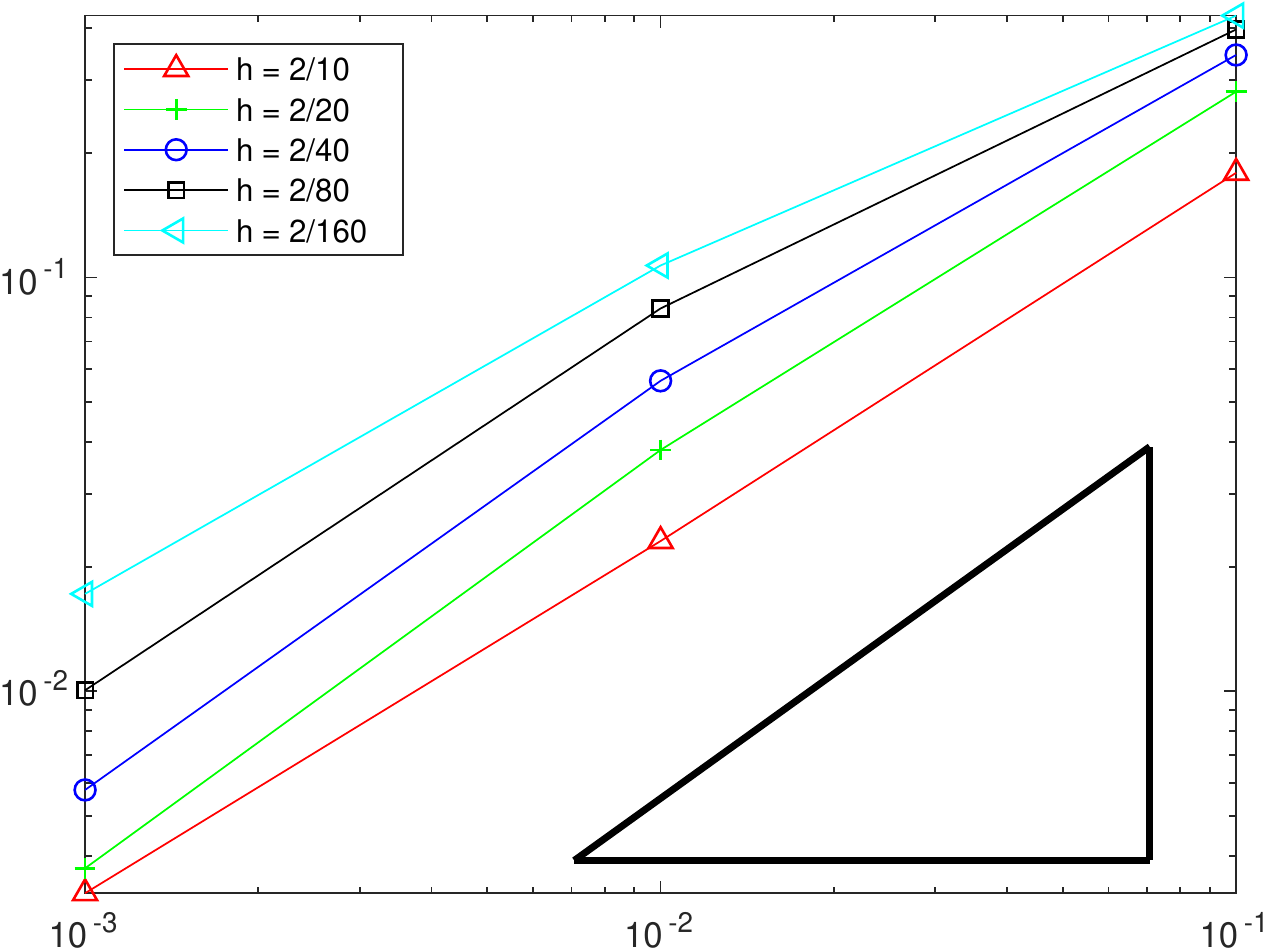}
		\caption{ \bf $\vu$}
	\end{subfigure}
	\begin{subfigure}{0.32\textwidth}
		\includegraphics[width=\textwidth]{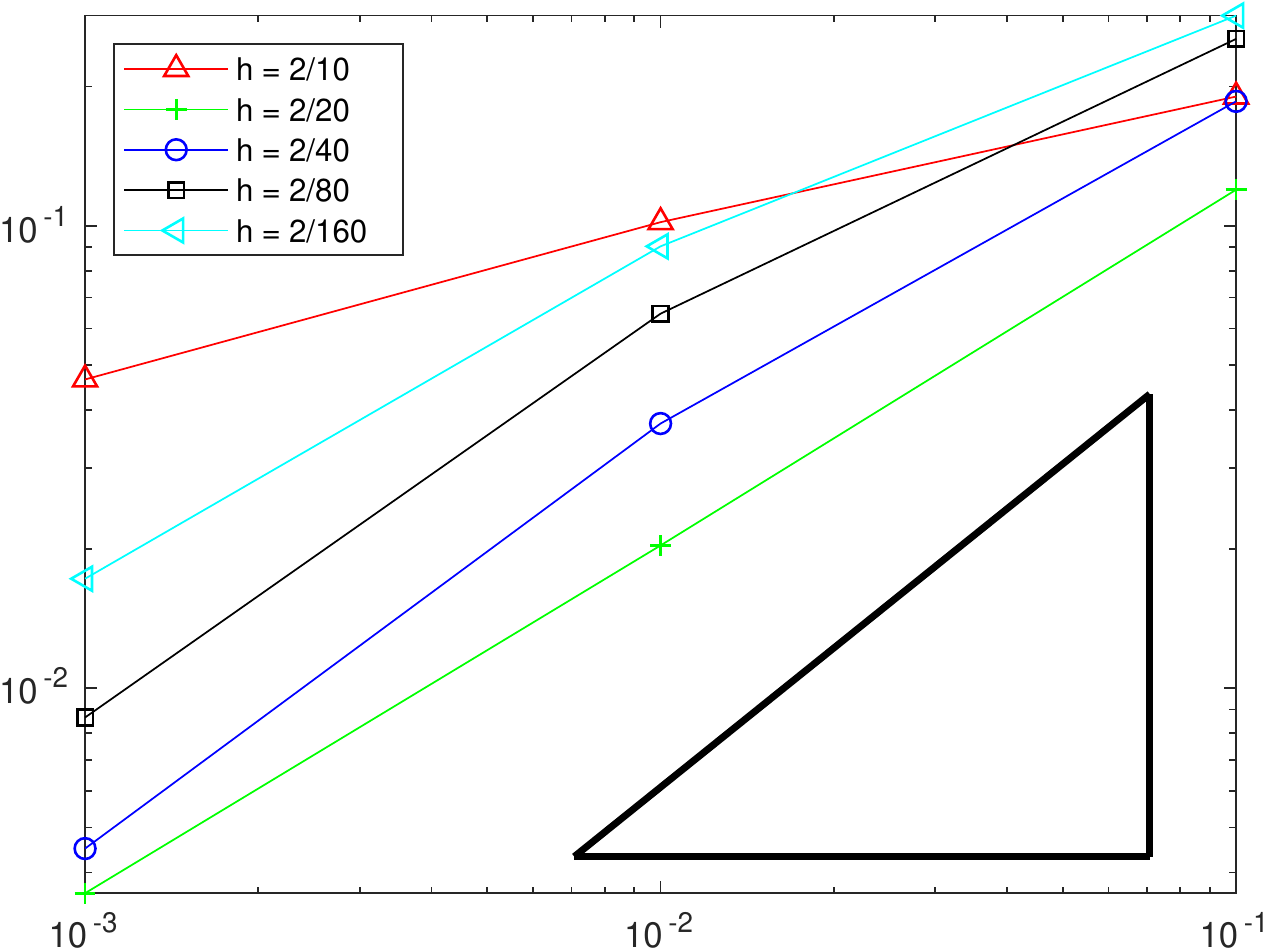}
		\caption{ \bf $\vartheta$}
	\end{subfigure}
	\caption{\small{Experiment~2:  $P(U_h^{\varepsilon})$ errors with respect to $\varepsilon.$}}\label{fig:ex2-2}
\end{figure}

\begin{figure}[htbp]
	\setlength{\abovecaptionskip}{0.cm}
	\setlength{\belowcaptionskip}{-0.cm}
	\centering
	\begin{subfigure}{0.2\textwidth}
		\includegraphics[width=\textwidth]{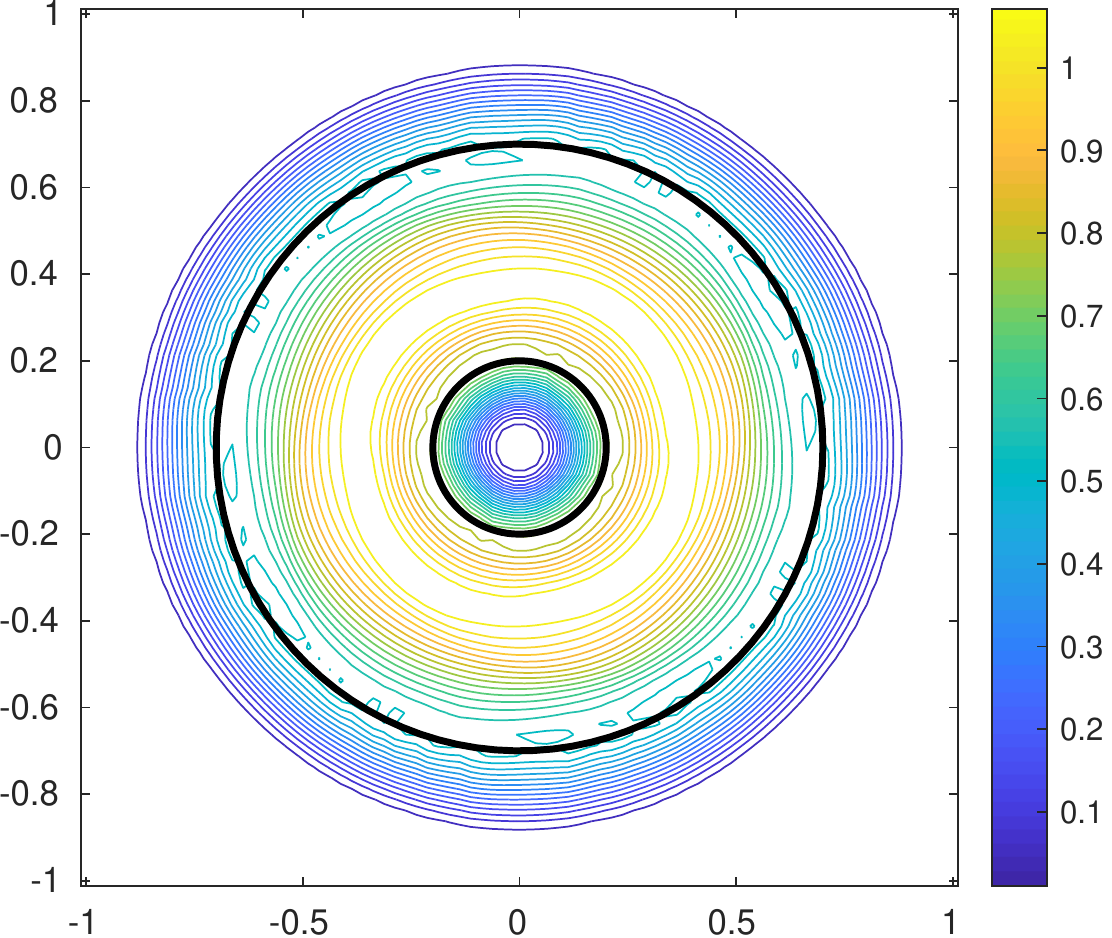}
		\caption{ \bf $\varrho, \varepsilon = 10^{-1}$}
	\end{subfigure}
	\begin{subfigure}{0.2\textwidth}
		\includegraphics[width=\textwidth]{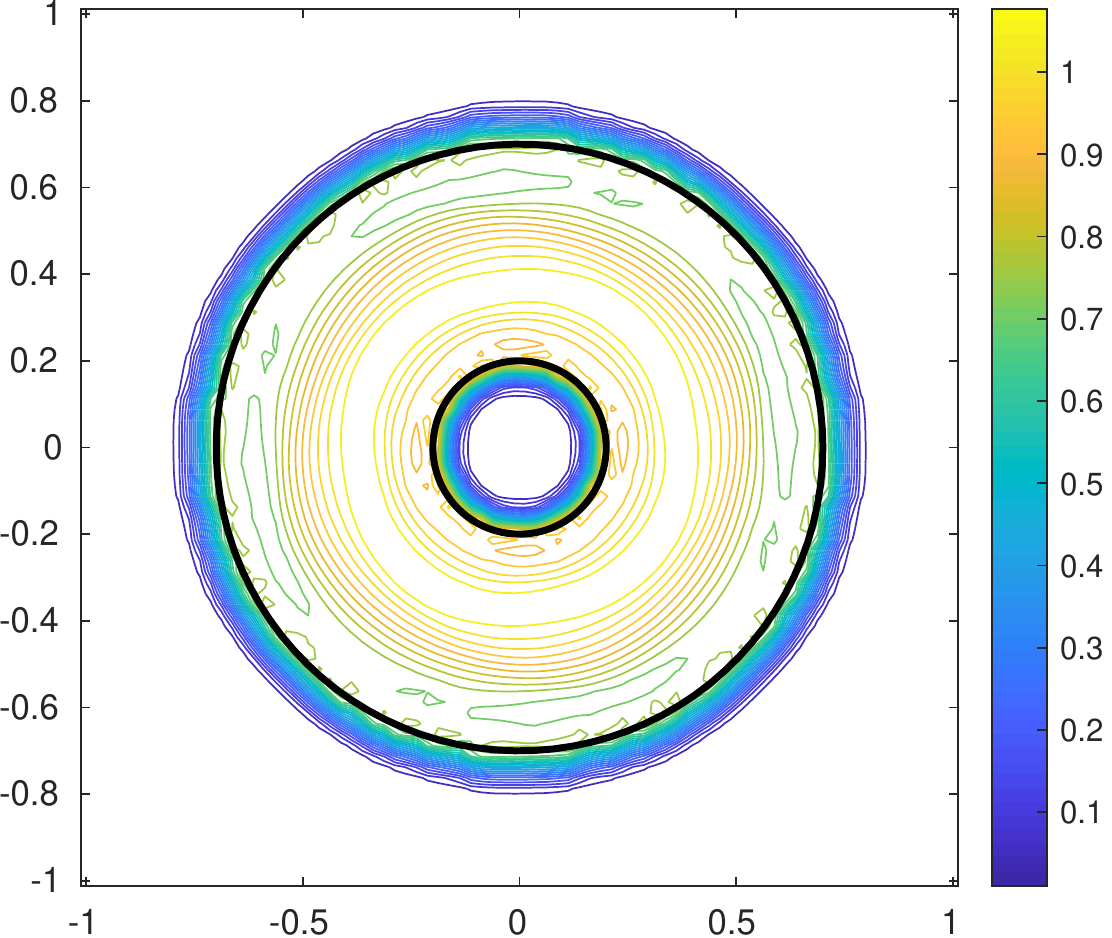}
		\caption{ \bf $\varrho, \varepsilon = 10^{-2}$}
	\end{subfigure}
	\begin{subfigure}{0.2\textwidth}
		\includegraphics[width=\textwidth]{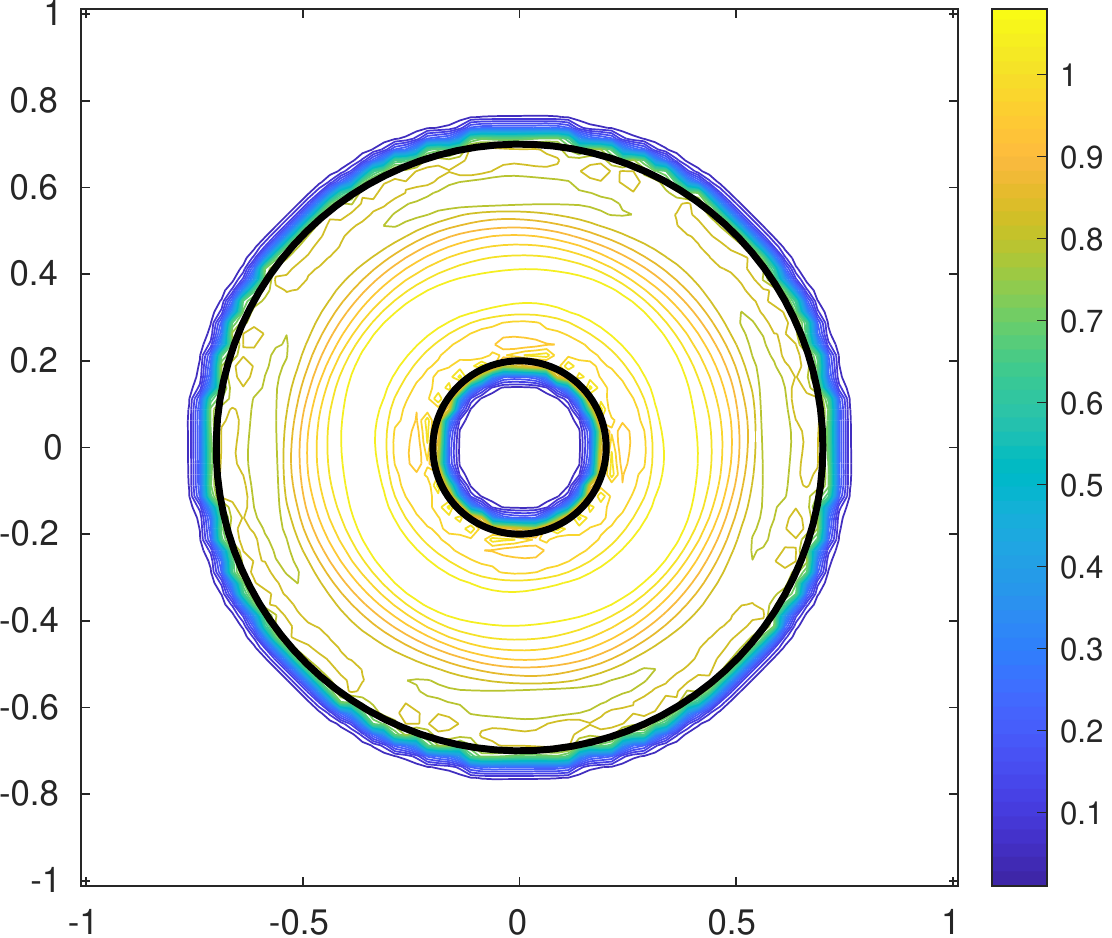}
		\caption{ \bf $\varrho, \varepsilon = 10^{-3}$}
	\end{subfigure}
	\begin{subfigure}{0.2\textwidth}
		\includegraphics[width=\textwidth]{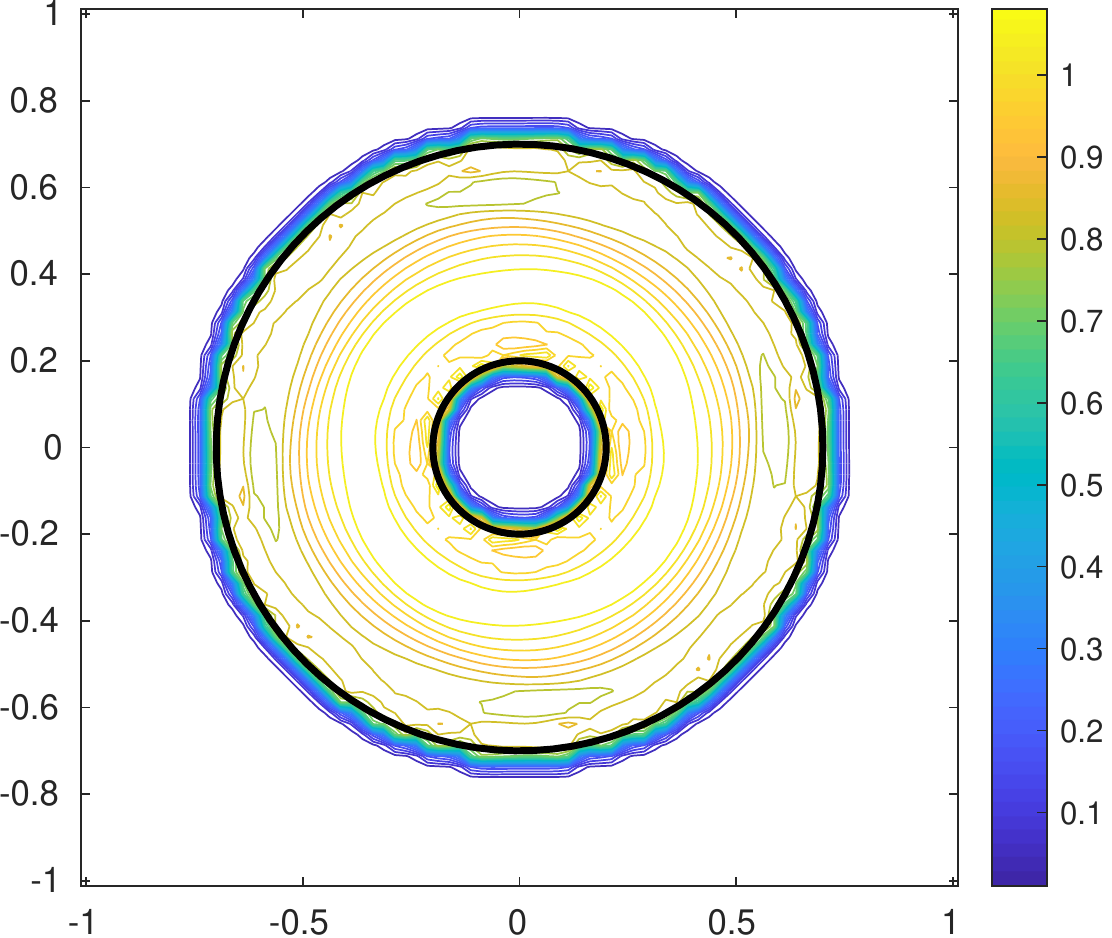}
		\caption{ \bf $\varrho, \varepsilon = 10^{-4}$}
	\end{subfigure}\\
	\begin{subfigure}{0.2\textwidth}
		\includegraphics[width=\textwidth]{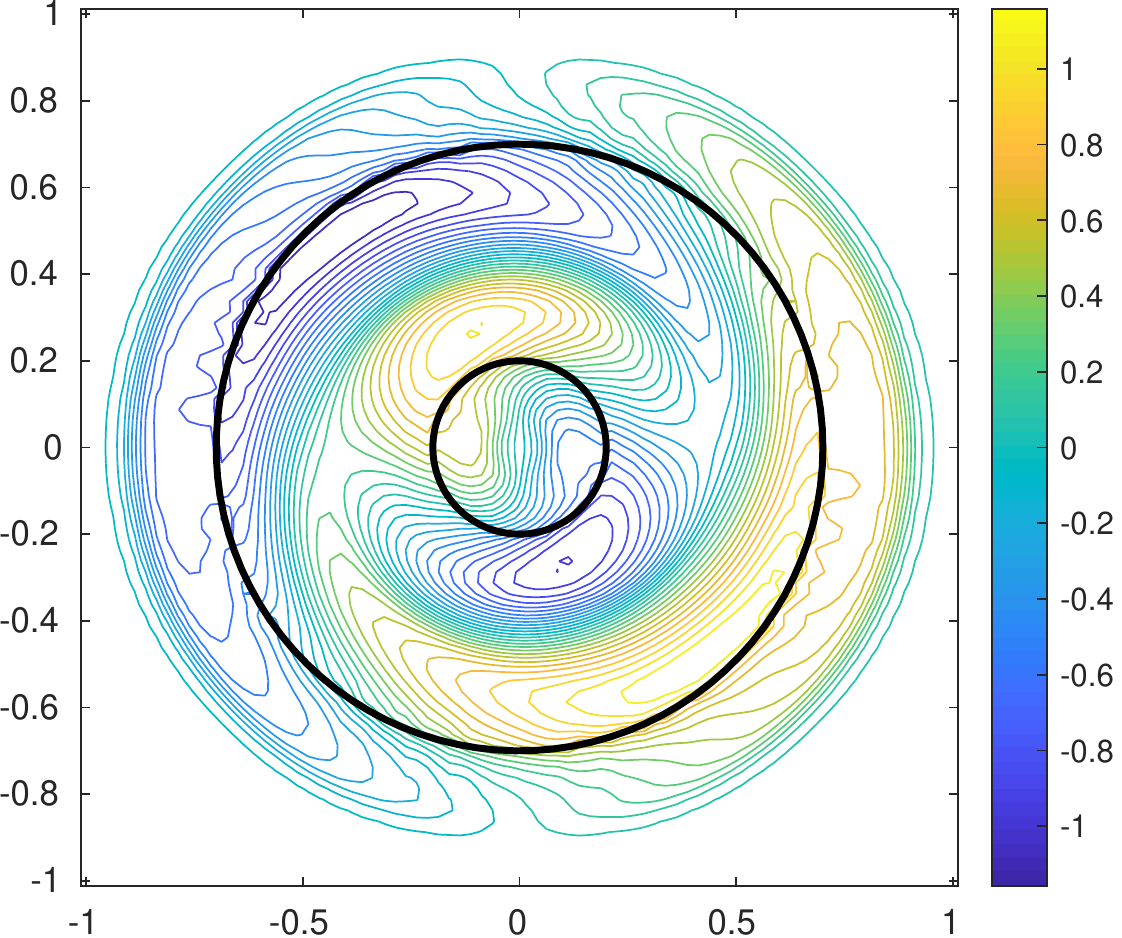}
		\caption{\bf $u_1, \varepsilon = 10^{-1}$}
	\end{subfigure}	
	\begin{subfigure}{0.2\textwidth}
		\includegraphics[width=\textwidth]{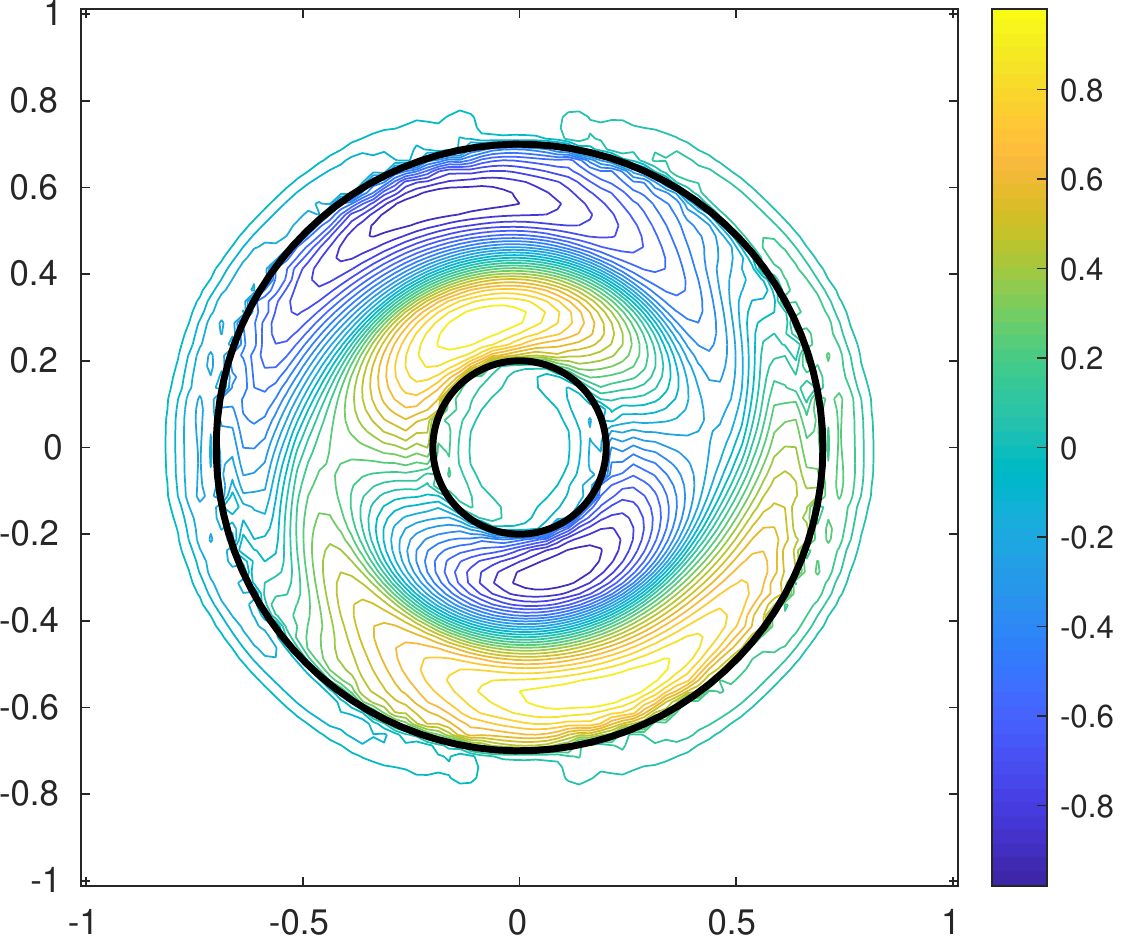}
		\caption{\bf $u_1, \varepsilon = 10^{-2}$}
	\end{subfigure}
	\begin{subfigure}{0.2\textwidth}
		\includegraphics[width=\textwidth]{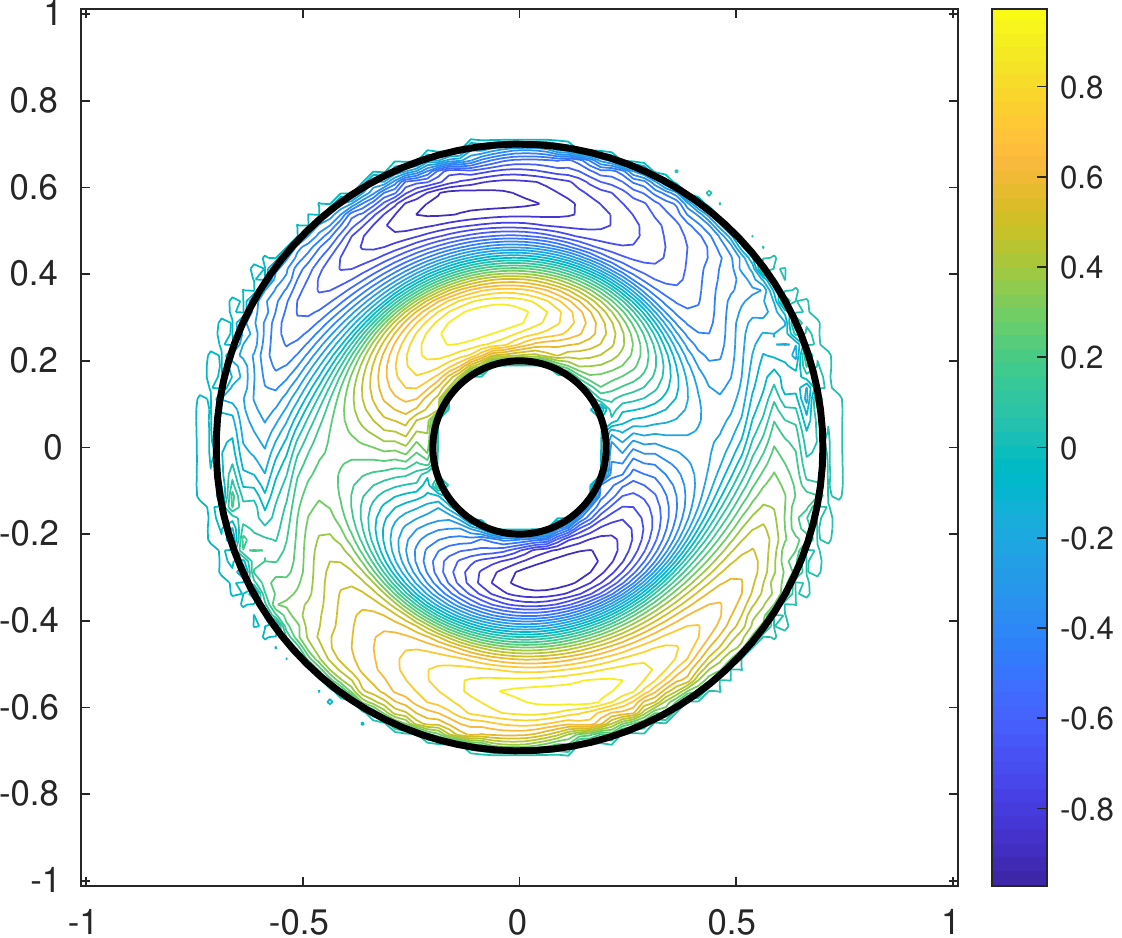}
		\caption{\bf $u_1, \varepsilon = 10^{-3}$}
	\end{subfigure}
	\begin{subfigure}{0.2\textwidth}
		\includegraphics[width=\textwidth]{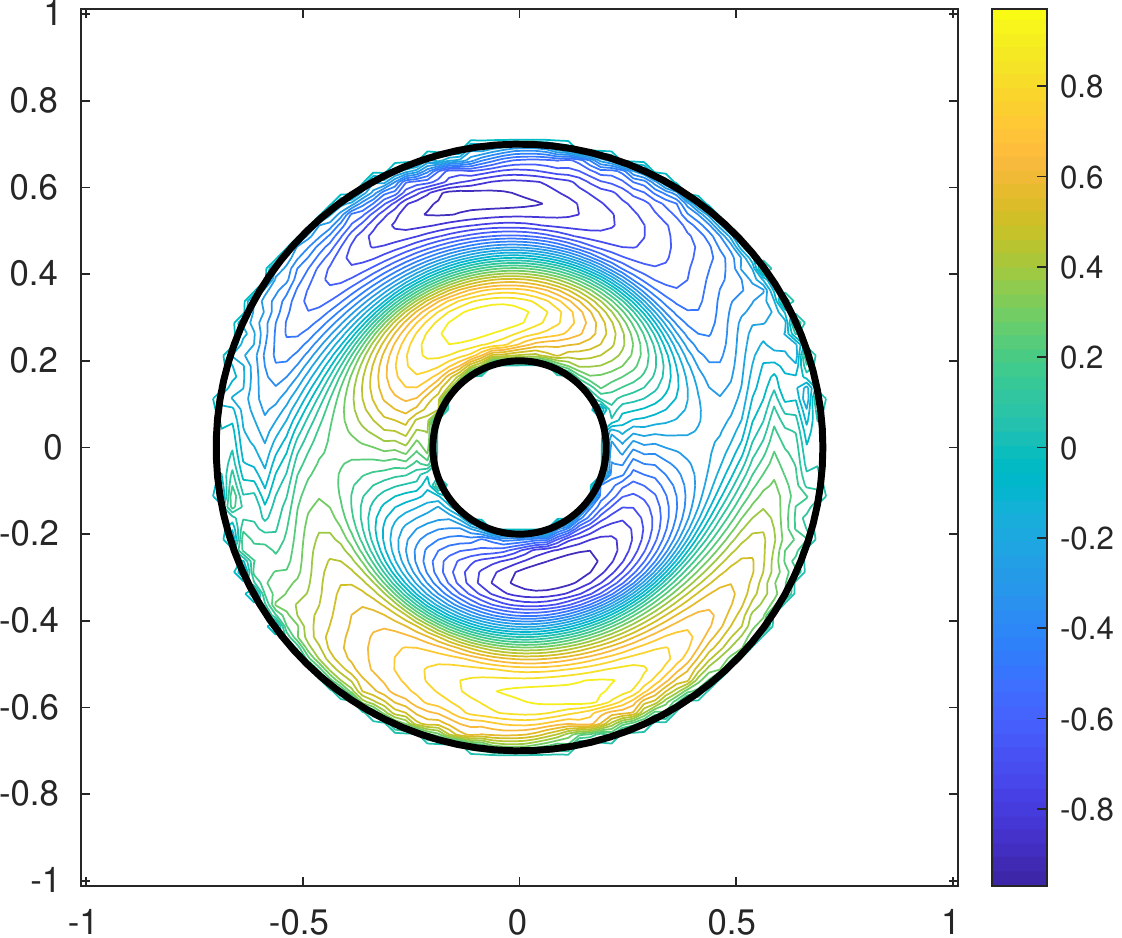}
		\caption{\bf $u_1, \varepsilon = 10^{-4}$}
	\end{subfigure}\\
	\begin{subfigure}{0.2\textwidth}
		\includegraphics[width=\textwidth]{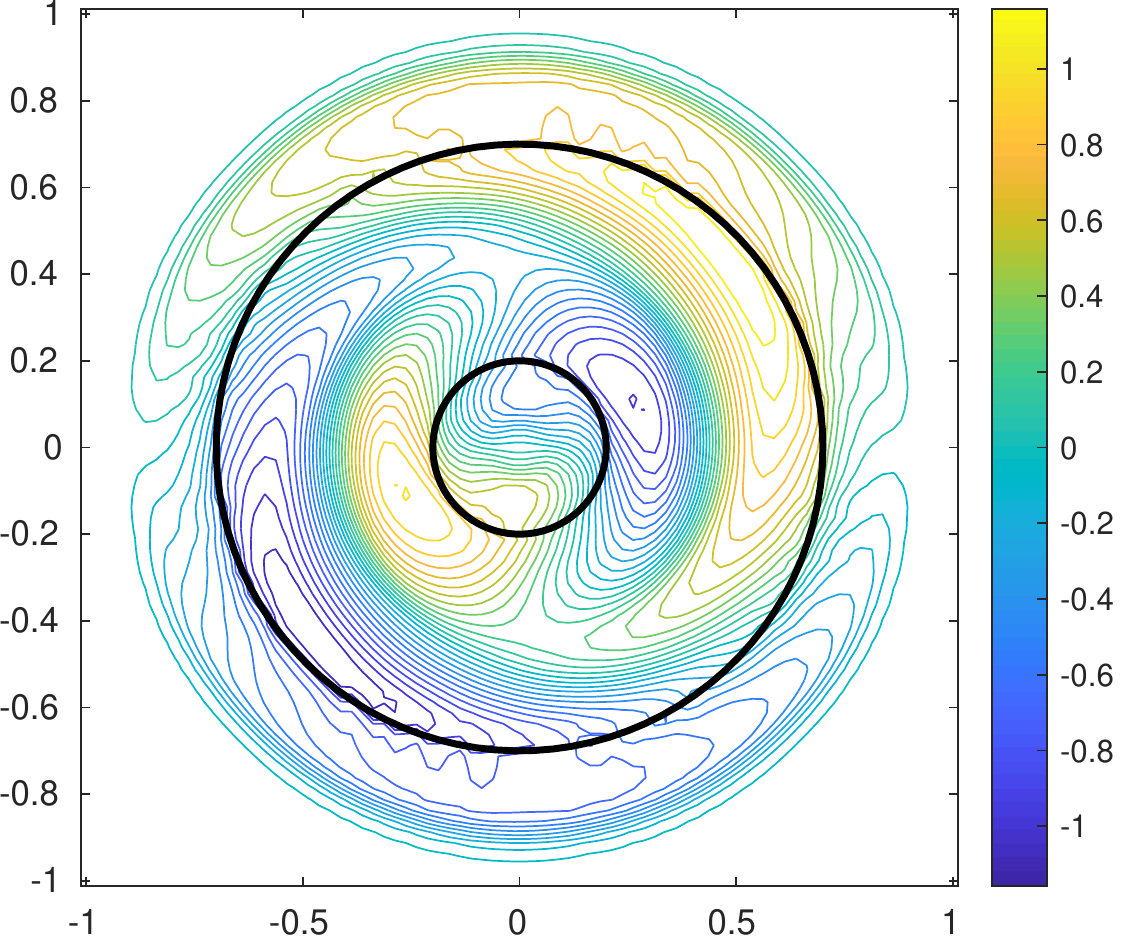}
		\caption{ \bf $u_2, \varepsilon = 10^{-1}$}
	\end{subfigure}	
	\begin{subfigure}{0.2\textwidth}
		\includegraphics[width=\textwidth]{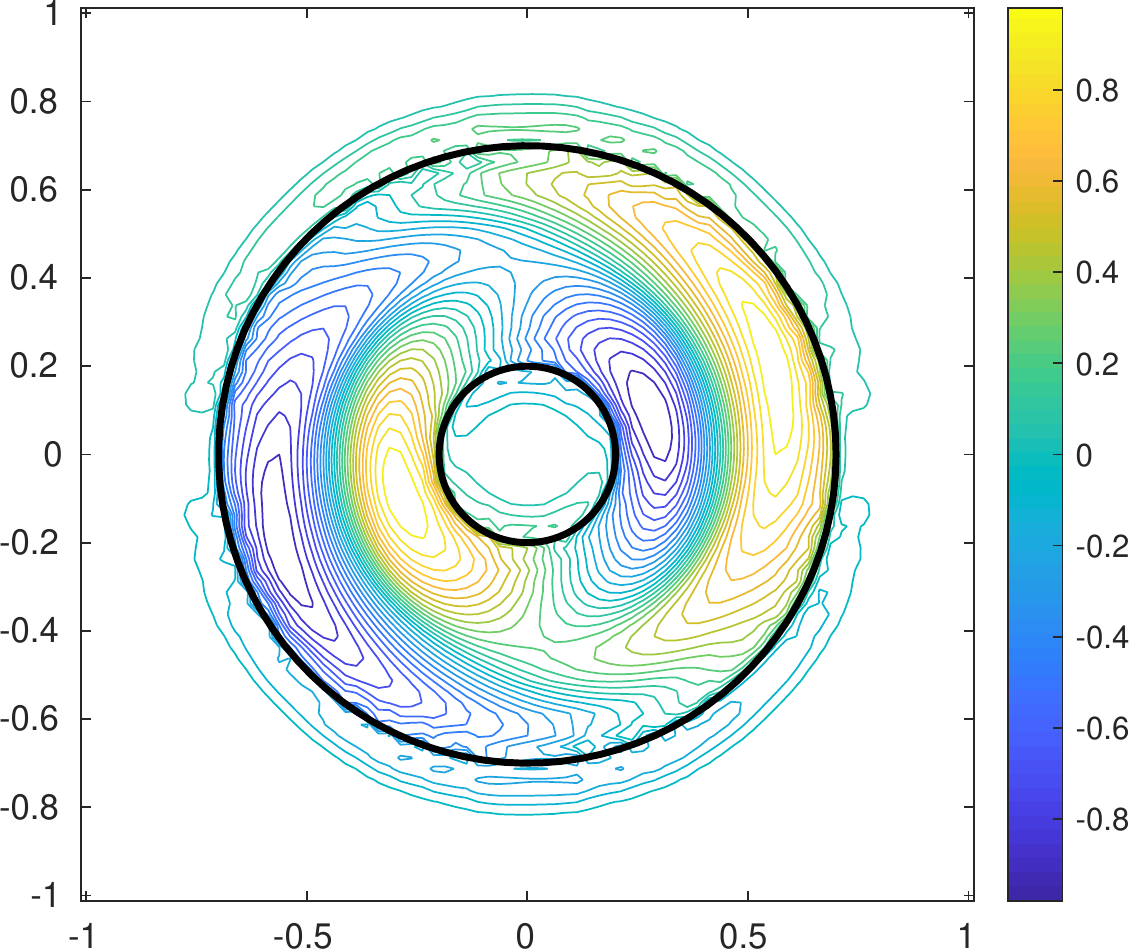}
		\caption{ \bf $u_2, \varepsilon = 10^{-2}$}
	\end{subfigure}
	\begin{subfigure}{0.2\textwidth}
		\includegraphics[width=\textwidth]{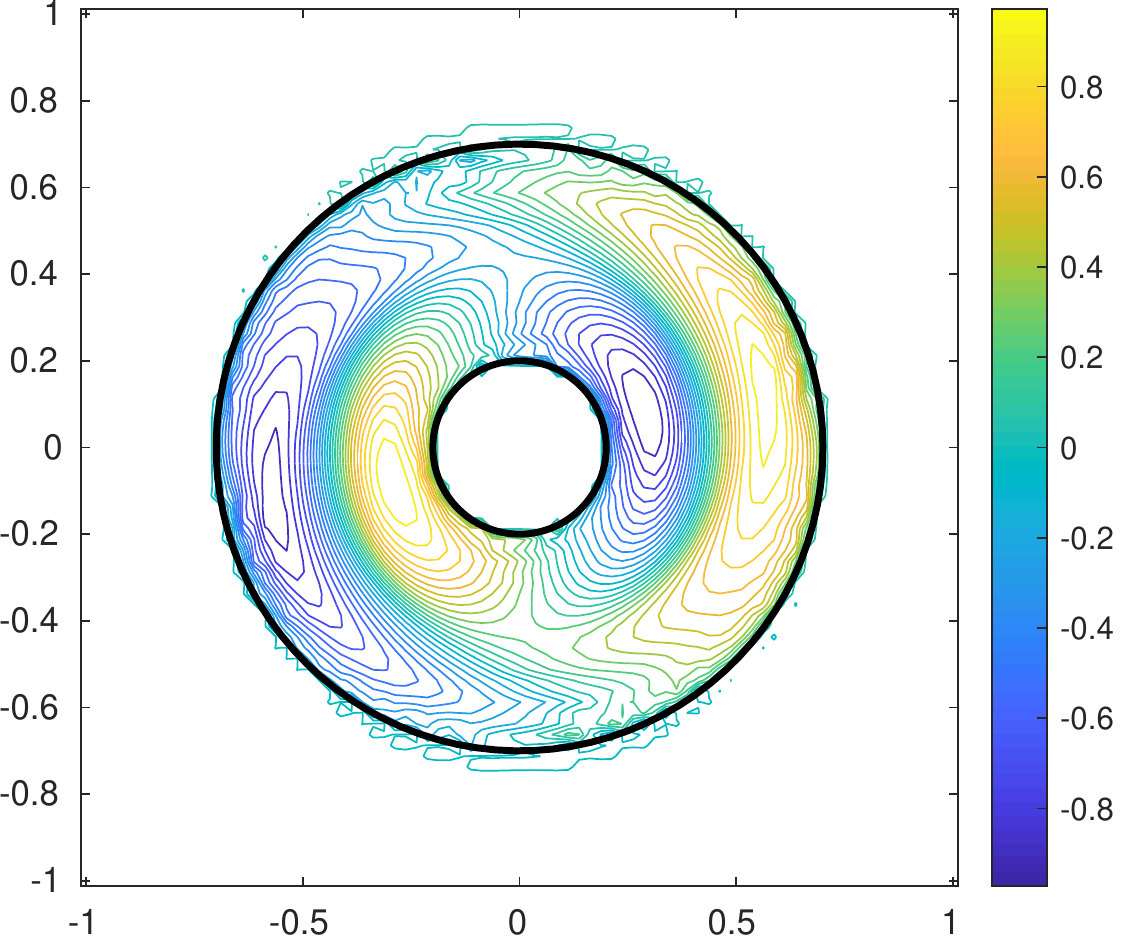}
		\caption{ \bf $u_2, \varepsilon = 10^{-3}$}
	\end{subfigure}
	\begin{subfigure}{0.2\textwidth}
		\includegraphics[width=\textwidth]{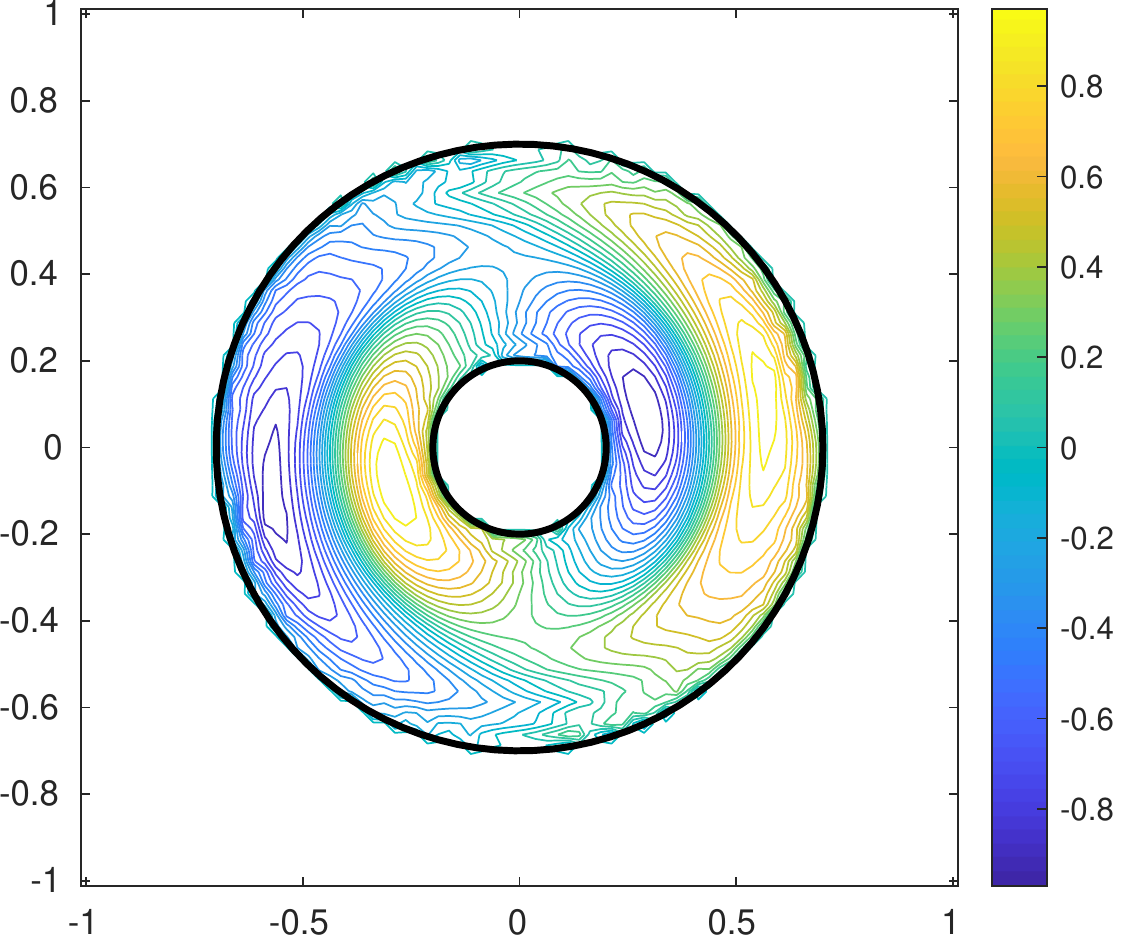}
		\caption{ \bf $u_2, \varepsilon = 10^{-4}$}
	\end{subfigure}\\
	\begin{subfigure}{0.2\textwidth}
		\includegraphics[width=\textwidth]{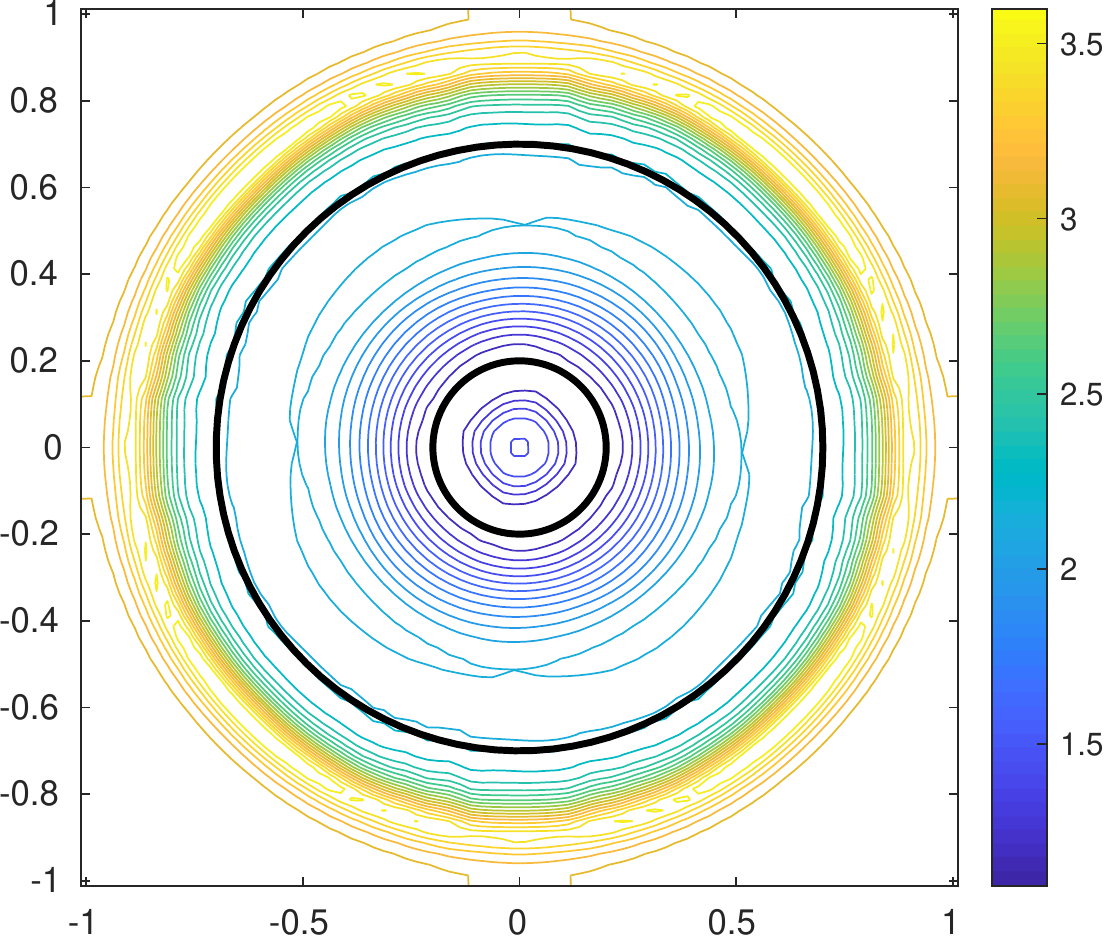}
		\caption{\bf $\vartheta, \varepsilon = 10^{-1}$}
	\end{subfigure}
	\begin{subfigure}{0.2\textwidth}
		\includegraphics[width=\textwidth]{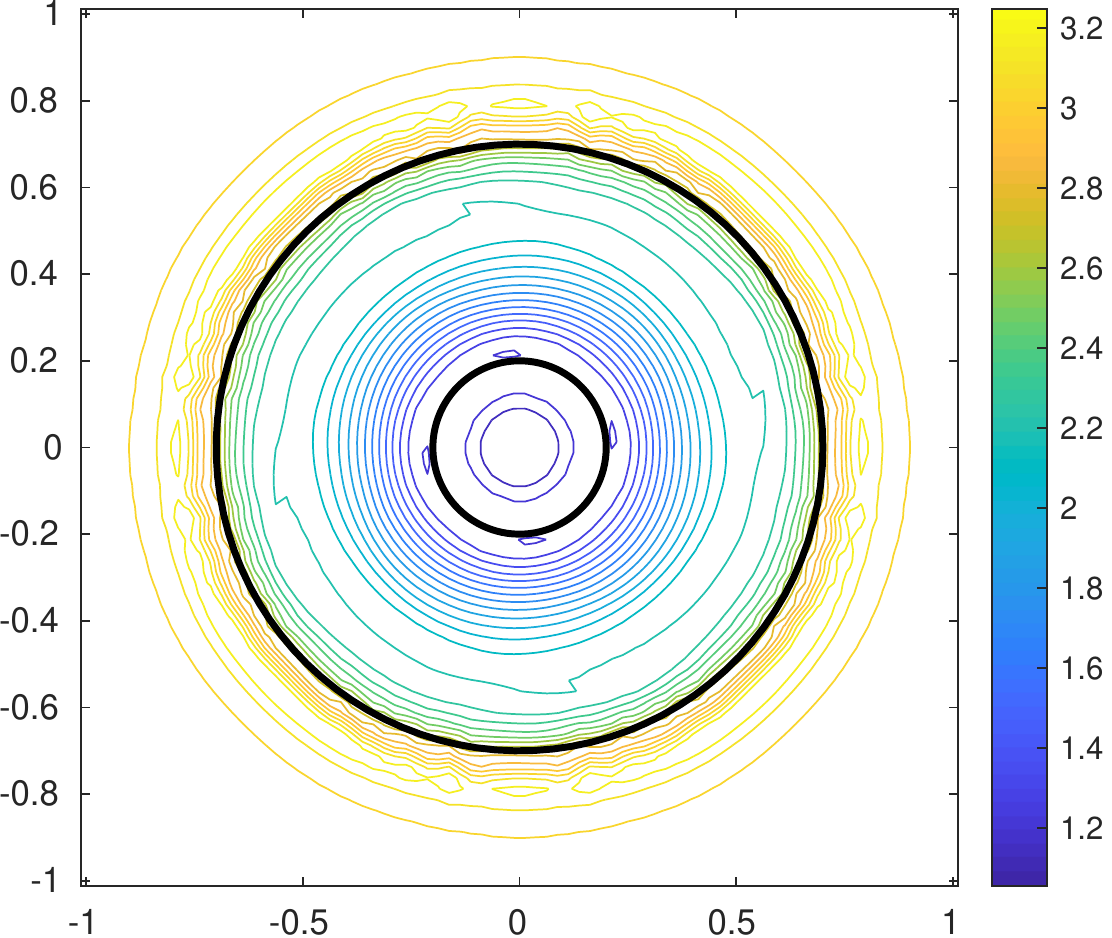}
		\caption{\bf $\vartheta, \varepsilon = 10^{-2}$}
	\end{subfigure}
	\begin{subfigure}{0.2\textwidth}
		\includegraphics[width=\textwidth]{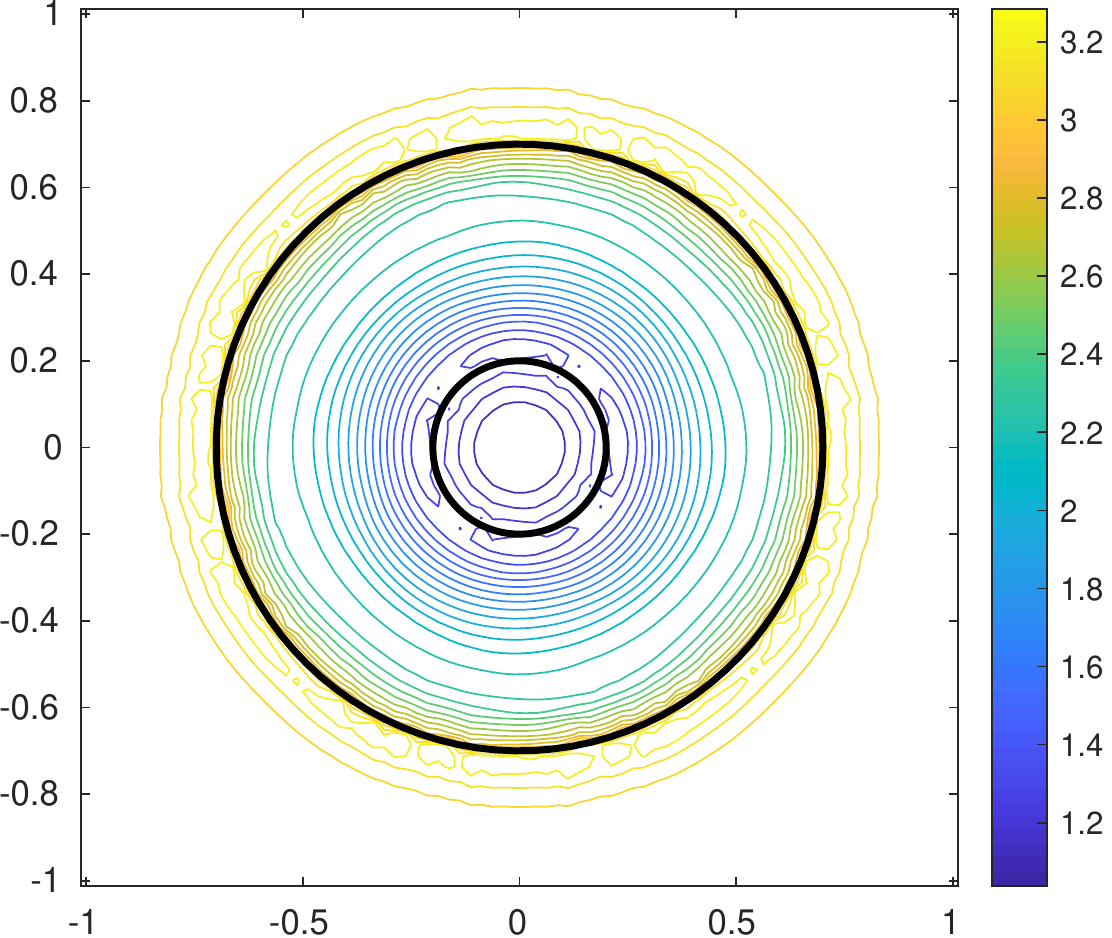}
		\caption{\bf $\vartheta, \varepsilon = 10^{-3}$}
	\end{subfigure}
	\begin{subfigure}{0.2\textwidth}
		\includegraphics[width=\textwidth]{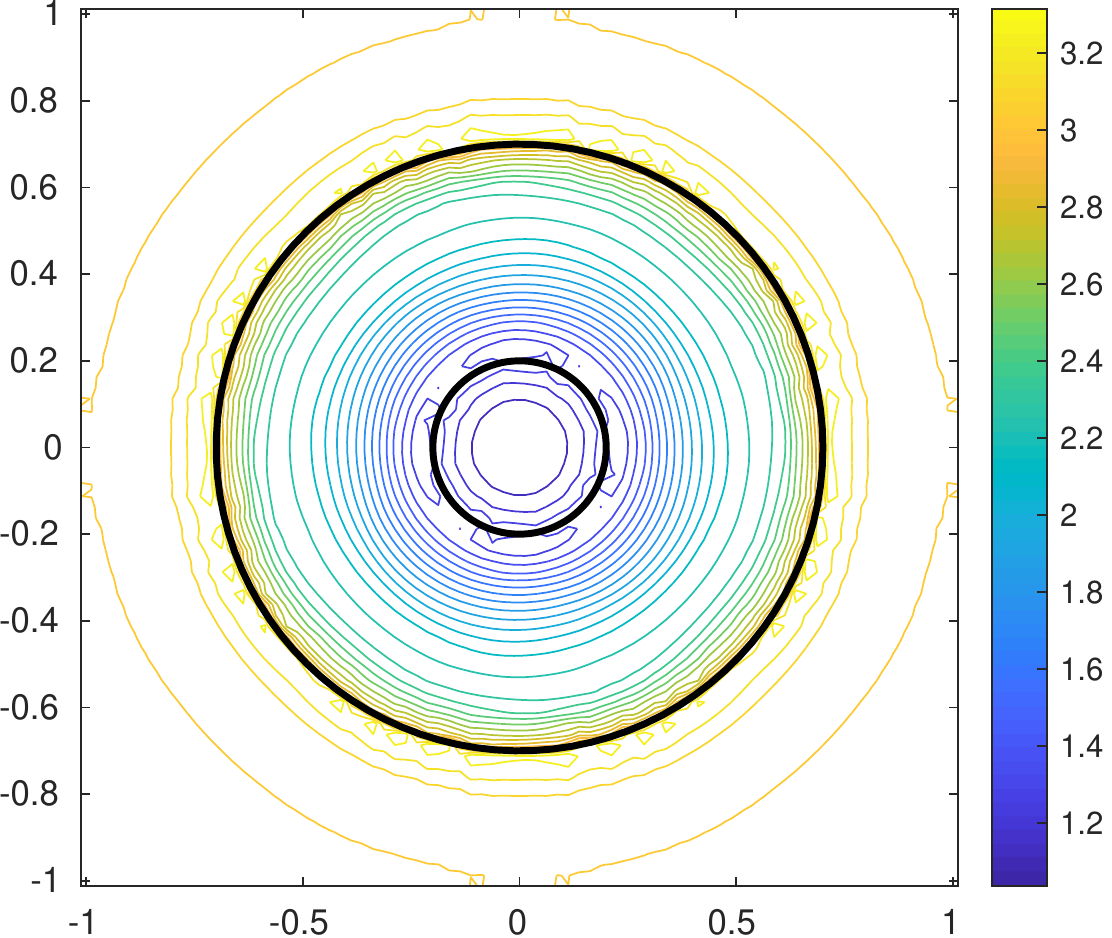}
		\caption{\bf $\vartheta, \varepsilon = 10^{-4}$}
	\end{subfigure}\\
	\begin{subfigure}{0.2\textwidth}
		\includegraphics[width=\textwidth]{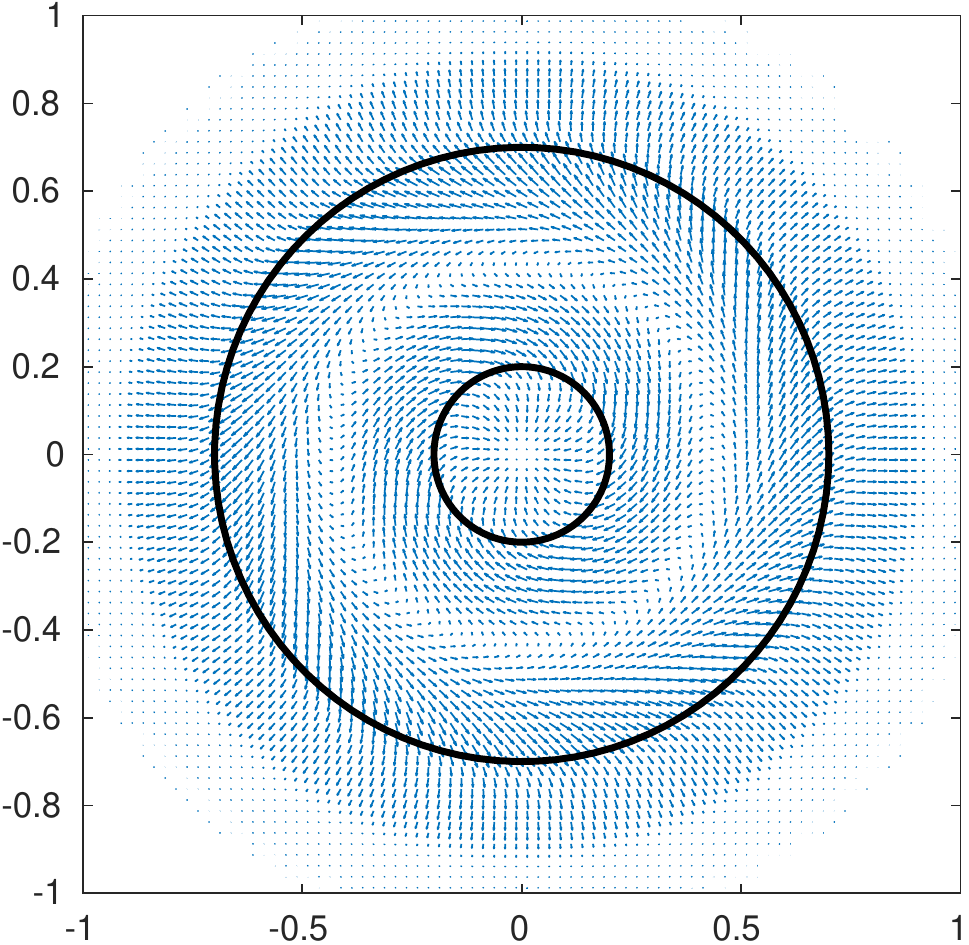}
		\caption{\bf $\vu, \varepsilon = 10^{-1}$}
	\end{subfigure}
	\begin{subfigure}{0.2\textwidth}
		\includegraphics[width=\textwidth]{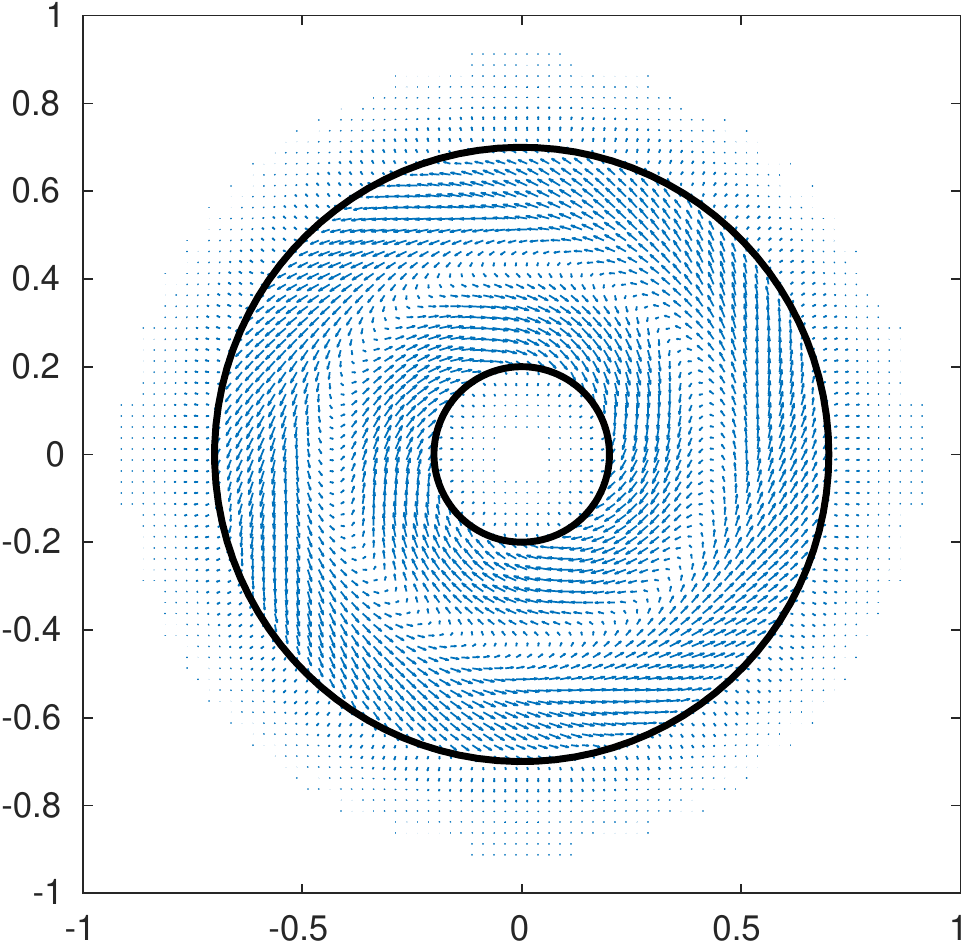}
		\caption{\bf $\vu, \varepsilon = 10^{-2}$}
	\end{subfigure}	
	\begin{subfigure}{0.2\textwidth}
		\includegraphics[width=\textwidth]{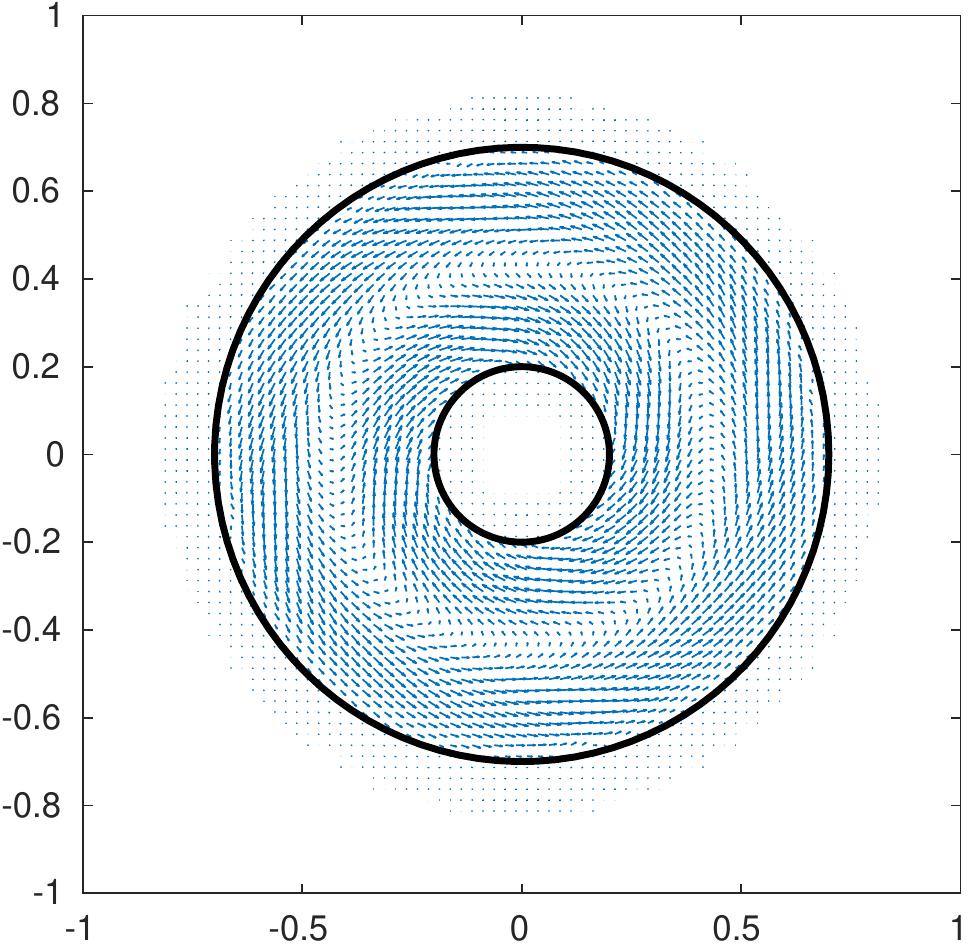}
		\caption{\bf $\vu, \varepsilon = 10^{-3}$}
	\end{subfigure}		
	\begin{subfigure}{0.2\textwidth}
		\includegraphics[width=\textwidth]{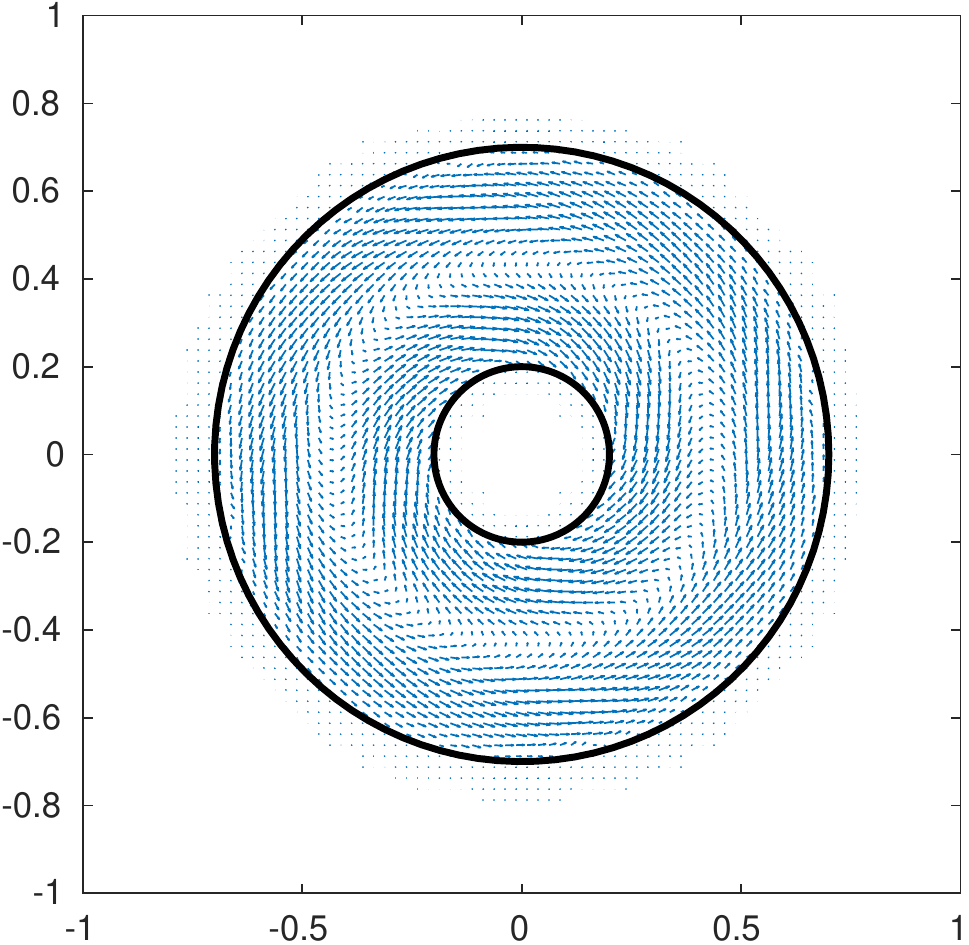}
		\caption{\bf $\vu, \varepsilon = 10^{-4}$}
	\end{subfigure}
	\caption{\small{Experiment~2:  ${U}_h^{\varepsilon}$ with $h = 2/80$ and $\varepsilon = 10^{-1},\dots, 10^{-4}$.}}\label{fig:ex2}
\end{figure}

\subsection{Experiment~3: Complex domain - zero density outside $\Omega$}
We consider the same initial data as in the previous  experiment but choose a more complicated geometry of the fluid domain; $\Omega \equiv \hat{B}_{0.7}\setminus B_{0.2},$ where
\begin{equation*}
\hat{B}_{0.7} := \left\{x ~\bigg|~ |x| = (0.7+\delta) + \delta \cos(8\phi), ~ \tan(\phi) =\frac{x}{y} \right\}.
\end{equation*}
The initial data are given by
\begin{equation*}
	(\varrho,\vu, \vartheta)(0,x)
	\; = \; \begin{cases}
	(10^{-2},0,0, 1 ) , & x \in B_{0.2}, \\
	\left(1, \frac{ \sin(4\pi (|x|-0.2)) x_2}{|x|} ,-\frac{ \sin(4\pi (|x|-0.2)) x_1}{|x|}, 0.2 + 4|x| \right) , & x \in  \Omega \equiv {B}_{0.7}\setminus B_{0.2}, \\
	(1,0 ,0,  3) , & x \in \hat{B}_{0.7}\setminus B_{0.7}, \\	
	(10^{-2},0 ,0,  3) , & x \in \mathbb{T}^2 \setminus \hat{B}_{0.7} .
	\end{cases}
\end{equation*}
In the simulations we set $\delta = 0.05$.
Figure \ref{fig:ex3-1} and Figure \ref{fig:ex3-2} demonstrate that  the experimental convergence rates with respect to $h$ and $\varepsilon$ are of the first order.
Figure \ref{fig:ex3} illustrates the effects of different penalization parameters $\varepsilon = 10^{-1},\dots,10^{-4}$
on the numerical solutions computed on the mesh with $80^2$ cells.

We can observe some oscillations near the inner and outer boundaries, whereas the oscillations at the outer boundary  are larger than in  Experiment~2.
Due to small outside density the fluid flows outside, meanwhile the temperature pushes the fluid to flow to the center. Consequently,
due to the complex geometry of the fluid domain, the oscillations become more visible  and vortex structure arises.
Interestingly, even that the oscillations are present, the penalization method still converges with rate $1$, which is consistent with our theoretical analysis.

\begin{figure}[htbp]
	\setlength{\abovecaptionskip}{0.cm}
	\setlength{\belowcaptionskip}{-0.cm}
	\centering
	\begin{subfigure}{0.32\textwidth}
		\includegraphics[width=\textwidth]{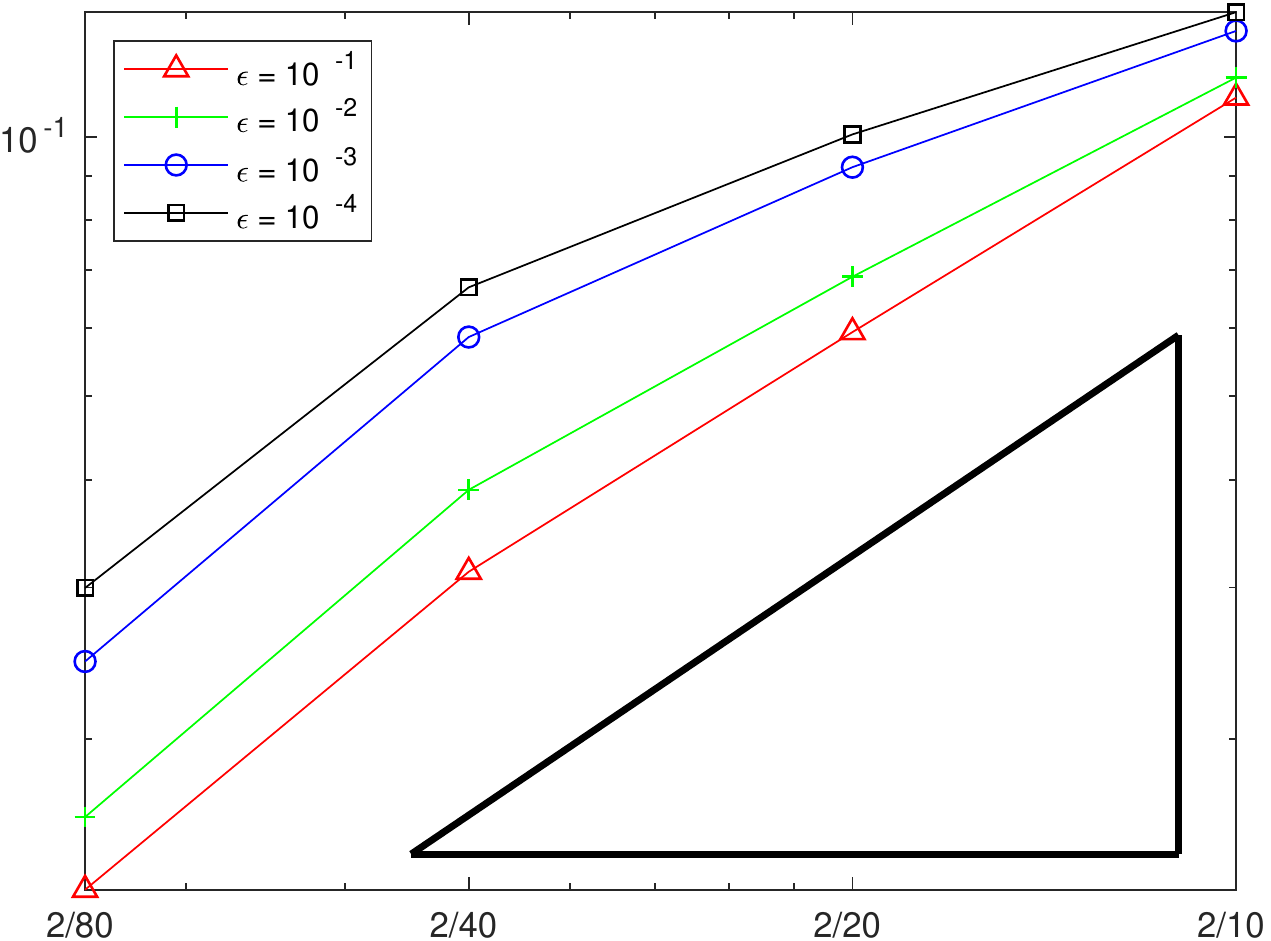}
		\caption{ \bf $\varrho$}
	\end{subfigure}
	\begin{subfigure}{0.32\textwidth}
		\includegraphics[width=\textwidth]{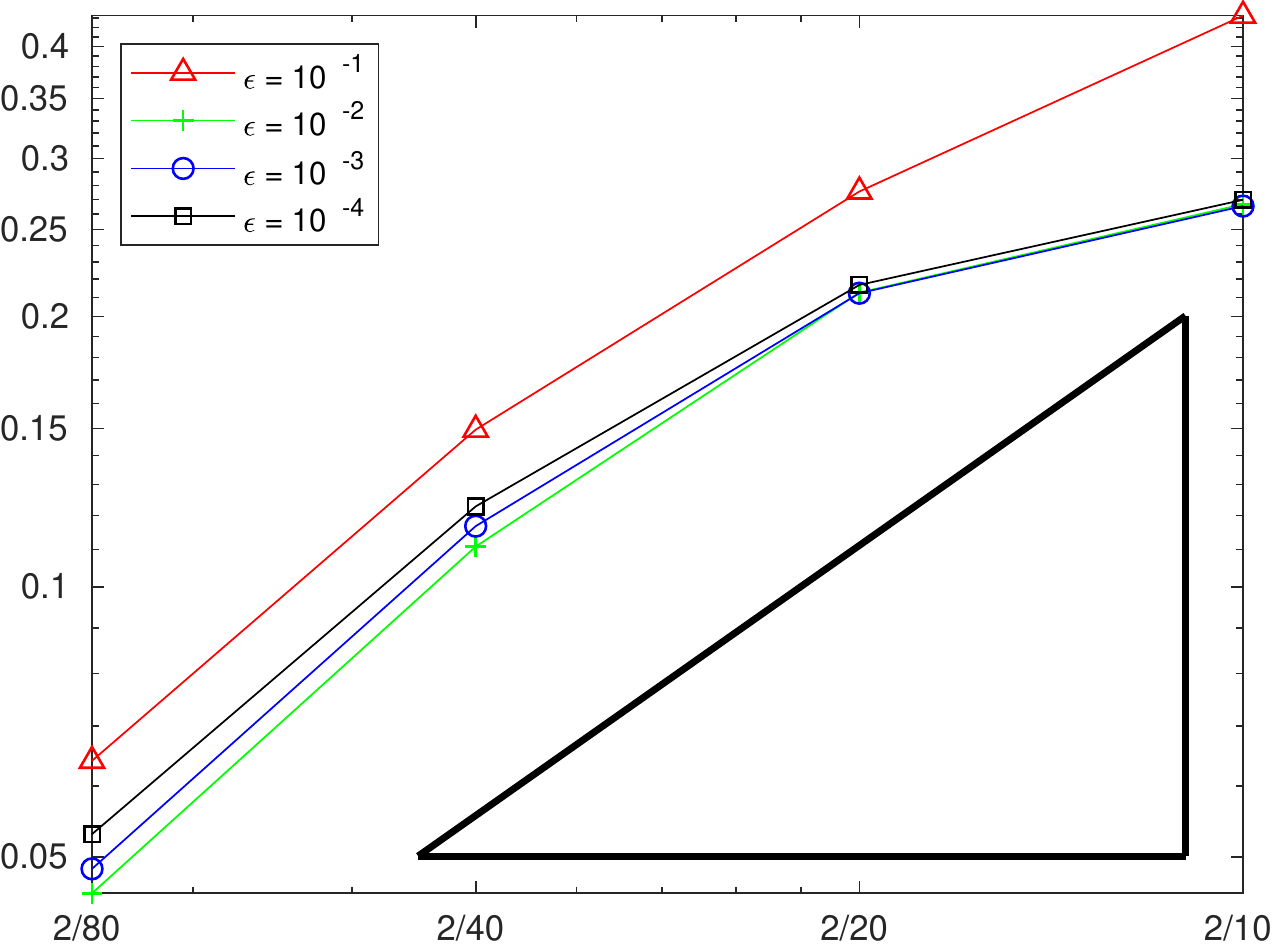}
		\caption{ \bf $\vu$}
	\end{subfigure}
	\begin{subfigure}{0.32\textwidth}
		\includegraphics[width=\textwidth]{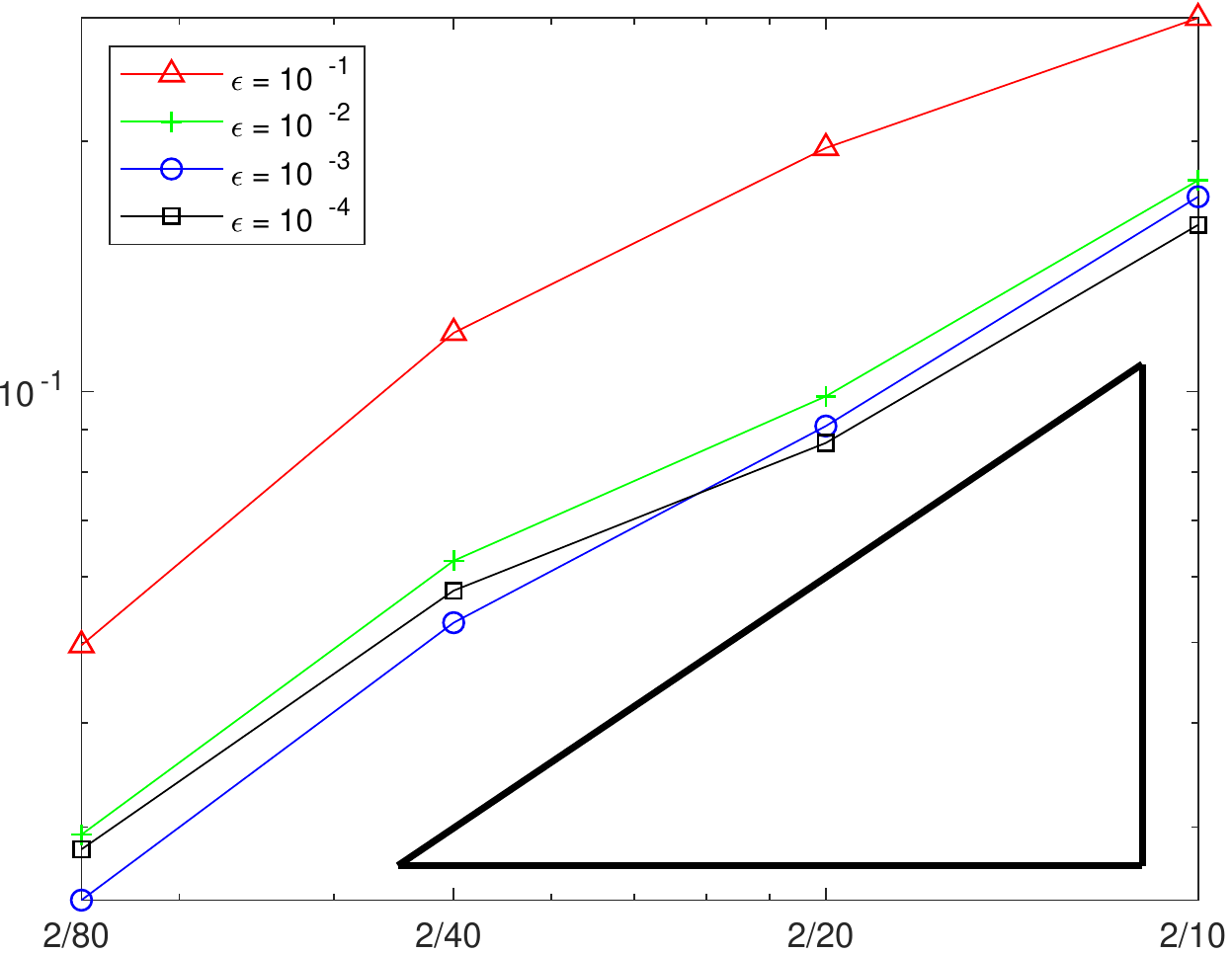}
		\caption{ \bf $\vartheta$}
	\end{subfigure}
	\caption{\small{Experiment~3:  $E(U_h^{\varepsilon})$ errors  with respect to $h$.}}\label{fig:ex3-1}
\end{figure}

\begin{figure}[htbp]
	\setlength{\abovecaptionskip}{0.cm}
	\setlength{\belowcaptionskip}{-0.cm}
	\centering
	\begin{subfigure}{0.32\textwidth}
		\includegraphics[width=\textwidth]{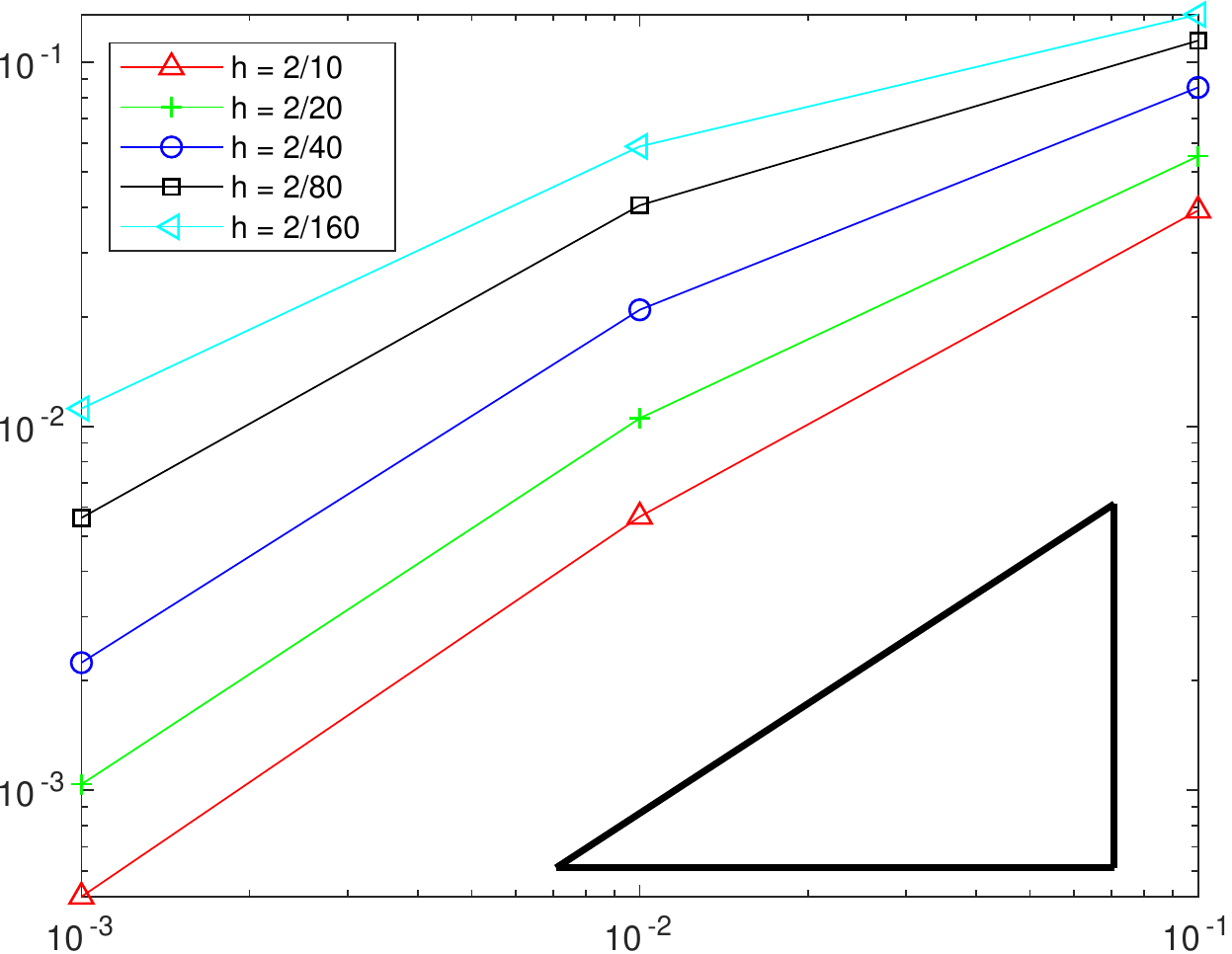}
		\caption{ \bf $\varrho$}
	\end{subfigure}
	\begin{subfigure}{0.32\textwidth}
		\includegraphics[width=\textwidth]{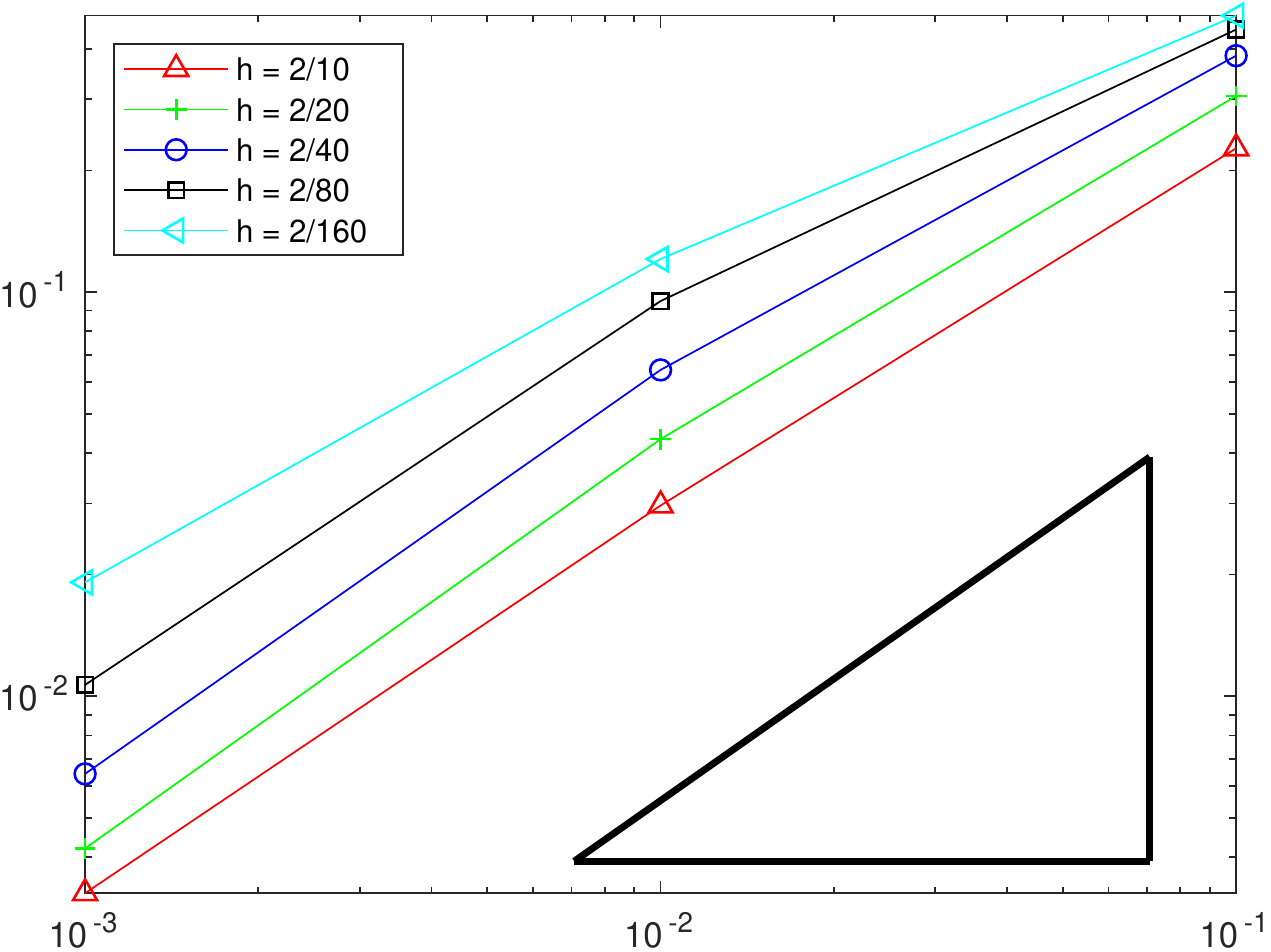}
		\caption{ \bf $\vu$}
	\end{subfigure}
	\begin{subfigure}{0.32\textwidth}
		\includegraphics[width=\textwidth]{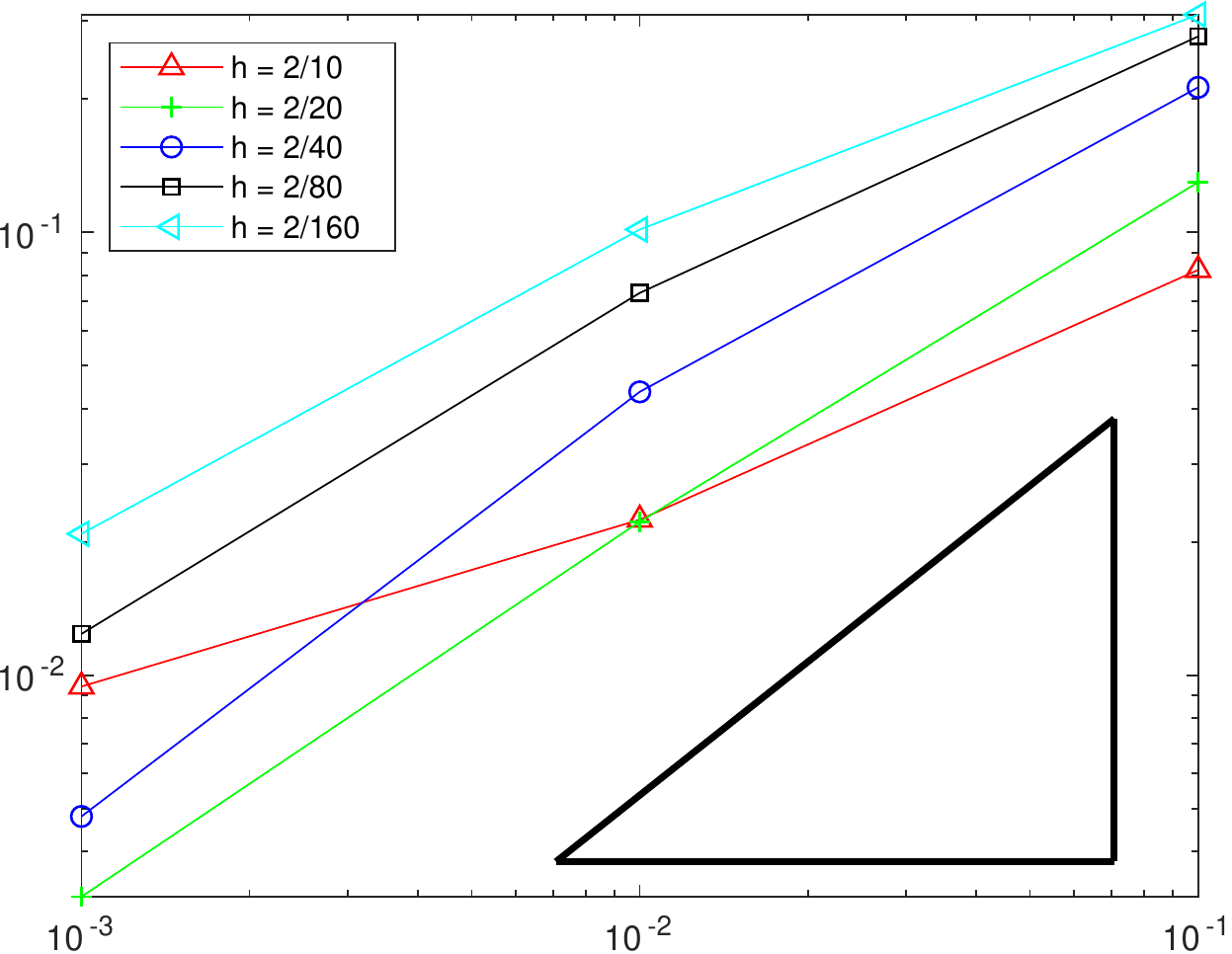}
		\caption{ \bf $\vartheta$}
	\end{subfigure}
	\caption{\small{Experiment~3:  $P(U_h^{\varepsilon})$ errors with respect to $\varepsilon$.}}\label{fig:ex3-2}
\end{figure}

\begin{figure}[htbp]
	\setlength{\abovecaptionskip}{0.cm}
	\setlength{\belowcaptionskip}{-0.cm}
	\centering
	\begin{subfigure}{0.2\textwidth}
		\includegraphics[width=\textwidth]{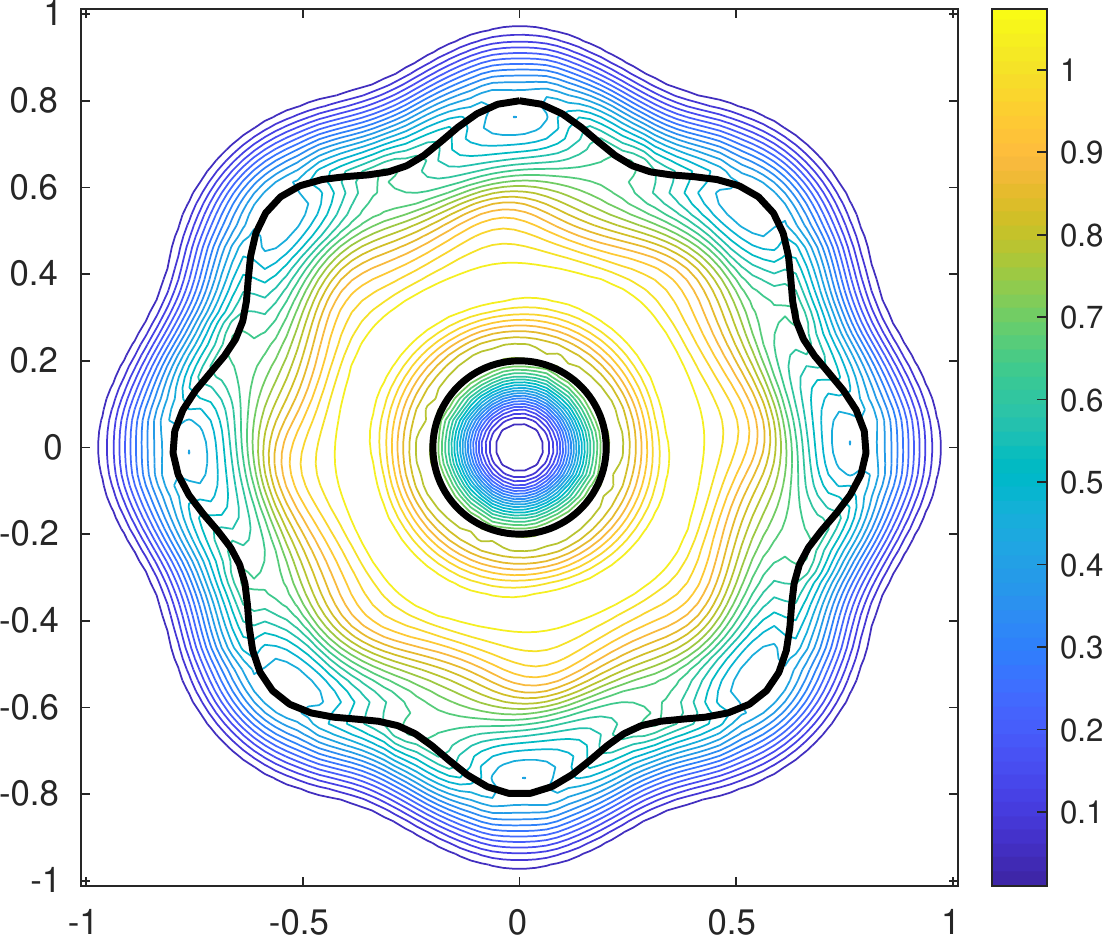}
		\caption{ \bf $\varrho, \varepsilon = 10^{-1}$}
	\end{subfigure}
	\begin{subfigure}{0.2\textwidth}
		\includegraphics[width=\textwidth]{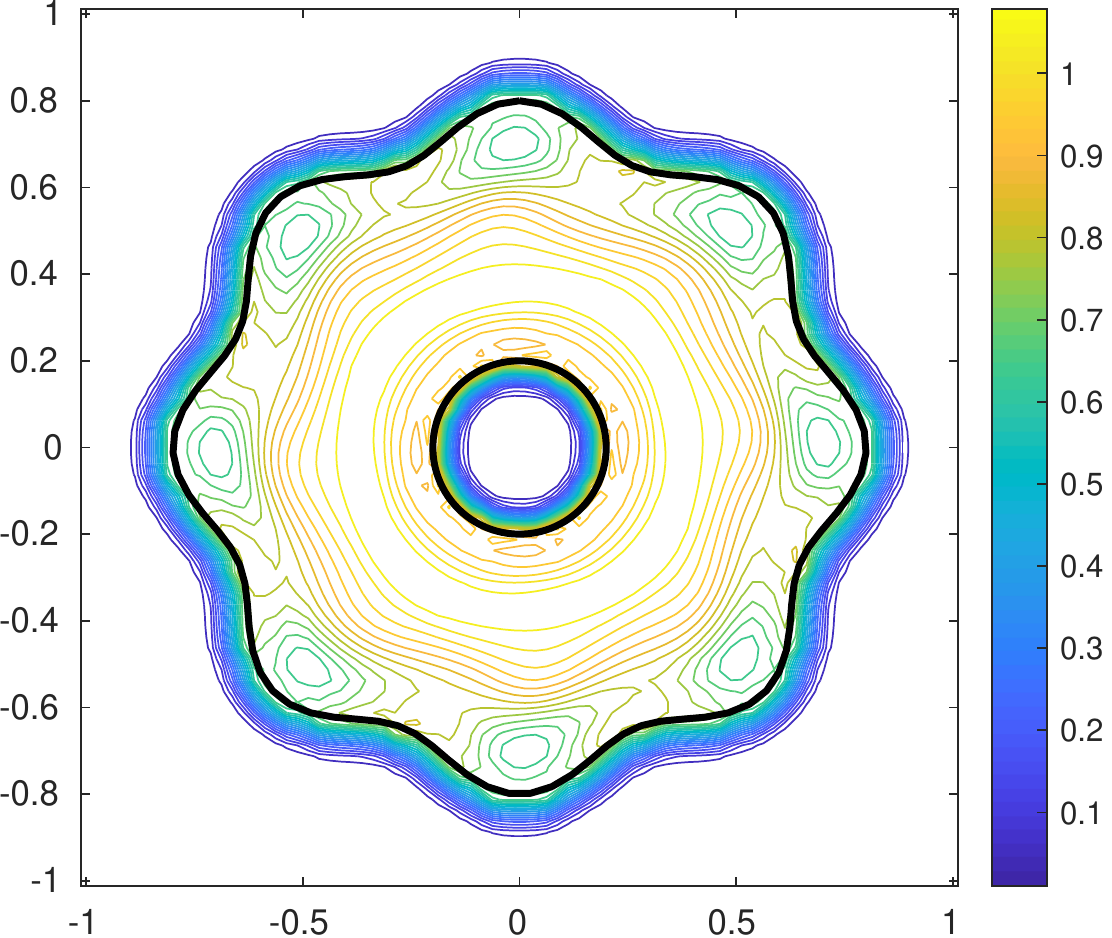}
		\caption{ \bf $\varrho, \varepsilon = 10^{-2}$}
	\end{subfigure}
	\begin{subfigure}{0.2\textwidth}
		\includegraphics[width=\textwidth]{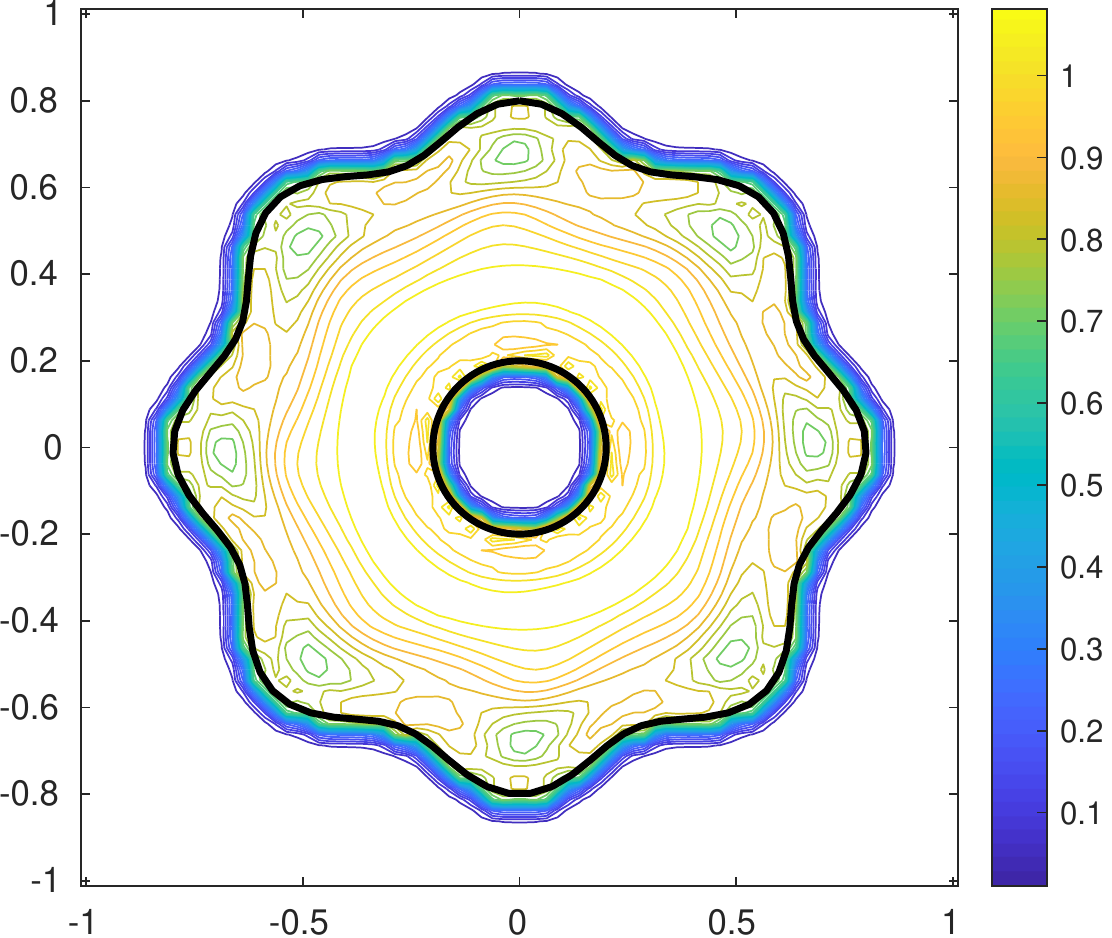}
		\caption{ \bf $\varrho, \varepsilon = 10^{-3}$}
	\end{subfigure}
	\begin{subfigure}{0.2\textwidth}
		\includegraphics[width=\textwidth]{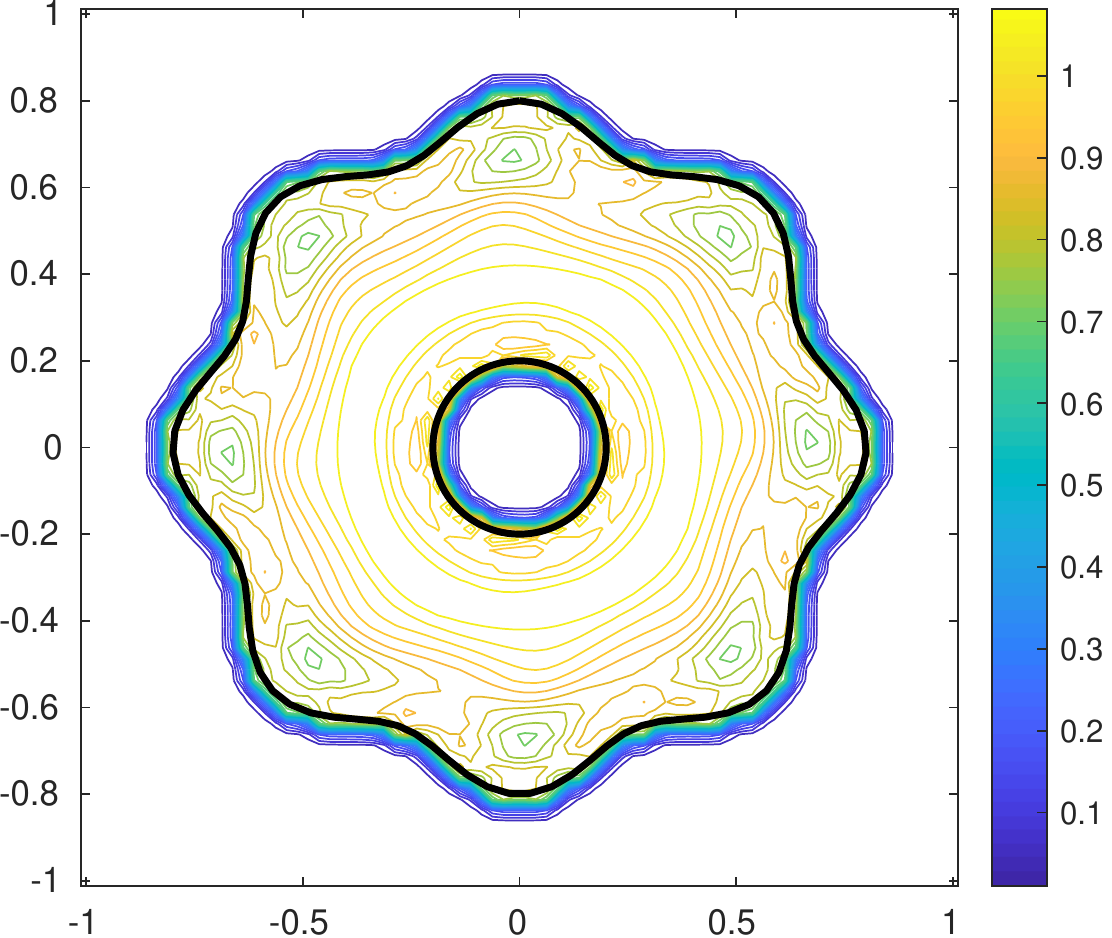}
		\caption{ \bf $\varrho, \varepsilon = 10^{-4}$}
	\end{subfigure}\\
	\begin{subfigure}{0.2\textwidth}
		\includegraphics[width=\textwidth]{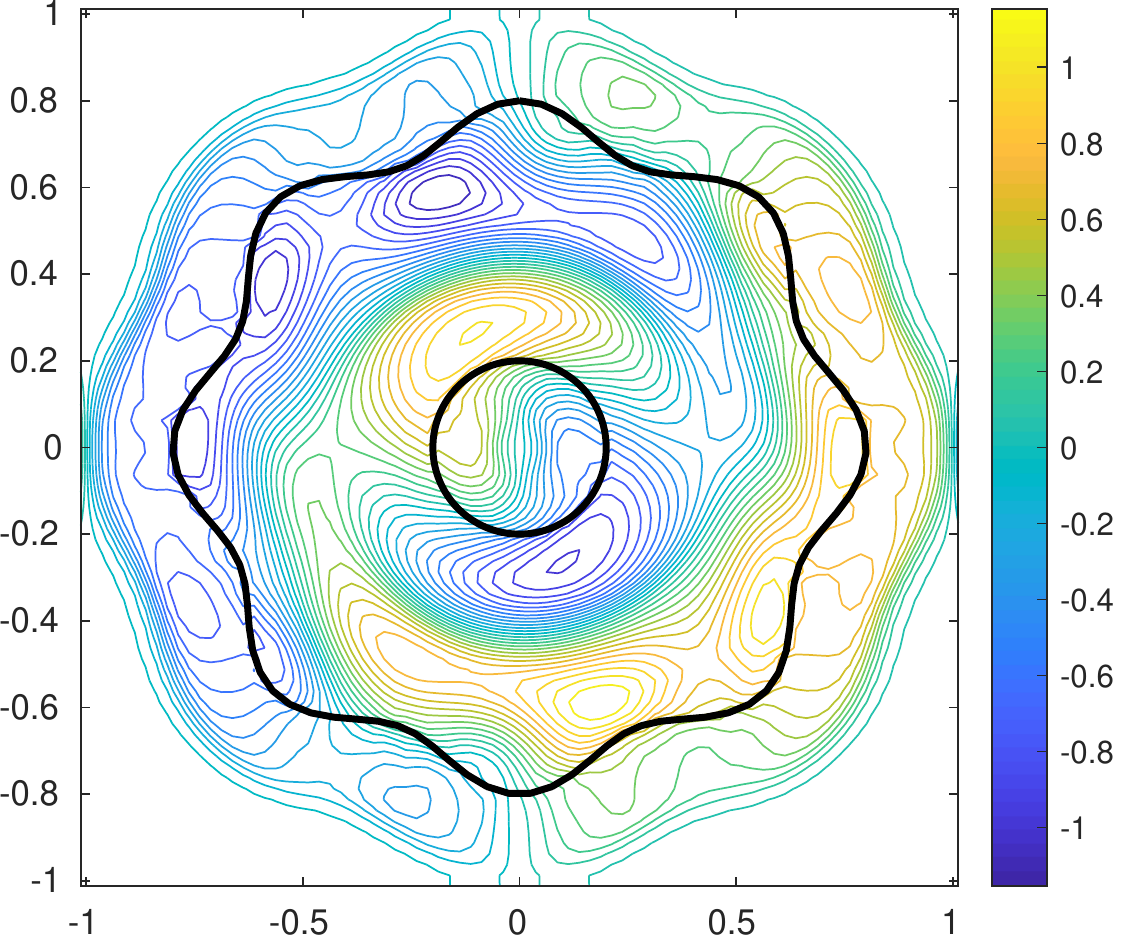}
		\caption{\bf $u_1, \varepsilon = 10^{-1}$}
	\end{subfigure}	
	\begin{subfigure}{0.2\textwidth}
		\includegraphics[width=\textwidth]{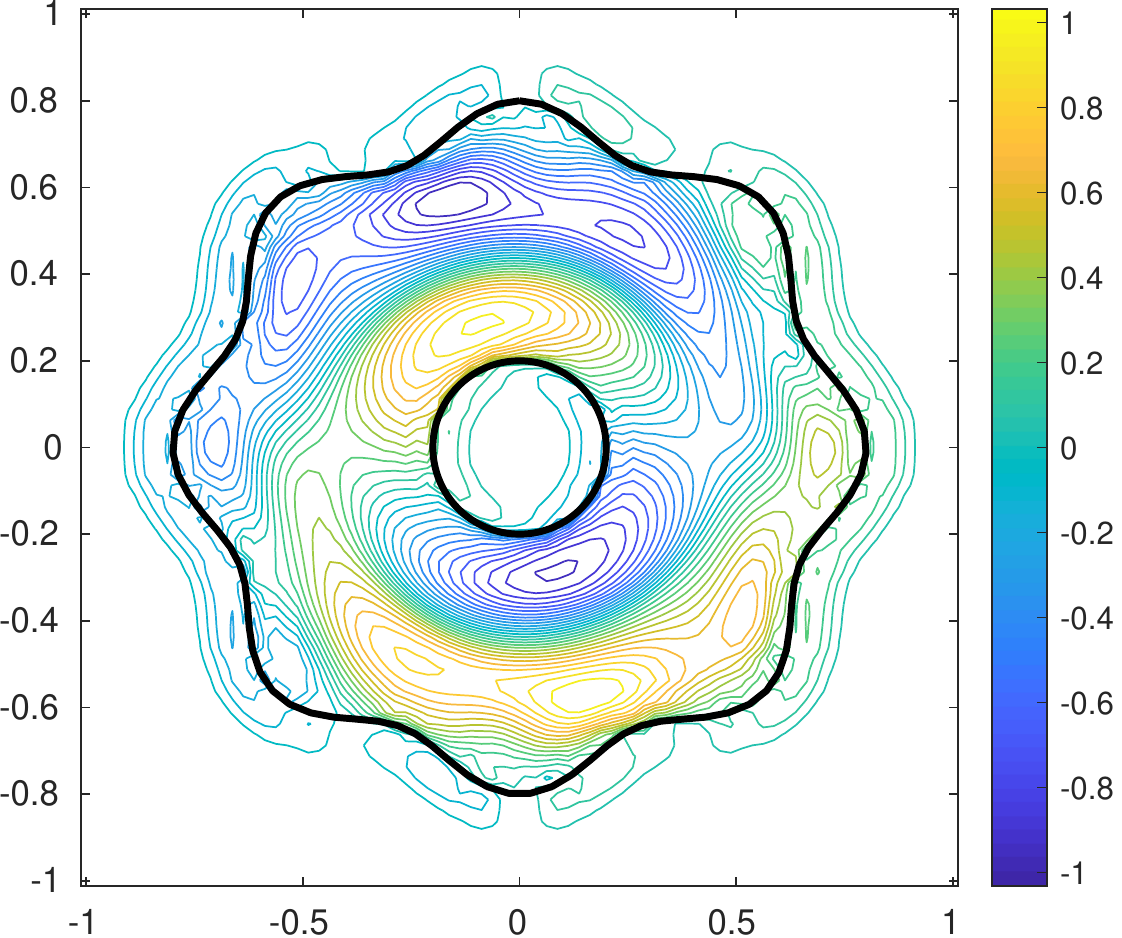}
		\caption{\bf $u_1, \varepsilon = 10^{-2}$}
	\end{subfigure}
	\begin{subfigure}{0.2\textwidth}
		\includegraphics[width=\textwidth]{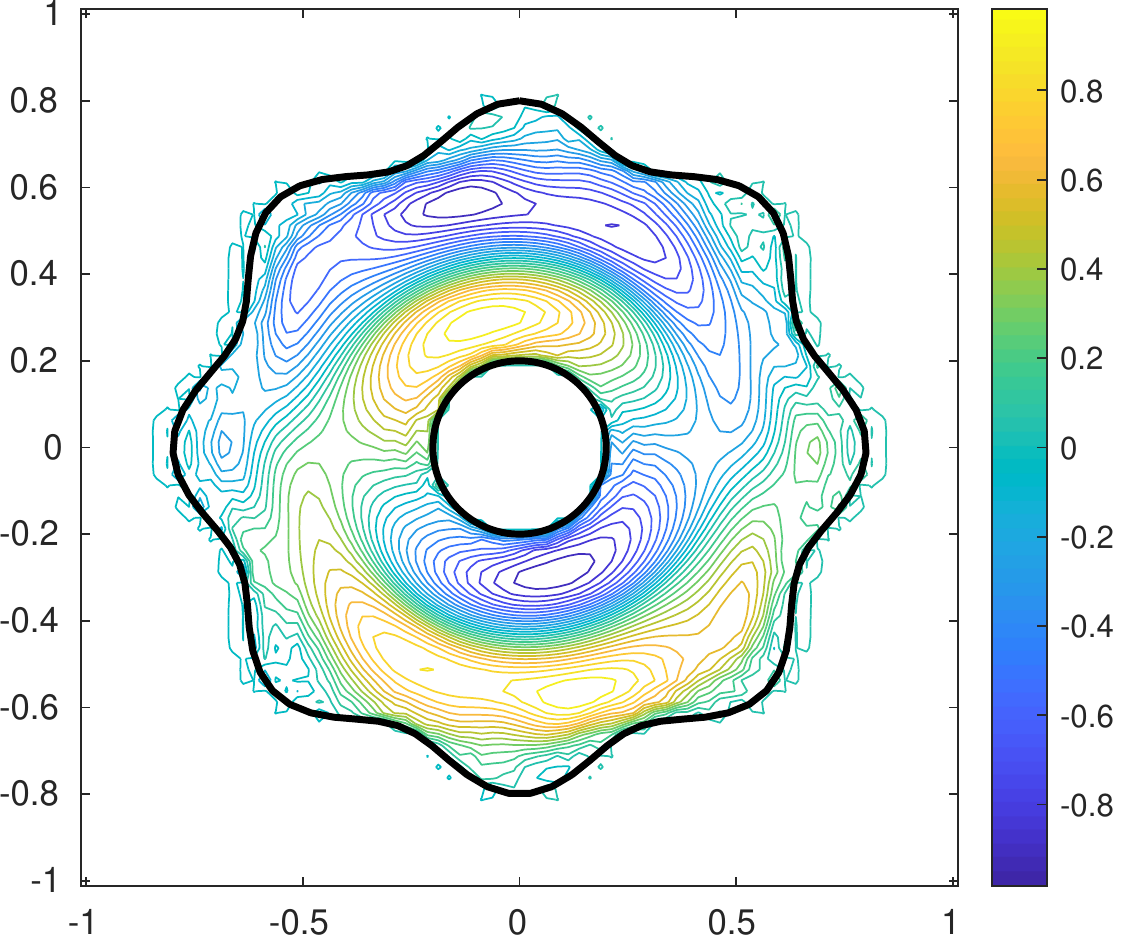}
		\caption{\bf $u_1, \varepsilon = 10^{-3}$}
	\end{subfigure}
	\begin{subfigure}{0.2\textwidth}
		\includegraphics[width=\textwidth]{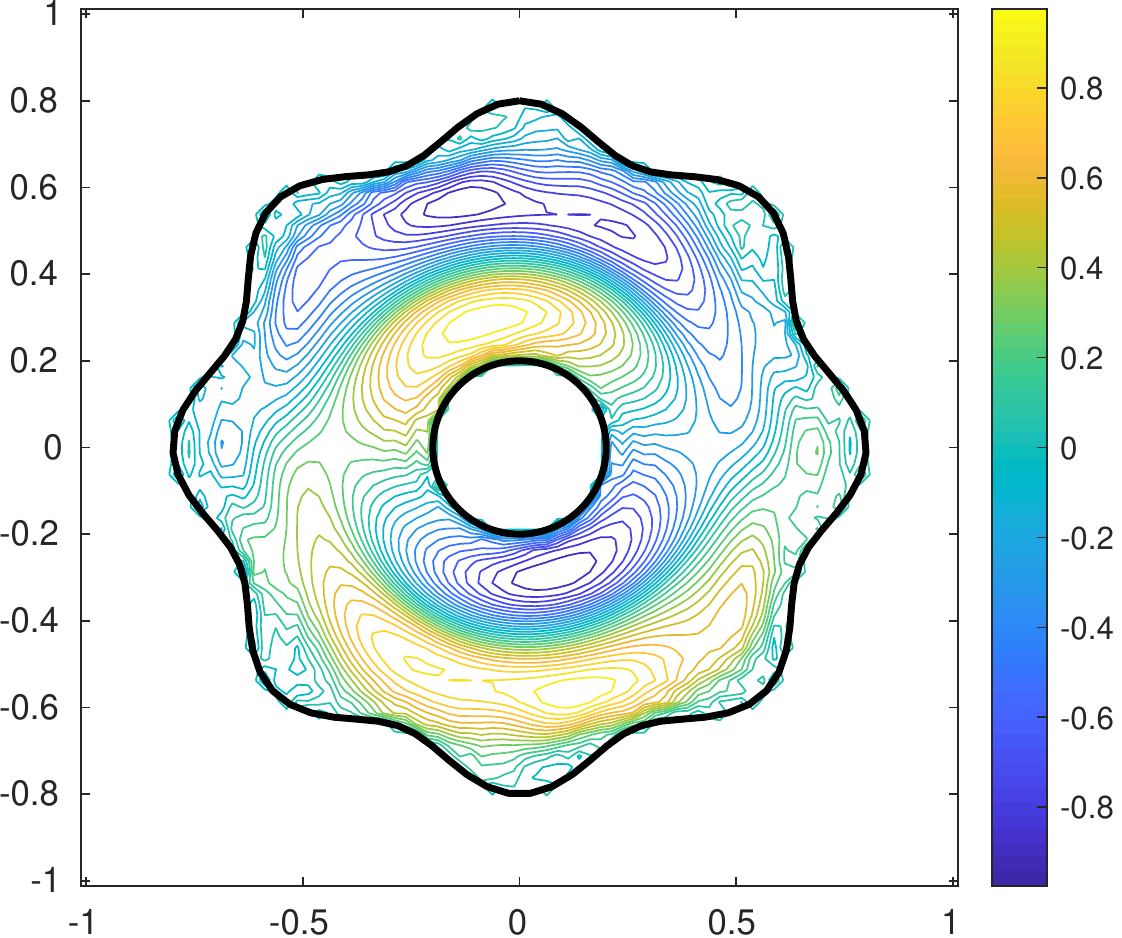}
		\caption{\bf $u_1, \varepsilon = 10^{-4}$}
	\end{subfigure}\\
	\begin{subfigure}{0.2\textwidth}
		\includegraphics[width=\textwidth]{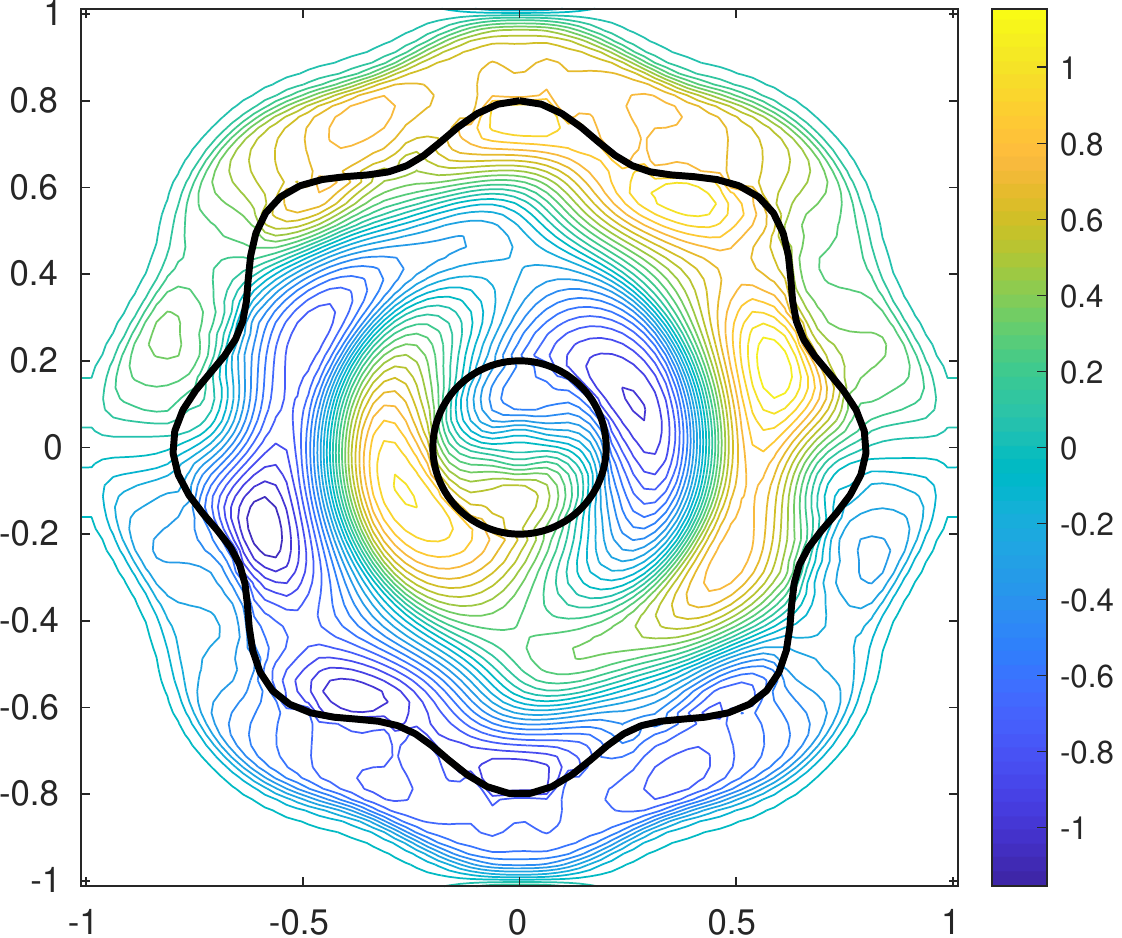}
		\caption{ \bf $u_2, \varepsilon = 10^{-1}$}
	\end{subfigure}	
	\begin{subfigure}{0.2\textwidth}
		\includegraphics[width=\textwidth]{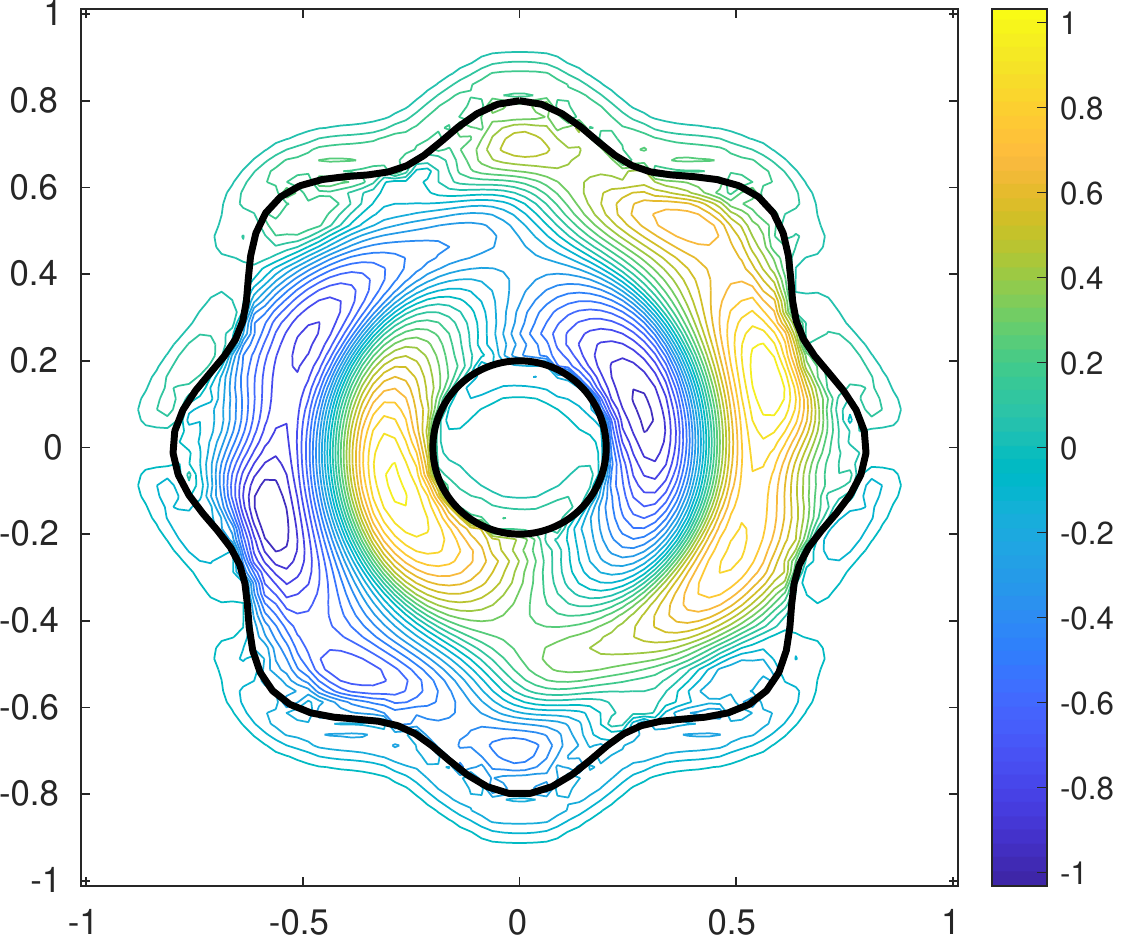}
		\caption{ \bf $u_2, \varepsilon = 10^{-2}$}
	\end{subfigure}
	\begin{subfigure}{0.2\textwidth}
		\includegraphics[width=\textwidth]{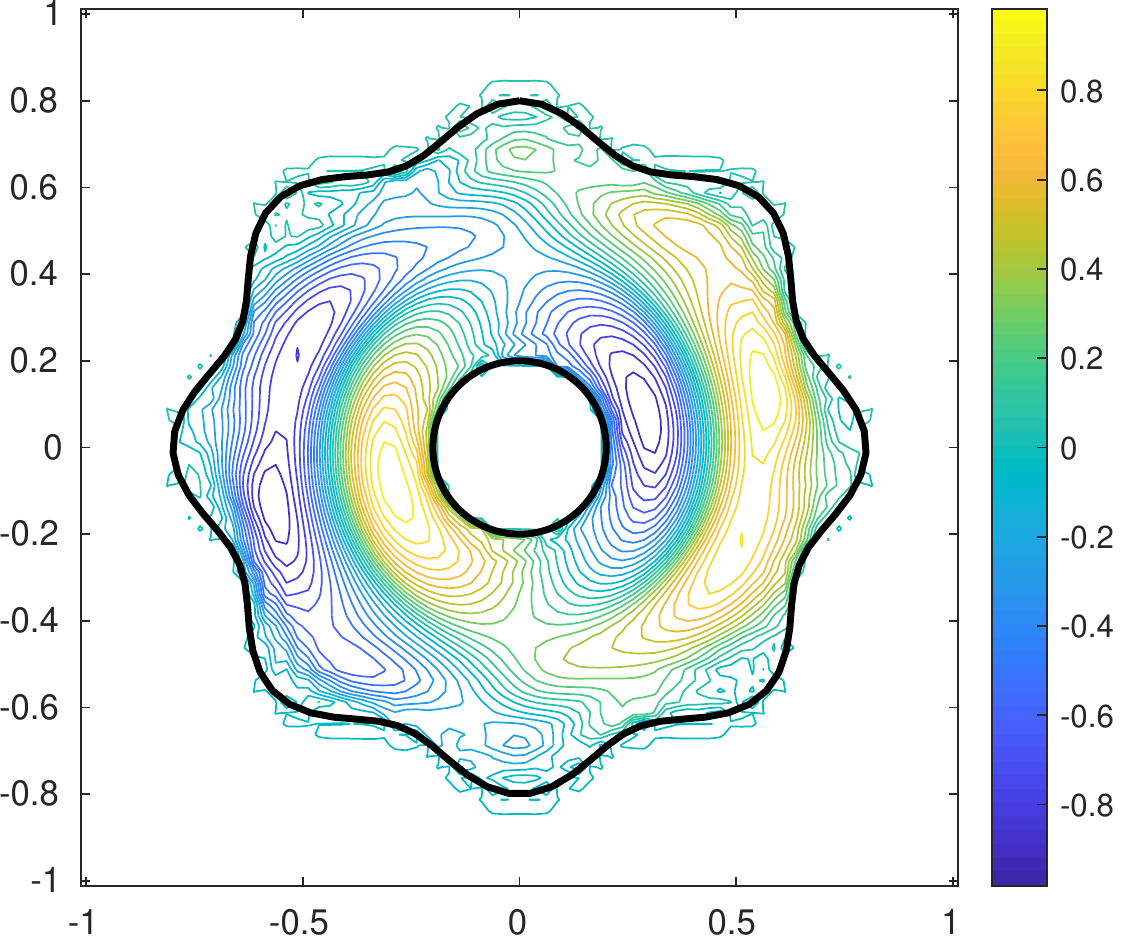}
		\caption{ \bf $u_2, \varepsilon = 10^{-3}$}
	\end{subfigure}
	\begin{subfigure}{0.2\textwidth}
		\includegraphics[width=\textwidth]{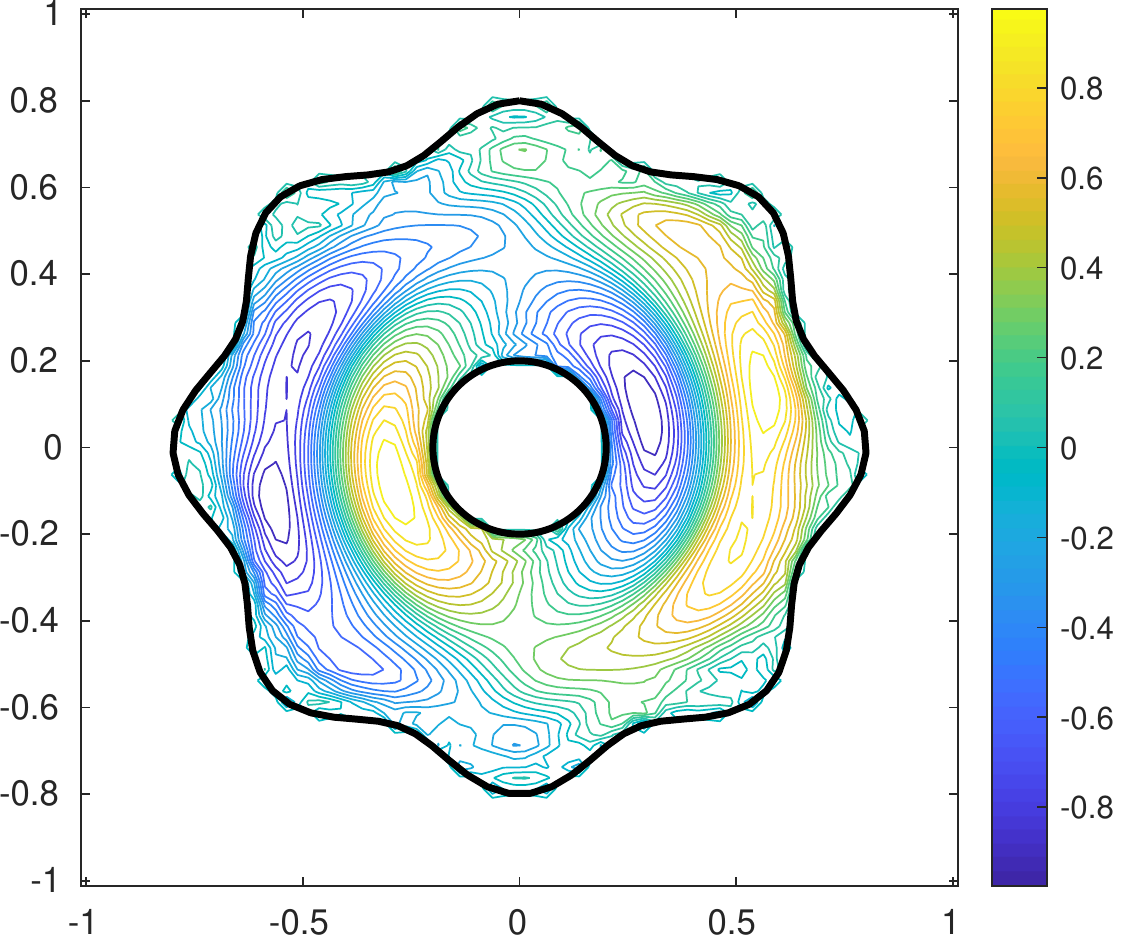}
		\caption{ \bf $u_2, \varepsilon = 10^{-4}$}
	\end{subfigure}\\
	\begin{subfigure}{0.2\textwidth}
		\includegraphics[width=\textwidth]{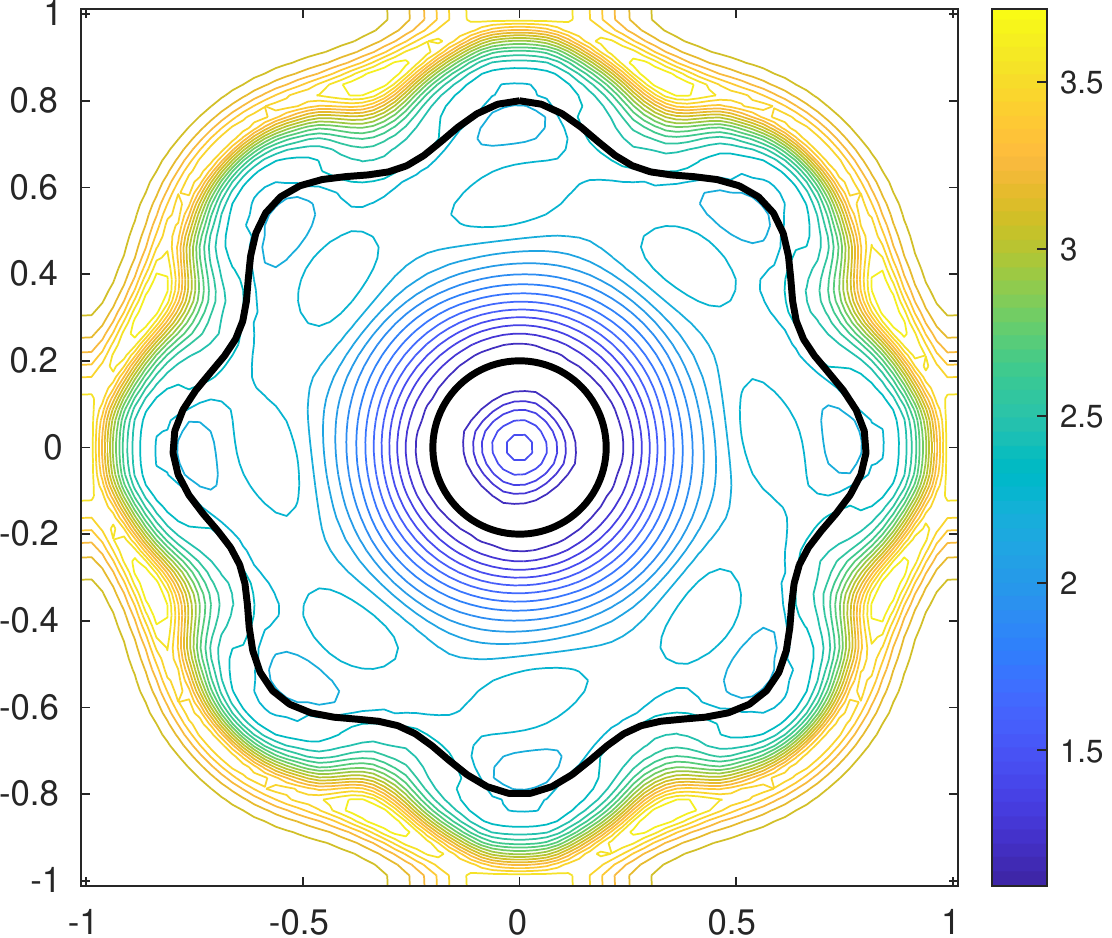}
		\caption{\bf $\vartheta, \varepsilon = 10^{-1}$}
	\end{subfigure}
	\begin{subfigure}{0.2\textwidth}
		\includegraphics[width=\textwidth]{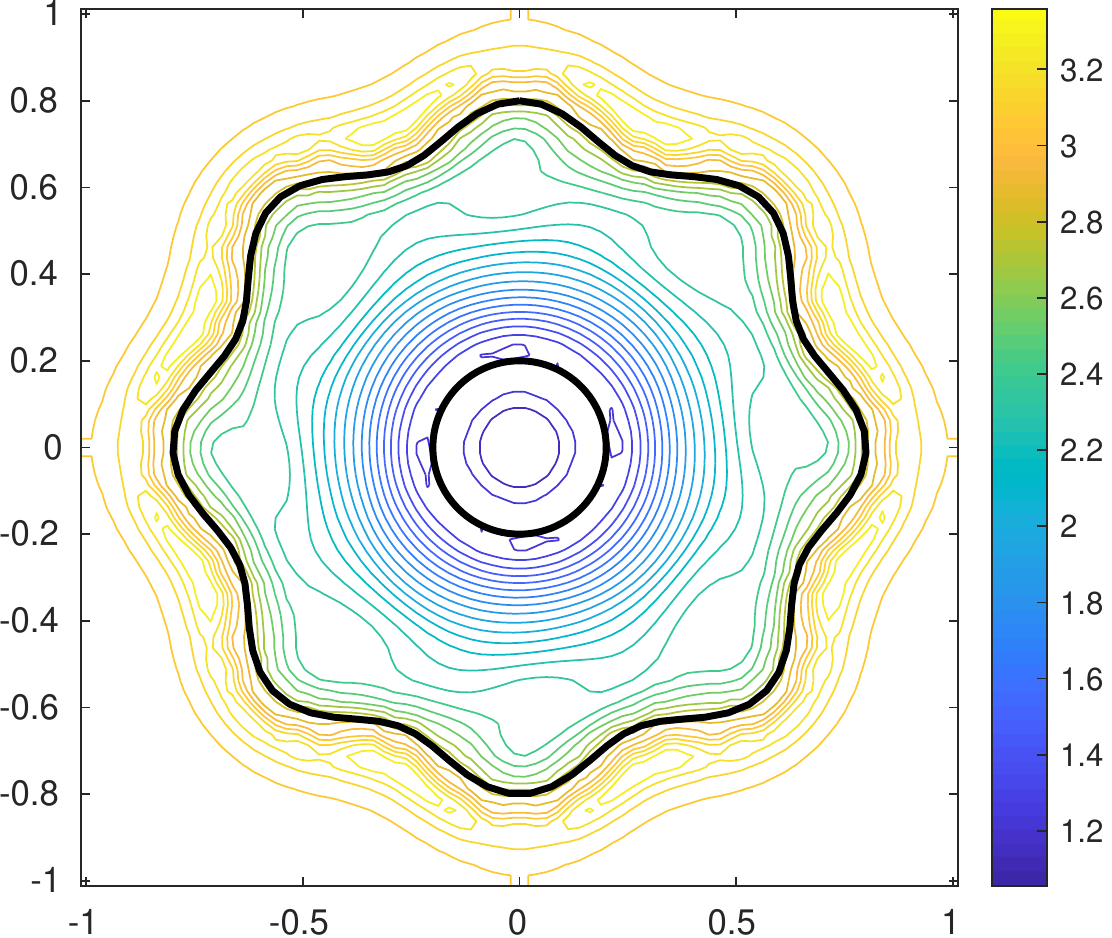}
		\caption{\bf $\vartheta, \varepsilon = 10^{-2}$}
	\end{subfigure}
	\begin{subfigure}{0.2\textwidth}
		\includegraphics[width=\textwidth]{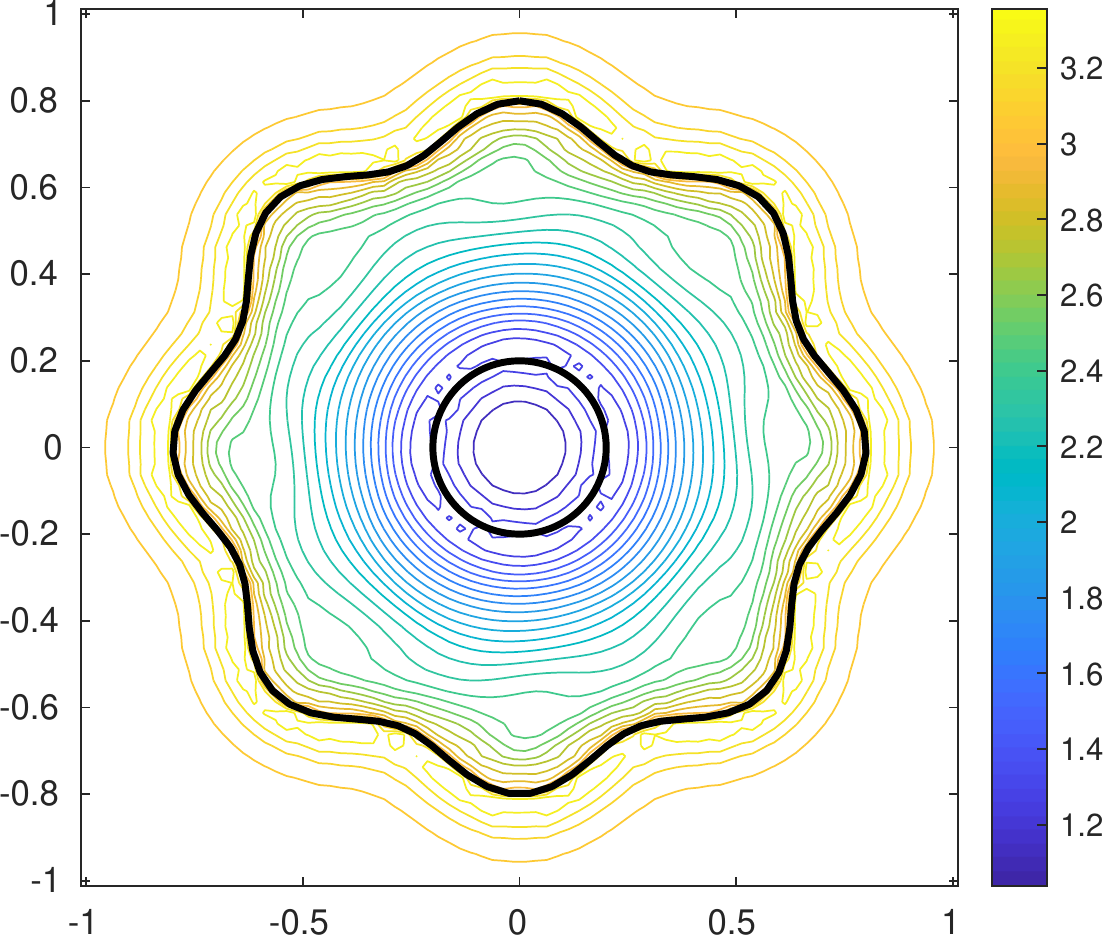}
		\caption{\bf $\vartheta, \varepsilon = 10^{-3}$}
	\end{subfigure}
	\begin{subfigure}{0.2\textwidth}
		\includegraphics[width=\textwidth]{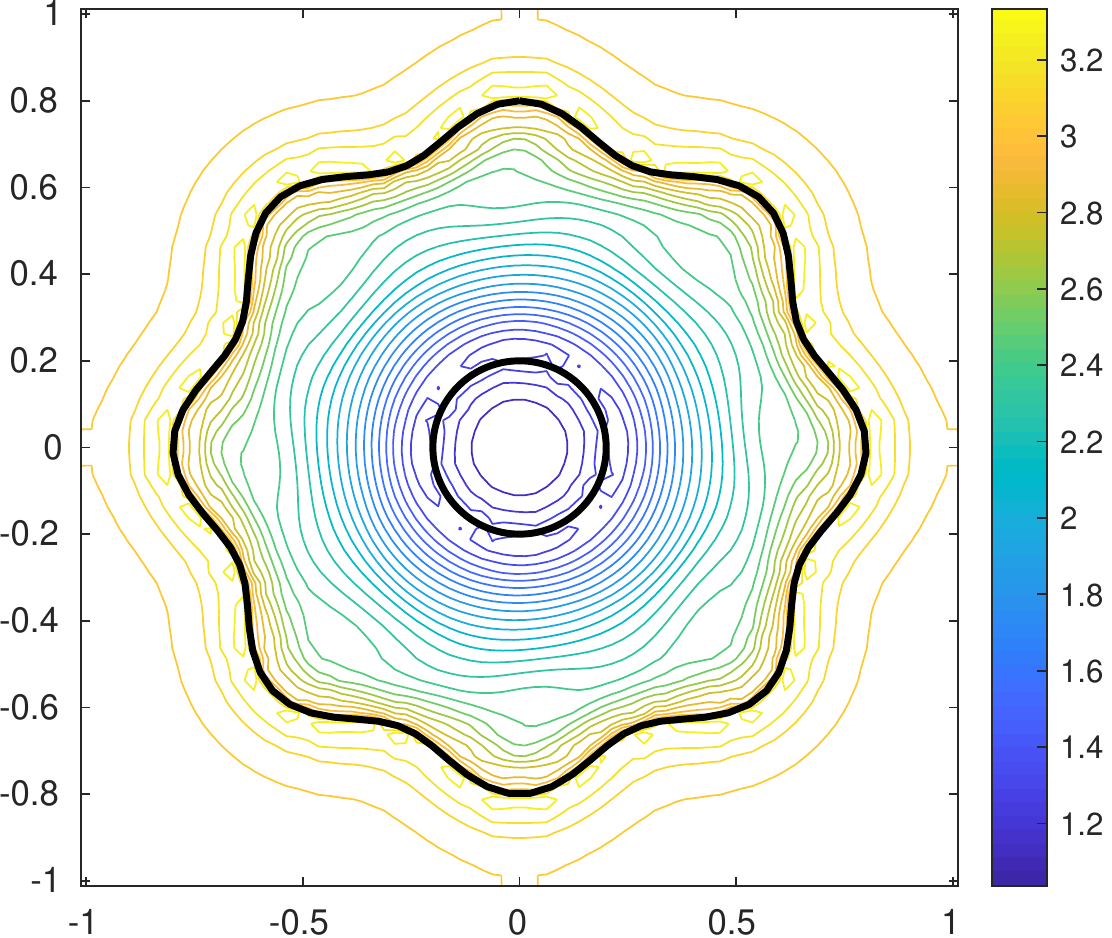}
		\caption{\bf $\vartheta, \varepsilon = 10^{-4}$}
	\end{subfigure}\\
	\begin{subfigure}{0.2\textwidth}
		\includegraphics[width=\textwidth]{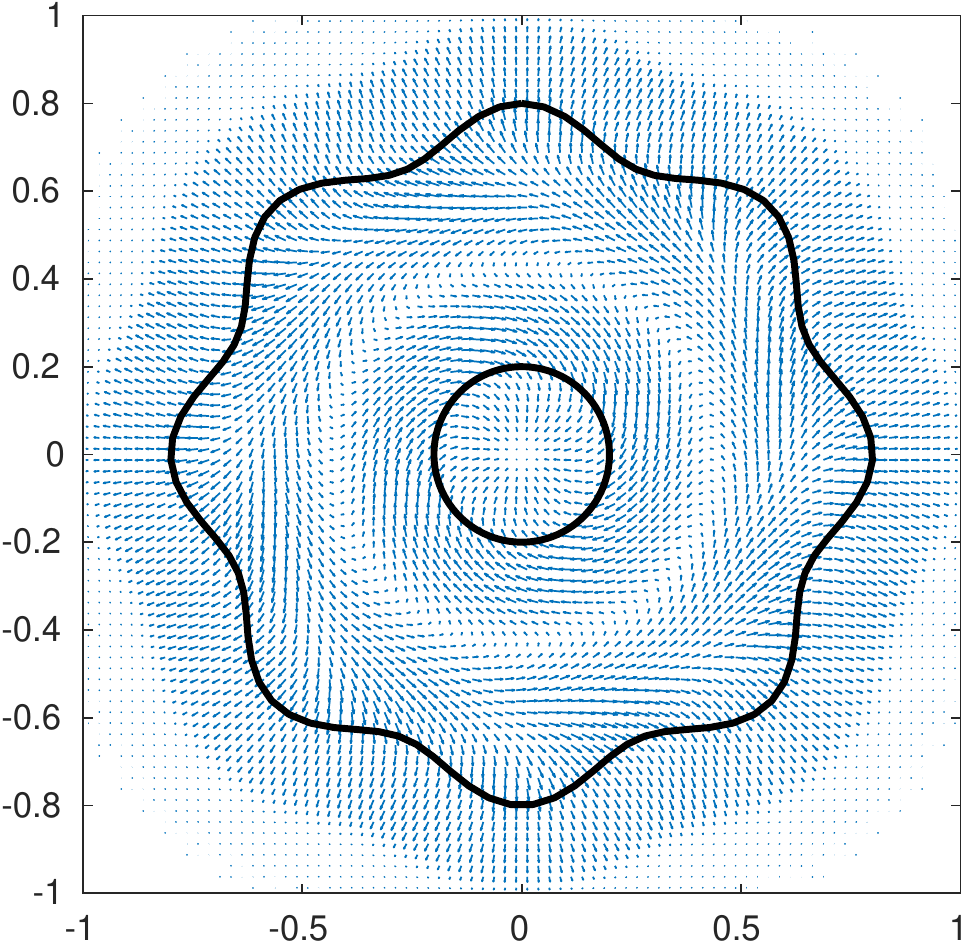}
		\caption{\bf $\vu, \varepsilon = 10^{-1}$}
	\end{subfigure}
	\begin{subfigure}{0.2\textwidth}
		\includegraphics[width=\textwidth]{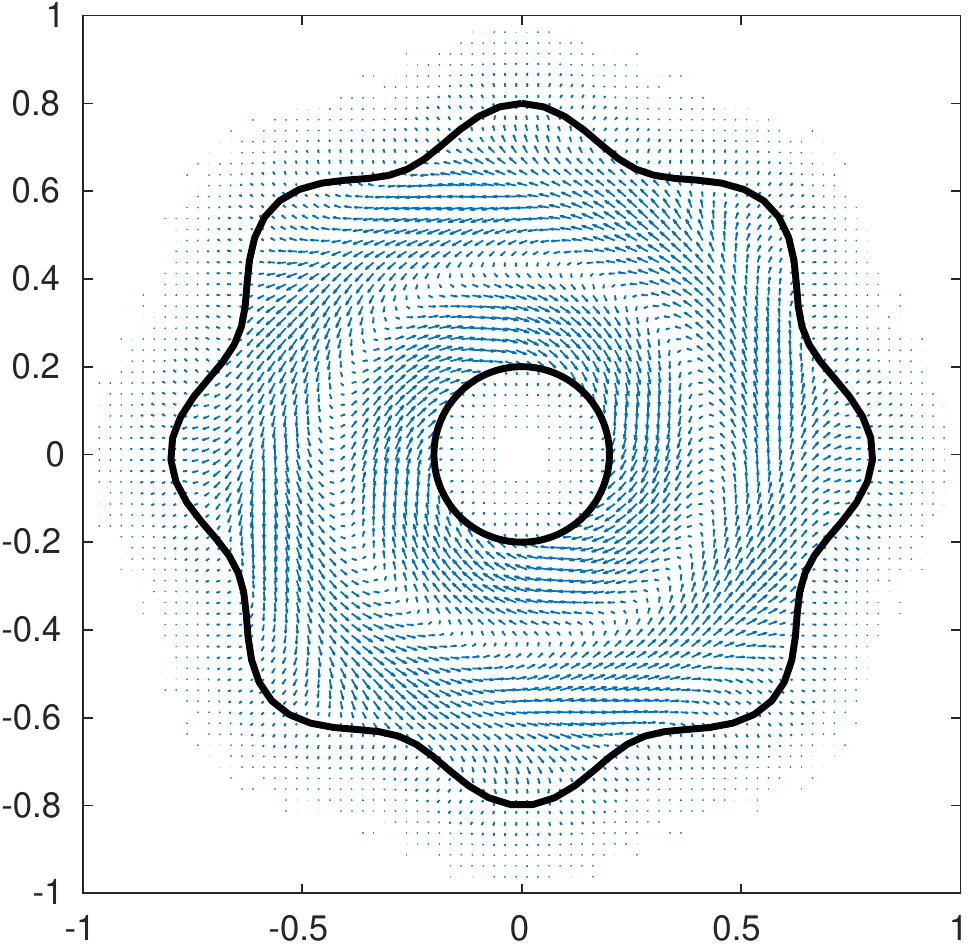}
		\caption{\bf $\vu, \varepsilon = 10^{-2}$}
	\end{subfigure}	
	\begin{subfigure}{0.2\textwidth}
		\includegraphics[width=\textwidth]{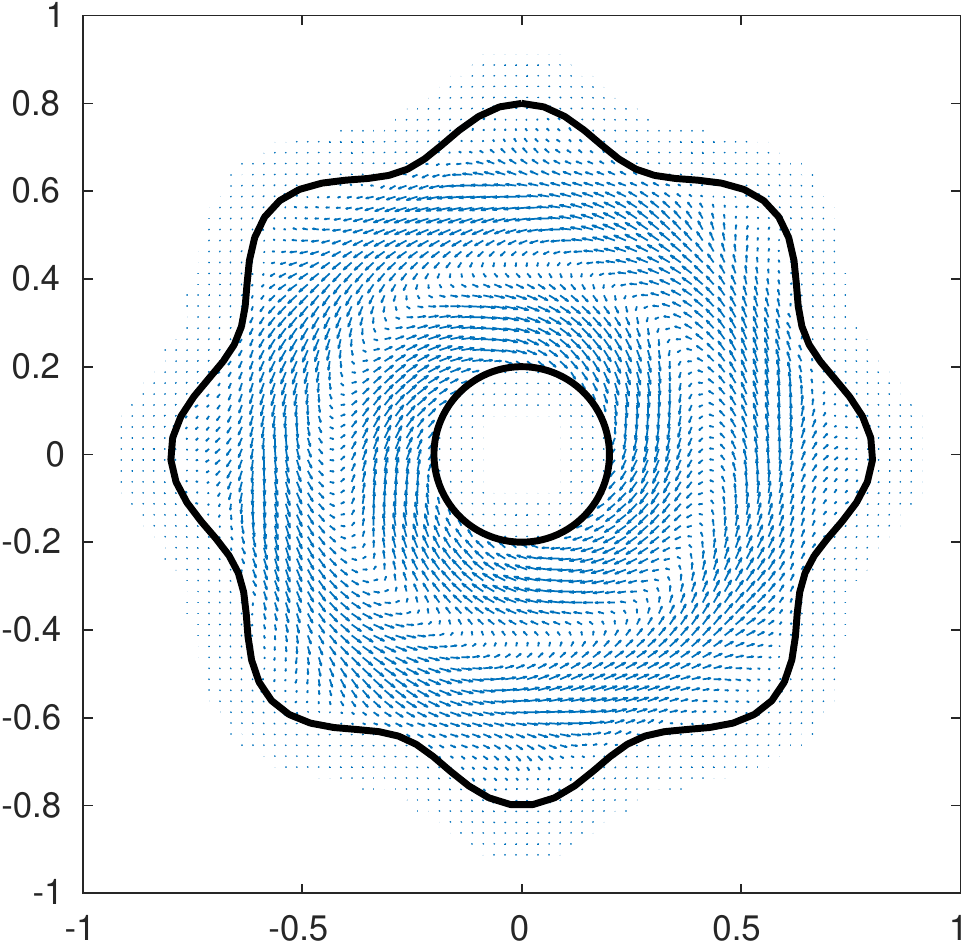}
		\caption{\bf $\vu, \varepsilon = 10^{-3}$}
	\end{subfigure}		
	\begin{subfigure}{0.2\textwidth}
		\includegraphics[width=\textwidth]{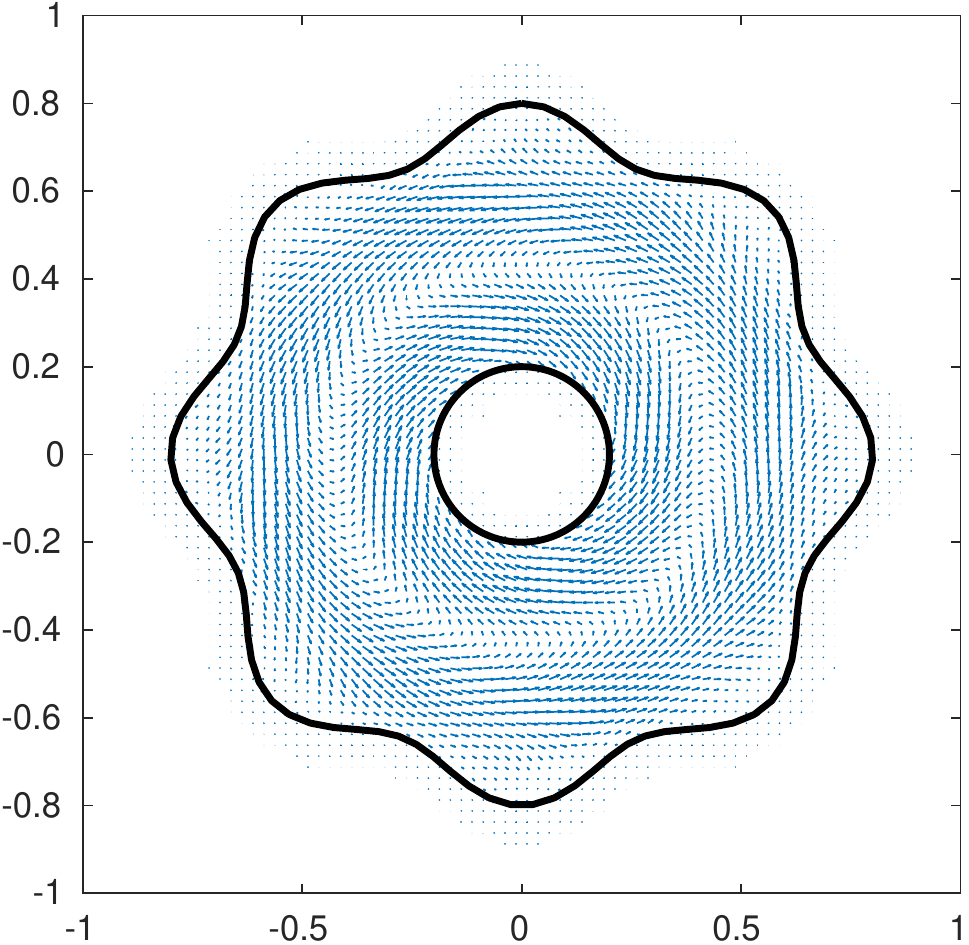}
		\caption{\bf $\vu, \varepsilon = 10^{-4}$}
	\end{subfigure}
	\caption{\small{Experiment~3: ${U}_h^{\varepsilon}$ with $h = 2/80$ and $\varepsilon = 10^{-1},\dots, 10^{-4}$.}}\label{fig:ex3}
\end{figure}

\subsection{Experiment~4: Ring domain - zero density outside $\Omega$, non-zero gravity force}

In the last experiment we extend the setting of Experiment~2 by adding an external force pointing to the center (0,0)
defined by
\begin{equation*}
\textbf{g} = \left(-100 \frac{x_1}{  |x|} , -100 \frac{x_2}{  |x|}\right).
\end{equation*}
In order to observe interesting phenomena the initial data are taken as
\begin{equation*}
	(\varrho, \vu, \vartheta)(0,x)
	\; = \; \begin{cases}
	(10^{-2},0, 0 , 30 ) , & x \in B_{0.2}, \\
	\left(1, \frac{5 \sin(4\pi (|x|-0.2)) x_2}{|x|} ,-\frac{ 5\sin(4\pi (|x|-0.2)) x_1}{|x|}, 41.6 - 58|x| \right) , & x \in  \Omega \equiv {B}_{0.7}\setminus B_{0.2}, \\
	(10^{-2},0 , 0,  1) , & x \in \mathbb{T}^2\setminus B_{0.7}.
	\end{cases}
\end{equation*}
The final time is set to $T = 0.2$.

The errors $E(U_h^{\varepsilon}), P(U_h^{\varepsilon})$ are plotted in Figure~\ref{fig:ex4-1} and Figure~\ref{fig:ex4-2}, respectively.
The results indicate that the numerical solutions converge with respect to $h$ and $\varepsilon$ with rate nearly $1$ and $1/2$, respectively.
The numerical solutions for different penalization parameters $\varepsilon = 10^{-1},\dots,10^{-4}$ on the mesh of $80^2$ cells are shown in
Figure~\ref{fig:ex4}.

Comparing these results with those of previous Experiments we can see oscillatory fluid behavior which is due to the development of the so-called Rayleigh-B\'enard convection rolls. These are  visible in the temperature plots in Figure~\ref{fig:ex4} and arise due to the temperate gradient acting against the outer force, see \cite{RB} for more details.

\begin{figure}[htbp]
	\setlength{\abovecaptionskip}{0.cm}
	\setlength{\belowcaptionskip}{-0.cm}
	\centering
	\begin{subfigure}{0.32\textwidth}
		\includegraphics[width=\textwidth]{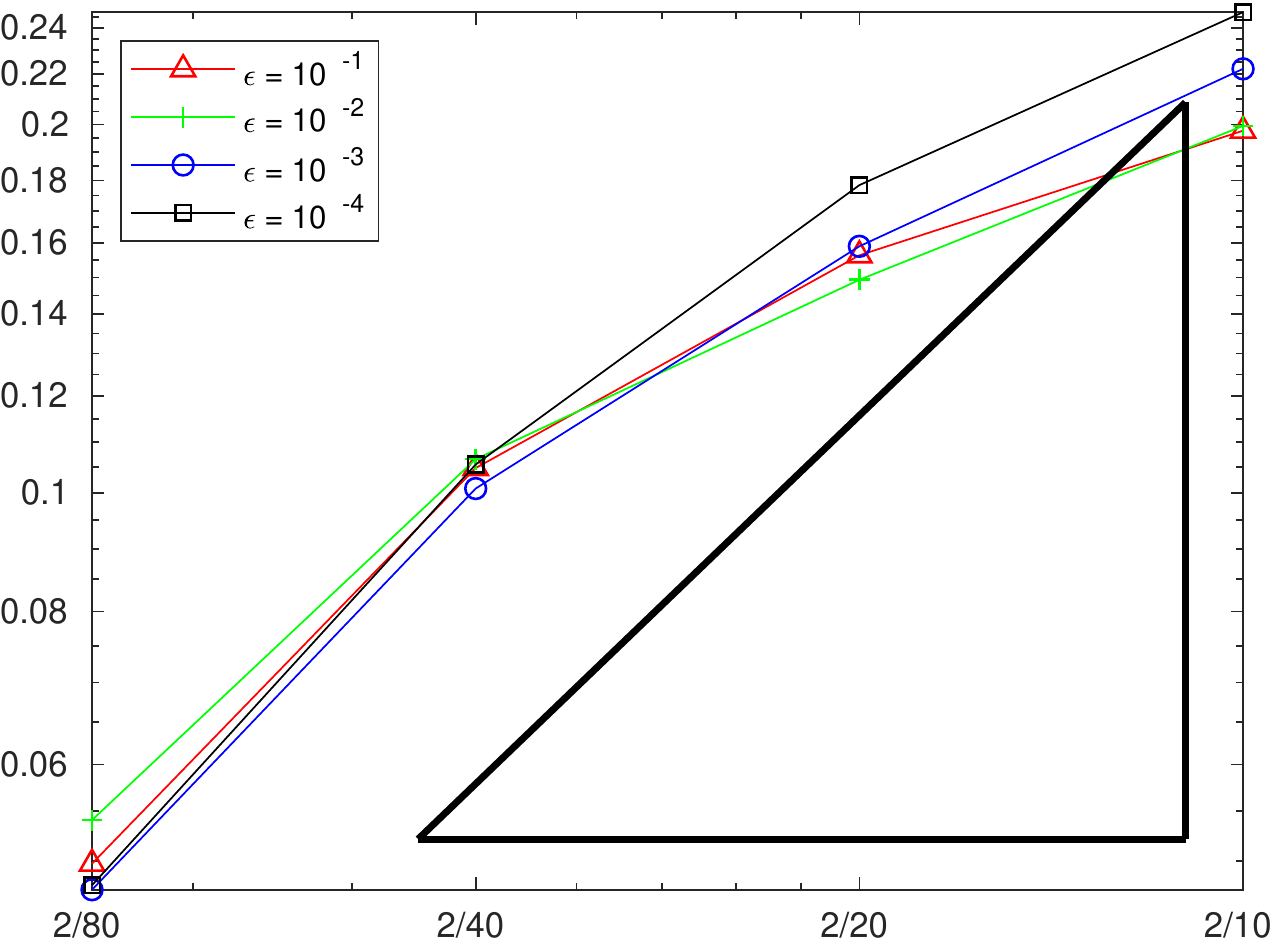}
		\caption{ \bf $\varrho$}
	\end{subfigure}
	\begin{subfigure}{0.32\textwidth}
	\includegraphics[width=\textwidth]{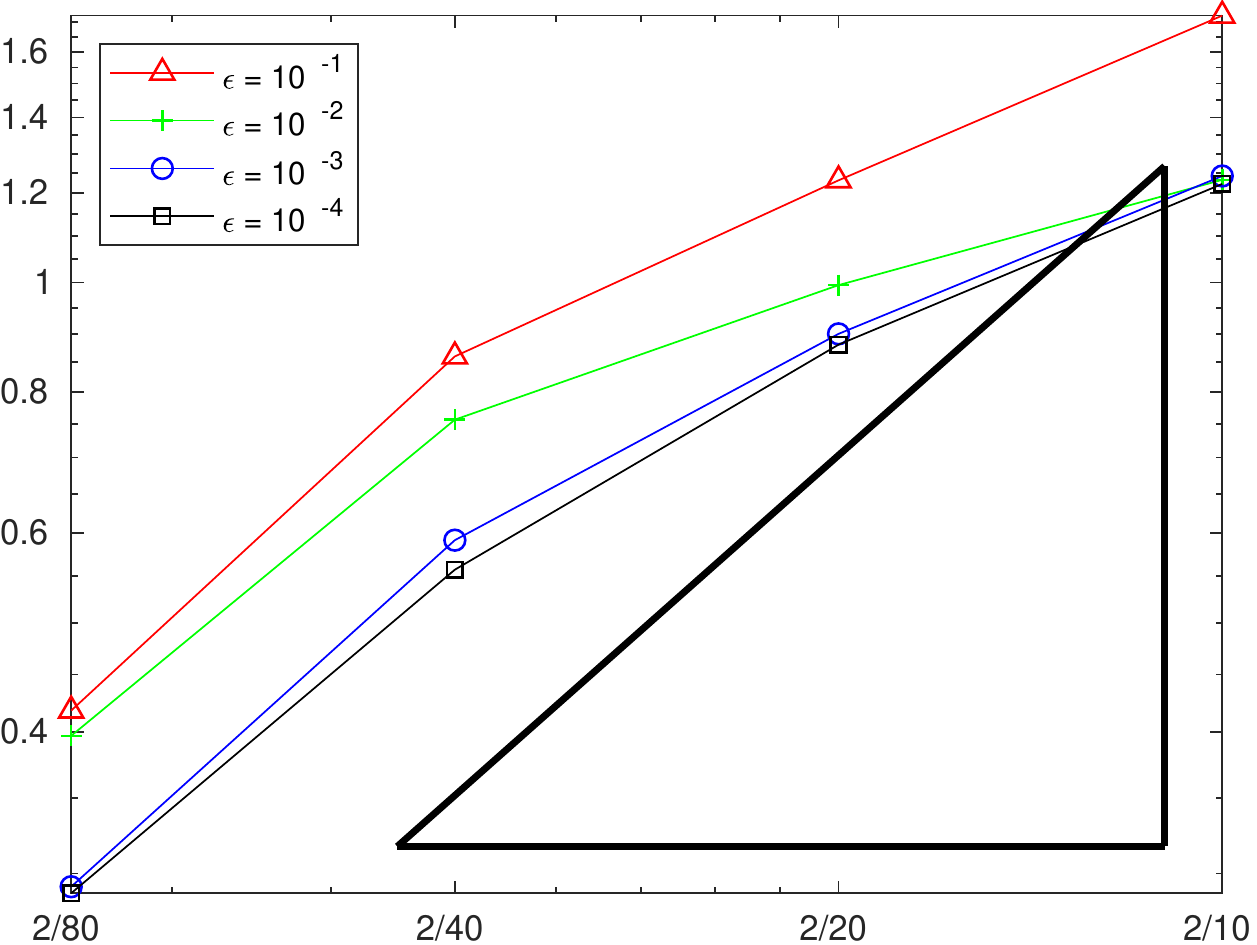}
	\caption{ \bf $\textbf{u}$}
\end{subfigure}
	\begin{subfigure}{0.32\textwidth}
	\includegraphics[width=\textwidth]{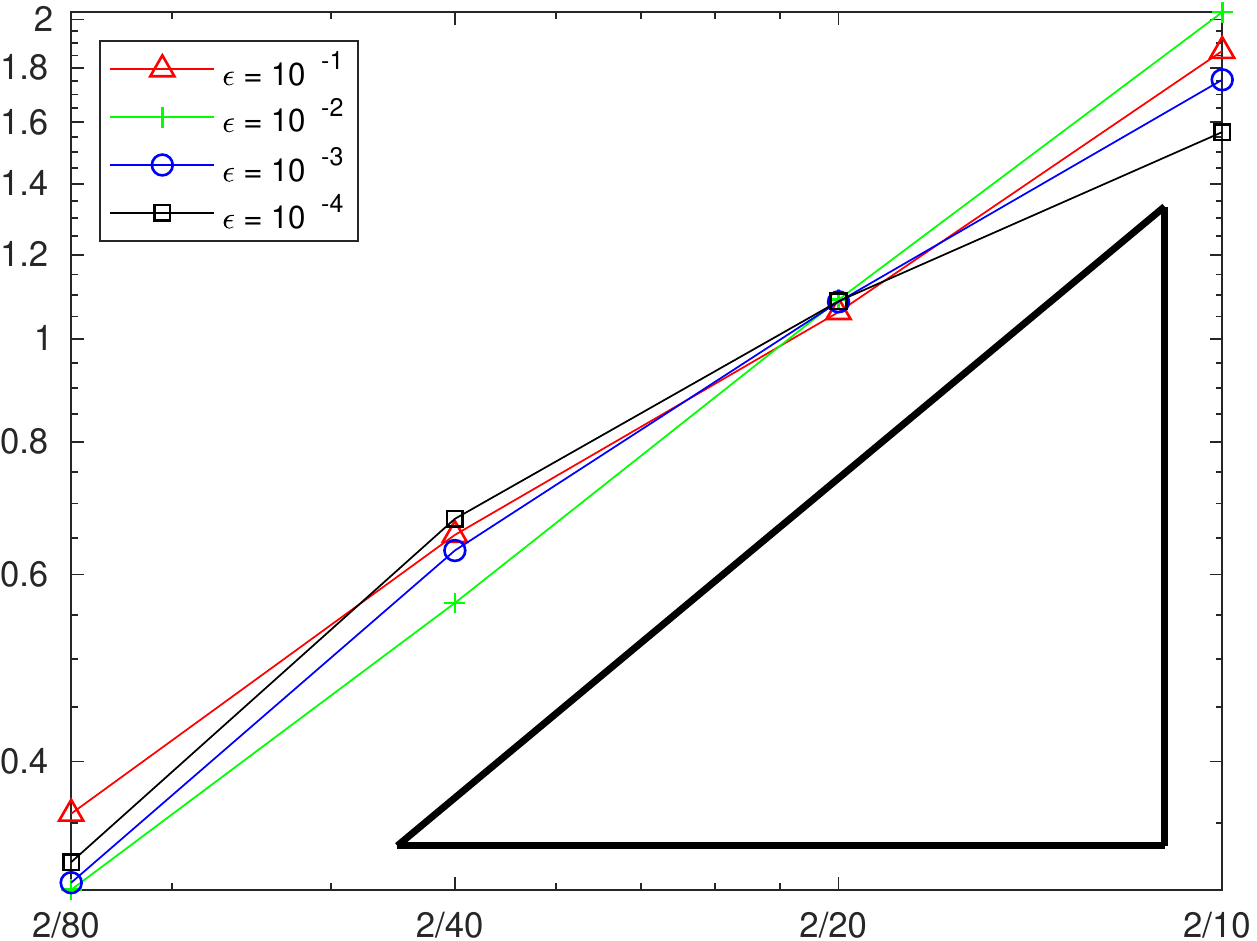}
	\caption{ \bf $\vartheta$}
\end{subfigure}
	\caption{\small{Experiment~4: $E(U_h^{\varepsilon})$ errors  with respect to $h$. }}\label{fig:ex4-1}
\end{figure}

\begin{figure}[htbp]
	\setlength{\abovecaptionskip}{0.cm}
	\setlength{\belowcaptionskip}{-0.cm}
	\centering
	\begin{subfigure}{0.32\textwidth}
		\includegraphics[width=\textwidth]{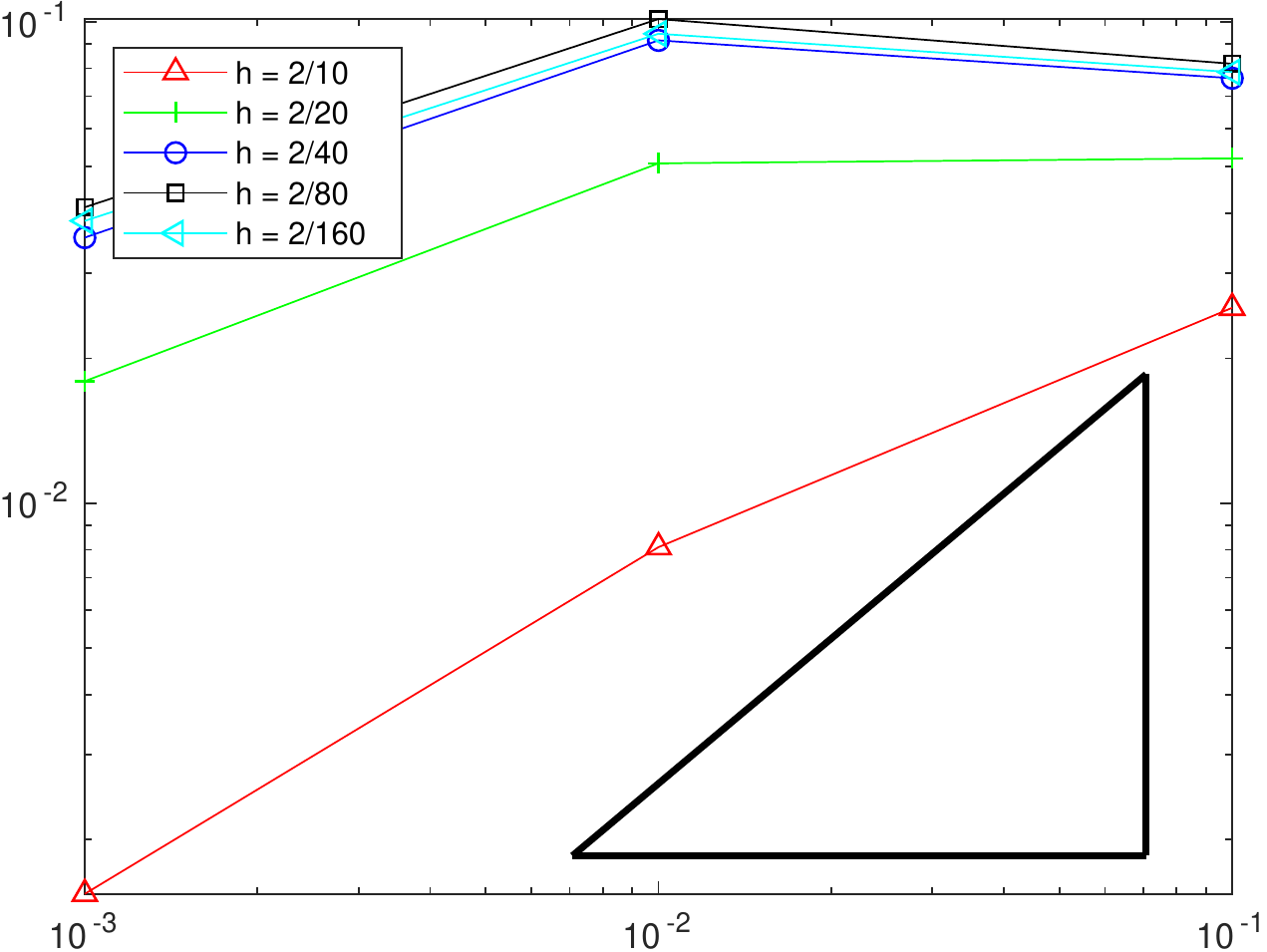}
		\caption{ \bf $\varrho$}
	\end{subfigure}
	\begin{subfigure}{0.32\textwidth}
		\includegraphics[width=\textwidth]{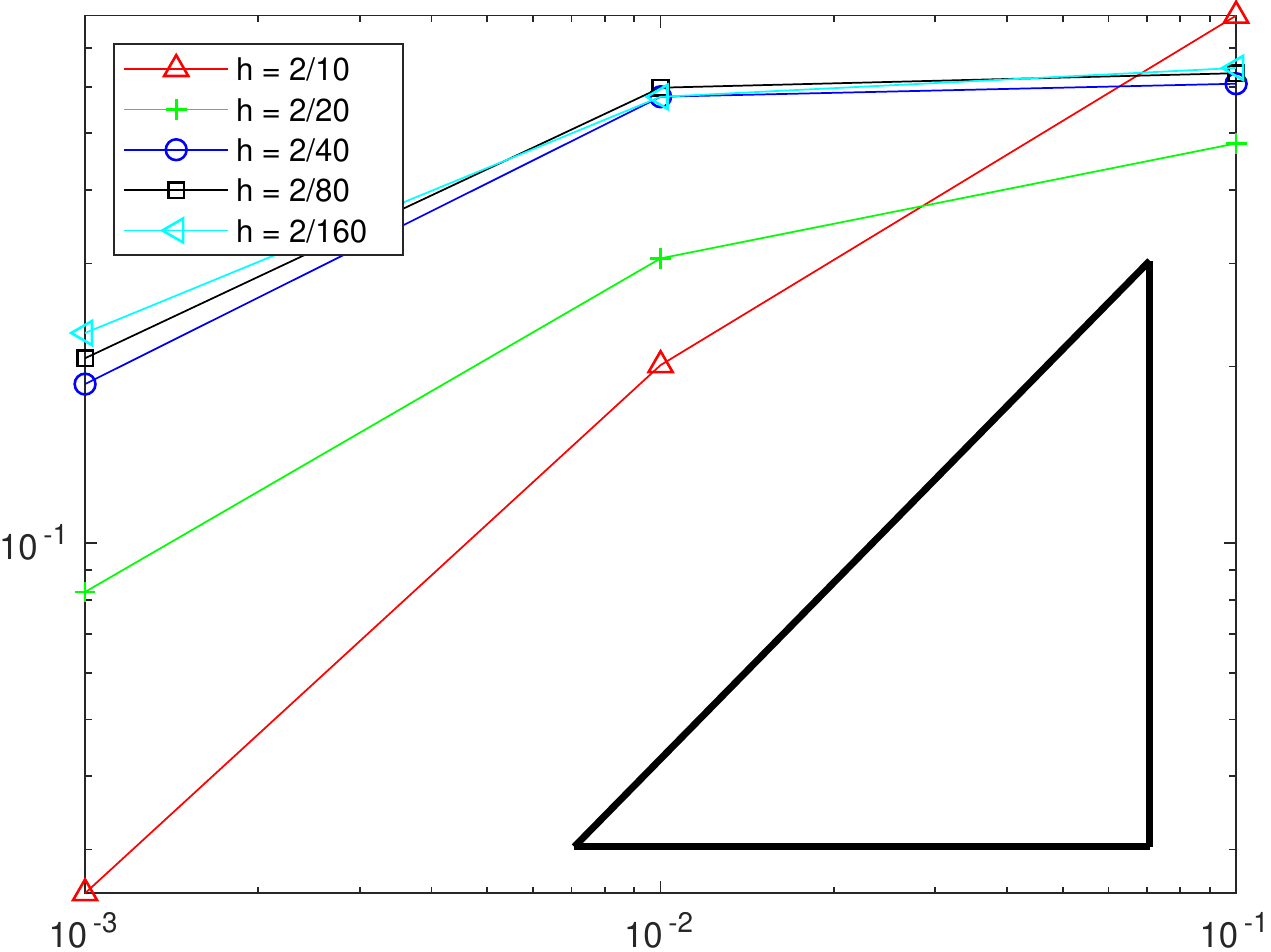}
		\caption{ \bf $\textbf{u}$}
	\end{subfigure}
	\begin{subfigure}{0.32\textwidth}
		\includegraphics[width=\textwidth]{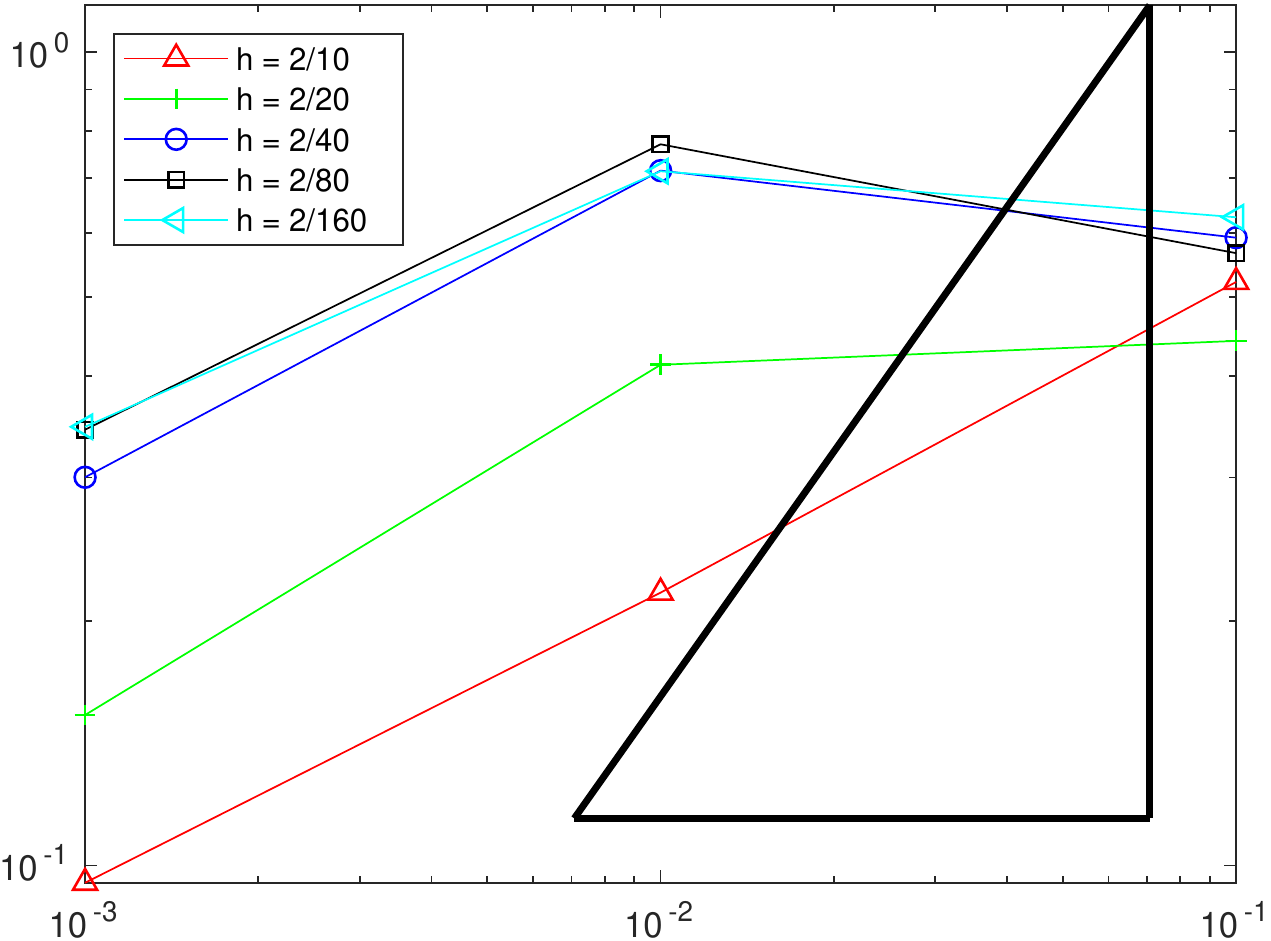}
		\caption{ \bf $\vartheta$}
	\end{subfigure}
	\caption{{\small Experiment~4: $P(U_h^{\varepsilon})$ errors with respect to $\varepsilon$. }}\label{fig:ex4-2}
\end{figure}

\begin{figure}[htbp]
	\setlength{\abovecaptionskip}{0.cm}
	\setlength{\belowcaptionskip}{-0.cm}
	\centering
	\begin{subfigure}{0.2\textwidth}
	\includegraphics[width=\textwidth]{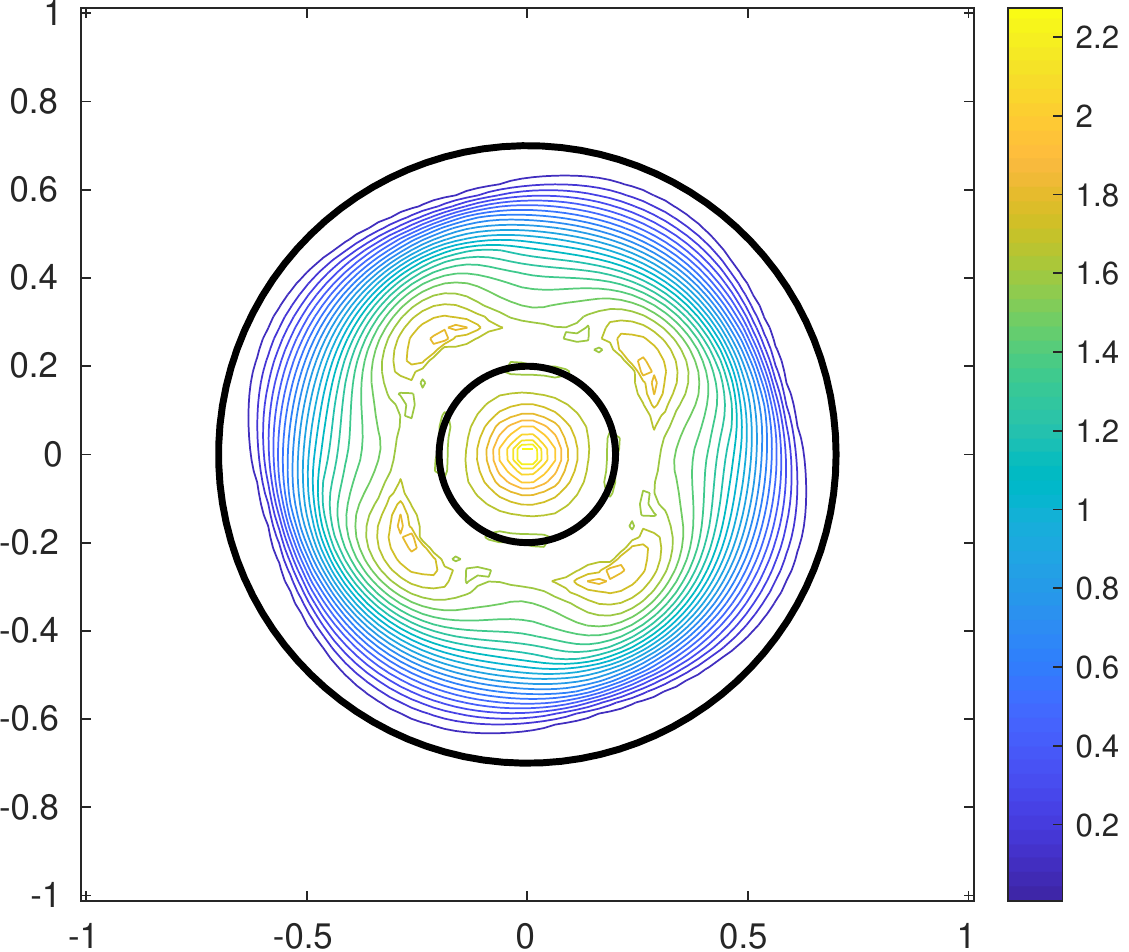}
	\caption{ \bf $\varrho, \varepsilon = 10^{-1}$}
	\end{subfigure}
	\begin{subfigure}{0.2\textwidth}
	\includegraphics[width=\textwidth]{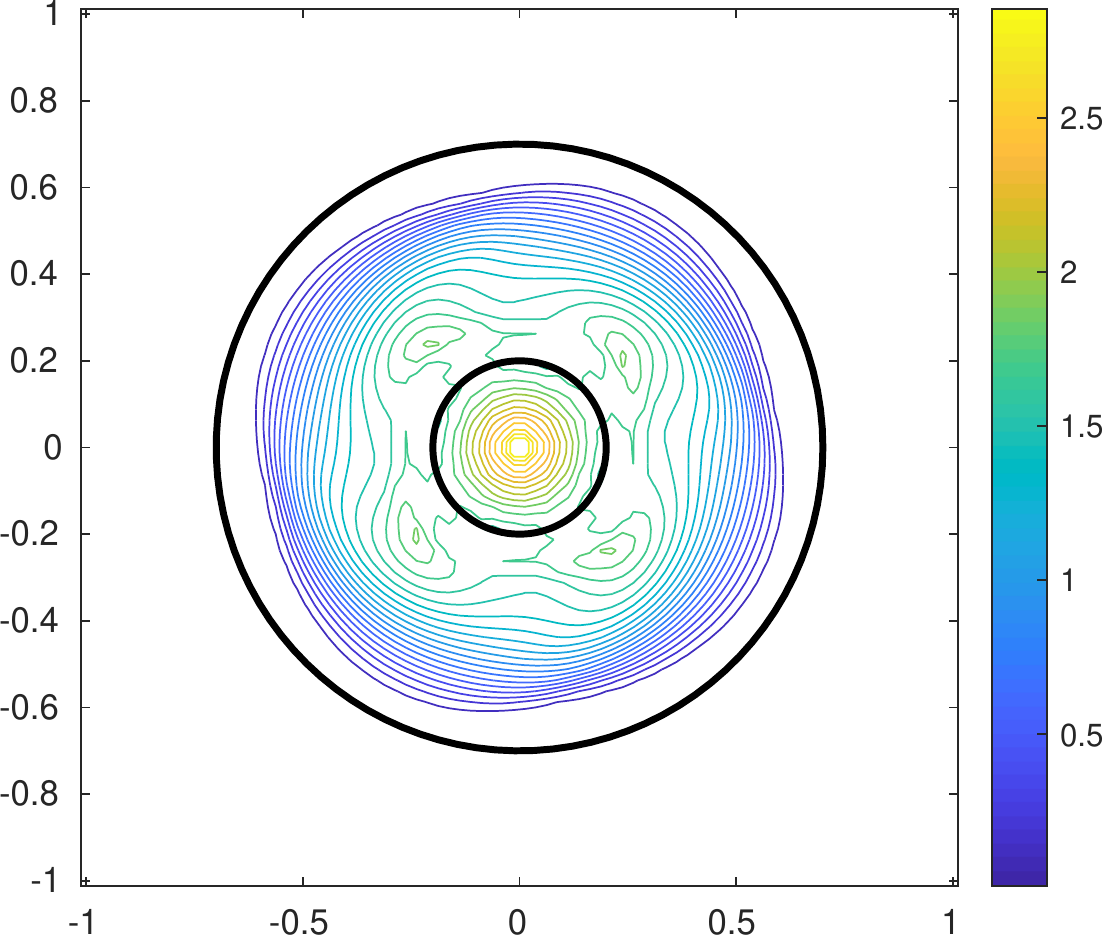}
	\caption{ \bf $\varrho, \varepsilon = 10^{-2}$}
	\end{subfigure}
	\begin{subfigure}{0.2\textwidth}
	\includegraphics[width=\textwidth]{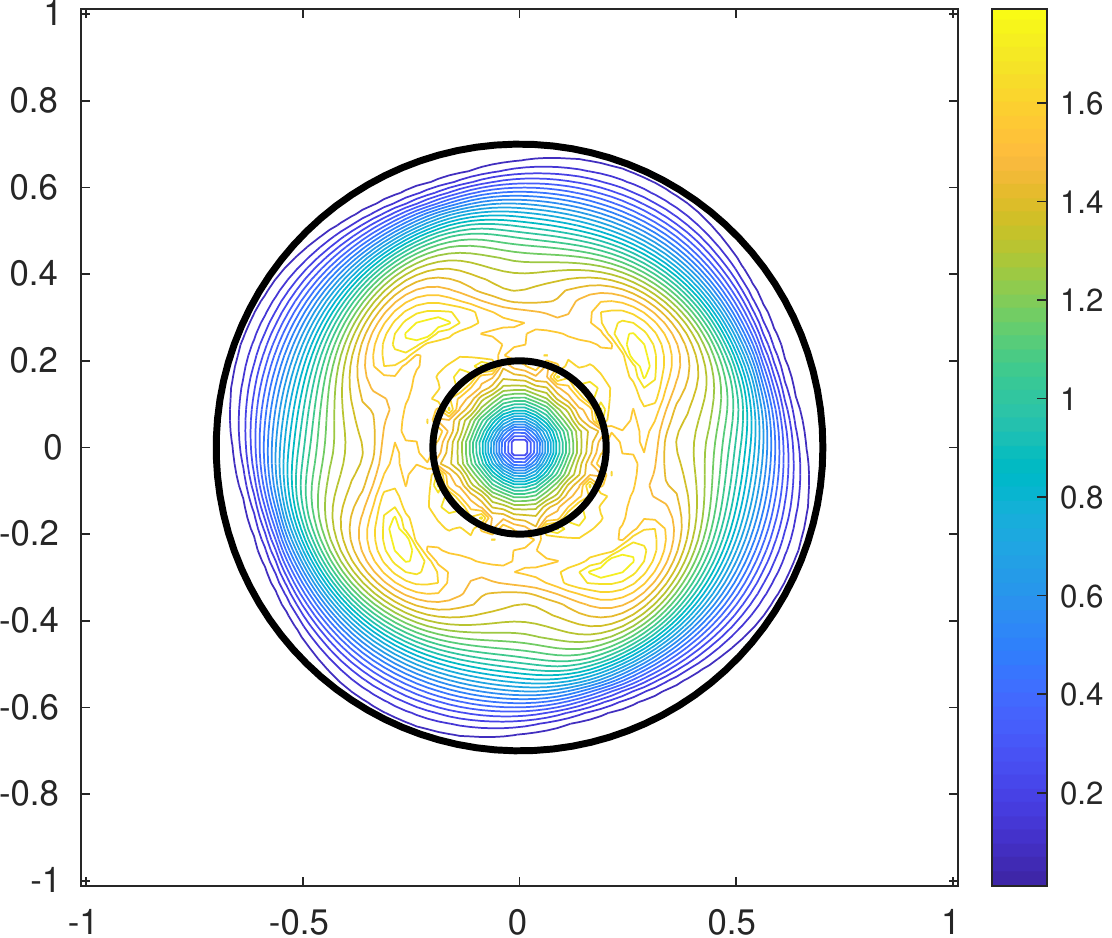}
	\caption{ \bf $\varrho, \varepsilon = 10^{-3}$}
	\end{subfigure}
	\begin{subfigure}{0.2\textwidth}
	\includegraphics[width=\textwidth]{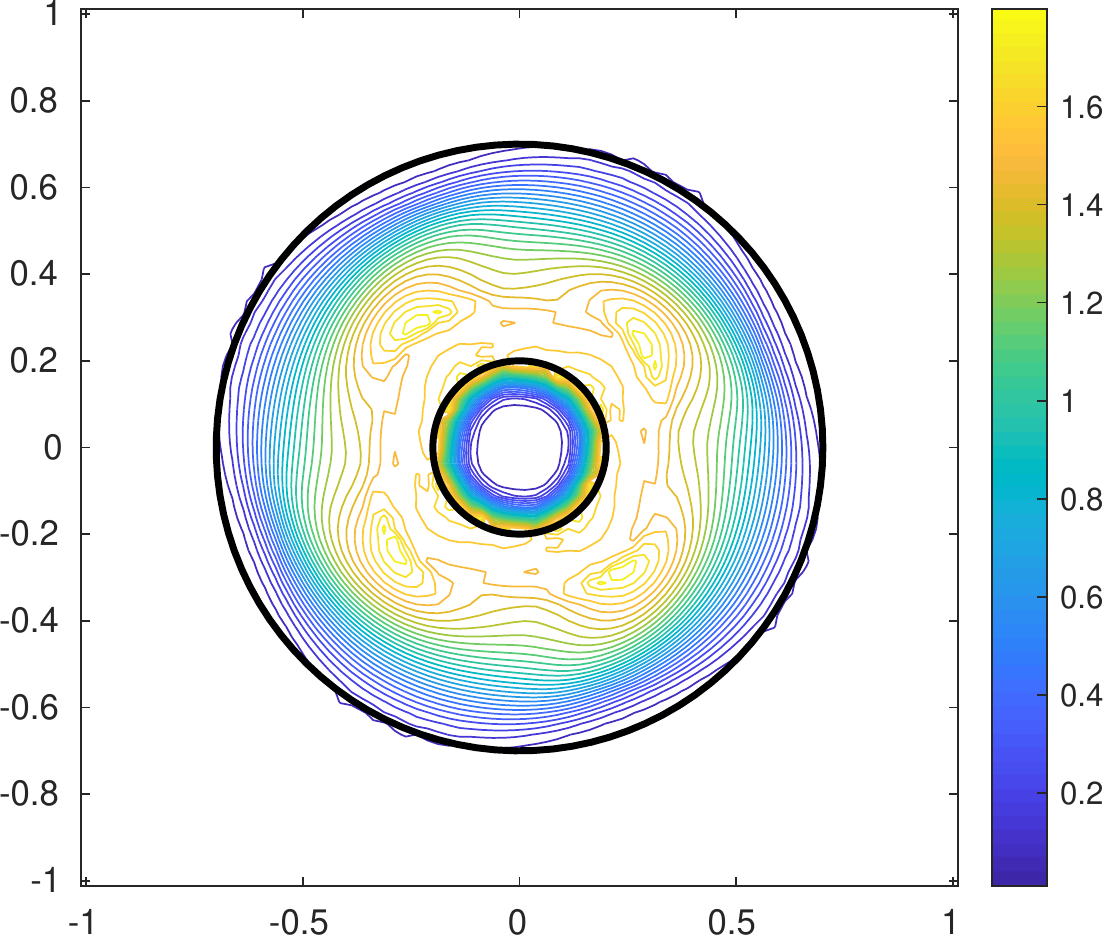}
	\caption{ \bf $\varrho, \varepsilon = 10^{-4}$}
	\end{subfigure}\\
	\begin{subfigure}{0.2\textwidth}
		\includegraphics[width=\textwidth]{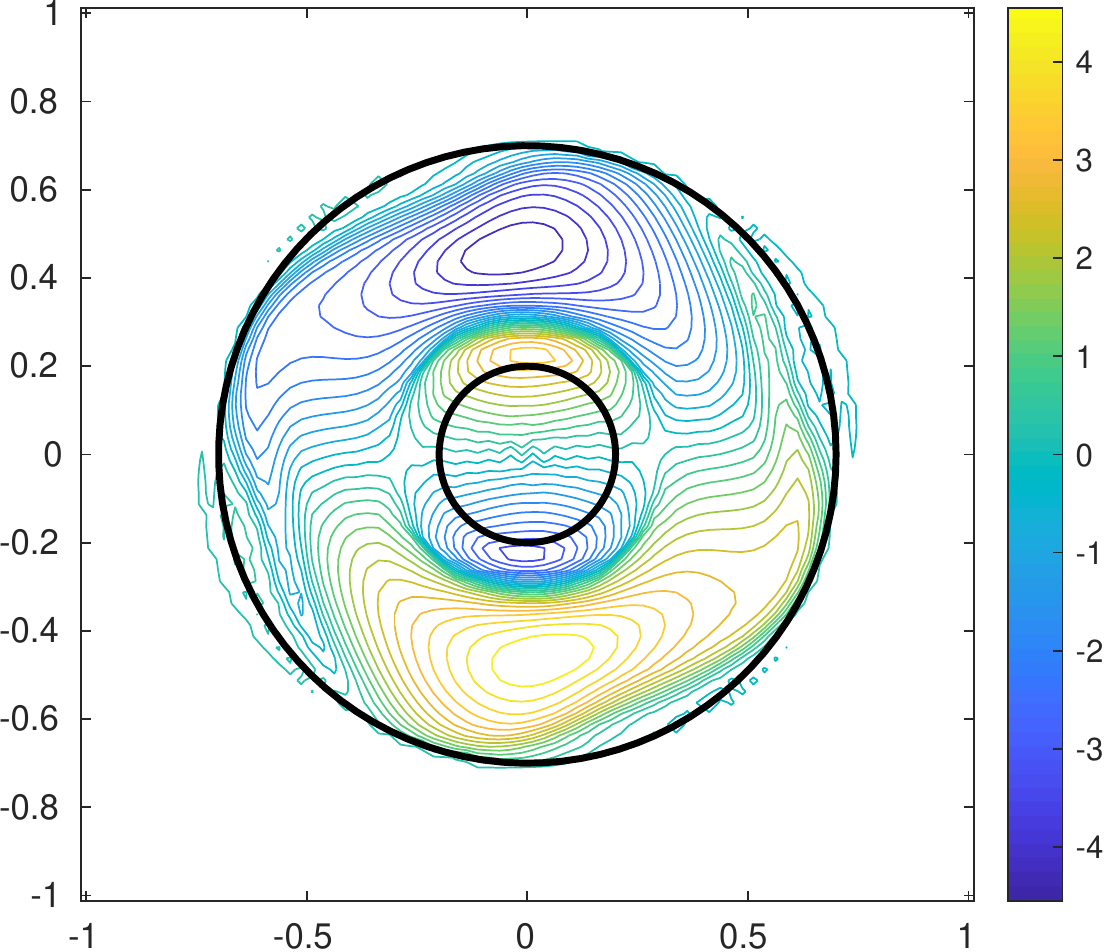}
		\caption{\bf $u_1, \varepsilon = 10^{-1}$}
	\end{subfigure}	
	\begin{subfigure}{0.2\textwidth}
	\includegraphics[width=\textwidth]{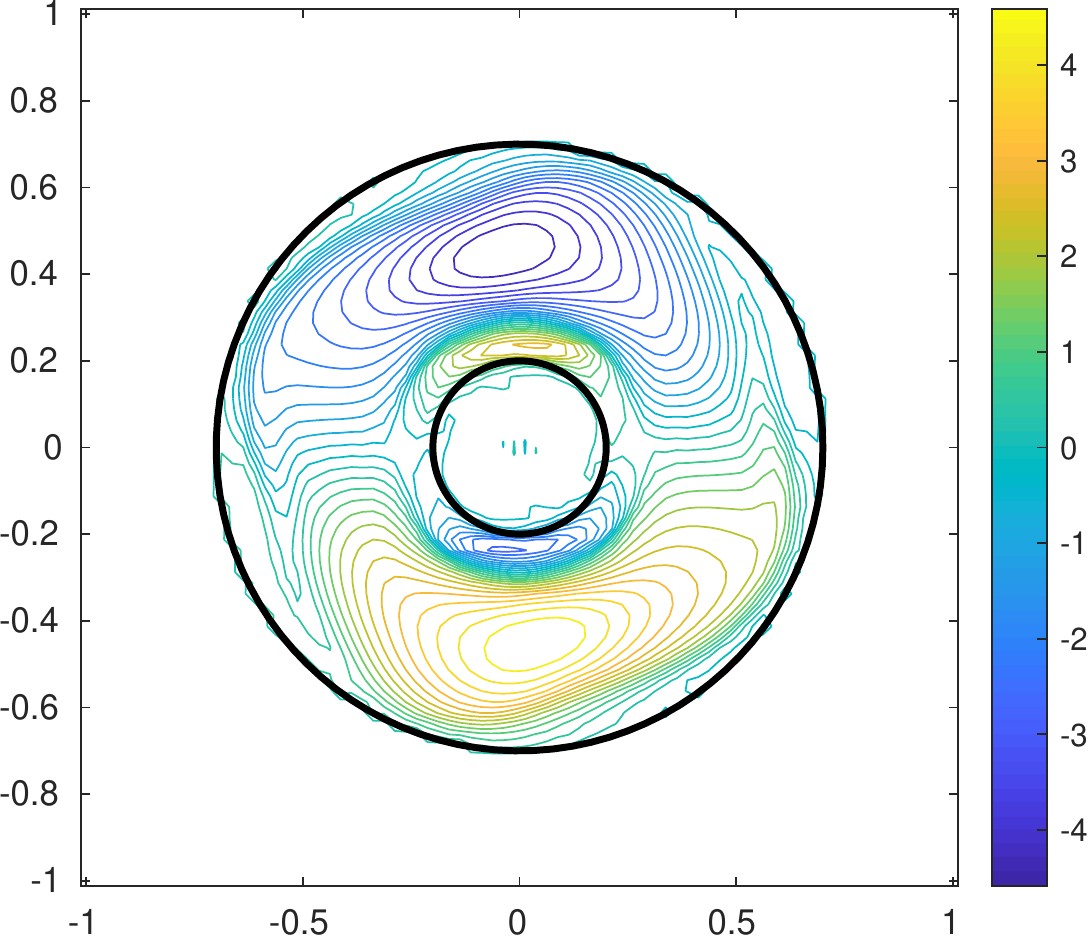}
	\caption{\bf $u_1, \varepsilon = 10^{-2}$}
	\end{subfigure}
	\begin{subfigure}{0.2\textwidth}
	\includegraphics[width=\textwidth]{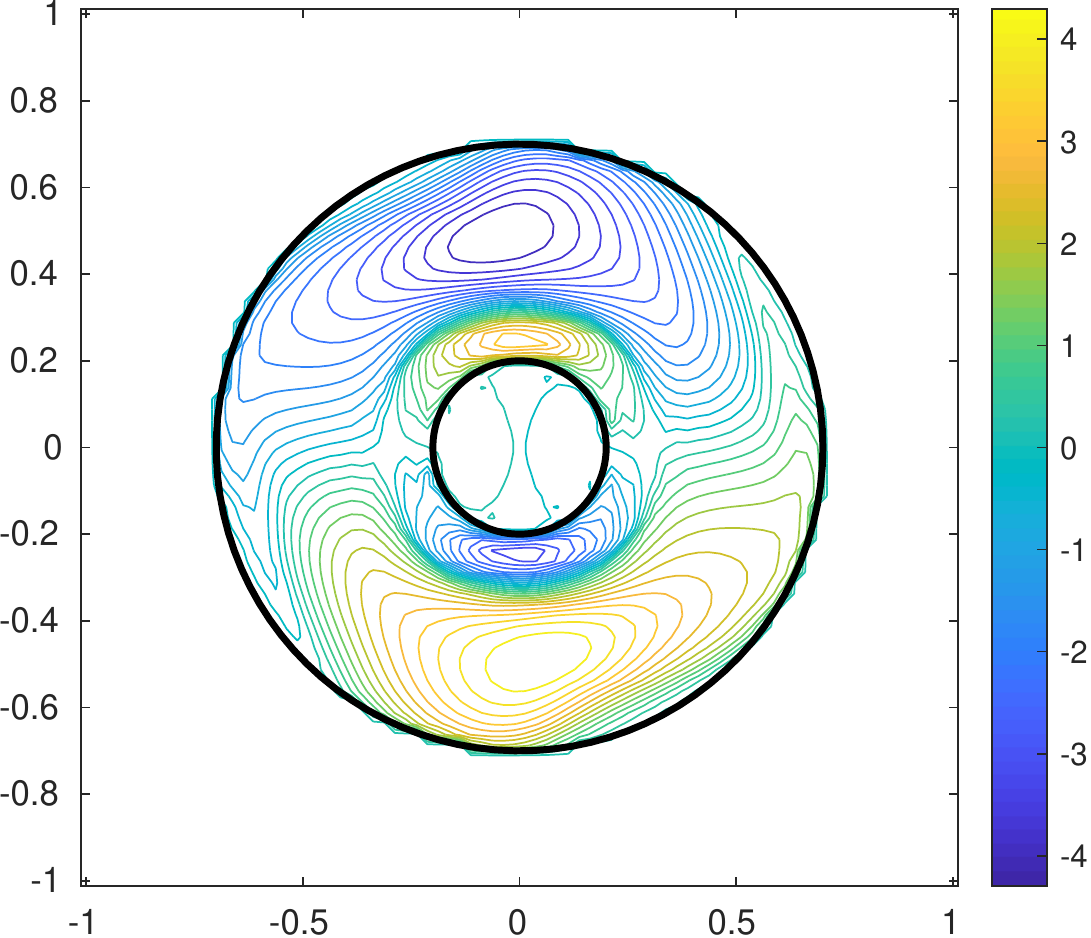}
	\caption{\bf $u_1, \varepsilon = 10^{-3}$}
	\end{subfigure}
	\begin{subfigure}{0.2\textwidth}
	\includegraphics[width=\textwidth]{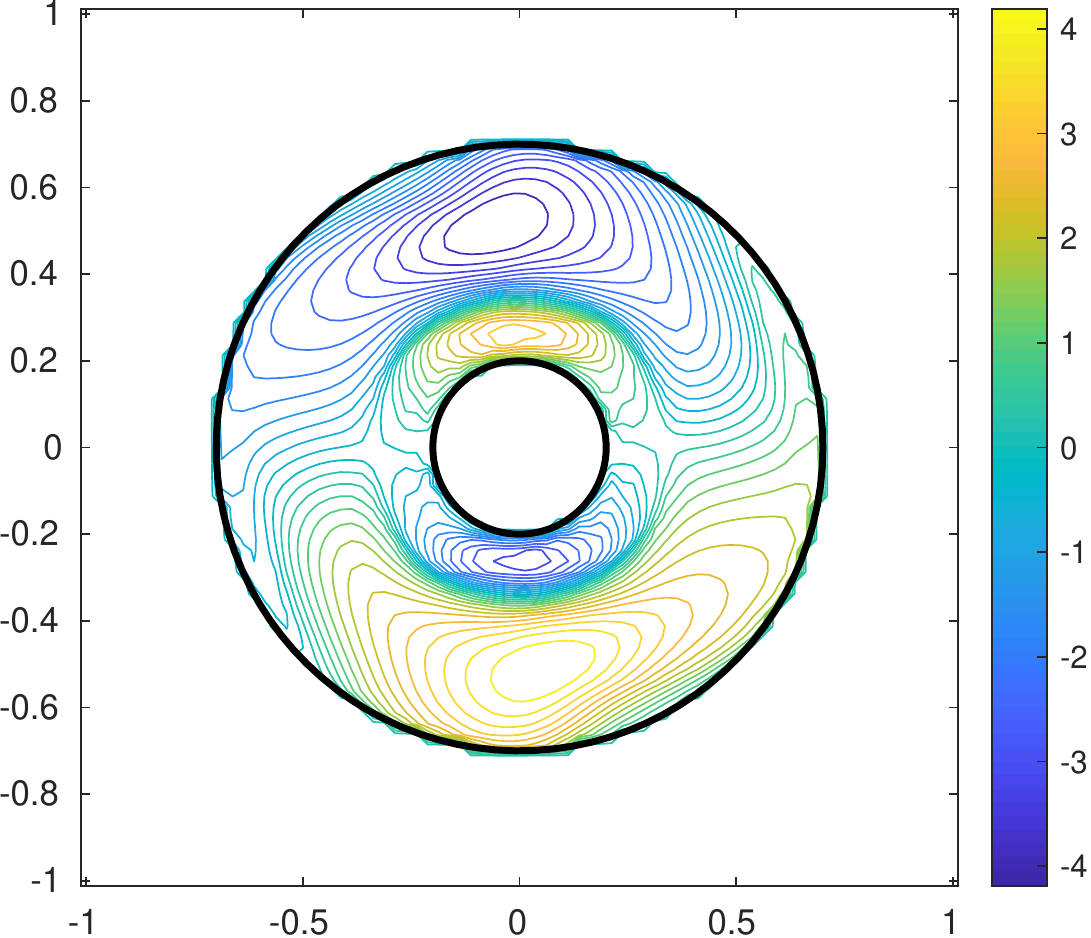}
	\caption{\bf $u_1, \varepsilon = 10^{-4}$}
	\end{subfigure}\\
	\begin{subfigure}{0.2\textwidth}
		\includegraphics[width=\textwidth]{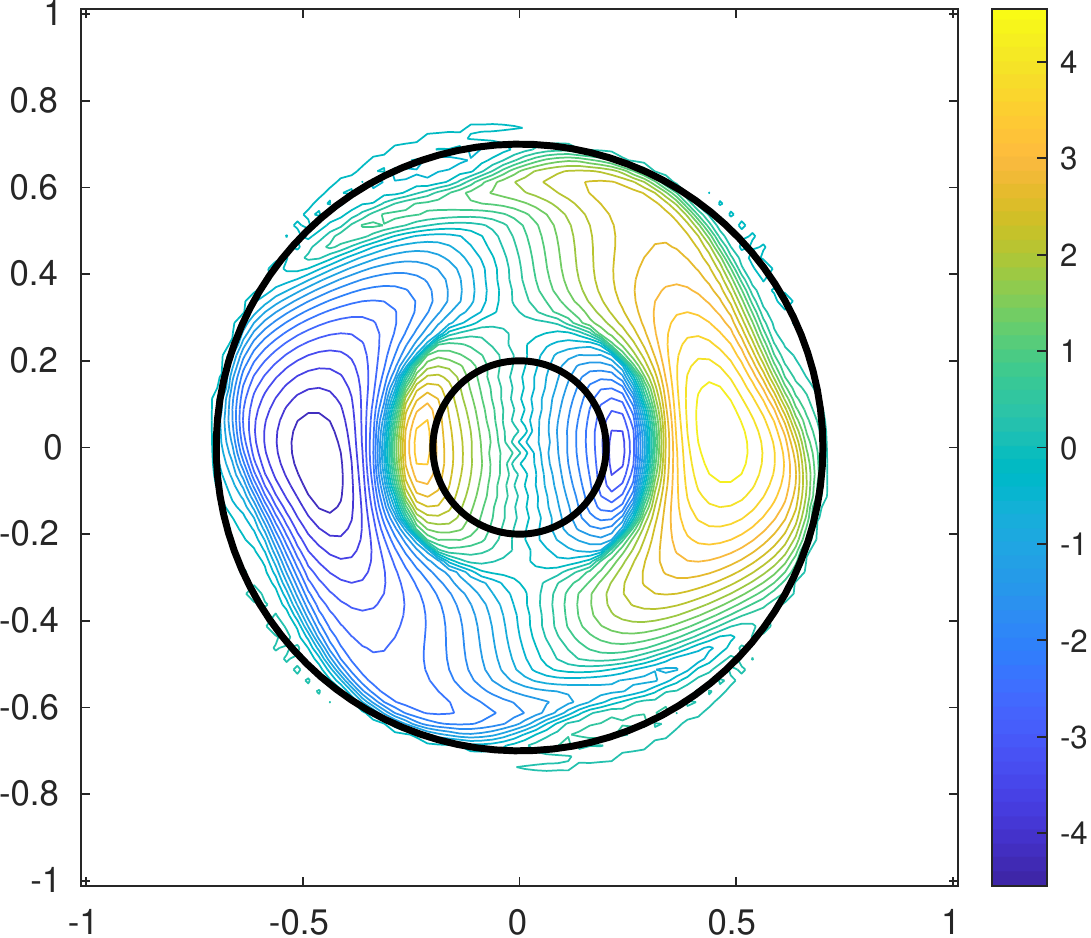}
		\caption{ \bf $u_2, \varepsilon = 10^{-1}$}
	\end{subfigure}	
	\begin{subfigure}{0.2\textwidth}
	\includegraphics[width=\textwidth]{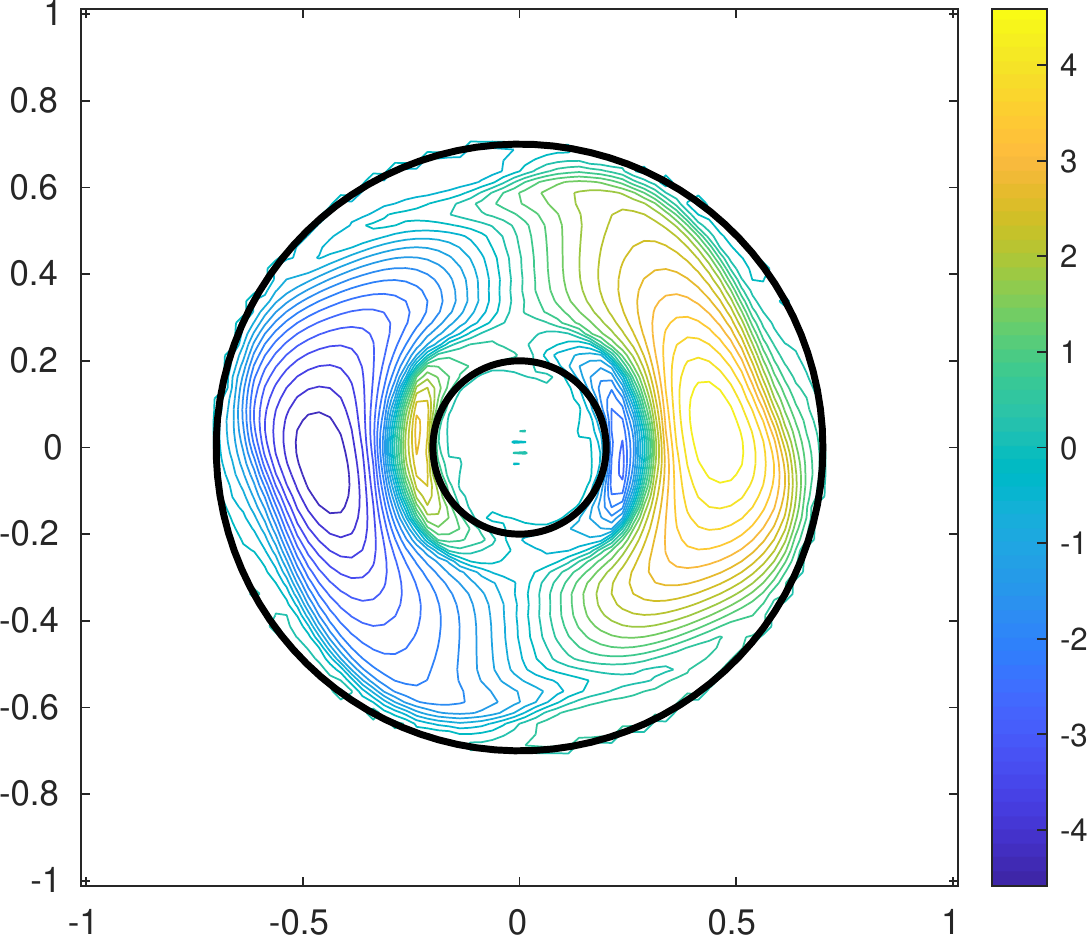}
	\caption{ \bf $u_2, \varepsilon = 10^{-2}$}
	\end{subfigure}
	\begin{subfigure}{0.2\textwidth}
	\includegraphics[width=\textwidth]{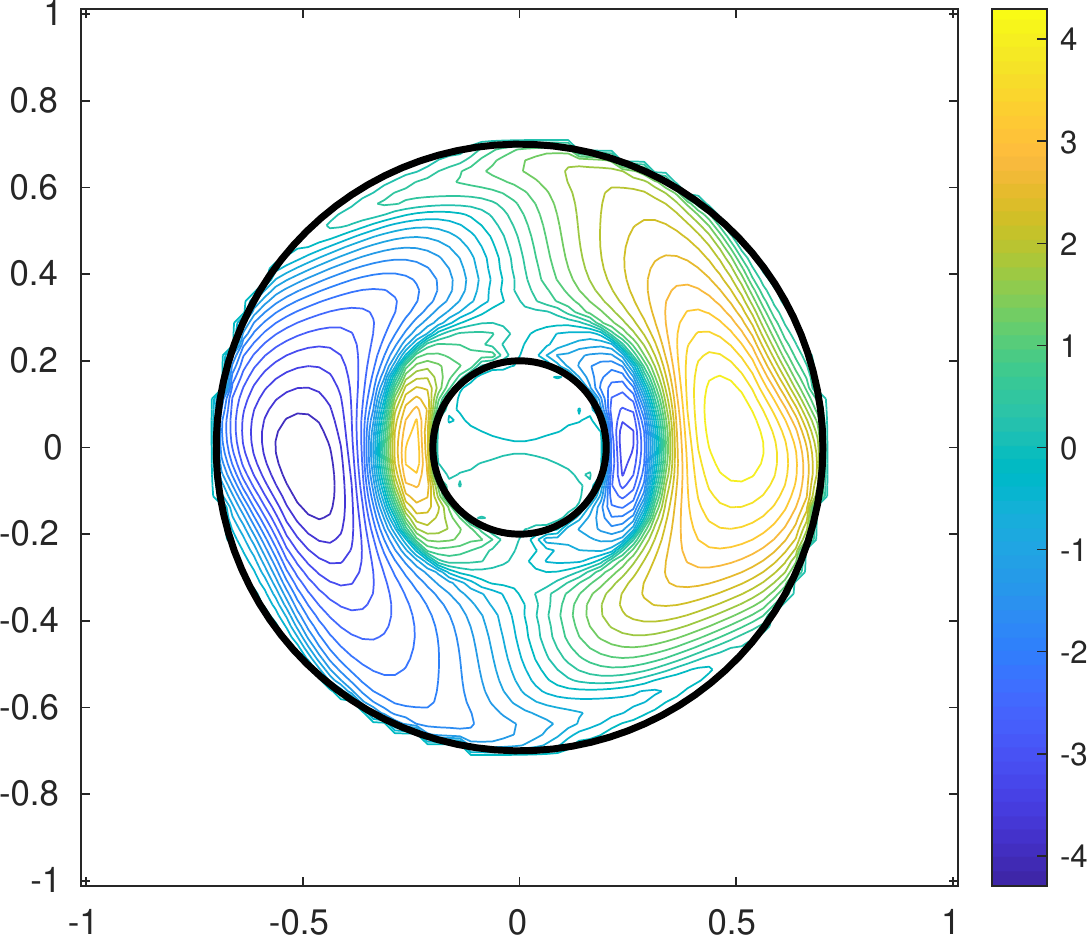}
	\caption{ \bf $u_2, \varepsilon = 10^{-3}$}
	\end{subfigure}
	\begin{subfigure}{0.2\textwidth}
	\includegraphics[width=\textwidth]{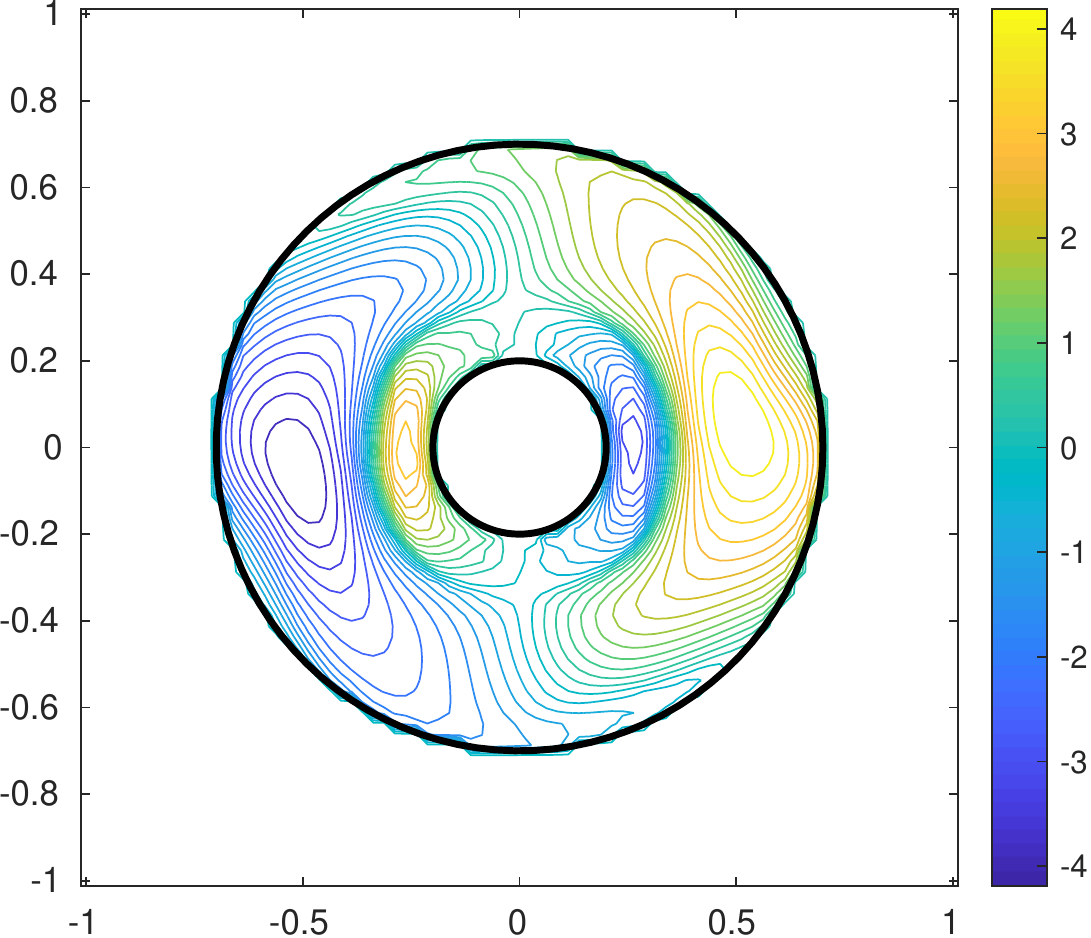}
	\caption{ \bf $u_2, \varepsilon = 10^{-4}$}
	\end{subfigure}\\
	\begin{subfigure}{0.2\textwidth}
	\includegraphics[width=\textwidth]{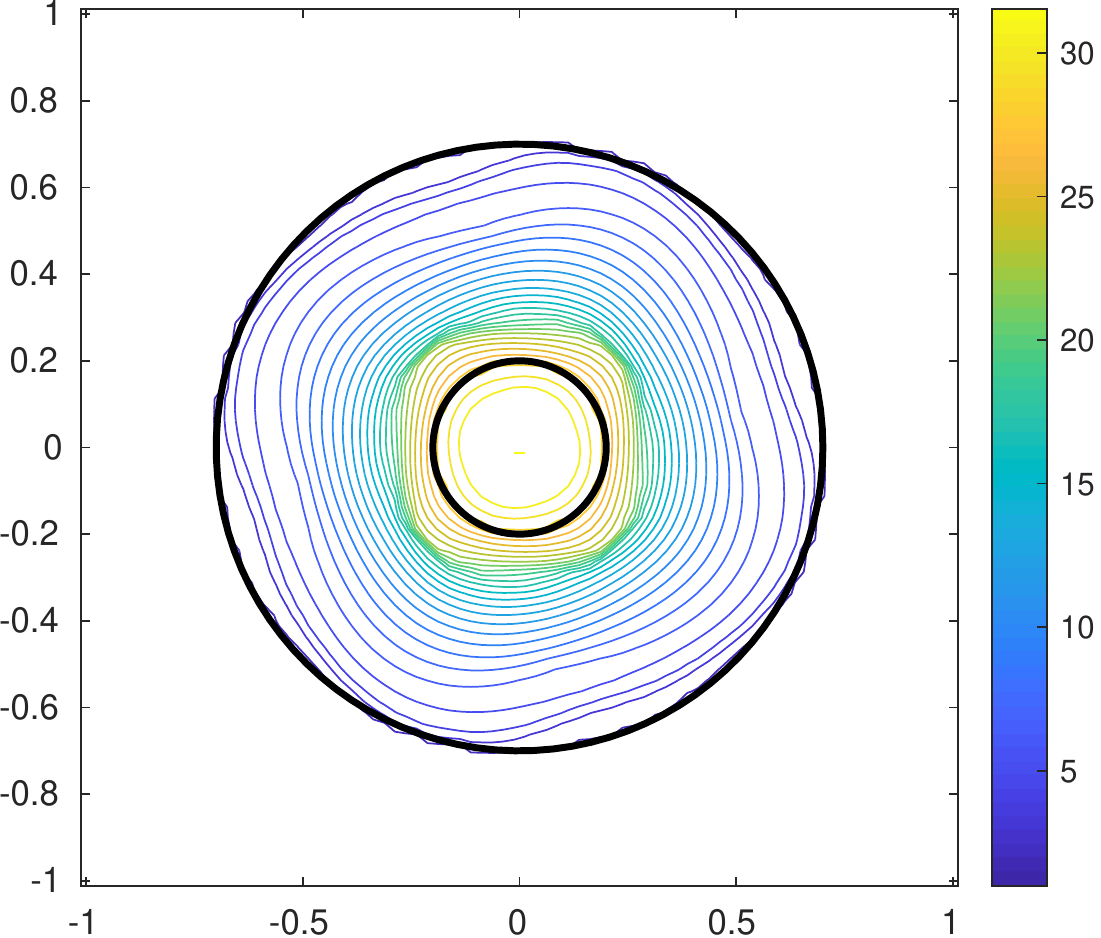}
	\caption{\bf $\vartheta, \varepsilon = 10^{-1}$}
	\end{subfigure}
	\begin{subfigure}{0.2\textwidth}
	\includegraphics[width=\textwidth]{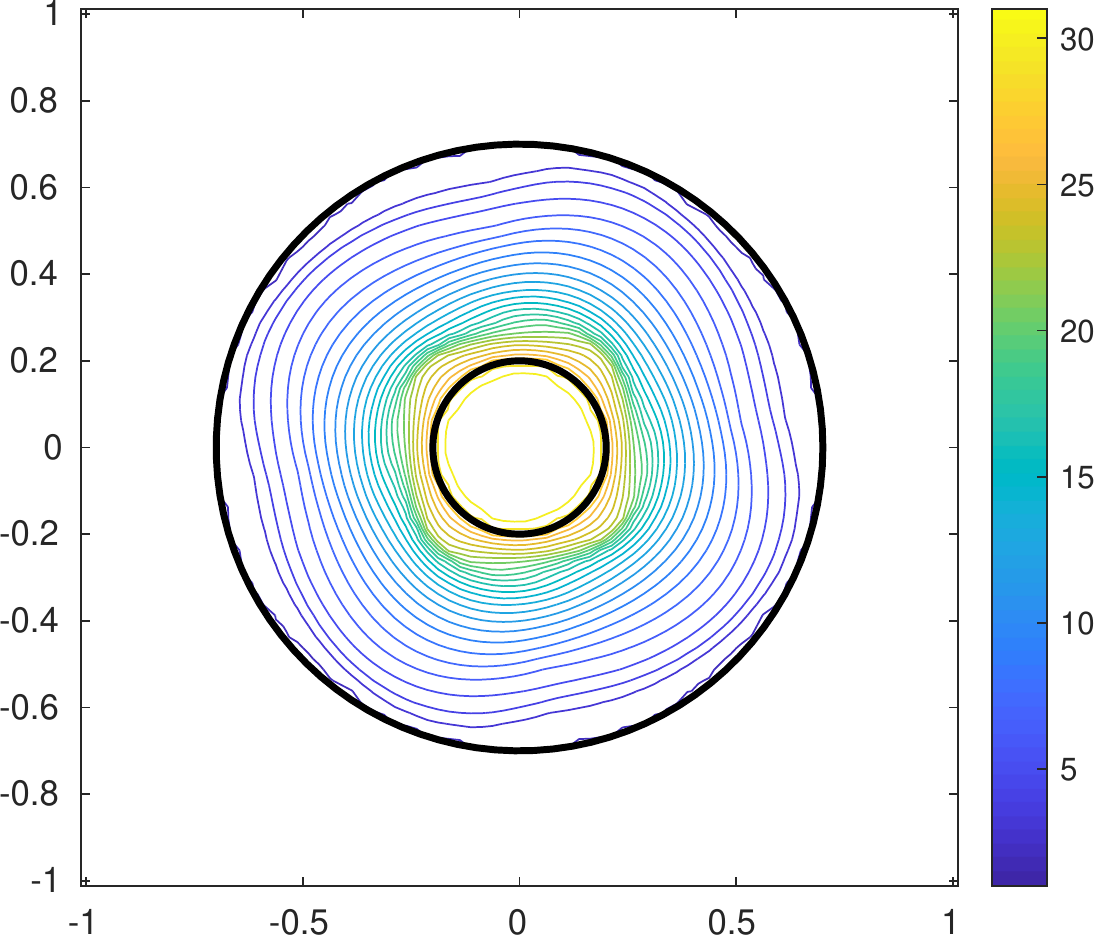}
	\caption{\bf $\vartheta, \varepsilon = 10^{-2}$}
	\end{subfigure}
	\begin{subfigure}{0.2\textwidth}
	\includegraphics[width=\textwidth]{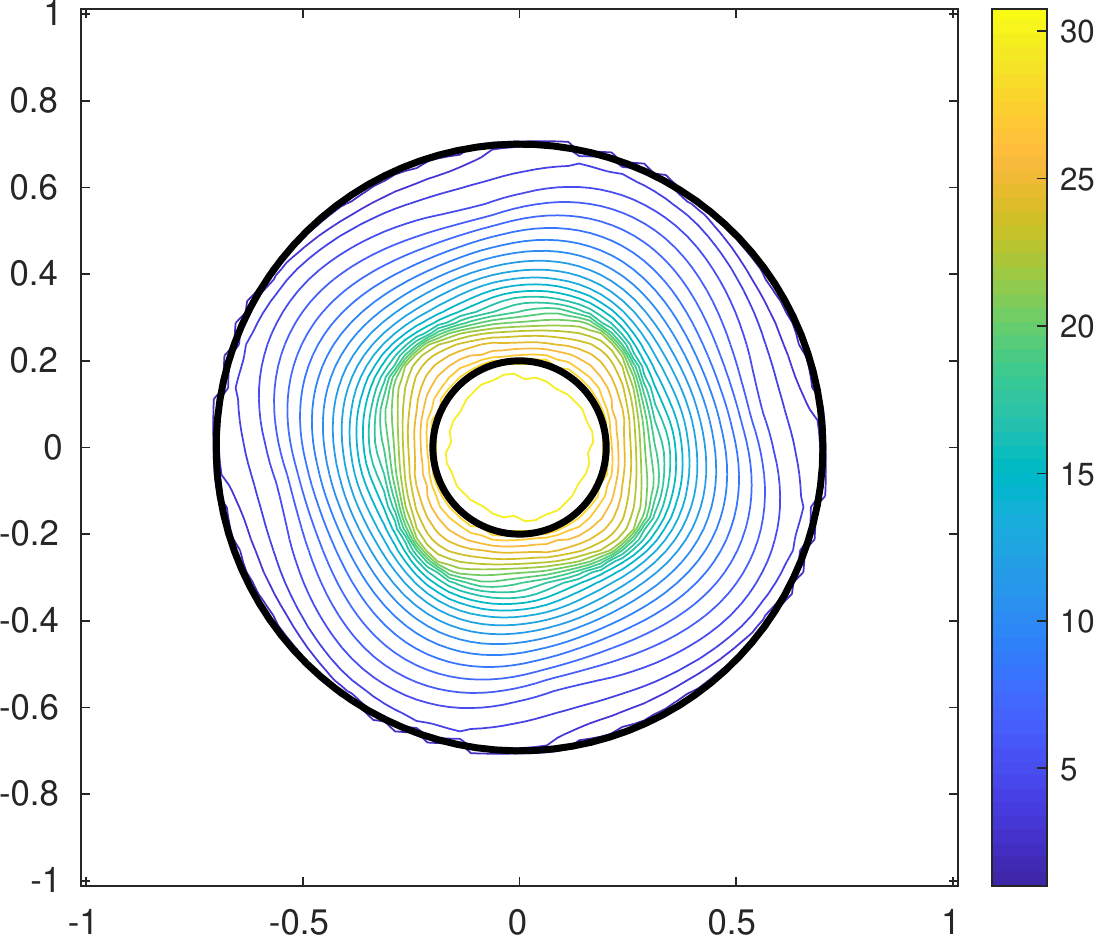}
	\caption{\bf $\vartheta, \varepsilon = 10^{-3}$}
	\end{subfigure}
	\begin{subfigure}{0.2\textwidth}
	\includegraphics[width=\textwidth]{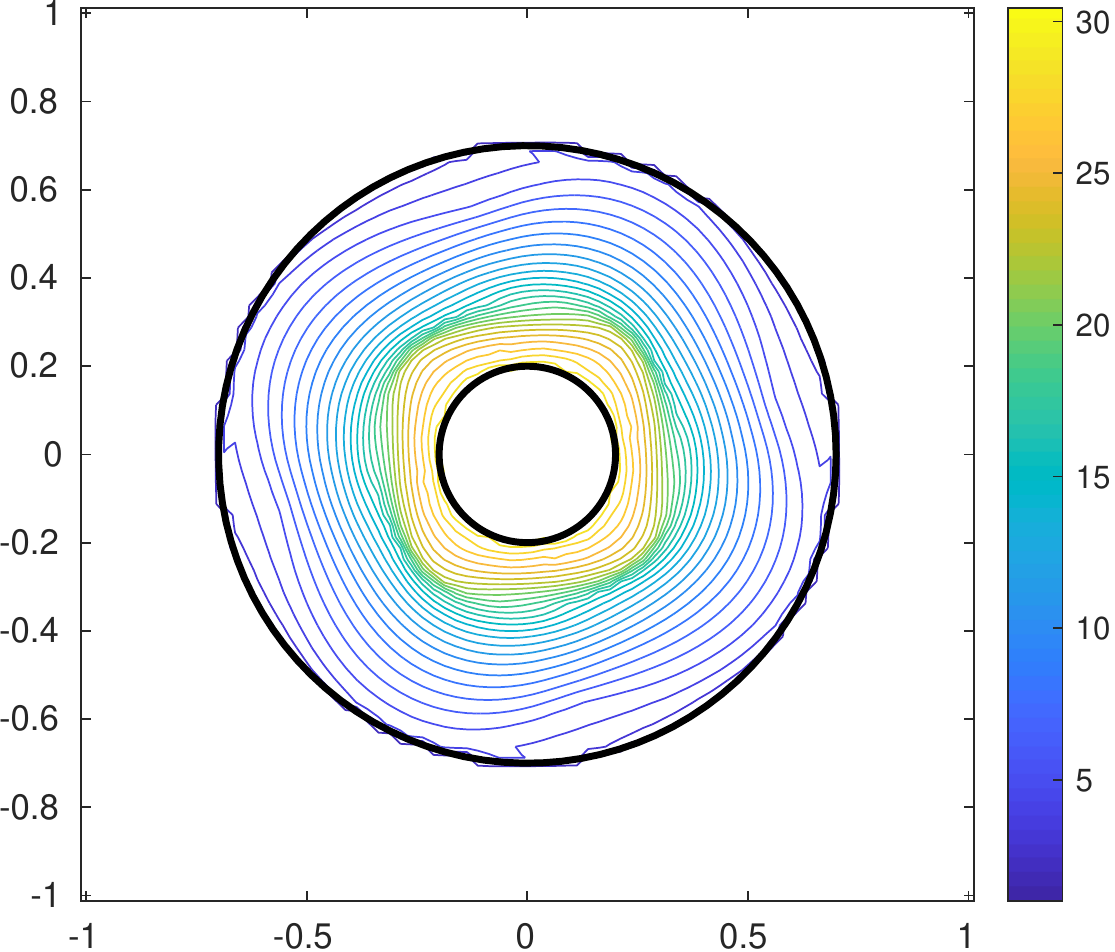}
	\caption{\bf $\vartheta, \varepsilon = 10^{-4}$}
	\end{subfigure}\\
	\begin{subfigure}{0.2\textwidth}
		\includegraphics[width=\textwidth]{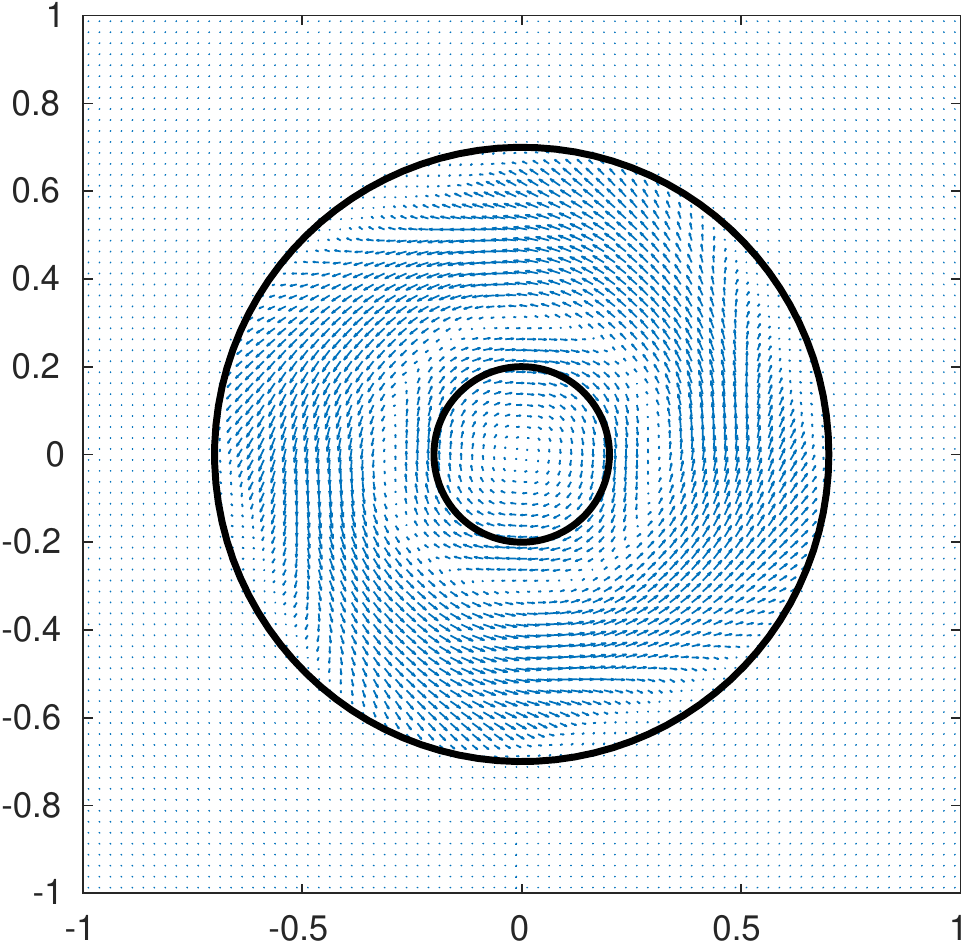}
		\caption{\bf $\textbf{u}, \varepsilon = 10^{-1}$}
	\end{subfigure}
	\begin{subfigure}{0.2\textwidth}
		\includegraphics[width=\textwidth]{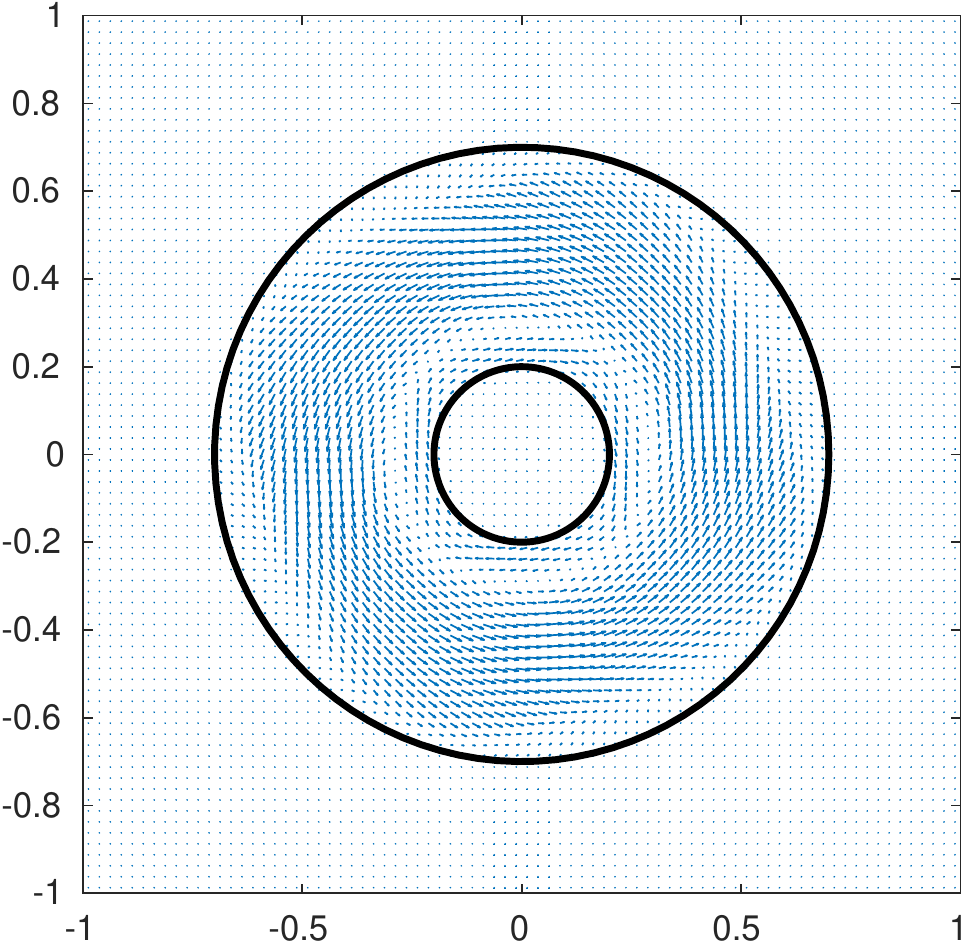}
		\caption{\bf $\textbf{u}, \varepsilon = 10^{-2}$}
	\end{subfigure}	
	\begin{subfigure}{0.2\textwidth}
		\includegraphics[width=\textwidth]{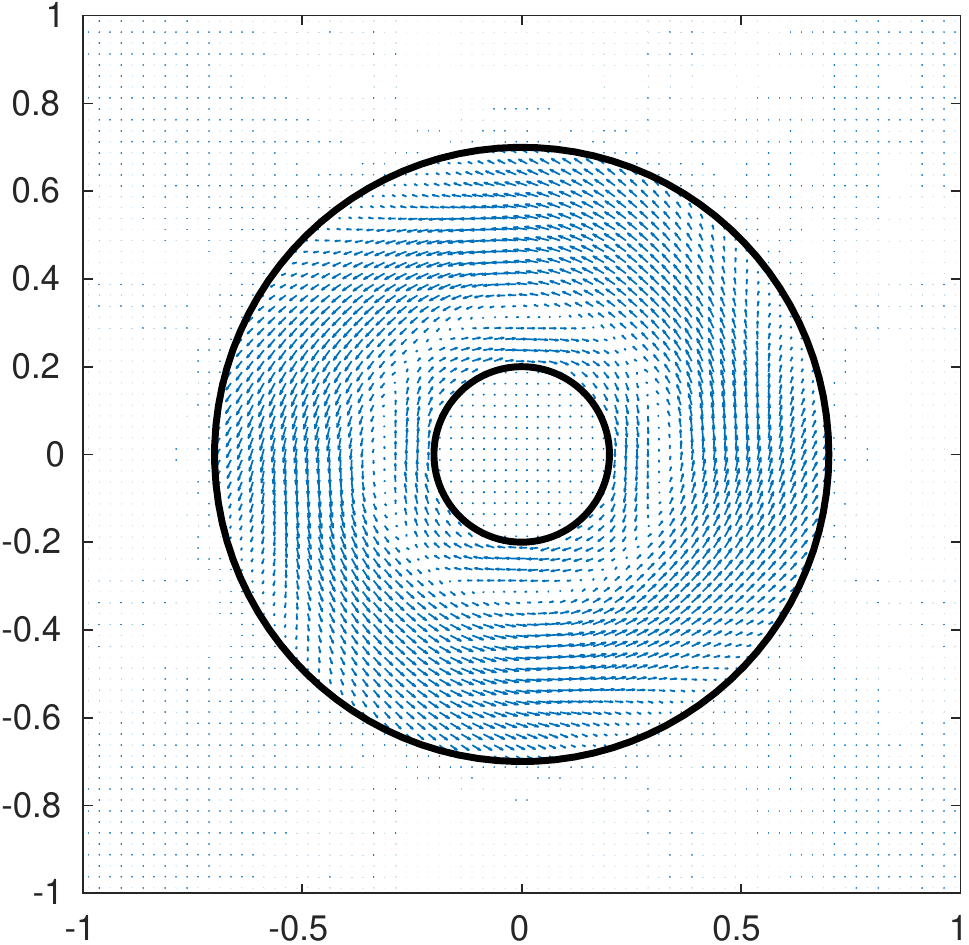}
		\caption{\bf $\textbf{u}, \varepsilon = 10^{-3}$}
	\end{subfigure}		
	\begin{subfigure}{0.2\textwidth}
		\includegraphics[width=\textwidth]{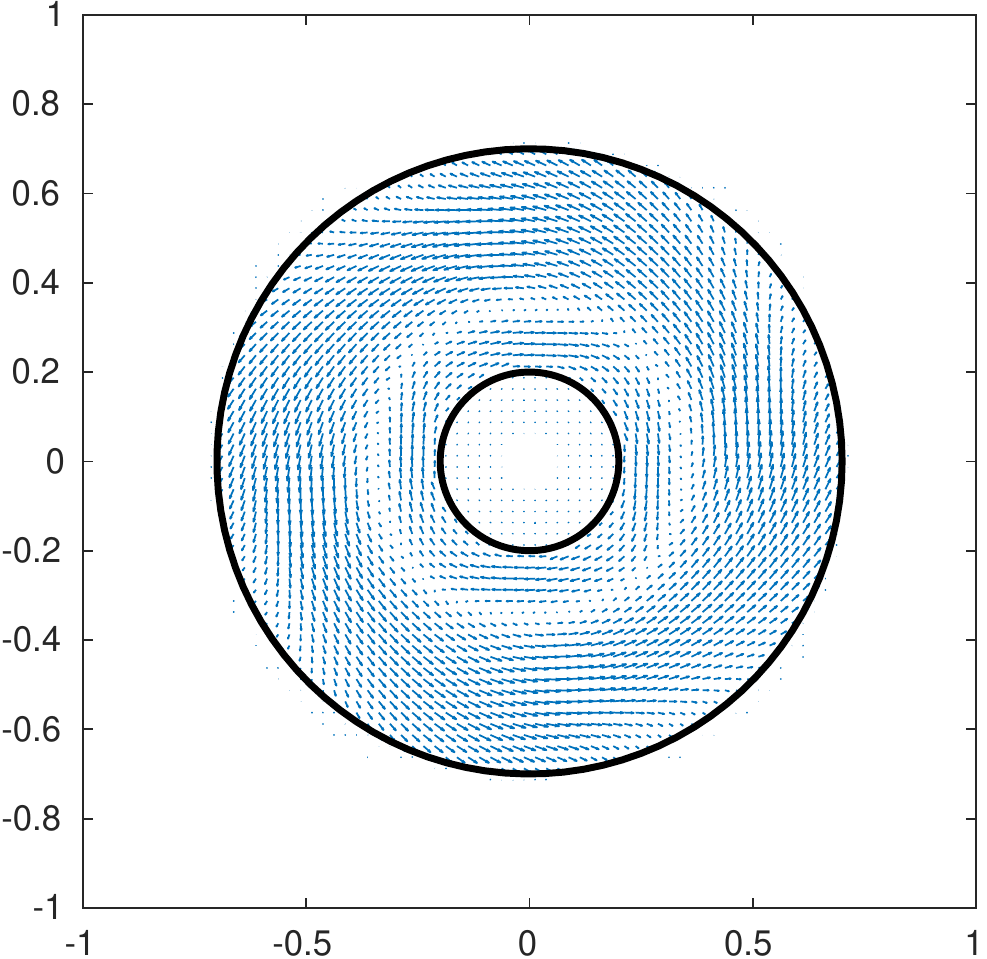}
		\caption{\bf $\textbf{u}, \varepsilon = 10^{-4}$}
	\end{subfigure}
	\caption{{\small Experiment~4: ${U}_h^{\varepsilon}$ with $h = 2/80$ and $\varepsilon = 10^{-1},\dots, 10^{-4}$.}}\label{fig:ex4}
\end{figure}

In summary, we have demonstrated experimentally that the penalization method \eqref{i5}-\eqref{i7} is robust and efficient. Penalized numerical solutions $(\rho_{h}^\varepsilon, \bm{u}_{h}^\varepsilon, \vt_{h}^\varepsilon)_{h\searrow 0, \varepsilon \searrow 0}$ converge to an exact  solution $(\vr, \bm{u}, \vt)$ of the Dirichlet boundary problem. We have tested experimentally the strong convergence with respect to $L^1(\mathbb{T}^2)$-norm. In future our goal will be to extend theoretical analysis presented in this paper to the FV method \eqref{eq:FVmethod} and prove  rigorously its convergence with respect to  both parameters,  the discretization parameter $h$ as well as the penalization parameter $\varepsilon$.

\section{Conclusion}

In the present paper we have studied convergence of a penalization method for the Navier--Stokes--Fourier system with the Dirichlet boundary conditions.  The physical fluid domain is embedded into a large cube on which
the periodic boundary conditions are imposed. The penalty terms act as the friction term in the momentum and the sink/source term in the internal energy balance, respectively. We have discussed two cases with zero and non-zero density outside of the physical fluid domain. In Theorem~\ref{MT1} we have shown that the penalized solutions converge to a weak solution of the Dirichlet boundary problem. A key ingredient of the convergence analysis is the use of the ballistic energy inequality \eqref{U3} as a source of uniform bounds, and the limiting process discussed in Section~\ref{ballistic}. Theoretical analysis is accompanied with  numerical simulations illustrating main ideas and efficiency of the proposed penalization method in order to approximate an exact solution of the Dirichlet problem for the Navier--Stokes--Fourier system on complex domains, cf.~Section~\ref{numerics}.

\section*{Acknowledgment}
We would like to thank Dr. Bangwei She (Prague/Beijing) for fruitful discussions on the numerical method and numerical results.%exchange of his computational code.

%%\bibliography{citace}

%%\bibliographystyle{plain}

\end{document}